\documentclass[11pt,letterpaper]{amsart}

\usepackage{euscript,amsfonts,amssymb,amsmath,amscd,amsthm,enumerate,hyperref,mathtools}
\usepackage{tikz}
\usetikzlibrary{graphs,patterns,decorations.markings,arrows,matrix}
\usetikzlibrary{calc,decorations.pathmorphing,decorations.pathreplacing,shapes}

\usepackage{comment}
\usepackage{enumerate,enumitem}

\usepackage{graphicx}

\usepackage{color}

\usepackage[margin=1in]{geometry}

\usepackage{bbm}
\usepackage{bm}

\newtheorem{theorem}{Theorem}[section]
\newtheorem*{theorem*}{Theorem}
\newtheorem{lemma}[theorem]{Lemma}
\newtheorem{lem}[theorem]{Lemma}
\newtheorem{definition}[theorem]{Definition}
\newtheorem{defn}[theorem]{Definition}
\newtheorem{proposition}[theorem]{Proposition}
\newtheorem{prop}[theorem]{Proposition}
\newtheorem{conjecture}[theorem]{Conjecture}

\newtheorem{corollary}[theorem]{Corollary}

\theoremstyle{remark}
\newtheorem{remark}[theorem]{Remark}

\theoremstyle{remark}
\newtheorem{rmk}[theorem]{Remark}

\newcommand{\eps}{\varepsilon}
\newcommand{\E}{\mathbb E}

\newcommand{\ii}{\mathbf i}

\newcommand{\ges}{{\scriptscriptstyle \geqslant}}
\newcommand{\Rho}{\mathcal{P}}
\newcommand{\dd}{\mathrm d}

\renewcommand{\H}{\mathcal H}

\newcommand{\Rd}{\mathcal R^d}
\newcommand{\In}{\mathcal I}

\newcommand{\A}{\mathbf A}
\newcommand{\B}{\mathbf B}
\newcommand{\C}{\mathbf C}
\newcommand{\D}{\mathbf D}

\renewcommand{\tikz}[2]{
    \begin{tikzpicture}[scale=#1,baseline=(current bounding box.center),>=stealth]
    #2
    \end{tikzpicture}}

\colorlet{lgray}{white!85!black}
\colorlet{lred}{white!85!red}
\colorlet{lgreen}{white!80!green}
\colorlet{dgreen}{black!30!green}
\definecolor{green}{rgb}{0.1,0.8,0.1}
\definecolor{yellow}{rgb}{1.0,0.85,0.25}

\def\wh{\widehat}

\def\ie{{\it i.e.}\/,}

\newcommand\fixin{\stackrel{\mathclap{\normalfont\mbox{f}}}{\in}}
\newcommand\sumin{\stackrel{\mathclap{\normalfont\mbox{s}}}{\in}}
\def\b{\bar}

\def\I{\bm{I}}

\newcommand{\Pp}{\mathbf P}

\newcommand{\As}[2]{\A_{[#1,#2]}}
\def\e{\bm{e}}

\numberwithin{equation}{section}

\title{Shift-invariance for vertex models and polymers}

\author{Alexei Borodin \and Vadim Gorin \and Michael Wheeler}

\begin{document}

\maketitle

\begin{abstract}

We establish a symmetry in a variety of integrable stochastic systems: Certain
multi-point distributions of natural observables are unchanged under a
shift of a subset of observation points. The property holds for stochastic vertex models,
(1+1)d directed polymers in random media, last passage percolation, the
Kardar-Parisi-Zhang equation, and the Airy sheet. In each instance it leads to computations of
previously inaccessible joint distributions. The proofs rely on a combination of the Yang-Baxter integrability of the inhomogeneous colored stochastic six-vertex model and Lagrange interpolation. We also show that a simplified (Gaussian) version of our theorems is related to the invariance in law of the local time of the Brownian bridge under the shift of the observation level.

\end{abstract}

%\textcolor{green}{[TODO: try to upgrade Beta-polymer statement to more general domains. We need to figure out, what is happening inside the domain when we make the boundary limit transition necessary for the step initial condition. Behavior at the intersection of two lines, where parameters are taken to infinity?]}

\tableofcontents

\section{Introduction}

\subsection{Preface}

This work is about a simple-looking property of a variety of integrable probabilistic systems that includes stochastic vertex models, (1+1)d directed polymers in random media and last passage percolation with specific weights, as well as universal objects of the Kardar-Parisi-Zhang universality class -- the KPZ equation and the Airy sheet. The property says that joint distributions of certain multi-dimensional observables in the system are unchanged under
a shift of a subset of observation points.

It can be thought of as a far reaching generalization of the following known feature of the Brownian bridge: Fix $a<b$ and let $B(t)$, $0\le t\le 1$, be a Brownian bridge such that $B(0)=a$ and $B(1)=b$. Let $\mathcal L_c$ denote the local time that $B(t)$ spends at level $c$. Then as long as $a\le c \le b$, the distribution of $\mathcal L_c$ does not depend on the choice of $c$.

While the above property of the invariance of Brownian local times under the shifts of $c$ admits a bijective proof, we have not been able to find anything similar for the more complicated systems that we deal with. Instead, our proofs rely on much more advanced machinery of Yang-Baxter integrable vertex models.

Beyond intrinsic interest, the shift-invariance property yields explicit formulas for certain multi-dimensional distributions that were not accessible before. The basic idea is that shifts sometimes allow one to reduce complicated configurations of observation points to simpler ones, for which exact expressions are already known.

\bigskip

Let us present an example that involves Brownian directed last passage percolation.

The Brownian last passage time $\mathfrak Z_{(n',t')\to (n,t)}$ is a random function of four arguments $n,n'\in\mathbb Z$, $t,t'\in \mathbb R$. Let $\{B_n(t)\}_{n\in\mathbb Z}$ be a collection of independent standard Brownian motions on the real line.  For any $n\ge n'$, $t\ge t'$, we define the  passage time as the maximum of the increments of the Brownian motions over monotone grid paths between $(n',t')$ and $(n,t)$:
\begin{equation}
      \label{eq_Brownian_passage_time_intro}
        \mathfrak Z_{(n',t')\to (n,t)}=\max_{t'=t_0< t_1< \dots< t_{n-n'+1}=t} \left[\sum_{i=0}^{n-n'} \bigl(B_{i+n'}(t_{i+1})-B_{i+n'}(t_{i})\bigr) \right] .
\end{equation}
The law of $\mathfrak Z_{(n',t')\to (n,t)}$ is a well-studied object. Kuperberg \cite{Kuperberg} and Baryshnikov \cite{Bary} showed that for fixed $n,t,n',t'$ the one-dimensional distribution coincides with that of the largest eigenvalue of a random Hermitian matrix from the Gaussian Unitary Ensemble (GUE). In particular, this implied that the asymptotic behavior of $\mathfrak Z_{(n',t')\to (n,t)}$ as the points $(n,t)$ and $(n',t')$ move away from each other is governed by the celebrated Tracy--Widom distribution. Numerous subsequent works yielded a fairly complete description of the joint distribution of $\mathfrak Z_{(n',t')\to(\cdot,\cdot)}$ and its asymptotic behavior when the \emph{starting point is fixed} and the ending one varies, and we refer to Johansson-Rahman \cite{JR} for the most recent results.

Here is the simplest open question: Given a pair of points $A$ and $B$ in $\mathbb Z\times \mathbb R$ and a similar pair of points $C$ and $D$, what is the joint law of the passage times $(\mathfrak Z_{A\to B}, \mathfrak Z_{C\to D})$? See the left panel in Figure \ref{Fig_Brownian_polymer} for an illustration. Note that if the lattice rectangles with opposite vertices $A$, $B$ and with opposite vertices $C$, $D$ are disjoint, then the random variables $\mathfrak Z_{A\to B}$ and $\mathfrak Z_{C\to D}$ are independent and the question is trivial. We are able to provide an answer in an antipodal situation, when any monotone path between $A$ and $B$ and any monotone path between $C$ and $D$ must intersect. The key is the following property of shift-invariance.

 \begin{figure}[t]
\begin{center}
{\scalebox{0.75}{\includegraphics{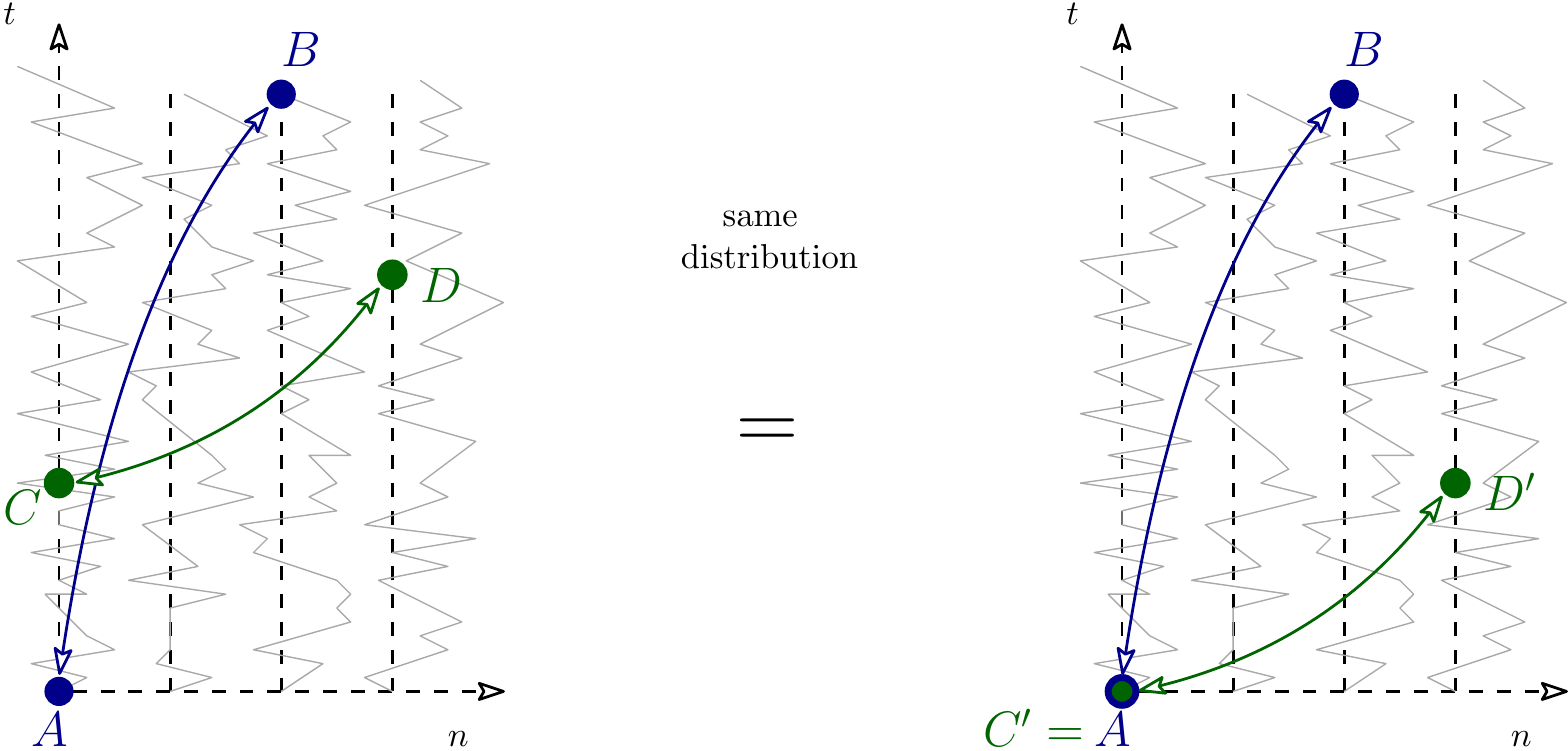}}}
 \caption{Brownian LPP: We have $(\mathfrak Z_{A\to B}, \mathfrak Z_{C\to D})\stackrel{d}{=} (\mathfrak Z_{A\to B}, \mathfrak Z_{C'\to D'})$.
 \label{Fig_Brownian_polymer}}
\end{center}
\end{figure}

\begin{theorem} \label{Theorem_Brownian_intro}
 Let $A,B,C,D$ be as above, take $\Delta\in\mathbb R$, and set $C'=C+(0,\Delta)$, $D'=D+(0,\Delta)$. Suppose that the following holds:
 \begin{itemize}
   \item The $n$-coordinates of points $A$, $C$, and $C'$ are the same, while the $t$--coordinate of $A$ is not larger than those of $C$ and $C'$.
   \item $B \succeq D$ and $B \succeq D'$, where $(n,t) \succeq (n',t')$ means $n\le n'$, $t\ge t'$.
 \end{itemize}
 Then we have a distributional identity $(\mathfrak Z_{A\to B}, \mathfrak Z_{C\to D})\stackrel{d}{=} (\mathfrak Z_{A\to B}, \mathfrak Z_{C'\to D'})$, cf.~ Figure \ref{Fig_Brownian_polymer}.
\end{theorem}
Theorem \ref{Theorem_Brownian_intro} can be extended to more than two pairs of points and shifts in $n$-directions, see Section \ref{Section_BLPP} for details.
Note that we can now choose $C'=A$. Then the joint law of the vector $(\mathfrak Z_{A\to B}, \mathfrak Z_{A\to D'})$ is explicitly known (as the starting points are the same); it can also be related to the joint distribution for the largest eigenvalues of a GUE random matrix and its symmetric submatrices evolving according to the Dyson Brownian Motion.

\begin{figure}[t]
    \begin{center}
        {\scalebox{0.75}{\includegraphics{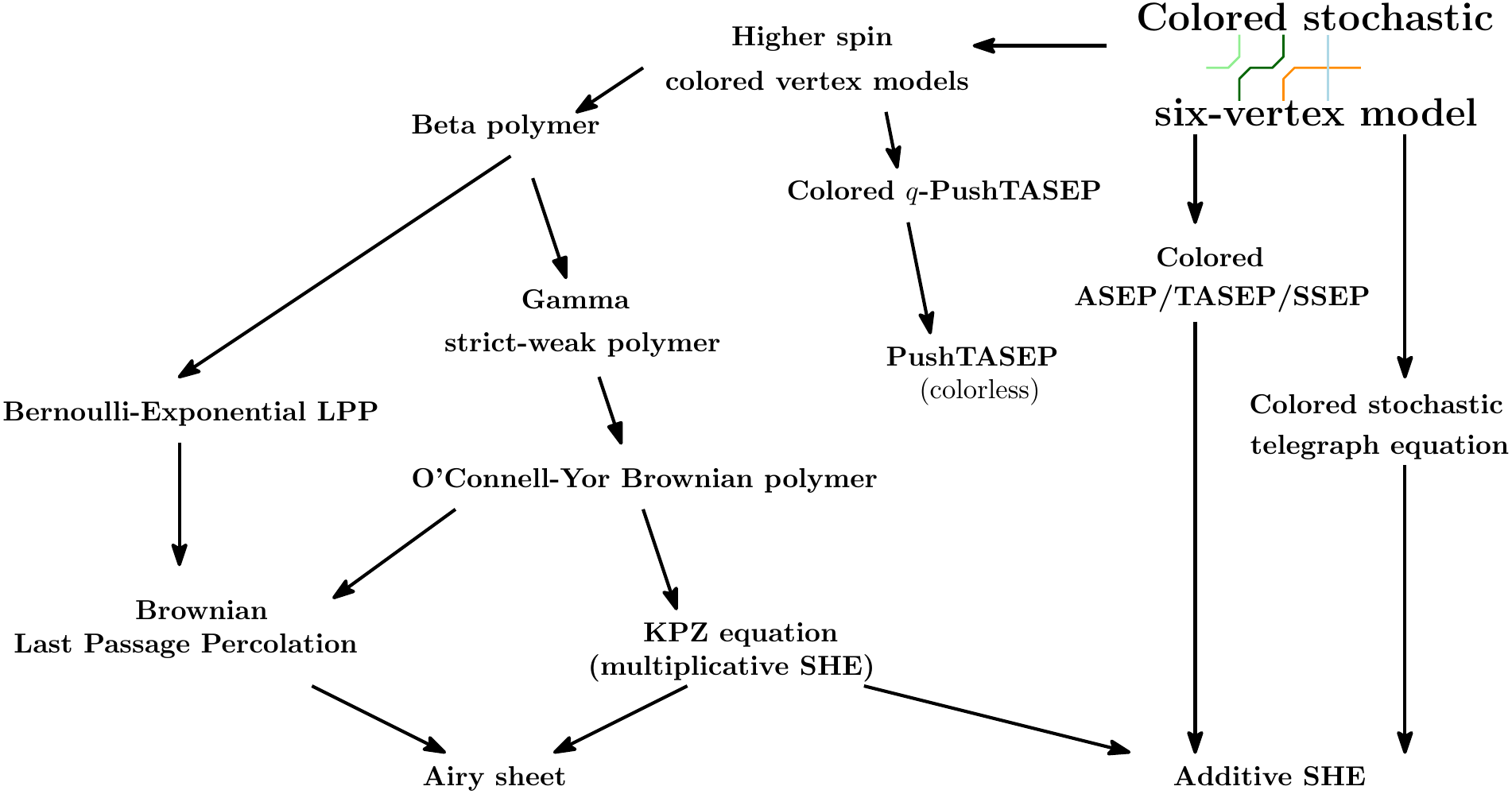}}}
        \caption{A chart of stochastic systems for which we can prove the shift--invariance.
            \label{Fig_Chart}}
    \end{center}
\end{figure}

The shift-invariance of Theorem \ref{Theorem_Brownian_intro} is not restricted to the Brownian last passage percolation. Far from it, the invariance extends to a whole class of related stochastic systems (all with specific ``integrable'' weights, though). The chart in Figure \ref{Fig_Chart} shows most of these systems and relations between them. Whenever two systems are linked by an arrow, we are able to deduce shift-invariance for one system from a similar statement for the other one. Exact mechanisms of deduction vary: We use stochastic fusion, analytic continuation, deterministic limits, and functional central limit theorems, as discussed in Sections \ref{Section_higher_spin} and \ref{Section_polymers}.

Our central result is a \emph{master theorem} that we prove for \emph{colored stochastic vertex models}, and we give its (simplified) formulation below. Not all of the implied statements are discussed in this paper as it is already quite long. But we do consider the implications for polymer models, the KPZ equation, and the Airy sheet; we will offer a few details about those in a later part of the introduction. A conjectural generalization of the proved shift-invariance will also be mentioned.

\subsection{Shift-invariance for colored stochastic vertex models}
\label{Section_intro_6v}

We start by describing a Markovian recipe to construct random colored up-right paths in the positive quadrant $\mathbb Z_{\ge 1}\times\mathbb Z_{\ge 1}$ with the colors labeled by natural numbers. In this section we assume that no lattice edge can be occupied by more than one path, although that restriction will be removed later.

The model depends on a quantization parameter $q\in (0,1)$ and real column and row \emph{rapidities} denoted by $u_1,u_2,\dots$ and $v_1,v_2,\dots$, respectively. The rapidities are assigned to the corresponding rows and columns, and we assume that $v_y\ge u_x\ge 0$ for all $x,y\ge 1$.

Along the boundary of the quadrant, we demand that no paths enter the quadrant from the bottom. On the other hand, a single path of color $i$ enters the quadrant from the left in row $i$ for each $i\ge 1$. Once the paths are specified along the boundary, they progress in the up-right direction within the quadrant using certain interaction probabilities, also known as \emph{vertex weights}.

For each vertex of the lattice, once we know the colors of the entering paths along the bottom and left adjacent edges, we decide on the colors of the exiting paths along the top and right edges according to those probabilities. They are given by the table of Figure \ref{Fig_colored_weights}, where $0\le i<j$ denote the colors of paths on the corresponding edges, and color $0$ encodes the absence of paths.

These vertex weights represent a stochastic version of the $R$-matrix for the quantum affine algebra $U_q(\widehat{\mathfrak{sl_{n+1}}})$ that goes back to mid-1980s works of Bazhanov \cite{Bazhanov}, Faddeev-Reshtikhin-Takhtadjan \cite{FRT}, and Jimbo \cite{Jimbo1,Jimbo2}. The stochastic version was introducted quite recently by Kuniba-Mangazeev-Maruyama-Okado in \cite{KMMO}, see also Bosnjak-Mangazeev \cite{BM}, Aggarwal-Borodin-Bufetov \cite{ABB}.\footnote{In the rank-1 or \emph{colorless} case, these objects are much older; the six-vertex model was introduced by Pauling in 1935 \cite{Pauling}, and its stochastic version was first considered by Gwa-Spohn in 1992 \cite{GwaSpohn}.}

A less graphical way of describing the same random ensemble of colored paths is by introducing \emph{colored height functions}. For each point $(X,Y)\in (\mathbb Z_{\ge 0}+\frac 12) \times (\mathbb Z_{\ge 0}+\frac 12)$ of the dual lattice and any $k\ge 1$, we define $\mathcal H^{\ges k}(X,Y)$ as the number of paths of colors $\geqslant k$ that pass below $(X,Y)$. See Figure \ref{Fig_colored_6v} for an illustration.

\begin{figure}[t]
    \begin{center}
        {\scalebox{0.9}{\includegraphics{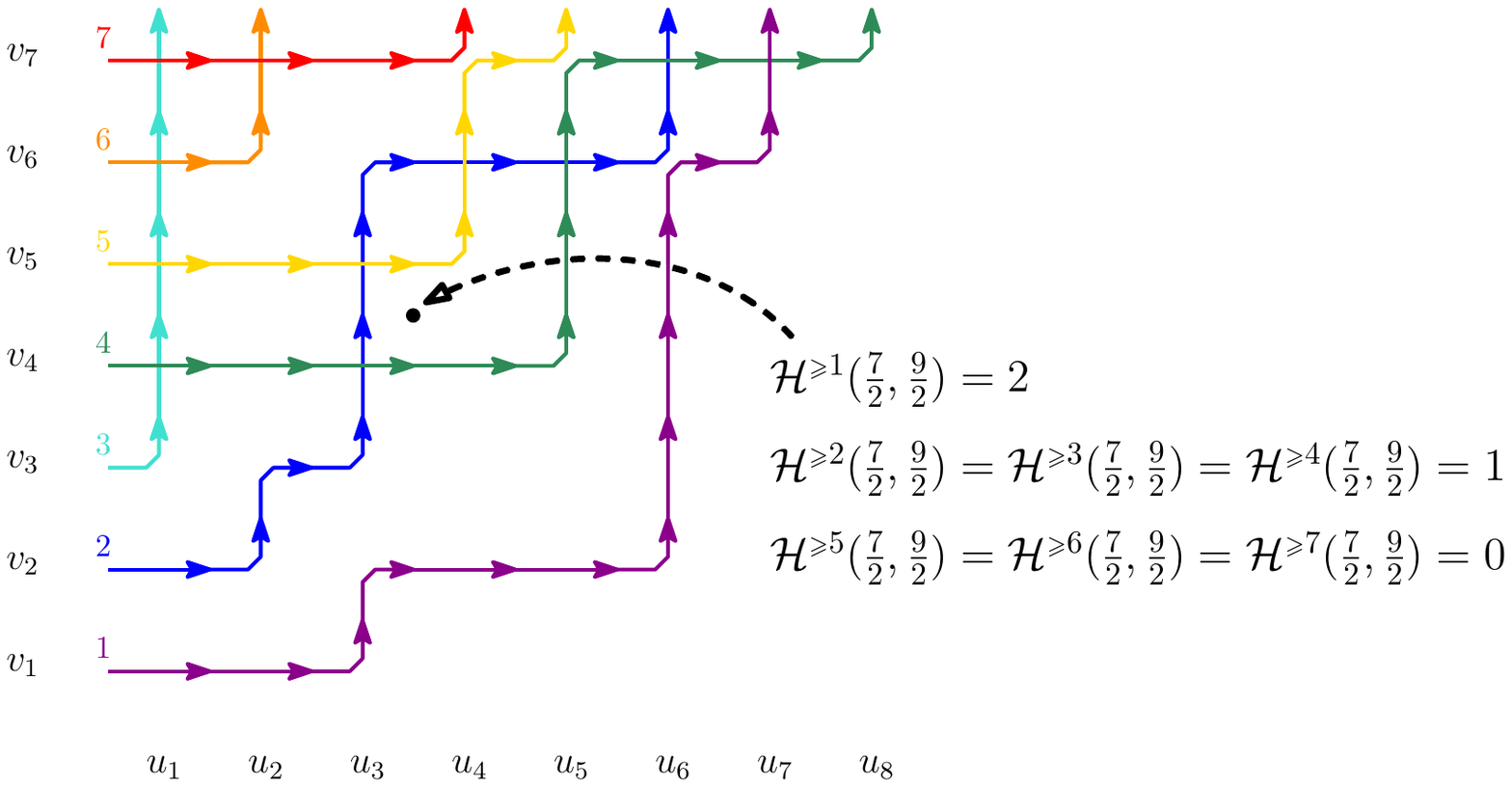}}
        \caption{A possible configuration of the colored vertex model.}
            \label{Fig_colored_6v}}
    \end{center}
\end{figure}

\begin{figure}[t]
    \begin{center}
        {\scalebox{0.7}{\includegraphics{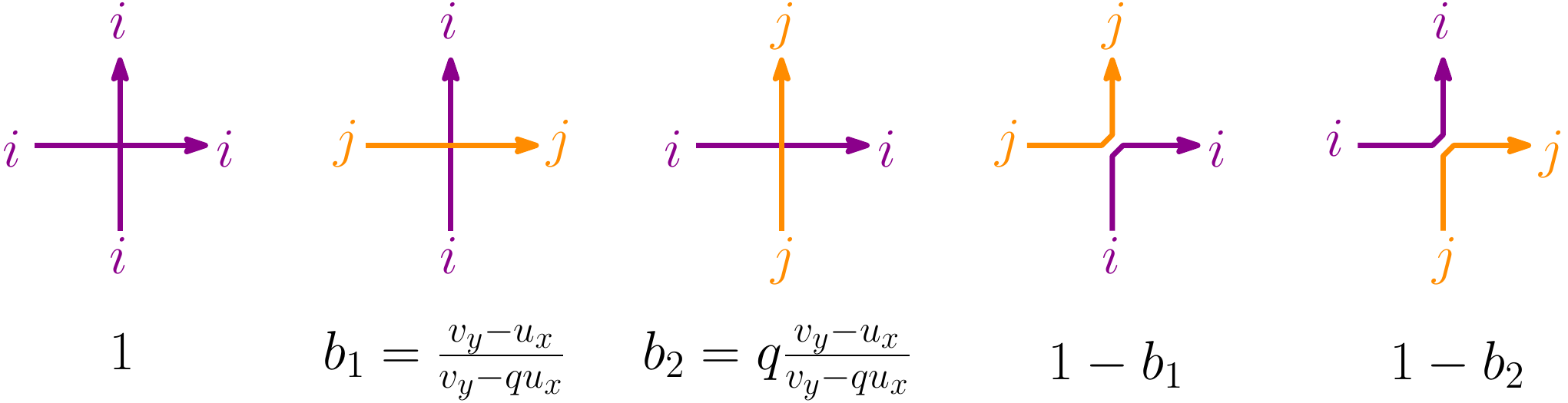}}
        \caption{When two paths of colors $i$ and $j$ with $0\le i \le j$ (color $0$ is identified with the absence of the path) meet at vertex $(x,y)$, they continue according to these weights.}
            \label{Fig_colored_weights}}
    \end{center}
\end{figure}

Here is what the simplest instance of shift-invariance for colored stochastic vertex models looks like.

As in Theorem \ref{Theorem_Brownian_intro}, for two points $\mathcal U=(x^{\mathcal U},y^{\mathcal U})$, $\mathcal V=(x^{\mathcal V},y^{\mathcal V})$ in the quadrant, we write $\mathcal U\succeq \mathcal V$ if $x^{\mathcal U}\le x^{\mathcal V}$ and $y^{\mathcal U}\ge y^{\mathcal V}$. In other words, $\mathcal V$ is in the down--right direction from $\mathcal U$.

\begin{theorem} \label{Theorem_6v_intro}
    In the above setting of the colored stochastic model in the quadrant, choose an index $\iota\ge 1$, color cutoff levels $k_1\dots,k_n\ge 1$, and a collection of points  $\{\mathcal U_i\}_{i=1}^n$. Set
    $$
    k_j'=\begin{cases} k_j,& j\ne \iota,\\ k_\iota+1, & j=\iota,\end{cases} \qquad \qquad \mathcal U'_j=\begin{cases}\mathcal U_j, & j\ne \iota, \\ \mathcal U_\iota + (0,1), & j=\iota. \end{cases}
    $$
    Assume that
    $$
    0\le k_1\le k_2\le \dots\le k_n, \qquad 0\le k'_1\le k'_2\le \dots\le k'_n,
    $$
    $$\mathcal U_1,\dots,\mathcal U_{\iota-1} \succeq \mathcal U_\iota \succeq\mathcal U_{\iota+1},\dots, \mathcal U_n, \qquad \mathcal U'_1,\dots,\mathcal U'_{\iota-1} \succeq \mathcal U'_\iota \succeq \mathcal U'_{\iota+1},\dots, \mathcal U'_n.
    $$
    Then the distribution of the vector of colored height functions
    $$
    \bigl(\H^{\ges k_1}(\mathcal U_1),\, \H^{\ges k_2}(\mathcal U_2),\, \dots, \H^{\ges k_n}(\mathcal U_n)\bigr)
    $$
    coincides with the distribution of a similar vector with shifted $\iota$-th point and cutoff
    $$
    \bigl(\H^{\ges k'_1}(\mathcal U'_1),\, \H^{\ges k'_2}(\mathcal U'_2),\, \dots, \H^{\ges k'_n}(\mathcal U'_n)\bigr),
    $$
under the condition that in the vertex model used to define the second vector one swaps the row rapidities $v_{k_\iota}$ and $v_{y}$, where $y$ is the row in which $\mathcal U'_\iota$ is located.
\end{theorem}

A variant of Theorem \ref{Theorem_6v_intro} that allows for more general bottom-left boundary conditions that restore the symmetry between horizontal and vertical axes will be proved in Section \ref{sec:shift}.

Arguably, Theorem \ref{Theorem_6v_intro} is not the most intuitive statement of all, and our path to both its formulation and its proof was not straightforward.

It started with a special case, when in the \emph{homogeneous} model (all rapidities are equal) it was shown to be possible to shift a collection of points to two extreme positions -- when all $k_j$'s are equal, and when all $\mathcal U_j$'s coincide. The coincidence of these two extreme distributions was proved by Borodin-Wheeler \cite{BW}, see also Borodin-Bufetov \cite{BB_sym} for a simpler proof and more general boundaries.
The shift-invariance of Theorem \ref{Theorem_6v_intro}, however, is much more powerful than that; for example, the coincidence of the two extremes turns into a tautology in the Brownian percolation limit of Theorem \ref{Theorem_Brownian_intro}.

Next, in the Gaussian models (additive stochastic heat equation and colored stochastic telegraph equation in Figure \ref{Fig_Chart}; see Borodin-Gorin \cite{BG_tele} for a description of the telegraph limit in the colorless situation) we noticed that shift-invariance reduces, still nontrivially, to that for local times of Brownian bridges and numbers of intersections of two persistent random walks. This argument generalizes further  to proving equality of second moments of the height function vectors in Theorem \ref{Theorem_6v_intro}.  Alas, pushing beyond that in a similar fashion seemed prohibitively complicated.

This partial progress, however, turned out to be sufficient for inferring the complete statement for vertex models. Its proof that we give in the present paper is an intricate verification argument based on an induction in the size of the domain swept by the participating colored paths. The induction is powered by two essential tools -- the Yang-Baxter equation for the vertex weights of Figure \ref{Fig_colored_weights}, and Lagrange interpolation for the distributions viewed as polynomials in the rapidities. Thus, the inhomogeneity of the model was essential for our proof, as it furnishes us with sufficiently many interpolation points.

Another step for which the inhomogeneity was indispensable was extending Theorem \ref{Theorem_6v_intro} to more general vertex models obtained from the one above by \emph{fusion}. Fusion is a representation theoretic procedure that constructs more complicated solutions of the Yang-Baxter equation from simpler ones; it goes back to the work of Kulish-Reshetikhin-Sklyanin \cite{KRS} in early 1980s. We employ a stochastic version of this procedure whose detailed description can be found in Section \ref{Section_fusion} below; another exposition was earlier given by Kuan \cite{Kuan} generalizing the colorless statements by Corwin-Petrov \cite{CP}.

Fusion allows us to cluster adjacent $L\ge 1$ horizontal and $M\ge 1$ vertical lines of the lattice together, generating ``fat'' lines that can carry up to $L$, correspondingly $M$ paths. Further, resulting vertex weights in the fused model end up being rational in $q^L$ and $q^M$, allowing for analytic continuation in those quantities (they are the analogs of the \emph{spin} parameters in the rank 1 case). We remark that additional efforts are needed to ``analytically continue'' the boundary condition, see Section \ref{Section_analytic_continuation} for details. Previously, such analytic continuations in the boundary parameters were realized for the (colorless) six-vertex model  in Aggarwal-Borodin \cite{AB}, Aggarwal \cite{Ag_Duke}, and the ability to handle inhomogeneous and fused models is also essential for that. In a certain special colored situation, a similar procedure was performed in Borodin-Wheeler \cite[Section 12.3.4]{BW}.

%More care is needed to ``analytically continue'' the boundary conditions. That required a special limit transition near the boundary, which, to our luck, ended up being quite similar to a previously known procedure in the colorless case performed in Borodin-Petrov \cite{BP1,BP2}, Aggarwal-Borodin \cite{AB}.

At the end of this road we obtained a fully fused version of Theorem \ref{Theorem_6v_intro} (see Theorem \ref{Theorem_6v_invariance_infinite} below), which finally allowed us to descend into the polymer world.

\subsection{Shift-invariance for directed polymers in random media} It has been known for some time that directed random polymers in (1+1)d arise as limits of vertex models. One route to seeing this goes through the papers of Gwa-Spohn \cite{GwaSpohn}, Borodin-Corwin-Gorin \cite{BCG}, and Aggarwal \cite{Ag} which link the stochastic six-vertex model to the asymmetric simple exclusion process (ASEP). The latter has a limit to the continuous directed polymer, as shown by Bertini-Giacomin \cite{BeGi} and Amir-Corwin-Quastel \cite{ACQ}.\footnote{Combining these two steps into a single (more general) asymptotic statement is also possible, see Corwin-Tsai \cite{CT}, Corwin-Ghosal-Shen-Tsai \cite{CGST}, and also Borodin-Olshanski \cite[Section 12]{BO}.}

We follow another route, which leads to a much richer family of polymers. It is based on a recently
discovered fact (see Borodin \cite{Bor_rat}, Corwin-Petrov \cite{CP}, Borodin-Petrov \cite{BP1})
that $q$-deformed exclusion processes and Boson systems are special cases of the fully fused
colorless stochastic vertex models. This allows us to exploit numerous existing instances of
convergence of such processes to polymers. The one closest to our context is the appearance of
Beta-polymer as the $q\to 1$ limit of the $q$-Hahn particle system in the work of Barraquand-Corwin
\cite{Bar-Cor}.

The role of the colors in the $q\to 1$ limit transitions in the fully fused colored stochastic
models (first introduced in \cite{KMMO}) was not investigated before, and we are led to study it in
this paper. It turns out that the presence of colors translates into varying one of the end-points
of the polymer. This is a nontrivial statement, and it boils down to a certain degeneration of
noise that takes place -- the multi-dimensional distributions at each vertex in the first
approximation behave as one-dimensional ones, with randomness in the perpendicular directions being
of smaller order.\footnote{The authors are very grateful to Pierre Le Doussal, in discussions with
whom in May 2018 such an effect showed up for the first time.}

Different observation points $\{\mathcal U_i\}$ of the fused version of Theorem \ref{Theorem_6v_intro} provide a variation of the other end-point, and we thus gain access to information on polymers with both ends moving. Furthermore, the exponentiated colored height function $q^{\mathcal H^{\ges i}}$ tends to the partition function of the corresponding polymer, paving the way to joint distributions of those.

On this road we obtain in Section \ref{Section_polymers}  shift-invariance for the following models:
\begin{itemize}
    \item Directed polymers in Beta-random environment, first studied by Barraquand-Corwin \cite{Bar-Cor}, see Section \ref{Section_Beta};
    \item The Gamma or strict-weak polymer, first studied by Corwin-Sepp\"al\"ainen-Shen \cite{CSS} and  O'Connell-Ortmann \cite{OCO}, see Section \ref{Section_Gamma};
    \item The O'Connell--Yor semi-discrete Brownian directed polymer \cite{OCY}; Theorem \ref{Theorem_Brownian_intro} above is obtained as the infinite temperature limit of the shift-invariance for this polymer. See Section \ref{Section_Brownian_polymer} for details.
\end{itemize}
In each of these cases we prove a property that is very similar to what we described for the Brownian directed percolation above -- multi-dimensional distributions of partition functions are invariant under a shift of the two end-points of one of the polymers under certain intersection conditions.

Taking a further limit to polymers in fully continuous space-time leads to the KPZ equation discussed in the next section.

\subsection{Shift-invariance for universal objects}
All the polymer systems from the previous section are very special; we can prove shift-invariance for random directed polymers with specifically positioned independent Beta or Gamma weights, but not for any other weight distributions. In fact, at this point we do not even know whether the shift-invariance extends to (discrete) polymers with any other weights.

What we do know, however, is that the shift-invariance holds for two universal objects in the world of random polymers and interacting particle systems: the KPZ equation and the Airy sheet. Let us start from the former one.

Define a random function $\mathcal Z^{(y)}(t,x)$, $x\in\mathbb R$, $t\ge 0$, as a solution to the stochastic heat equation with multiplicative white noise and with
$\delta$-initial condition at point $y$ at $t=0$:
\begin{equation}
\label{eq_multi_SHE_intro}
 \mathcal Z^{(y)}_{t}=\tfrac12 \mathcal Z^{(y)}_{xx}+ \eta \mathcal Z^{(y)}, \quad t\ge 0, x\in\mathbb R;\qquad \mathcal Z^{(y)}(0,x)=\delta(x-y).
\end{equation}
Here $\eta$ is the space-time $2d$ white noise (the same for each $y$).
The Feynman-Kac representation for the solution to \eqref{eq_multi_SHE_intro} leads to its representation as a (regularized) integral of exponentiated noise over paths of Brownian bridges, thus clarifying the name Continuum Directed Random Polymer for $\mathcal Z^{(y)}(x,t)$, see Alberts-Khanin-Quastel \cite{AKQ1}, Quastel \cite{Q_CDM}. Computing (formally) the logarithm of $\mathcal Z$, one finds that $\mathcal H:=-\ln(\mathcal Z^{(y)})$ satisfies the KPZ equation:
\begin{equation}
 \mathcal H_t= \tfrac12 \mathcal H_{xx}-\tfrac12 (\mathcal H_x)^2-\eta.
\end{equation}
The KPZ equation is a universal scaling limit for a large family of stochastic systems such as directed polymers in intermediate disorder regime, cf.~Alberts-Khanin-Quastel \cite{AKQ2}; ASEP and its relatives, cf.~Bertini-Giacomin \cite{BeGi}, Amir-Corwin-Quastel \cite{ACQ}, Dembo-Tsai \cite{DT}, Corwin-Shen-Tsai \cite{CST}, Corwin-Shen \cite{CS}; stochastic vertex models, cf.~Corwin-Tsai \cite{CT}, Corwin-Ghosal-Shen-Tsai \cite{CGST}; and several SPDEs, cf.~Hairer-Shen \cite{HS}, Hairer-Quastel \cite{HQ}.  We also refer to Corwin \cite{Corwin}, Quastel-Spohn \cite{QS} for reviews.

Our shift-invariance leads to the following statement for $\mathcal Z^{(y)}$ (or, equivalently, for $\mathcal H$), cf.\ Figure \ref{Fig_KPZ_shift} and see Section \ref{Section_KPZ} for more details.

 \begin{figure}[t]
\begin{center}
{\scalebox{1.1}{\includegraphics{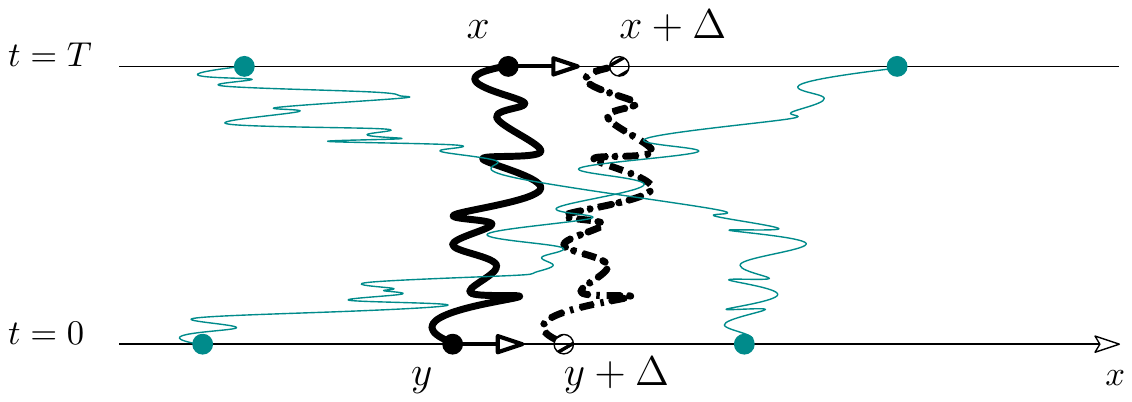}}}
 \caption{The law of  Continuum Directed Random Polymer $\mathcal Z^{(y)}(t,x)$ coincides with the law of $\mathcal Z^{(y+\Delta)}(t,x+\Delta)$, even when taken jointly with the laws of other polymers in the same noise, as long as the intersection condition holds.  \label{Fig_KPZ_shift}}
\end{center}
\end{figure}

\begin{theorem} \label{Theorem_KPZ_invariance_intro}
 Fix $t>0$, $\Delta>0$, $x,y\in\mathbb R$. In addition, choose a collection of points $(x_i,y_i)$, $i=1,\dots,n$, such that for each $i$ either $x_i\le x$, $y_i\ge y+\Delta$, or $x_i\ge x+\Delta$, $y_i\le y$. Then we have the following distributional identity:
 $$
   \bigl(\mathcal Z^{(y)}(t,x); \, \mathcal Z^{(y_i)}(t,x_i), i=1,\dots,n \bigr)\stackrel{d}{=}
   \bigl(\mathcal Z^{(y+\Delta)}(t,x+\Delta); \, \mathcal Z^{(y_i)}(t,x_i), i=1,\dots,n\bigr).
 $$
\end{theorem}

Conjecturally (see Corwin-Quastel-Remenik \cite{CQR}), the large time limit of KPZ (as a function of $x$ and $y$) is described, after proper centering and rescaling, by another prominent object called \emph{Airy sheet} $\mathcal A(x,y)$. The same object is believed to serve as a universal limit for directed polymers and last passage percolation models, with the only rigorous result available at this time being the convergence of the Brownian directed percolation established by Dauvergne-Ortmann-Virag \cite{DOV}. However, if we focus only on various marginals of $\mathcal A(x,y)$, rather than on the full law of the function of two variables, then much more is known. For instance, convergence to $\mathcal A(0,0)$ -- the Tracy-Widom distribution -- is proved for the KPZ equation and a variety of (integrable) polymers and percolation models. Moreover, for the latter one also typically proves  joint convergence to the function of one variable $\mathcal A(0,\cdot)$ known as the \emph{Airy process}.

The following statement can be found in Section \ref{Section_Airy_sheet} below.

\begin{theorem} \label{Theorem_Airy_invariance_intro}
 Fix $t>0$, $\Delta>0$, $x,y\in\mathbb R$. In addition, choose a collection of points $(x_i,y_i)$, $i=1,\dots,n$, such that for each $i$ either $x_i\le x$, $y_i\ge y+\Delta$, or $x_i\ge x+\Delta$, $y_i\le y$. Then we have the following distributional identity:
 $$
   \bigl(\mathcal A(x,y); \, \mathcal A(x_i,y_i), i=1,\dots,n \bigr)\stackrel{d}{=}
   \bigl(\mathcal A(x+\Delta,y+\Delta); \, \mathcal A(x_i,y_i), i=1,\dots,n\bigr).
 $$
\end{theorem}

One corollary of Theorem \ref{Theorem_Airy_invariance_intro} is that it reduces a family of previously unknown joint laws of $\mathcal A(x,y)$ to those of the Airy process. For example, take $x_1<x_2< x_3$ and $y_1>y_2>y_3$, as in Figure \ref{Fig_Airy_shift}. Applying Theorem \ref{Theorem_Airy_invariance_intro} twice and using translation-invariance of $\mathcal A(x,y)$ with respect to simultaneous shift of all arguments, we obtain
\begin{multline*}
%\label{eq_Airy_shift}
(\mathcal A(x_1,y_1), \mathcal A(x_2,y_2), \mathcal A(x_3,y_3))\stackrel{d}{=} (\mathcal A(x_2,y_1+x_2-x_1), \mathcal A(x_2,y_2), \mathcal A(x_2,y_3+x_2-x_3))\\ \stackrel{d}{=} (\mathcal A(0,y_1-x_1), \mathcal A(0,y_2-x_2), \mathcal A(0,y_3-x_3)).
\end{multline*}
The same computation can be done for $n$ pairs of points satisfying $x_1<x_2<\dots<x_n$, $y_1>y_2>\dots>y_n$.

 \begin{figure}[t]
\begin{center}
{\scalebox{0.7}{\includegraphics{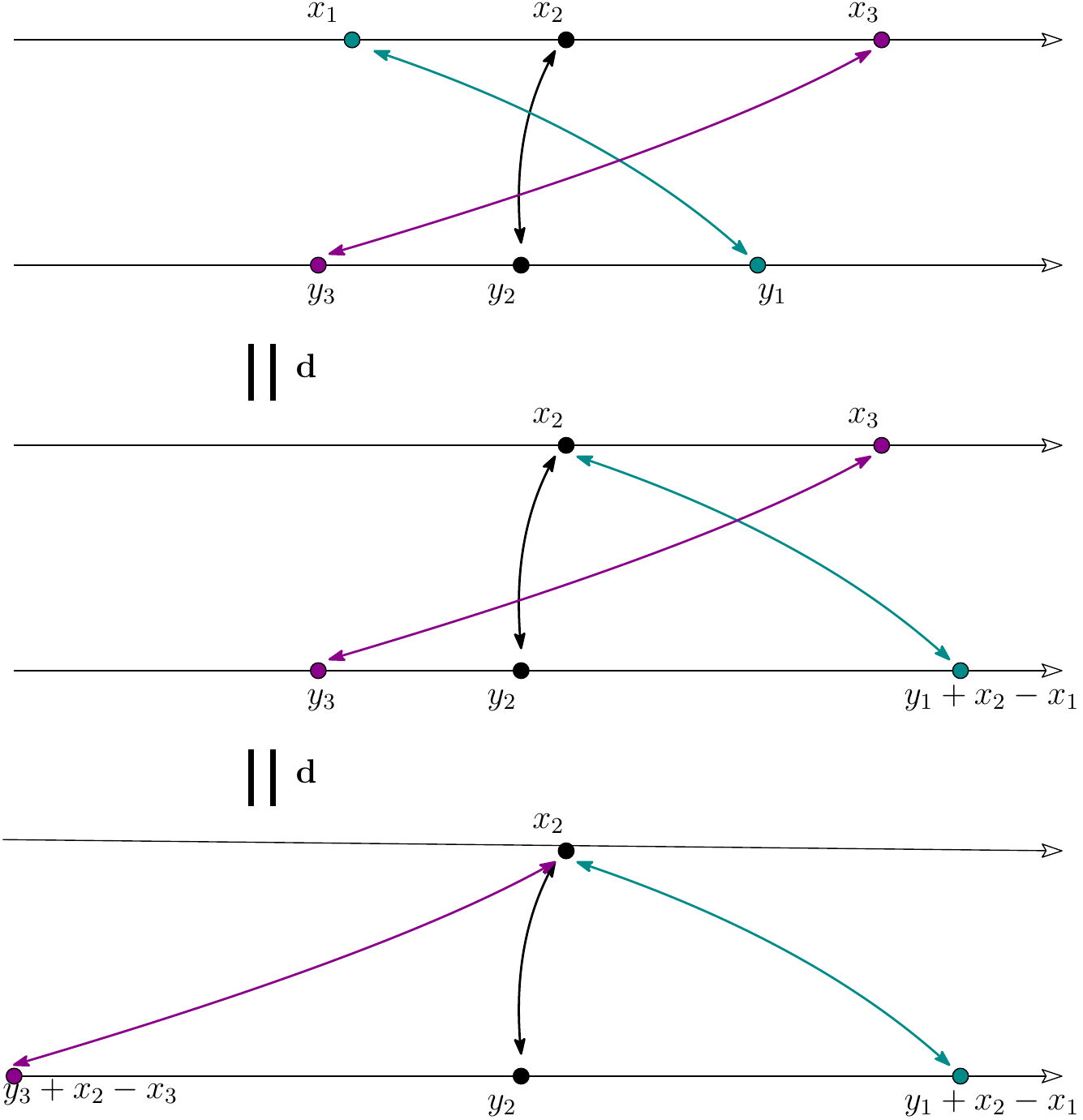}}}
 \caption{Airy sheet shifts: $(\mathcal A(x_1,y_1), \mathcal A(x_2,y_2), \mathcal A(x_3,y_3)) \stackrel{d}{=} (\mathcal A(x_2,y_1+x_2-x_1), \mathcal A(x_2,y_2), \mathcal A(x_2,y_3+x_2-x_3))$. \label{Fig_Airy_shift}}
\end{center}
\end{figure}

\subsection{Additive stochastic heat equation}
 While the large time limit of the KPZ equation is given by the Airy sheet, the small time limit is a much simpler Gaussian object, which is a version of the stochastic heat equation  \emph{with additive noise}\footnote{This should not be confused with the different stochastic heat equation related to KPZ}, see  \cite[Section 6.2]{ACQ}) and Section \ref{Section_back_to_SHE}. The same object can be also obtained as the fluctuating profile of the density of particles in the diffusive scaling limit of the colored symmetric simple exclusion process, as outlined in Section \ref{Section_6v_degenerations}. For the Gaussian processes the shift-invariance of the distributions is equivalent to the shift-invariance of the covariances and, therefore, it suffices to verify the latter. We show in Section \ref{Section_SHE} that shift-invariance of the covariances for the stochastic heat equation with additive noise becomes equivalent to the invariance of the expected intersection local times of two Brownian bridges under shifts of the end-points for one of them. In turn, the last invariance can be reduced to invariance of the expected local time at level $c$ for the Brownian bridge travelling from $a$ at time $0$ to $b$ at time $1$. The law of such a local time is known, see  Ray \cite{Ray}, Williams \cite{Williams}, Borodin \cite{ANBorodin}, Biane--Yor \cite{BianeYor}, Pitman \cite{Pitman}. Its density is proportional to
$$
  (|c-a|+|c-b|+y)\exp\left(-\tfrac{1}{2}(|c-a|+|c-b|+y)^2\right) \, dy, \qquad y>0,
$$
which is clearly independent of $c$ as long as $a<c<b$, thus, giving the desired shift-invariance.

In Section \ref{Section_6v} we show that the interplay between shift-invariance of the \emph{covariances} and intersection local times is preserved up to the level of the homogeneous version of our master statement for the colored stochastic six-vertex model. Our covariance computation is based on the four point relation for the model, which generalizes the colorless version in Borodin-Gorin \cite{BG_tele}. The Brownian bridges get replaced by a pair of discrete persistent random walks; it seems that shift-invariance of the latter has not appeared in the literature before, and we prove it in Section \ref{Section_local_times}.

\subsection{A conjecture} Numerical experiments and certain formulaic evidence indicate that shift-invariance should hold in a greater generality than what we have stated so far. Let us record this as a conjecture for the KPZ equation, extending Theorem \ref{Theorem_KPZ_invariance_intro}, cf.\ Figure \ref{Fig_KPZ_double_shift}. It is not hard to come up with similar conjectural statements for all the other models mentioned above.

 \begin{figure}[t]
\begin{center}
{\scalebox{1.1}{\includegraphics{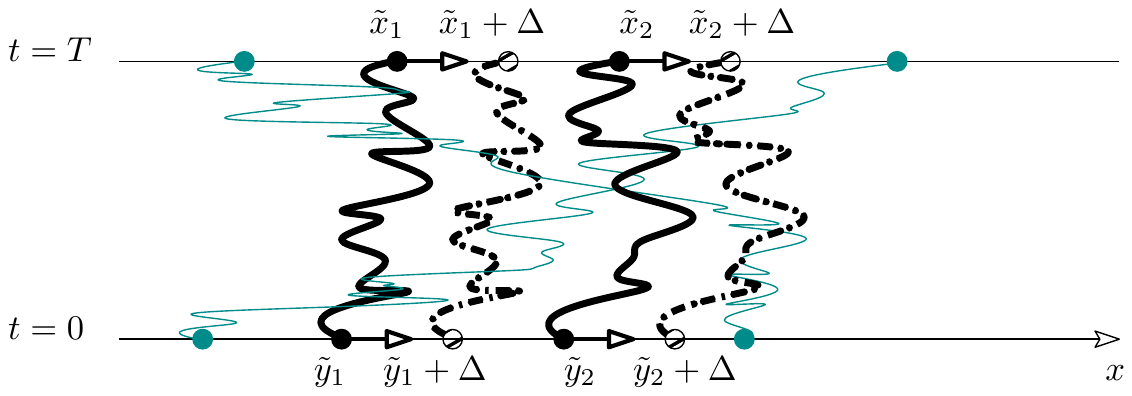}}}
 \caption{Conjecturally, two pairs of the end-points for Continuum Directed Random Polymer can also be shifted, while preserving the joint distribution with non-moving polymers. \label{Fig_KPZ_double_shift}}
\end{center}
\end{figure}

\begin{conjecture} \label{Conjecture_KPZ_invariance_intro}
 Fix $t>0$, $\Delta>0$, two segments $[X,X'],[Y,Y']\subset\mathbb R$, and four collections of real numbers $(x_1,\dots,x_n)$, $(y_1,\dots,y_n)$, $(\tilde x_1,\dots,\tilde x_k)$, $(\tilde y_1,\dots,\tilde y_k)$. Assume that for each $1\le j \le k$, $[\tilde x_j , \tilde x_j+\Delta]\subset[X,X']$ and $[\tilde y_j, \tilde y_j +\Delta]\subset [Y,Y']$. In addition, assume that for each $1\le i \le n$ either $x_i\le X$, $y_i\ge Y'$, or $x_i\ge X'$, $y_i\le Y$. Then we have the following distributional identity:
 \begin{multline*}
   \bigl(\mathcal Z^{(\tilde y_j)}(t,\tilde x_j), j=1,\dots,k; \, \mathcal Z^{(y_i)}(t,x_i), i=1,\dots,n \bigr)\\ \stackrel{d}{=}
   \bigl(\mathcal Z^{(\tilde y_j+\Delta)}(t,\tilde x_j+\Delta), j=1,\dots,k; \, \mathcal Z^{(y_i)}(t,x_i), i=1,\dots,n\bigr).
 \end{multline*}
\end{conjecture}
Note that for $\tilde x_1<\tilde x_2<\dots<\tilde x_k$ and $\tilde y_1>\tilde y_2>\dots>\tilde y_k$, this statement is proved by $k$ applications of Theorem \ref{Theorem_KPZ_invariance_intro}. The conjecture says that such a $k$-dimensional shift-invariance also holds without these inequalities.

\subsection*{Acknowledgements} The authors are grateful to Pierre Le Doussal for discussions on limit transitions from colored models to polymers. A.~B.~ was partially supported by NSF grants DMS-1664619 and DMS-1853981, V.G.~ was partially supported by NSF grants DMS-1664619, DMS-1855458, by the NEC Corporation Fund for Research in Computers and Communications, and by the Office of the Vice Chancellor for Research and Graduate Education at the University of Wisconsin--Madison with funding from the Wisconsin Alumni Research Foundation.
 M.W.~ was partially supported by  ARC grant DP190102897.

\section{Shift-invariance for the stochastic heat equation}

In this section we provide a proof for the shift-invariance in the simplest Gaussian case --- for the stochastic heat equation with additive noise. Two ways to obtain this equation by degenerating other systems in the chart of Figure \ref{Fig_Chart} are outlined in Section \ref{Section_6v_degenerations} and Section \ref{Section_back_to_SHE}.

\label{Section_SHE}

\subsection{Statement}

The central objects of this section are an $(N+1)$--tuple of Gaussian fields $\eta^i(t,y)$, $i=0,\dots,N$, $t\ge 0$, $y\in\mathbb R$, and an $(N+1)$--tuple of deterministic fields $\rho^i(t,y)$. At each $(t,y)$, the sum of the Gaussian fields is identically zero, while the sum of deterministic fields is identically equal to one. The deterministic fields solve the heat equation with prescribed initial conditions at $t=0$. The Gaussian fields solve the stochastic heat equation with additive white (in $(t,y)$) noises in the right-hand side. At fixed $(t,y)$, the covariance of white noises can be identified with that of an $(N+1)$--dimensional vector of independent Gaussians with variances $\rho^i(t,y)$ conditioned on their sum being zero.

\smallskip

One possible interpretation is that $\rho^i(t,y)$, $i=0,\dots,N$, represent macroscopic densities of $N+1$ different species at time $t$ and position $y$. The densities are subject to the conservation law of total density $1$. Simultaneously, $\eta^i(t,y)$ represent random Gaussian fluctuations linked to these densities; as $\rho^i$ becomes larger, so does the white noise driving $\eta^i(t,y)$.

\bigskip

We now proceed to the formal definition. Take $N+1$ non-negative real functions $\rho^i_0(y)$, $i=0,1,\dots,N$, which sum up to $1$: $\rho^0_0(y)+\rho^1_0(y)+\dots+\rho^N_0(y)=1$, $y\in\mathbb R$.
 Define the functions $\rho^i:\mathbb R_{\ge 0} \times \mathbb R\to\mathbb R$, $i=1,\dots,N$, as the solutions to the heat equation
  \begin{equation}
 \label{eq_limit_shape_HE}
 \frac{\partial}{\partial t} \rho^i(t,y)-\frac{1}{2}\frac{\partial^2}{\partial y^2} \rho^i(t,y)=0,\quad t\ge 0,\, y\in\mathbb R, \quad \rho^i(0,y)=\rho^{i}_0(y).
 \end{equation}
 Since the fundamental solution of the heat equation is given by the Gaussian kernel, the functions $\rho^i$ can be expressed as partial integrals of this kernel. We further set
 $$
  \rho^{\geqslant i}(t,y)=\sum_{a=i}^{N} \rho^a(t,y),
 $$
 so that
 \begin{equation} \label{eq_SHE_ineq}1=\rho^{\ges 0}(t,y) \ge \rho^{\ges 1}(t,y) \ge \dots \ge\rho^{\ges (N-1)}(t,y)\ge \rho^{\ges N}(t,y)\ge 0.
 \end{equation}
  Since the heat equation is linear, functions $\rho^{\ges i}(t,y)$ are also its solutions.

Further, define $N+1$ random Gaussian functions $\eta^{i}(t,y)$, $i=0,1,\dots,N$,  which solve the stochastic heat equation
\begin{equation}
\label{eq_SHE_single}
 \frac{\partial}{\partial t} \eta^{i}(t,y)-\frac{1}{2}\frac{\partial^2}{\partial y^2} \eta^{i}(t,y)
  = \Rho^{i}, \quad t\ge 0,\, y\in\mathbb R,\qquad \eta^{i}(0,y)=0,
 \end{equation}
 where $\Rho^{i}$ are spatially uncorrelated centered generalized Gaussian noises with covariance  given by
\begin{equation}
\label{eq_covariance_noise_colors_SHE_single}
\E \Rho^{i}(t,y) \Rho^{j}(t',y')=\delta_{y=y'}\delta_{t=t'}
\cdot \begin{cases} -\rho^i (t,y) \rho^{j}(t,y),& i<j,\\
  \rho^{i}(t,y)(1-\rho^{i}(t,y)), & i=j.
  \end{cases}
\end{equation}
Note that the noises $\Rho^i(t,y)$ can be thought of as independent Gaussian white noises of variances $\rho^i(t,y)$ and conditioned to have zero sum:
$$\Rho^0(t,y)+\Rho^1(t,y)+\dots+\Rho^{N}(t,y)=0.$$
We also define
$$\eta^{\ges i}=\eta^i+\eta^{i+1}+\dots+\eta^N,\qquad  \Rho^{\ges i}=\Rho^i+\Rho^{i+1}+\dots+\Rho^N,\qquad i=1,\dots,N,$$
so that
\begin{equation}
\label{eq_SHE}
 \frac{\partial}{\partial t} \eta^{\ges i}(t,y)-\frac{1}{2}\frac{\partial^2}{\partial y^2} \eta^{\ges i}(t,y)
  = \Rho^{\ges i}, \quad t\ge 0,\, y\in\mathbb R, \qquad \eta^{\ges i}(0,y)=0,
 \end{equation}
 where $\Rho^{\ges i}$ can be equivalently defined as spatially uncorrelated centered generalized Gaussian noises with covariances for $i\le j$ given by
\begin{equation}
\label{eq_covariance_noise_colors_SHE}
\E \Rho^{\ges i}(t,y) \Rho^{\ges j}(t',y')=\delta_{y=y'}\delta_{t=t'}
\cdot
    (1- \rho^{\ges i}(t,y))  \rho^{\ges j}(t,y).
\end{equation}
The covariance in \eqref{eq_covariance_noise_colors_SHE} can be also given an interpretation. For that we use inequalities \eqref{eq_SHE_ineq} and think about $\rho^{\ges N}(t,y),\dots,\rho^{\ges 1}(t,y)$ as $N$ \emph{times} $0< s_1<\dots<s_N<1$ for the standard Brownian bridge, which travels from $0$ to $0$ in time $1$. Then the right-hand side of \eqref{eq_covariance_noise_colors_SHE} is  the covariance of the values for such a Brownian bridge.

\bigskip

Fix  $T>0$, two indices $1\le i<j \le N$ and two reals $B_i>B_j$. Consider the two-dimensional random vector
$$
 \xi_{\rho^{\ges i}_0, \rho^{\ges j}_0}(B_i,B_j)=(\eta^{\ges i}(T,B_i), \eta^{\ges j}(T,B_j)).
$$
For $\Delta\in\mathbb R$, let $S_\Delta$ denote the operator of shifting the argument of a function by $\Delta$:
$$
  [S_\Delta f](x)=f(x+\Delta).
$$

\begin{figure}[t]
\begin{center}
{\scalebox{1.2}{\includegraphics{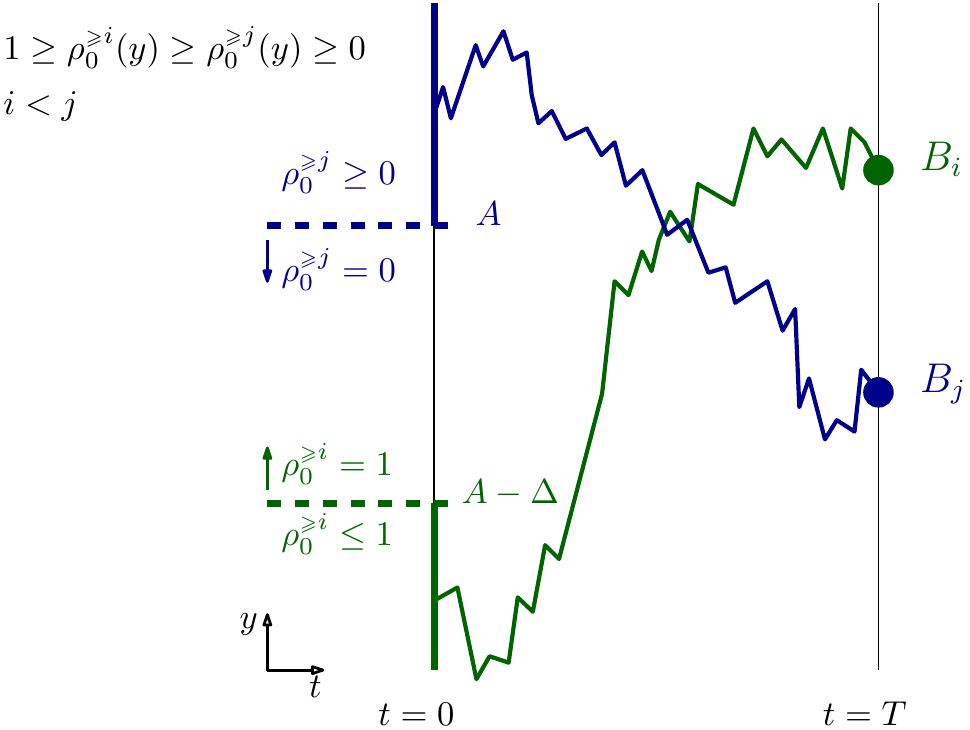}}}
 \caption{Restriction on parameters in Theorem \ref{Theorem_SHE_invariance} and two intersecting Brownian bridges appearing in its proof.
 \label{Fig_intersecting_BM}}
\end{center}
\end{figure}

\begin{theorem}
\label{Theorem_SHE_invariance}
  Fix $\Delta\ge 0$. Suppose that for some $i<j$ and $A\in\mathbb R$, $\rho^{\ges j}_0(y)=0$ on $[-\infty, A)$ and $\rho^{\ges i}_0(y)=1$ on $[A-\Delta,+\infty)$, cf.\ Figure \ref{Fig_intersecting_BM}. Then the following identity in distribution holds for every $B_i>B_j$:
  \begin{equation}
  \label{eq_SHE_invariance}
   \xi_{\rho^{\ges i}_0, \rho^{\ges j}_0}(B_i,B_j)\stackrel{d}{=} \xi_{\rho^{\ges i}_0, S_\Delta \rho^{\ges j}_0}(B_i,B_j-\Delta).
  \end{equation}
  In words, the simultaneous shift in the same direction and by the same $\Delta$ of the initial condition $\rho^{\ges j}_0$ and the observation point $B_{j}$ does not change the joint distribution of the two-dimensional vector.
\end{theorem}
\begin{remark}
 Due to translation-invariance of the system in $y$--direction, the random variables of \eqref{eq_SHE_invariance} will also have the same distribution as $\displaystyle \xi_{S_{-\Delta}\rho^{\ges i}_0,\rho^{\ges j}_0}(B_i+\Delta,B_j)$.
\end{remark}

We can extend Theorem \ref{Theorem_SHE_invariance} to a statement involving $N$--dimensional vectors. Fix $T>0$, $\Delta\ge 0$, and an index $\iota\in\{1,2,\dots,N\}$. Take two real vectors $\vec{A}, \vec{B} \in\mathbb R^N$, such that
 \begin{equation}
 \label{eq_x14}
  A_1<A_2<\dots<A_{\iota-1}<A_{\iota}-\Delta<A_{\iota}<A_{\iota+1}<\dots<A_N,
 \end{equation}
 \begin{equation}
 \label{eq_x15}
  B_1,\dots,B_{\iota-1}> B_{\iota}> B_{\iota}-\Delta > B_{\iota+1},\dots,B_N.
 \end{equation}
 (Note that $B_1,\dots, B_{\iota-1}$ and $B_{\iota+1},\dots,B_N$ are not ordered.)
Given this data, we construct the initial conditions $\rho^{\ges i}(y)=\mathbf 1(y\ge A_i)$ and consider an $N$--dimensional random vector:
$$
 \xi_{\vec A}(\vec B)=(\eta^{\ges 1}(T, B_1),\dots,\eta^{\ges N}(T,B_N)).
$$
\begin{corollary}
\label{Corollary_SHE_invariance}
   Take $T,\Delta>0$, $\iota\in\{1,\dots,N\}$, and two vectors $\vec{A}, \vec{B} \in\mathbb R^N$ satisfying \eqref{eq_x14}, \eqref{eq_x15}.
 With the notation $\mathbf{e}_\iota$ for the $\iota$--th coordinate vector in $N$--dimensional space, the following distributional identity holds:
 \begin{equation}
 \label{eq_SHE_invariance_vector}
  \xi_{\vec{A}}(\vec{B})\stackrel{d}{=}
  \xi_{\vec{A}-\Delta \mathbf{e}_\iota}(\vec{B}-\Delta \mathbf{e}_\iota).
 \end{equation}
\end{corollary}
Since the vector \eqref{eq_SHE_invariance_vector} is Gaussian, the distributional identity \eqref{eq_SHE_invariance_vector} of $N$--dimensional vectors follows from the same identity for its pairs of coordinates, which is Theorem \ref{Theorem_SHE_invariance} when one of the coordinates is $\iota$, and a tautological statement otherwise. Our choice of $\rho^{\ges i}(y)$ as indicator functions of semi-infinite intervals guarantees that the conditions on the support of Theorem \ref{Theorem_SHE_invariance} holds. More complicated choices for $\rho^{\ges i}(y)$ are also possible, but we will not pursue this direction further.

\subsection{Proof of Theorem \ref{Theorem_SHE_invariance}}
\label{Section_SHE_proof}
Note that the individual distributions of the first components of the vectors in \eqref{eq_SHE_invariance} coincide by the definition. Similarly, the individual distributions of the second components coincide. Since we deal with Gaussian vectors, it remains to check that the covariances coincide as well.

Let $G(x,y; t)$ denote the Gaussian kernel:
$$
 G(x,y;t)=\frac{1}{\sqrt{2 \pi t}} \exp\left( -\frac{(x-y)^2}{2t}\right).
$$
The solution to \eqref{eq_limit_shape_HE} is written as
\begin{equation}
\rho^i(t,y)=\int_{z\in\mathbb R} \rho^i_0(z) G(z,y; t)\, \dd z,
\end{equation}
and the solution to \eqref{eq_SHE} is
\begin{equation}
 \eta^{\ges i}(t,y)=\int_{s=0}^t \int_{z\in\mathbb R} \Rho^{\ges i}(s,z) G(z,y; t-s)\, \dd z \dd s.
\end{equation}
Hence, the covariance of $\eta^{\ges i}(T,B_i)$ and $\eta^{\ges j}(T,B_j)$ for the left-hand side of  \eqref{eq_SHE_invariance} becomes
\begin{multline}
\label{eq_cov1}
 \E  \eta^{\ges i}(T,B_i) \eta^{\ges j}(T,B_j)\\=  \int_{t=0}^T \int_{y\in\mathbb R}  (1- \rho^{\ges i}(t,y)) \rho^{\ges j}(t,y)
 G(y,B_i; T-t) G(y,B_j; T-t)\,\dd y \dd t
 \\=  \int_{z_i\in \mathbb R} \int_{z_j\in\mathbb R} (1-\rho^{\ges i}_0(z_i))\rho^{\ges j}_0(z_j)\,  \dd z_j \dd z_i \\ \times \int_{t=0}^T \int_{y\in\mathbb R}   G(z_i,y;t) G(z_j,y;t)
 G(y,B_i; T-t) G(y,B_j; T-t)\, \dd y \dd t.
\end{multline}
Making the same computation for the right-hand side of \eqref{eq_SHE_invariance} we get a similar, yet different expression:
\begin{multline}
  \int_{z_i\in \mathbb R} \int_{z_j\in\mathbb R}  (1-\rho^{\ges i}_0(z_i))\rho^{\ges j}_0(z_j+\Delta)\, \dd z_j \dd z_i \\ \times \int_{t=0}^T \int_{y\in\mathbb R}   G(z_i,y;t) G(z_j,y;t)
 G(y,B_i; T-t) G(y,B_j-\Delta; T-t)\, \dd y \dd t.
\end{multline}
We shift $z_j$ by $\Delta$ to get
\begin{multline}
\label{eq_cov2}
  \int_{z_i\in \mathbb R} \int_{z_j\in\mathbb R}  (1-\rho^{\ges i}_0(z_i))\rho^{\ges j}_0(z_j) \dd z_j \dd z_i \\ \times \int_{t=0}^T \int_{y\in\mathbb R}  G(z_i,y;t) G(z_j-\Delta,y;t)
 G(y,B_i; T-t) G(y,B_j-\Delta; T-t)\, \dd y \dd t.
\end{multline}
We need to show that \eqref{eq_cov1} is the same as \eqref{eq_cov2}, which would follow from the $\Delta$--invariance of the double integral of the product of four Gaussian kernels in the last line of both formulas.

For that we notice a stochastic interpretation of the double integral over $t$ and $y$. Recall that the functions $G$ are transition probabilities for the Brownian motion. Imagine for a second that they were transition probabilities for discrete random walks instead. Then the product of four probabilities under the double integral over $y$ and $t$ in \eqref{eq_cov1} can be interpreted as the probability that random walk from $(z_j,0)$ to $(B_j,T)$ intersects another (independent) random walk from $(z_i,0)$ to $(B_i,T)$ at point $(y,t)$. The integral over all $(y,t)$ (or rather a sum, if we speak about discretization) then counts the total expected number of intersections. We show in Section \ref{Section_local_times} (see $b_1=b_2$ case of Theorem \ref{Theorem_intersection_RW}) that the distribution of the number of intersections is unchanged upon shifting $z_j$ and $B_j$ simultaneously by the same amount. It is crucial for the theorem that our two random walks necessarily intersect with probability $1$, which is implied by  $z_i<z_j-\Delta$ and $B_i>B_j$ (the first inequality follows from our restriction on the supports of $\rho^{\ges j}_0$ and $1-\rho^{\ges i}_0$). Taking the expectation of this distributional identity and performing a standard limit transition between discrete random walks and Brownian motions we arrive at the desired $\Delta$--independence of \eqref{eq_cov2}.

If we wanted to avoid discretizations, then we could have argued with Brownian motions directly. For that take two independent Brownian bridges $\mathfrak{Br}_i$, $\mathfrak{Br}_j$, where the first one is the standard Brownian motion conditioned to travel from $z_i$ to $B_i$ in time $1$, and the second one is the standard Brownian motion conditioned to travel from $z_j-\Delta$ to $B_j-\Delta$ in time $1$. Then our double integral computes the expected intersection time of $\mathfrak{Br}_i$ and $\mathfrak{Br}_j$, or, equivalently the local time at $0$ for the difference $\mathfrak{Br}_i-\mathfrak{Br}_j$. The key property, which guarantees the $\Delta$--independence, is that $z_i+\Delta<z_j$ and $B_i>B_j$. Therefore, the trajectories of the Brownian bridges $\mathfrak{Br}_i$ and $\mathfrak{Br}_j$ almost surely intersect, cf.\ Figure \ref{Fig_intersecting_BM}, while
$\mathfrak{Br}_i-\mathfrak{Br}_j$ starts below zero and ends above zero, so that its trajectory necessarily intersects the $0$ level. Now a shift by $\Delta$ is a translation of one of the Brownian bridges (equivalently, of their difference) by $\Delta$. At this point, we can use a well-known property of the local time of the Brownian bridge: the distribution of the local time at $c$ for a Brownian bridge travelling from $a$ to $b$ is independent of $c$ as long as $a<c<b$ (or $b<c<a$), see \cite{Pitman} and references therein. Since the distribution is unchanged under shifts, so is its expectation, and therefore, \eqref{eq_cov2} does not depend on $\Delta$, as desired.

\section{Homogeneous stochastic six-vertex model: covariance match}

\label{Section_6v}

\begin{figure}[t]
\begin{center}
{\scalebox{0.5}{\includegraphics{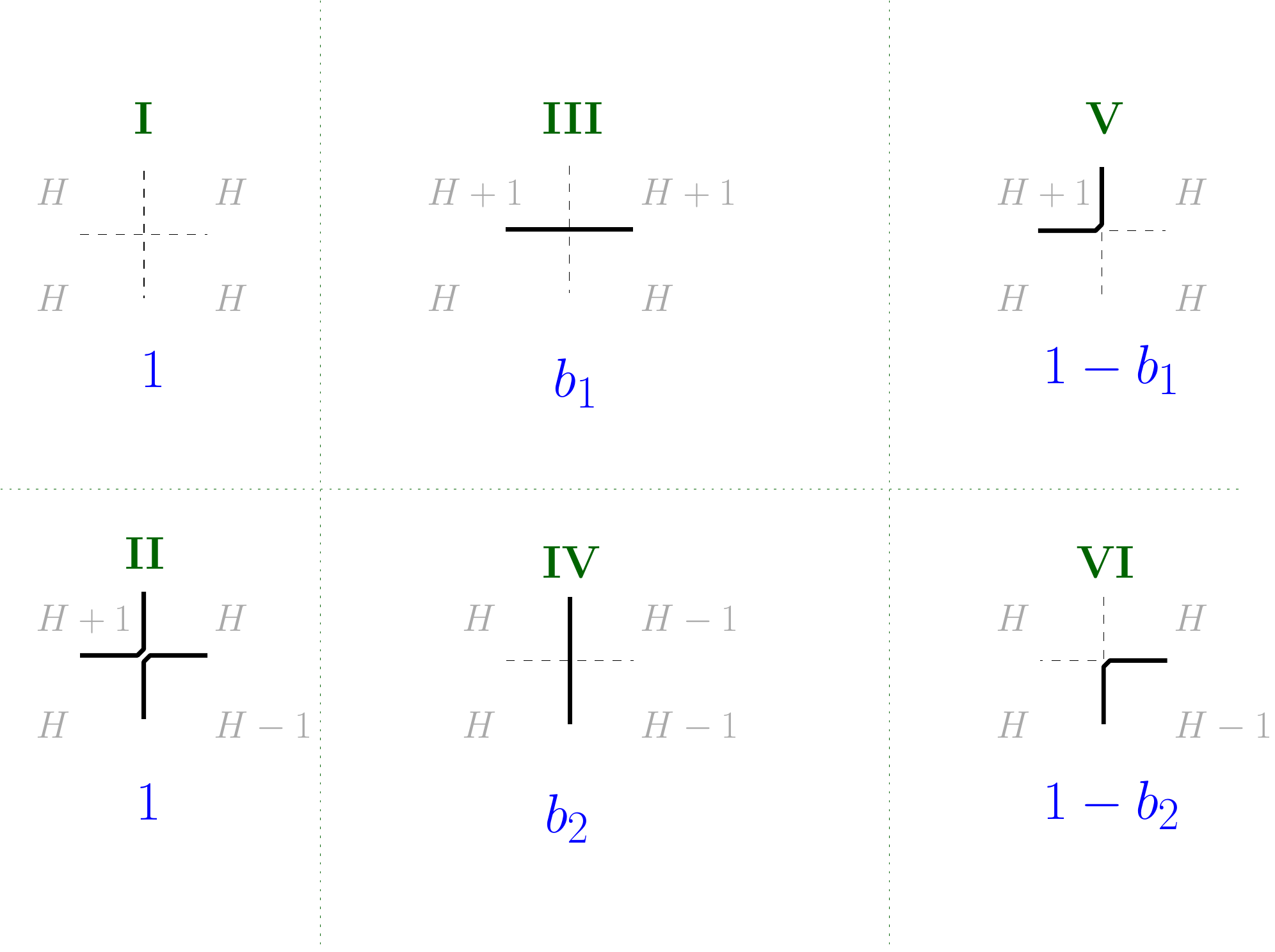}}}
 \caption{Basic weights of the six-vertex model. Same weights are used in the colored version upon the convention that the dashed (empty) edge corresponds to the lower of the two colors. Local changes of the height function $\H(x,y)$ are shown in gray.
 \label{Fig_6v}}
\end{center}
\end{figure}

Next, we switch to the homogeneous stochastic $(N+1)$--colored six-vertex model with colors $0,1,2,\dots,N$. Configurations of the model are colorings of all edges in (a part of) the square grid into $N+1$ colors, subject to the rule ``number of incoming edges = number of outgoing edges'' for each color and each vertex of the grid.

When $N=1$, we have only two colors, and if treat the color $0$ as the absence of the edge and color $1$ as the presence of the edge, then we get the standard (colorless) stochastic six-vertex model, cf.\ Figure \ref{Fig_6v}. We deal with the model in the quadrant $x\ge 1$, $y\ge 1$, and sample the vertices sequentially in the up-right direction (i.e., the first vertex to sample is in the corner of the quadrant) according to probabilities of Figure \ref{Fig_6v}. The probabilities depend on two parameters $0<b_1<1$, $0<b_2<1$, and we also set $q=\frac{b_2}{b_1}$. In the notations of Section \ref{Section_intro_6v}, all the rapidities $u_x$ and $v_y$ are taken to be independent of $x$ and $y$ and, therefore, the weights $b_1$ and $b_2$ of Figure \ref{Fig_colored_weights} do not change throughout the quadrant.

When we have $N+1$ colors, we still sample the vertices sequentially. Whenever we need to sample the vertex at $(x,y)$, it comes with the colors of two incoming edges: $i$ and $j$. We assume that $0\le i < j$ and make the sampling according to probabilities of Figure \ref{Fig_colored_weights} in Section \ref{Section_intro_6v}.
Note that if we treat the color $i$ as the absence of the edge, and the color $j$ as the presence of the edge, then the probabilities are the same as those of Figure \ref{Fig_6v} for sampling.

For example, if $N=2$, then we can think about absence of edges (color $0$), narrow black edges (color $1$), and bold red edges (color $2$).  All 15 types
of vertices and corresponding probabilities (computed by the rule from the previous paragraph) are shown in Figure \ref{Fig_weights_1}. We enumerate the
vertices by rows and columns, so that vertices A1, A2, A3 have weight 1, there is no vertex of type
A4, etc. The weights are stochastic in the sense that for each pair of inputs from the left and
from the bottom, the weights of possible vertices (outputs) sum up to $1$.

\begin{figure}[t]
\begin{center}
{\scalebox{1.2}{\includegraphics{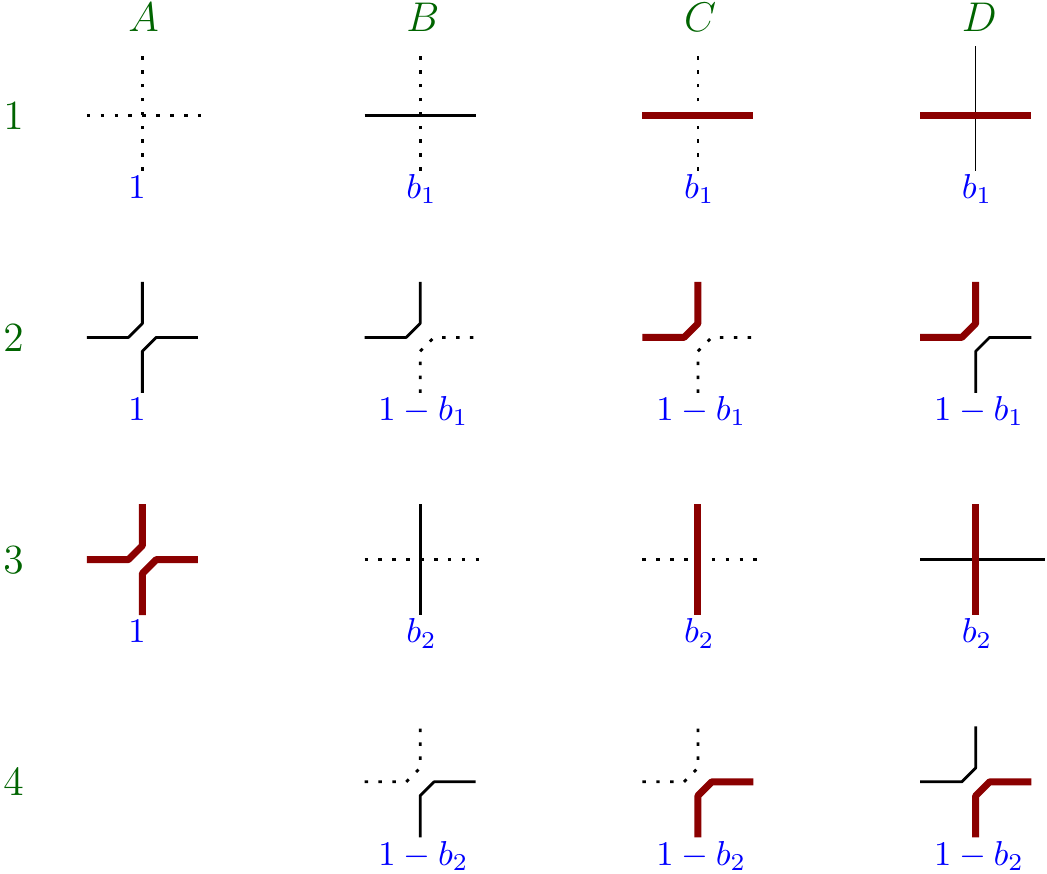}}}
 \caption{Weights of 15 types of vertices in the two-colored model.
 \label{Fig_weights_1}}
\end{center}
\end{figure}

We deal with the model in the quadrant, which means that the vertices have integer
coordinates $(x,y)$ satisfying $x>0$, $y>0$. Along the borders of the quadrant we allow arbitrary (deterministic) boundary conditions.
This means that for each vertex adjacent to the left boundary, we choose the color of the
incoming edge from the left. Similar
choices of color of edges are made for each vertex adjacent to the bottom boundary. The model is then defined by
stochastic sampling, treating the weights of Figure \ref{Fig_6v} (or Figure \ref{Fig_weights_1}) as probabilities: we first
sample a single vertex with $x+y=1$, then two vertices with $x+y=2$, then three vertices with
$x+y=3$, etc.

Each color $i$ comes with its \emph{height function} $\H^i(x,y)$, which is defined by setting $\H^{i}(\frac12,\frac12)=0$ and
$$
\H^{i}(x,y+1)-\H^{i}(x,y)=\begin{cases}1, & \text{there is an edge of color }i\text{ at } (x,y+\frac12),\\0,&\text{otherwise,}\end{cases}
$$
$$
\H^{i}(x+1,y)-\H^{i}(x,y)=\begin{cases}-1, & \text{there is an edge of color }i\text{ at } (x+\frac12,y),\\0,&\text{otherwise,}\end{cases}
$$
Note that in order to avoid ambiguities, the height function is defined not at vertices $(x,y)\in\mathbb Z\times \mathbb Z$, but at vertices of the dual grid  $(\mathbb Z+\frac12)\times(\mathbb Z+\frac12)$, which are in natural bijection with faces of the original grid. We also set
$$
 \H^{\ges i}(x,y)=\H^i(x,y)+\H^{i+1}(x,y)+\dots+\H^{N}(x,y).
$$
See Figure \ref{Fig_heights_1} for an illustration. Our definitions imply that
\begin{multline*}
0\le \H^{\ges N}(x,y+1)-\H^{\ges N}(x,y) \le \H^{\ges N-1}(x,y+1)-\H^{\ges N-1}(x,y)\\ \le \dots \le \H^{\ges 1}(x,y+1)-\H^{\ges 1}(x,y)\le \H^{\ges 0}(x,y+1)-\H^{\ges 0}(x,y)=1,
\end{multline*}
 and
\begin{multline*}
 0\ge \H^{\ges N}(x+1,y)-\H^{\ges N}(x,y) \ge \H^{\ges N-1}(x+1,y)-\H^{\ges N-1}(x,y)\\ \ge \dots \ge \H^{\ges 1}(x+1,y)-\H^{\ges 1}(x,y)\ge \H^{\ges 0}(x+1,y)-\H^{\ges 0}(x,y)=-1.
\end{multline*}
Note that specification of the boundary conditions is the same as specification of all the values of the height functions $\H^{\ges i}(x,y)$ for $i=1,2,\dots,N$, and either $x=\frac12$ or $y=\frac12$.

\begin{figure}[t]
\begin{center}
{\scalebox{0.6}{\includegraphics{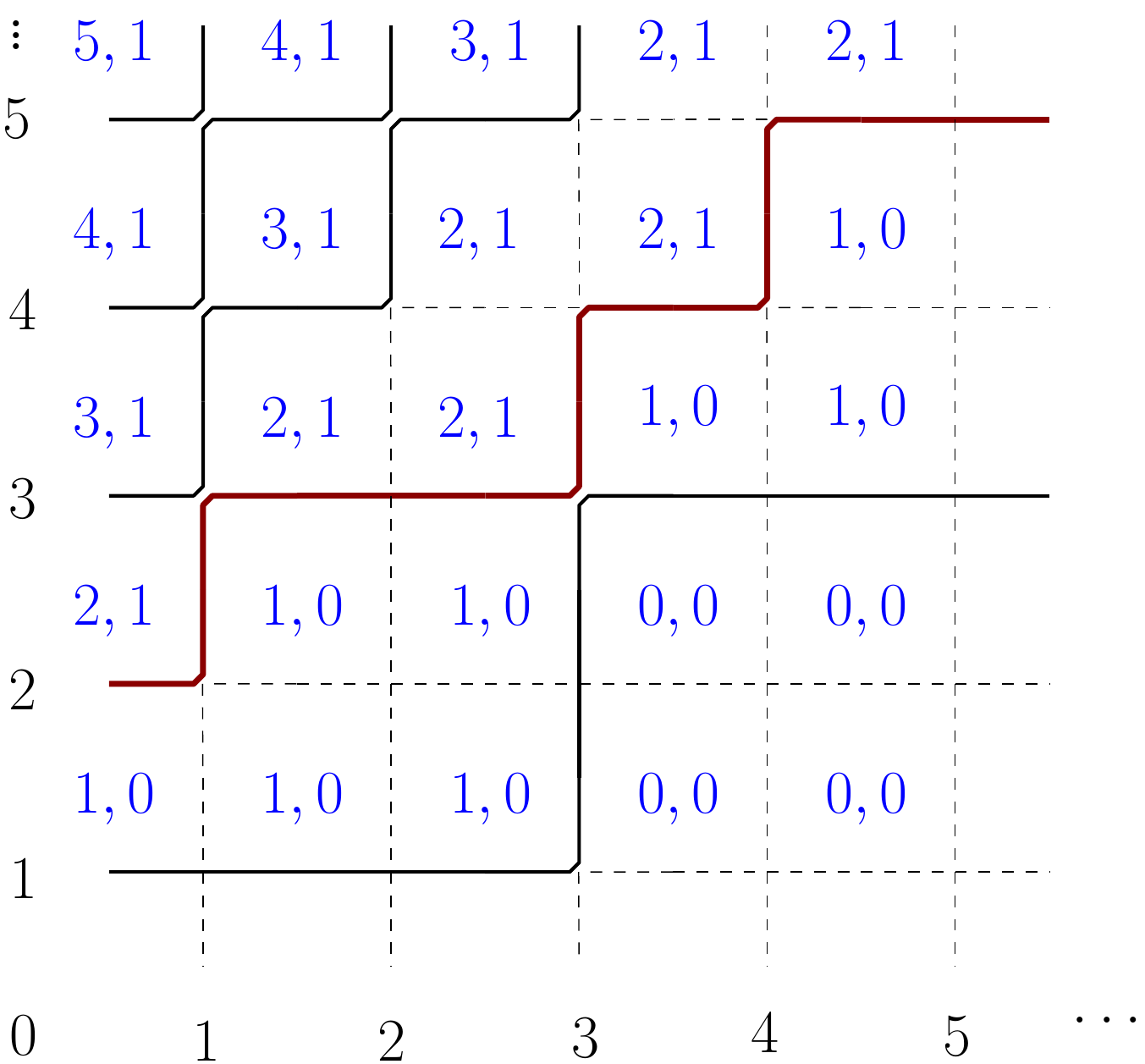}}}
 \caption{One possible configuration of the three-colored ($N=2$, shown with bold red paths, narrow black paths, and absence of paths) model in quadrant. The values of the vector of height functions $(\H^{\ges 1}(x,y),\H^{\ges 2}(x,y))$ are shown.
 \label{Fig_heights_1}}
\end{center}
\end{figure}

\bigskip

We now produce analogues of Theorem \ref{Theorem_SHE_invariance} and Corollary \ref{Corollary_SHE_invariance}. We start by specifying the boundary conditions via the functions $\rho^i_{b}(x,y)$, $i=0,1,\dots,N$, which are discrete derivatives of the height functions $\H^i$ along the boundary of the quadrant. Formally, for $(x,y)=(\frac12, n)_{n\in\mathbb Z_{>0}}$ and $(x,y)=(n,\frac12)_{n\in\mathbb Z_{>0}}$:
$$
 \rho^{i}_b\left(\tfrac12,y\right)=\begin{cases} 1,& \text{there is an incoming horizontal edge of color } i\text{ at }(\tfrac12,y),\\ 0,&\text{otherwise.} \end{cases}
$$
$$
 \rho^{i}_b\left(x,\tfrac12\right)=\begin{cases} 1,& \text{there is an incoming vertical edge of color } i\text{ at }(x,\tfrac12),\\ 0,&\text{otherwise.} \end{cases}
$$
We set $\rho^{\ges i}_b(x,y)=\rho^{i}_b(x,y)+\dots+\rho^{N}_b(x,y)$, so that
\begin{multline*}
 0\le \rho^{\ges N}_b(x,y)\le \rho^{\ges N-1}_b(x,y)\le \dots\le \rho^{\ges 1}_b(x,y)\le \rho^{\ges 0}_b(x,y)=1, \\ (x,y)\in \left\{\left(\tfrac12,n\right), \left(n,\tfrac12\right)\right\}_{n=1,2,\dots}.
\end{multline*}

Fix  $X\in\mathbb Z_{>0}+\frac12$, two indices $1\le i<j \le N$, and two numbers $B_i>B_j\in\mathbb Z+\frac12$. Consider the two-dimensional random vector
$$
 \xi_{\rho^{\ges i}_{b}, \rho^{\ges j}_{b}}(B_i,B_j)=\left(q^{\H^{\ges i}(X,B_i)}, q^{\H^{\ges j}(X,B_j)}\right), \qquad q=\frac{b_2}{b_1}.
$$
For $\Delta\in\mathbb Z$ let $S_\Delta$ denote the operator of shifting the second variable (i.e., $y$) of a function $f(x,y)$ by $\Delta$:
$$
[S_\Delta f](x,y)=f(x,y+\Delta).
$$

\begin{theorem}
\label{Theorem_6v_invariance}
  Fix $\Delta\in\mathbb Z_{\ge 0}$ and $A\in \mathbb Z_{>0}$. Suppose that for some $i<j$, $\rho^{\ges j}_b(x,y)=0$ for $y< A$ and $\rho^{\ges i}_b(x,y)=1$ for $y\ge A-\Delta$, cf.\ Figure \ref{Fig_intersecting_BM}. Then for every $B_i>B_j$ the following two random vectors have the same first and second moments:
  \begin{equation}
  \label{eq_6v_invariance}
   \xi_{\rho^{\ges i}_b, \rho^{\ges j}_b}(B_i,B_j) \stackrel{1,2 \text{ moments}}{=}  \xi_{\rho^{\ges i}_b,S_\Delta \rho^{\ges j}_b}(B_i,B_j-\Delta).
  \end{equation}
  In words, the simultaneous shift in the same direction and by the same $\Delta$ of the boundary condition $\rho^{\ges j}_b$ and the observation point $B^{j}$ does not change the first two moments of the two-dimensional vector.
\end{theorem}
\begin{remark}
 Let us clarify the meaning of the boundary conditions $S_\Delta\rho^{\ges j}_b$.
 Note that since there are no edges of colors $\geqslant j$ underneath the line $y=A+\Delta$, we can simply shift down by $\Delta$ the function $\rho^{\ges j}$, i.e., all horizontal incoming edges of colors $\geqslant j$. Moreover, we can make this shift without affecting $\rho^{\ges i}(x,y)$, since it equals to $1$ everywhere where any changes happen, and, therefore, no contradiction with inequality $\rho^{\ges i}\ge \rho^{\ges j}$ can occur. There is also no need to change the quadrant $x>0$, $y>0$ in this transformation.
\end{remark}
\begin{remark}
 Since the definition of the stochastic colored six-vertex model is translationally invariant, we can also deal with $\displaystyle \xi_{S_{-\Delta}\rho^{\ges i}_b,\rho^{\ges j}_b}(B_i+\Delta,B_j)$, which is the same as $\xi_{\rho^{\ges i}_b,S_\Delta \rho^{\ges j}_b}(B_i,B_j-\Delta)$ by shifting everything (including the quadrant and the observation points $B_i$, $B_j$) by $\Delta$ in the vertical direction.
 \end{remark}
 \begin{remark}
  In the current form
  \eqref{eq_6v_invariance} is an empty statement at $q=1$. However, subtracting $1$ from all coordinates of the vectors in \eqref{eq_6v_invariance}, dividing by $q-1$ and sending $q\to 1$, we get a non-trivial (and valid) statement at $q=1$. We return to this in Section \ref{Section_6v_degenerations}.
 \end{remark}

We are going to prove Theorem \ref{Theorem_6v_invariance} by finding appropriate extensions of the two main ingredients of the proof of Theorem \ref{Theorem_SHE_invariance}:
\begin{itemize}
\item Expression of the covariance of the height functions of different colors as a $4$--fold integral with the main part of the integrand being the product of four transition probabilities of the Brownian motion.
\item Shift--invariance for the intersection local times for Brownian bridges.
\end{itemize}
Discrete analogues of these two steps are presented in Sections \ref{Section_4_point} and \ref{Section_local_times}, respectively. While for the heat equation, Theorem \ref{Theorem_SHE_invariance} immediately followed, for the discretization given by the colored six-vertex model additional efforts are necessary and we finish the proof of Theorem \ref{Theorem_6v_invariance} in Section \ref{Section_proof_of_inv}.

\bigskip

We can extend Theorem \ref{Theorem_6v_invariance} to a statement involving $N$--dimensional vectors. Fix $X\in\mathbb Z_{>0}+\frac12$, $\Delta=1,2,\dots$, and an index $\iota\in\{1,2,\dots,N\}$. Take two  vectors $\vec{A}\in\mathbb Z^N$, $\vec{B} \in (\mathbb Z+\frac12)^N$, such that
 \begin{equation}
 \label{eq_x16}
  0<A_1<A_2<\dots<A_{\iota-1}<A_{\iota}-\Delta<A_{\iota}<A_{\iota+1}<\dots<A_N,
 \end{equation}
 \begin{equation}
 \label{eq_x17}
  B_1,\dots,B_{\iota-1}> B_{\iota}> B_{\iota}-\Delta > B_{\iota+1},\dots,B_N>0.
 \end{equation}
 (Note that there are no inequalities between $B_1,\dots,B_{\iota-1}$ and $B_{\iota+1},\dots,B_N$.)
Given this data, we construct the initial conditions $\rho^{\ges i}(\frac12,y)=\mathbf 1(y\ge A_i)$, $i=1,\dots,N$, set $\rho^{\ges i}(x,\frac12)=0$, $i=1,\dots,N$, and consider an $N$--dimensional random vector:
$$
 \xi_{\vec A}(\vec B)=(q^{\H^{\ges 1}(X, B_1)},\dots,q^{\H^{\ges N}(X,B_N)}).
$$
\begin{corollary}
\label{Corollary_6v_invariance}
   Take $X\in\mathbb Z_{>0}+\frac12$, $\Delta=0,1,2,\dots$,  an index $\iota\in\{1,2,\dots,N\}$, and two vectors $\vec{A}, \vec{B}$ satisfying \eqref{eq_x16}, \eqref{eq_x17}.
 With the notation $\mathbf{e}_\iota$ for the $\iota$--th coordinate vector in $N$--dimensional space, the following identity of the first two moments holds:
 \begin{equation}
 \label{eq_6v_invariance_vector}
  \xi_{\vec{A}}(\vec{B})\stackrel{1,2 \text{ moments}}{=}
  \xi_{\vec{A}-\Delta \mathbf{e}_\iota}(\vec{B}-\Delta \mathbf{e}_\iota).
 \end{equation}
\end{corollary}

 Similarly to Corollary \ref{Corollary_SHE_invariance} and Theorem \ref{Theorem_SHE_invariance}, Corollary \ref{Corollary_6v_invariance} is a direct consequence of Theorem \ref{Theorem_6v_invariance}, and also more general choices of the boundary conditions are possible in the statement.

 In fact, not only the first two moments, but the full distributions of the vectors in \eqref{eq_6v_invariance_vector} coincide and Theorem \ref{Theorem_6v_intro} is an inhomogeneous version of such a statement. However, checking this is based on a different set of techniques (where the relation to intersection local times of random walks is no longer visible); it will be discussed in Sections \ref{Section_inhom_shift}, \ref{Section_Shift_inhom_proof}.

\subsection{Four point relation for the colored stochastic six-vertex model}
\label{Section_4_point}

The computation of the covariance for the vectors in \eqref{eq_6v_invariance} is based on the following relation generalising \cite[Theorem 3.1]{BG_tele}.

\begin{theorem}
\label{Theorem_4_point_colors}
 Consider the stochastic $(N+1)$--colored six--vertex model (with weights $b_1$, $b_2$, and $q=\frac{b_2}{b_1}$) in the quadrant with arbitrary (possibly, even random) boundary conditions.
 For each $x,y\in \mathbb Z_{\ge 0}+\frac12 $ we have the identities for $j=0,1,\dots,N$:
 \begin{multline}
 \label{eq_4_point_relation_colors}
  q^{\H^{\ges j}(x+1,y+1)}-b_1 \cdot q^{\H^{\ges j}(x,y+1)}
   - b_2 \cdot q^{\H^{\ges j}(x+1,y)}+(b_1+b_2-1) \cdot q^{\H^{\ges j}(x,y)}= \xi^{\ges j}(x+1,y+1),
 \end{multline}
 where the conditional expectation for fields $\xi^{\ges j}$, $j=0,1,\dots,N$, is
 \begin{equation}
 \label{eq_4_point_no_correlation_colors}
 \E\bigl[ \xi^{\ges j}(x+1,y+1) \mid \H^{\ges 1}(u,v), \dots, \H^{\ges N}(u,v), u\le x \text{ or } v\le y  \bigr]=0,
   \end{equation}
 and the conditional covariance is computed for $0\le i\le j\le N$ through
 \begin{multline}
 \label{eq_4_point_covariance_colors}
   \E \bigl[ \xi^{\ges i}(x+1,y+1) \xi^{\ges j} (x+1,y+1) \mid \H^{\ges 1}(u,v),\dots, \H^{\ges N}(u,v), u\le x \text{ or } v\le y  \bigr]\\=
   b_2(1-b_1)\Delta_x^{\ges i} \Delta_y^{\ges j}+b_1(1-b_2) \Delta_y^{\ges i} \Delta_x^{\ges j}
   \\
   - b_1(1-b_1)(1-q) q^{\H^{\ges i}(x,y)} \Delta_y^{\ges j}  + b_1(1-b_2)(1-q) q^{\H^{\ges i}(x,y)} \Delta_x^{\ges j},
 \end{multline}
 with
 $$
  \Delta_x^{\ges i}=q^{\H^{\ges i}(x+1,y)}-q^{\H^{\ges i}(x,y)},\quad \Delta_y^{\ges i}=q^{\H^{\ges i}(x,y+1)}-q^{\H^{\ges i}(x,y)},\quad i=1,\dots, N.
 $$
 \end{theorem}
\begin{proof}
 The identity is checked by sampling one vertex at $(x+\frac12,y+\frac12)$ with all possible configurations of colors of two incoming edges.

 First, suppose that $N=1$. When $i=j=1$, the statement of the theorem coincides with \cite[Theorem 3.1]{BG_tele} for the colorless six-vertex model. The proof given there involved only the local update rule (vertex weights) at $(x+\frac12,y+\frac12)$. It remains to deal with the cases when either $i$ or $j$ is equal to $0$. Note that since each edge has some color, $\H^0$ is a deterministic function, which grows linearly with slope $1$ in $y$ variable and decreases linearly with slope $-1$ in $x$ variable. Hence, for $j=0$, \eqref{eq_4_point_relation_colors} gives identical $0$ value for $\xi^{\ges 0}$. Therefore, $\xi^{\ges 0}$ satisfies \eqref{eq_4_point_no_correlation_colors} automatically. As for
 \eqref{eq_4_point_covariance_colors}, note that whenever $i=0$, the expression in the right-hand side of  \eqref{eq_4_point_covariance_colors} vanishes (as follows from $b_2 \Delta^{\ges 0}_x= b_1 (1-q) q^{\H^{\ges 0}}$ and $\Delta^{\ges 0}_y=(q-1) q^{\H^{\ges 0}}$) matching the vanishing of the left-hand side.

 The case of general $N>1$ is reduced to $N=1$, since each vertex has edges of at most two colors, which are sampled by the same rule as for $N=1$.
\end{proof}
\begin{remark}
 It is easy to see that there should be \emph{some} formula for the conditional covariance \eqref{eq_4_point_covariance_colors} because each vertex is being sampled from finitely many options. However, we do not have any \emph{a priori} reason for the formula to have the particularly simple quadratic form of \eqref{eq_4_point_covariance_colors}.
\end{remark}

We now recall how difference equations of the sort \eqref{eq_4_point_relation_colors} can be solved.
 Consider the following equation for an unknown function $\Phi(x,y)$, $x,y\in \mathbb Z_{\ge 0}+\frac12$:
\begin{equation}
\label{eq_discrete_PDE}
 \Phi(x+1,y+1)-b_1 \Phi(x,y+1)-b_2\Phi(x+1,y)+(b_1+b_2-1)\Phi(x,y)=u(x+1,y+1)
\end{equation}
with a given right-hand side $u$ and subject to boundary conditions
\begin{equation}
\label{eq_discrete_PDE_boundary}
\Phi\left(x,\tfrac12\right)=\chi(x), \quad  \Phi\left(\tfrac12,y\right)=\psi(y), \quad x,y=\tfrac12,\tfrac32,\tfrac52,\dots,\quad \chi\left(\tfrac12\right)=\psi\left(\tfrac12\right).
\end{equation}
We take $b_1$ and $b_2$ to be arbitrary distinct real numbers satisfying $0<b_1\ne b_2<1$.

Define the discrete Riemann function through
\begin{multline}
\label{eq_Discrete_R}
 \Rd(X,Y; x,y)=\frac{1}{2\pi \ii} \oint_{-\frac{1}{b_2(1-b_1)}} \left(\frac{1+  b_1(1-b_1)z}{1+ b_2(1-b_1) z}
  \right)^{X-x} \left( \frac{1+b_2 (1-b_2)z}{1+b_1(1-b_2)z} \right)^{Y-y}\\
  \times  \frac{ (b_2-b_1)\, dz}{(1+ b_2(1-b_1) z)(1+ b_1(1-b_2) z)}, \quad X\ge x, Y\ge y,
\end{multline}
where the integration goes in the positive direction and encircles $-\frac{1}{b_2(1-b_1)}$, but not $-\frac{1}{b_1(1-b_2)}$. Note that we can also integrate in the negative
direction around ${-\frac{1}{b_1(1-b_2)}}$ for the same result. For convenience, we also set $R(X,Y;x,y)=0$ whenever $X<x$ or $Y<y$.

\begin{theorem}\cite[Theorem 4.7]{BG_tele} \label{Theorem_discrete_PDE_through_integrals} The unique solution to \eqref{eq_discrete_PDE}, \eqref{eq_discrete_PDE_boundary} has the form
\begin{multline}
\label{eq_discrete_PDE_solution}
 \Phi(X,Y)=\chi\left(\tfrac12\right) \Rd\left(X,Y;\tfrac12,\tfrac12\right)+\sum_{y=\frac32}^Y \Rd\left(X,Y;\tfrac 12,y\right) \bigl(\psi(y)-b_2 \psi(y-1)\bigr)\\+\sum_{x=\tfrac32}^X \Rd\left(X,Y;x,\tfrac12\right) \bigl(\chi(x)-b_1 \chi(x-1)\bigr)
  +\sum_{x=\tfrac32}^X \sum_{y=\tfrac32}^Y \Rd(X,Y;x,y) u(x,y),
\end{multline}
where all summations run over the lattices of mesh $1$.
\end{theorem}

\subsection{Intersection local times for persistent random walks}
\label{Section_local_times}

The second ingredient of the proof of Theorem \ref{Theorem_6v_invariance} is  a result on shift-invariance of the distribution of intersection local times for persistent random walks. We could not locate the result in the literature and present it here.

We deal with a pair of independent random walks. The trajectory of each walker consists of straight segments and turns, its weight is the product of weights of elementary steps. The latter depend on two parameters $0<b_1<1$, $0<b_2<1$ and are given in Figure \ref{Fig_RW_weights}.

\begin{figure}[t]
\begin{center}
{\scalebox{1.2}{\includegraphics{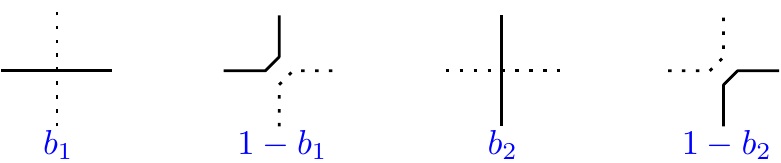}}}
 \caption{Weights for independent random walks
 \label{Fig_RW_weights}}
\end{center}
\end{figure}

We would like to count the number of intersections for two random walks. By an intersection we mean two paths passing through the same vertex. Since for each path there are four types of passages through a vertex (as in Figure \ref{Fig_RW_weights}), there are $4\cdot 4=16$ types of intersections --- at some of them the paths intersect transversally, at others they just touch each other, or they can even share the same edges of the lattice. We denote one of these 16 types by $\varpi$.

We consider a rectangle $S(A,B)$ drawn on the square grid with lower-left vertex $A$ and upper-right vertex $B$. Our fist random path $\pi_1$ is a random walk, which starts by entering horizontally from the left into vertex $A$ and travels inside the rectangle until exiting horizontally through the vertex $B$ to the right (this is a discrete version of the Brownian bridge of Section \ref{Section_SHE_proof}). The probability of a trajectory of $\pi_1$ is proportional to the product of the weights of Figure \ref{Fig_RW_weights} corresponding to its steps.

In addition, we take a point $C$ on the bottom side of $S(A,B)$ and a point $D$ on the top side of $S(A,B)$, such that $D$ is to the right (and up) from $C$. We consider the rectangle $S(C,D)$, and the second random path $\pi_2$ starts by entering $S(C,D)$ vertically through $C$ and ends by exiting $S(C,D)$ vertically through $D$. We refer to Figure \ref{Fig_RW_intersection} for an illustration  of $\pi_1$ and $\pi_2$.

For $\varpi$ being one of the 16 types of intersections, let $\In^{\varpi}(A,B;C,D)$ be the random variable, which counts the total number of intersections of $\pi_1$ and $\pi_2$ of type $\varpi$, cf.\ Figure \ref{Fig_RW_intersection}.

\begin{figure}[t]
\begin{center}
{\scalebox{1.2}{\includegraphics{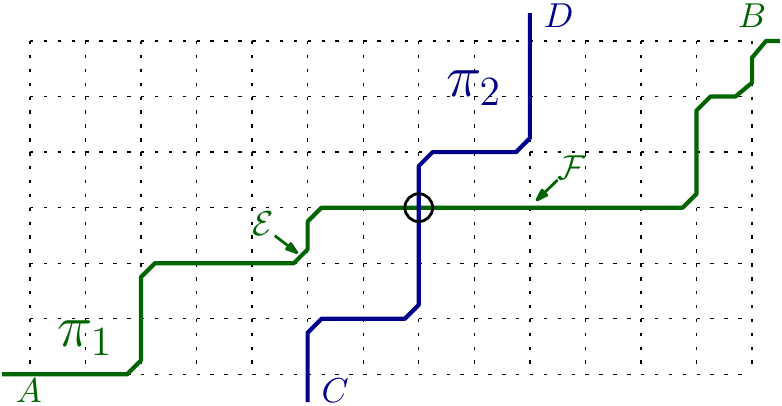}}}
 \caption{Intersections of two paths
 \label{Fig_RW_intersection}}
\end{center}
\end{figure}

\begin{theorem}
\label{Theorem_intersection_RW}
 Fix a type of intersection $\varpi$ and a rectangle $S(A,B)$. The distribution of $\In^{\varpi}(A,B;C,D)$ does not change when we shift $C,D$ simultaneously and by the same amount in $(1,0)$ direction.
\end{theorem}
\begin{proof}
 We argue by induction in the area of $S(A,B)$. If $A=B$, then necessarily also $C=D=A$ and the statement is empty.

 For the general case, let $x(A), y(A)$ denote the coordinates of $A$, and similarly for the points $B$, $C$, and $D$. We want to show that the dependence of $\In^{\varpi}$ on $C$ and $D$ is only through the difference $x(D)-x(C)$. This difference can take the values $0,1,\dots, x(B)-x(A)$. If the value is maximal, then necessarily $A=C$, $B=D$, and the statement of the theorem is empty. Hence, is suffices to consider the case when $x(B)-x(A)>x(D)-x(C)$. This implies that the area of $S(A,B)$ is larger than the area of $S(C,D)$, which will allow us to use the induction hypothesis.

 Let $\mathcal E$ denote the random point where the path $\pi_1$ enters into the rectangle $S(C,D)$. Formally, $\mathcal E$ is the unique point with $x(\mathcal E)=x(C)$ such that the edge $(\mathcal E-(1,0), \mathcal E)$ belongs to $\pi_1$. Similarly, let $\mathcal F$ denote the random point where the path $\pi_1$ leaves the rectangle $S(C,D)$. Note that all the intersections of $\pi_1$ and $\pi_2$ necessarily happen inside $S(\mathcal E,\mathcal F)$. Therefore, we can write
 \begin{multline}
 \label{eq_path_decomposition}
  \mathrm{Prob} (\In^{\varpi}(A,B;C,D)=k)\\ =\sum_{E,F} \frac{w(-A\to -E) w(-E\to F+)   w(F+\to B+)}{w(-A\to B+)} \cdot \mathrm{Prob} (\tilde \In^{\varpi} (C,D; E,F)=k),
 \end{multline}
 where the summation goes over $E$ and $F$ representing all possible values for $\mathcal E$ and $\mathcal F$, respectively,   $w(-A\to -E)$ is the total weight of paths starting by $(A-(1,0),A)$ and ending by $(E-(1,0),E)$ (i.e., product of weights of involved vertices, including the one at $A$, but not the one at $E$), $w(-E\to F+)$ is the total weight of paths starting by $(E-(1,0),E)$ and ending by $(F,F+(1,0))$ (including vertices at both $E$ and $F$), $w(F+\to B+)$ is the total weight of paths starting by $(F,F+(1,0))$ and ending by $(B,B+(1,0))$ (including $B$, but not $F$), and $w(-A\to B+)$  is the total weight of paths starting by $(A-(1,0),A)$ and ending by $(B,B+(1,0))$ (including both $A$ and $B$). Finally, $\tilde \In^{\varpi} (C,D; E,F)$ counts the number of intersection points of a random walk from $C$ to $D$ with another random walk from $E$ to $F$.

 Note that $\tilde \In^{\varpi}$ is a random variable of the same type as $\In^{\varpi}$ but for domains and paths reflected with respect to the $x=y$ axis (which, in particular, interchanges the probabilities $b_1$ and $b_2$). Hence, we can use the induction hypothesis and conclude that $\mathrm{Prob} (\tilde \In^{\varpi} (C,D; E,F)=k)$ depends on $E$ and $F$ only through the difference $y(F)-y(E)$.

 We further perform the summation in \eqref{eq_path_decomposition} in two steps: first summing over $E$ and $F$ such that $y(F)-y(E)=r$ and then summing over all possible choices of $r$. Hence, we write
 \begin{equation}
 \label{eq_path_decomposition_2}
  \mathrm{Prob} (\In^{\varpi}(A,B;C,D)=k) =\sum_{r=0}^{y(B)-y(A)} \mathrm{Prob}(y(E)-y(F)=r) \mathrm{Prob} (\tilde \In^{\varpi} (C,D; E,F)=k),
 \end{equation}
 where as $E$ and $F$ in the last sum we can choose any two points on the vertical sides of $S(C,D)$ satisfying $y(E)-y(F)=r$. We are now ready prove the invariance with respect to the simultaneous horizontal shifts of $C$ and $D$. Note that the second factor in the sum \eqref{eq_path_decomposition_2} is invariant by its definition, and therefore, we only need to prove the invariance of the first factor. Following \eqref{eq_path_decomposition}, this factor is
 \begin{equation}
  \sum_{\begin{smallmatrix}E,F:\\ y(F)-y(E)=r\end{smallmatrix}} \frac{w(-A\to -E) w(-E\to F+)   w(F+\to B+)}{w(-A\to B+)}.
 \end{equation}
By definitions, the factors $w(-E\to F+)$ and $w(-A\to B+)$ are shift-invariant. Hence, it remains to prove the shift invariance of
 \begin{equation}
 \label{eq_sum_inv}
  \sum_{\begin{smallmatrix}E,F:\\ y(F)-y(E)=r\end{smallmatrix}} w(-A\to -E)   w(F+\to B+).
 \end{equation}
By translating the path going from $F$ to $B$ so that the edge $(F,F+(1,0))$ becomes $(E-(1,0),E)$, we see that \eqref{eq_sum_inv} is the total weight of paths staring by $(A-(1,0),A)$ and ending by $(\hat B,\hat B+(1,0))$, where $\hat B=B-(x(D)-x(C)-1 ,r)$. This weight depends on $C$ and $D$ only through $x(D)-x(C)$, as desired.
\end{proof}

The following extension of Theorem \ref{Theorem_intersection_RW} is obtained by following exactly the same argument, and we omit the proof. See Figure \ref{Fig_RW_intersection_2} for some of the cases covered by the extension.

\begin{theorem} \label{Theorem_intersection_RW_extended}
  In the setting of Theorem \ref{Theorem_intersection_RW}, we can allow the direction of the incoming path at $A$ to be vertical and/or the direction of the outgoing path at $B$ to be vertical.  Moreover, the point $C$ may be below the bottom border of $S(A,B)$, and in this case, the incoming path at $C$ is allowed to have horizontal direction. Similarly, the point $D$ may be above the top border of $S(A,B)$ and in this case, the outgoing path at $D$ is allowed to have horizontal direction.

  In all these cases, the shift invariance will hold, provided that the paths $\pi_1$ and $\pi_2$ still almost surely intersect both before and after the shift.
\end{theorem}

\begin{figure}[t]
\begin{center}
{\scalebox{0.8}{\includegraphics{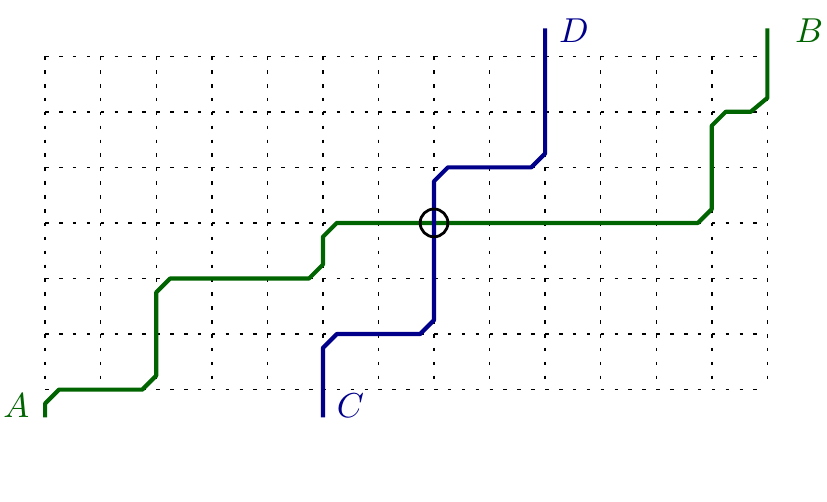}}}\qquad {\scalebox{0.8}{\includegraphics{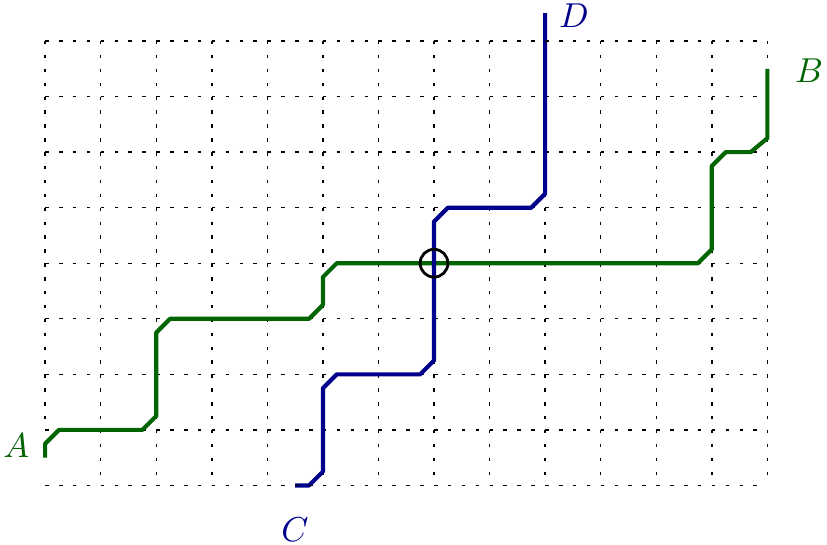}}}
 \caption{Two sample cases for Theorem \ref{Theorem_intersection_RW_extended}
 \label{Fig_RW_intersection_2}}
\end{center}
\end{figure}

Let us formulate an analytic corollary of the probabilistic statement of Theorem \ref{Theorem_intersection_RW_extended}. For two points $A,B\in\mathbb Z^2$, a monotone lattice path from $A$ to $B$ is a sequence of points $x_1,x_2,\dots,x_N$ with $x_1=A$, $x_N=B$ and each increment $x_{i+1}-x_i$ being either $(1,0)$ or $(0,1)$. In particular, in order for such a path to exist, $B-A$ must be a vector with non-negative coordinates and $N-1$ must be equal to the sum of these coordinates.

\begin{definition}
 We say that two pairs of points $A,B\in\mathbb Z^2$ and $C,D\in\mathbb Z^2$ are in \emph{intersecting position} if every monotone lattice path from $A$ to $B$ intersects (i.e., shares a vertex) with every monotone lattice path from $C$ to $D$.
\end{definition}

\begin{corollary} \label{Corollary_Rd_sum_shift}
 Take four points on the integer plane $A=(A_X,A_Y)$, $B=(B_X,B_Y)$, $C=(C_X,C_Y)$, $D=(D_X,D_Y)$ and let $\Delta=(\Delta_X,\Delta_Y)$ be an integer vector. Suppose that $A,B$ is in intersecting position both with $C,D$ and with $C+\Delta$, $D+\Delta$. Then
 \begin{multline}
 \label{eq_Rd_sum_shift}
  \sum_{x,y\in\mathbb Z} \Rd(x,y;A_X,A_Y) \Rd (B_X,B_Y;x,y) \Rd (x,y;C_X,C_Y) \Rd (D_X, D_Y;x,y)
\\= \sum_{x,y\in\mathbb Z} \Rd(x,y;A_X,A_Y) \Rd (B_X,B_Y;x,y)   \Rd (x,y;C_X+\Delta_X,C_Y+\Delta_Y) \\ \times \Rd (D_X+\Delta_X, D_Y+\Delta_Y;x,y).
 \end{multline}
\end{corollary}

 \noindent Note that since $\Rd$ vanishes for non-ordered arguments, both sums in \eqref{eq_Rd_sum_shift} are finite: the first runs over $(x,y)\in S(A,B)\cap S(C,D)$, while the second one runs over $(x,y)\in S(A,B)\cap S(C+\Delta,D+\Delta)$.
\begin{remark}
 Using \eqref{eq_Discrete_R}, the identity \eqref{eq_Rd_sum_shift} becomes an equality of two four-dimensional contour integrals. We are not aware of a more direct way for checking this equality.
\end{remark}

The proof of Corollary \ref{Corollary_Rd_sum_shift} is based on the following stochastic interpretation of the function $\Rd$. As before, we deal with persistent random walks, which travel in the up-right direction of the grid according to the probabilities of Figure \ref{Fig_RW_weights}.

\begin{proposition}
\label{Proposition_R_as_transition}
The probability that a path that started horizontally at
$(x_0+\frac12,y_0)$, ends vertically at $(X+1,Y+\frac12)$ (i.e.\ the path \emph{enters} into
$(X+1,Y+1)$ from below) is
\begin{equation}
\label{eq_x5}
 P_{-,|}(x_0,y_0; X,Y)=(1-b_1)\Rd(X,Y;x_0,y_0).
\end{equation}
The probability for a path, which started
vertically at $(x_0,y_0+\frac12)$, to end horizontally at $(X+\frac12,Y+1)$ (i.e.\ the path \emph{enters} into $(X+1,Y+1)$ from the left) is
\begin{equation}
\label{eq_x6}
 P_{|,-}(x_0,y_0; X,Y)=(1-b_2)\Rd(X,Y;x_0,y_0).
\end{equation}
\end{proposition}
\begin{proof}

We consider a particular case of the stochastic six--vertex model in the quadrant (with colorless weights) when we have only one
path. In this case the expectation of the height function has a simple probabilistic meaning: for $(X,Y)\in(\mathbb Z_{\ge 0}+\frac12) \times (\mathbb Z_{\ge 0}+\frac 12)$
\begin{align}
\label{eq_height_one_path}
 \E\left[ \frac{1-q^{H\left(X,Y\right)}}{1-q} \right] &=
 {\rm Prob}\bigl(\text{ the path passes to the right from }(X,Y)\, \bigr)
\\\notag &= {\rm Prob}\bigl(\text{ the path passes below }(X,Y)\, \bigr).
\end{align}
Suppose that the path enters the positive quadrant through the point $(1,y_0)$ coming from the left. Then by Theorem
\ref{Theorem_4_point_colors}, \eqref{eq_height_one_path} denoted as $F^-_{y_0}(X,Y)$ (the superscript
$^-$ indicates that the path enters horizontally) solves
\begin{equation}
\label{eq_x41}
 F^-_{y_0}(X+1,Y+1)-b_1 F^-_{y_0}(X,Y+1)-b_2 F^-_{y_0}(X+1,Y)+(b_1+b_2-1) F^-_{y_0}(X,Y)=0,
\end{equation}
with the boundary conditions
\begin{equation}
  F^-_{y_0}(X,\tfrac12)=0,\qquad F^-_{y_0}(\tfrac12,Y)=\begin{cases} 0,& Y<{y_0-\frac12},\\ 1, & Y\ge {y_0+\frac12}. \end{cases}
\end{equation}
 Theorem \ref{Theorem_discrete_PDE_through_integrals} gives a closed formula:
\begin{equation}
\label{eq_x42}
F^-_{y_0}(X,Y)=\Rd(X,Y;\tfrac12,y_0+\tfrac12)+(1-b_2) \sum_{y=y_0+\tfrac32}^Y \Rd(X,Y;0,y).
\end{equation}
Consider the difference in the $X$--direction:
$$
 P_{-,|}(0,y_0; X,Y):=F_{y_0}(X+\tfrac12,Y+\tfrac12)-F_{y_0}(X+\tfrac32,Y+\tfrac12), \quad X,Y\in\mathbb Z_{\ge 0}.
$$
\eqref{eq_height_one_path} implies that it computes the probability that the path, which started
horizontally at $(\frac12,y_0)$, ends vertically at $(X+1,Y+\frac12)$ (i.e.\ the path \emph{enters} into
$(X+1,Y+1)$ from below). Using \eqref{eq_x42} and invariance of $\Rd$ under simultaneous shifts of the first and third or of the second and fourth arguments, we compute
\begin{multline}
\label{eq_x1}
 P_{-,|}(0,y_0; X,Y)=\sum_{y=y_0+1}^{Y} (\Rd(X,Y;0,y)-\Rd(X+1,Y;0,y)) (1-b_2) \\ +\Rd(X,Y;0,y_0)-\Rd(X+1,Y;0,y_0).
\end{multline}
Let us investigate the sum in this formula using the shift--invariance of $\Rd$:
\begin{multline}
\label{eq_x7}
 \sum_{y=y_0+1}^{Y} (\Rd(X,Y;0,y)-\Rd(X+1,Y;0,y))\\=
 \sum_{y=y_0+1}^{Y} (\Rd(X,Y-y;0,0)-\Rd(X+1,Y-y;0,0))
 \\= \sum_{y=0}^{Y-y_0-1} (\Rd(X,y;0,0)-\Rd(X+1,y;0,0)).
\end{multline}
Using the definition \eqref{eq_Discrete_R} and omitting the third and fourth arguments being zeros, the $\Rd$ function satisfies
$$\Rd(X+1,Y')-b_1 \Rd(X,Y')-b_2\Rd(X+1,Y'-1) +(b_1+b_2-1) \Rd(X,Y'-1)=0.$$
Summing this formula for $Y'=1,\dots, Y-y_0$, we get
\begin{multline}
\label{eq_x2}
0= -(1-b_2) \sum_{y=1}^{Y-y_0-1} (\Rd(X,y;0,0)-\Rd(X+1,y;0,0))
 \\-b_2\Rd(X+1,0) +(b_1+b_2-1) \Rd(X,0)+\Rd(X+1,Y-y_0)-b_1 \Rd(X,Y-y_0).
\end{multline}
Hence, adding \eqref{eq_x2} to \eqref{eq_x1} and using \eqref{eq_x7}, we obtain
\begin{multline}
\label{eq_x3}
 P_{-,|}(0,y_0; X,Y)=(1-b_2)(\Rd(X,0;0,0)-\Rd(X+1,0;0,0))-b_2\Rd(X+1,0;0,0) \\+(b_1+b_2-1) \Rd(X,0;0,0)+\Rd(X+1,Y-y_0;0,0) -b_1 \Rd(X,Y-y_0;0,0) \\ +\Rd(X,Y-y_0;0,0)-\Rd(X+1,Y-y_0;0,0).
\end{multline}
Using $\Rd(X+1,0;0,0)=b_1 \Rd(X,0;0,0)$, we arrive at
\begin{equation}
\label{eq_x4}
 P_{-,|}(0,y_0; X,Y)=(1-b_1) \Rd(X,Y-y_0;0,0).
\end{equation}
By translation invariance, the same formula holds for the path which starts not by
entering from the left into $(1,y_0)$, but into an arbitrary point $(x_0+1,y_0)$, which yields
\eqref{eq_x5}.
By symmetry, we can also obtain similar formulas for the case when the path starts by entering  from below into a point $(x_0,y_0+1)$, arriving at \eqref{eq_x6}.
\begin{comment}
\begin{figure}[t]
\begin{center}
{\scalebox{0.6}{\includegraphics{six_vertices_lines_simple_reversed.pdf}}}
 \caption{Transition probabilities for the random walks towards the origin
 \label{Fig_RW_reversed}}
\end{center}
\end{figure}

We also need to invert time in the random walks. The paths now come with weights
 given in Figure \ref{Fig_RW_reversed} and travel in the direction of decreasing $x$
 and $y$ coordinates. The new probabilities are obtained from the ``forward paths''
 as in Figure \ref{Fig_RW_weights} by central symmetry. Let us denote the
 transition probabilities for these new random walks through $P^*$, so that the
 transition from the horizontal direction at $(X-\frac12,Y)$ to vertical direction at
 $(x_0-1,y_0-\frac12)$ happens with probability
 \begin{multline}
 \label{eq_transition_1}
  P^*_{-,|}(X,Y; x_0,y_0)=P^*_{-,|}(-X,-Y; -x_0,-y_0)= (1-b_1)\Rd(-x_0,-y_0;
  -X,-Y)\\=(1-b_1) \Rd(X,Y;x_0,y_0),
 \end{multline}
 and transition from the vertical direction at $(X,Y-\frac12)$ to the horizontal
 direction at $(x_0-\frac12,Y-1)$ happens with probability
 \begin{multline}
 \label{eq_transition_2}
  P^*_{|,-}(X,Y; x_0,y_0)=P^*_{|,-}(-X,-Y; -x_0,-y_0)= (1-b_2)\Rd(-x_0,-y_0;
  -X,-Y)\\=(1-b_2) \Rd(X,Y;x_0,y_0),
 \end{multline}
\end{comment}
\end{proof}

\begin{proof}[Proof of Corollary \ref{Corollary_Rd_sum_shift}]
The idea of the proof is to use Proposition \ref{Proposition_R_as_transition} to interpret the sum in the left-hand side of \eqref{eq_Rd_sum_shift} (multiplied by $(1-b_1)^2 (1-b_2)^2$) as the expected number of intersections of two paths, and then to use Theorem \ref{Theorem_intersection_RW_extended} to match that to the similarly interpreted right-hand side of \eqref{eq_Rd_sum_shift}. Since there are two ways to interpret the functions $\Rd$ in Proposition \ref{Proposition_R_as_transition}, we have some freedom, and different choices are necessary for different configurations of the points $A$, $B$, $C$, $D$; similarly, we use Theorem \ref{Theorem_intersection_RW_extended} for different types of intersections.

\begin{figure}[t]
\begin{center}
{\scalebox{1.3}{\includegraphics{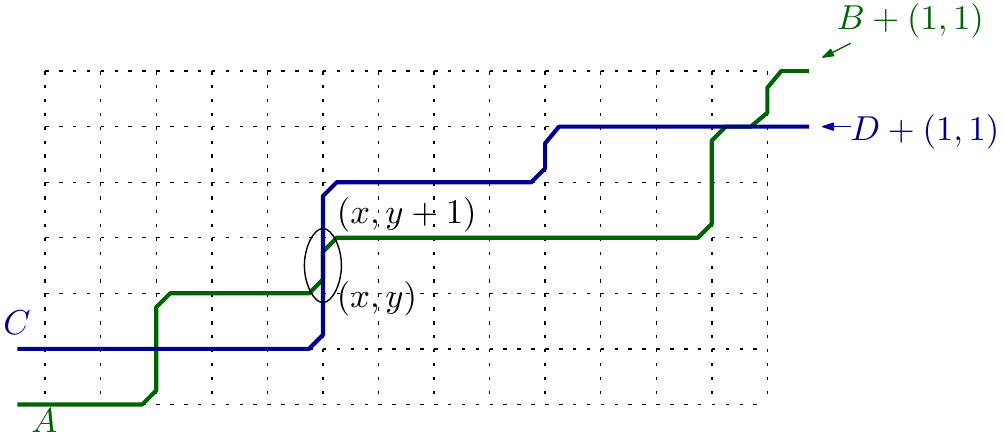}}}
 \caption{Two paths sharing $(x,y)-(x,y+1)$ edge in the proof of Corollary \ref{Corollary_Rd_sum_shift}.
 \label{Fig_RW_intersection_3}}
\end{center}
\end{figure}

Let us give more details for the cases when $A_X=C_X$, $B_X=D_X$, and $\Delta_X=0$, i.e.\ we shift in the vertical direction, cf.\ Figure \ref{Fig_RW_intersection_3}; all other cases are studied in the same way.
The product $ (1-b_1) (1-b_2) \Rd(x,y;A_X,A_Y) \Rd (B_X,B_Y;x,y)$ is the probability that a path which started by entering into the vertex $A$ from the left, passes through the edge $(x,y)-(x,y+1)$ and further enters into $B+(1,1)$ from the left.  Similarly, $(1-b_1) (1-b_2)\Rd (x,y;C_X,C_Y) \Rd (D_X, D_Y;x,y)$ is the probability that a path which started by entering into the vertex $C$ from the left, passes through the edge $(x,y)-(x,y+1)$ and further enters into $D+(1,1)$ from the left. Hence, $ (1-b_1)^2 (1-b_2)^2 \Rd(x,y;A_X,A_Y) \Rd (B_X,B_Y;x,y) \Rd (x,y;C_X,C_Y) \Rd (D_X, D_Y;x,y)$ computes the probability that two paths share an edge $(x,y)-(x,y+1)$, which is one of the allowed types of intersections in Theorem \ref{Theorem_intersection_RW_extended} (paths enter and exit $(x,y)$ vertically). Therefore, the sum over all $x$ and $y$ computes the expected number of such shared edges (equivalently, intersections). Since the right-hand side of \eqref{eq_Rd_sum_shift} admits the same interpretation as the expected number of intersections, Theorem \ref{Theorem_intersection_RW_extended} implies \eqref{eq_Rd_sum_shift}.
\end{proof}

\subsection{Proof of Theorem \ref{Theorem_6v_invariance}}
\label{Section_proof_of_inv}

The basic idea of the proof is similar to that for Theorem \ref{Theorem_SHE_invariance}. The expectations of each coordinate and each squared coordinate of the vectors in \eqref{eq_6v_invariance} coincide (by the definition of the system) between the left-hand side and the right-hand side, and we only need to show that the expectations of the product of coordinates are the same. For that we will write these expectations as a large sum, using Theorems \ref{Theorem_4_point_colors} and \ref{Theorem_discrete_PDE_through_integrals}. The invariance of this sum with respect to the shifts eventually will be a corollary of Theorem \ref{Theorem_intersection_RW_extended} (to be more precise, we rely on Corollary \ref{Corollary_Rd_sum_shift}), yet the reduction to this theorem needs some efforts.

We fix $\Delta\ge 0$ and study the covariance of the coordinates of the vector in the right-hand side of \eqref{eq_6v_invariance}. Our aim is to show that this covariance does not depend on the choice of $\Delta$. By Theorems \ref{Theorem_4_point_colors} and \ref{Theorem_discrete_PDE_through_integrals}, the covariance is
\begin{equation}
\label{eq_x18}
 \sum_{x=3/2}^X \, \sum_{y=3/2}^{B_j-\Delta}  \E \bigl[ \xi^{\ges i}(x,y) \xi^{\ges j} (x,y) \bigr] \cdot \Rd( X, B_i;x,y) \cdot \Rd(X, B_j-\Delta;x,y).
\end{equation}
We further use \eqref{eq_4_point_covariance_colors} to express the expectation in the last formula. Note that the colors $\geqslant j$ are distributed according to the boundary conditions $S_\Delta \rho_b^{\ge j}$, which implies that $\H^{\ges j}(x,y)$ is identically $0$ for $y\le A-\Delta-1/2$. Hence, the right-hand side of \eqref{eq_4_point_covariance_colors} vanishes at such points, and we conclude that the $y$--summation in \eqref{eq_x18} is $\sum_{y=A-\Delta+1/2}^{B_j-\Delta}$.

 The right-hand side of \eqref{eq_4_point_covariance_colors} is a sum of four quadratic expressions in $q^{\H^{\ges i}}$, $q^{\H^{\ges j}}$. We shift the $(x,y)$--coordinates by $1$ and rewrite  $\E \bigl[ \xi^{\ges i}(x,y) \xi^{\ges j} (x,y) \bigr]$  as the expectation of the sum of  two products:
 \begin{equation}
 \label{eq_x19}
  (1-b_1)\cdot  \E \left[ \Bigl( b_2 q^{\H^{\ges i}(x,y-1)}     - b_1 q^{\H^{\ges i}(x-1,y-1)}\Bigr) \cdot \Bigl( q^{\H{\ges j}(x-1,y)}- q^{\H{\ges j}(x-1,y-1)} \Bigr)\right]
 \end{equation}
 and
 \begin{equation}
 \label{eq_x20}
 (1-b_2)\cdot \E\left[ \Bigl( b_1q^{\H^{\ges i}(x-1,y)} -b_2 q^{\H^{\ges i}(x-1,y-1)}  \Bigr) \cdot
  \Bigl( q^{\H{\ges j}(x,y-1)}- q^{\H{\ges j}(x-1,y-1)} \Bigr)\right].
 \end{equation}
The first factors $(1-b_1)$, $(1-b_2)$ in \eqref{eq_x19}, \eqref{eq_x20} do not depend on $\Delta$, and we will be ignoring them. We split the expectation of the product in \eqref{eq_x19}, \eqref{eq_x20} into two further parts: product of expectations and covariance of the factors. We will first deal with the product of expectations and represent them as sums involving the boundary conditions $\rho_b^{\ges i}$, $\rho_b^{\ges j}$.

\begin{lemma} \label{Lemma_covariance_through_boundaries} For each $(x,y)\in (\mathbb Z_{>0}+\frac12)\times (\mathbb Z_{>0}+\frac 12)$ we have
 \begin{enumerate}
  \item $\displaystyle\E \Bigl[ q^{\H{\ges j}(x-1,y)}- q^{\H{\ges j}(x-1,y-1)} \Bigr]= \sum_{y'=A-\Delta +1/2}^{y} \psi_1^\Delta(y') \Rd\left(x,y; \frac 32, y'\right)$;
  \item $\displaystyle\E \Bigl[ q^{\H{\ges j}(x,y-1)}- q^{\H{\ges j}(x-1,y-1)} \Bigr]= \sum_{y'=A-\Delta +1/2}^{y} \psi_2^{\Delta}(y') \Rd\left(x,y; \frac 32, y'\right)$;

  \item $\displaystyle \E \Bigl[  b_2 q^{\H^{\ges i}(x,y-1)}     - b_1 q^{\H^{\ges i}(x-1,y-1)}\Bigr]=
   \sum_{y'=\frac 12}^{A-\Delta-\frac12} \psi_3(y') \Rd\left(x,y; \frac 32, y'\right) $ \\  $\displaystyle + \sum_{x'=\frac 12}^x \chi_3(x') \Rd\left(x,y; x', \frac32 \right)$;
   \item $\displaystyle \E \Bigl[  b_1q^{\H^{\ges i}(x-1,y)} -b_2 q^{\H^{\ges i}(x-1,y-1)}\Bigr]=
   \sum_{y'=\frac 12}^{A-\Delta-\frac12} \psi_4(y') \Rd\left(x,y; \frac 32, y'\right) $ \\  $\displaystyle + \sum_{x'=\frac 12}^x \chi_4(x') \Rd\left(x,y; x', \frac32 \right)$.
 \end{enumerate}
 The boundary functions $\psi_1^{\Delta},\psi_2^{\Delta}$ depend in an explicit way on $\rho^{\ges j}$; in particular, their dependence on $\Delta$ is in the linear shift of the argument: $\psi_{1/2}^{\Delta}=S_\Delta(\psi_{1/2})$. Other functions $\psi_3$, $\psi_4$, $\chi_3$, $\chi_4$ depend on $\rho^{\ges i}$ and do not depend on $\Delta$.
\end{lemma}
\begin{proof}
 The identity \eqref{eq_4_point_no_correlation_colors} implies that each of the expectations in the left-hand sides of $(1)-(4)$ satisfies the homogeneous four-point relation \eqref{eq_discrete_PDE} with $u=0$. Hence, Theorem \ref{Theorem_discrete_PDE_through_integrals} says that it can be written as a sum over boundaries of the quadrant $(\mathbb Z_{\ge 0}+\frac32)\times(\mathbb Z_{\ge 0}+\frac 32)$. What remains to show is that contributions of some parts of the boundaries vanish. For the first two cases involving $\psi_1^\Delta$ and $\psi_2^\Delta$ this is immediate, as the remaining boundary conditions for the colors $\geqslant j$ lead to identical $\H^{\ges j}\equiv 0$, and hence identically zero expectations $\E \Bigl[ q^{\H{\ges j}(x-1,y)}- q^{\H{\ges j}(x-1,y-1)} \Bigr]$ and $\E \Bigl[ q^{\H{\ges j}(x,y-1)}- q^{\H{\ges j}(x-1,y-1)} \Bigr]$ near the boundaries. For the fourth case we need to show that $\E \Bigl[  b_1q^{\H^{\ges i}(x-1,y)} -b_2 q^{\H^{\ges i}(x-1,y-1)}\Bigr]$ vanishes at $x=\frac 32$ and $y>A-\Delta$. This immediately follows from $\rho_b^{\ges i}$ being equal to $1$ for such points, as then $\H^{\ges i}(\frac32,y)$ linearly grows in $y$ with slope $1$ and two terms under the expectation cancel (recall that $\frac{b_2}{b_1}=q$).

 The third case needs a bit more care. Looking at the $y$--sum in \eqref{eq_discrete_PDE_solution}, we conclude that we need to show
\begin{equation}
\label{eq_x21}
  \E \Bigl[  b_2 q^{\H^{\ges i}(x,y-1)}     - b_1 q^{\H^{\ges i}(x-1,y-1)} -  b_2^2 q^{\H^{\ges i}(x,y-2)}     +b_1 b_2 q^{\H^{\ges i}(x-1,y-2)} \Bigr]\stackrel{?}{=}0
\end{equation}
 at $x=\frac 32$ and $y>A-\Delta$. Recall that the four-point relation \eqref{eq_4_point_no_correlation_colors} yields that for the same choice of $x$ and $y$:
 \begin{equation}
 \label{eq_x22}
  \E \Bigl[  q^{\H^{\ges i}(x,y-1)}     - b_1 q^{\H^{\ges i}(x-1,y-1)} -  b_2 q^{\H^{\ges i}(x,y-2)}     +(b_1+b_2-1) q^{\H^{\ges i}(x-1,y-2)} \Bigr]=0.
\end{equation}
Multiplying \eqref{eq_x22} by $b_2$ and subtracting from \eqref{eq_x21}, it remains to show
\begin{equation}
\label{eq_x23}
  \E \Bigl[      - b_1(1-b_2) q^{\H^{\ges i}(x-1,y-1)}      +b_2(1-b_2) q^{\H^{\ges i}(x-1,y-2)} \Bigr]\stackrel{?}{=}0,
\end{equation}
and the last identity holds at $x=\frac 32$, $y>A-\Delta$ due to linear, slope $1$ growth of $\H^{\ges i}$ in $y$ for such boundary points.
\end{proof}

We continue the proof of Theorem \ref{Theorem_6v_invariance} and deal with the first part of \eqref{eq_x19} --- product of the expectations of the factors. Combining \eqref{eq_x18} with the result of Lemma \ref{Lemma_covariance_through_boundaries}, we get the expression
\begin{multline}
\label{eq_x24}
 \sum_{x=3/2}^X \, \sum_{y=3/2}^{B_j-\Delta}  \E  \Bigl[ b_2 q^{\H^{\ges i}(x,y-1)}     - b_1 q^{\H^{\ges i}(x-1,y-1)}\Bigr] \E \Bigl[ q^{\H{\ges j}(x-1,y)}- q^{\H{\ges j}(x-1,y-1)} \Bigr] \\ \times  \Rd(X, B_i; x,y) \cdot \Rd(X, B_j-\Delta;x,y)
 \\=   \sum_{x=3/2}^X \, \sum_{y=3/2}^{B_j-\Delta} \left[\sum_{y'=A-\Delta +1/2}^{y} \psi_1^{\Delta}(y') \Rd\left(x,y; \frac 32, y'\right)\right]
 \\ \times \left[\sum_{y''=\frac 12}^{A-\Delta-\frac12} \psi_3(y'') \Rd\left(x,y; \frac 32, y''\right) + \sum_{x'=\frac 12}^x \chi_3(x') \Rd\left(x,y; x', \frac32 \right)\right]\\ \times  \Rd(X, B_i; x,y) \cdot \Rd( X, B_j-\Delta; x,y).
\end{multline}
Note that we can extend the summation in $y'$ up to $+\infty$, as the added terms do not contribute due to vanishing of the $\Rd$ function on unordered arguments. Similarly, we can extend  the summation in $x'$ up to $+\infty$. We can also think about the summation domain in $y''$ to be $\Delta$--independent, as $\psi_3$ (and the areas where it vanishes) actually does not depend on $\Delta$.

At this point we can replace $y'$ by $y'-\Delta$ in \eqref{eq_x24}, and for any fixed values of $y'$, $x'$, $y''$ apply Corollary \ref{Corollary_Rd_sum_shift} to the remaining sum in $x$ and $y$ to show its $\Delta$--independence. Note that $\psi_1^{\Delta}(y'-\Delta)=\psi_1(y')$ is $\Delta$--independent, cf.\ Lemma \ref{Lemma_covariance_through_boundaries}.

In the same way we deal with \eqref{eq_x20} and show that
\begin{multline}
\label{eq_x25}
 \sum_{x=3/2}^X \, \sum_{y=3/2}^{B_j-\Delta}  \E\Bigl[x b_1q^{\H^{\ges i}(x-1,y)} -b_2 q^{\H^{\ges i}(x-1,y-1)}  \Bigr] \E
  \Bigl[ q^{\H{\ges j}(x,y-1)}- q^{\H{\ges j}(x-1,y-1)} \Bigr]\\ \times  \Rd(X, B_i; x,y) \cdot \Rd(X, B_j-\Delta; x,y)
\end{multline}
also does not depend on $\Delta$.

Hence, it remains to deal with the covariances of the factors in \eqref{eq_x19}, \eqref{eq_x20}\footnote{This is the new part, as compared to Theorem \ref{Theorem_SHE_invariance} dealing with the stochastic heat equation. In the limit transition from the stochastic six-vertex model to the SHE such covariances vanish.}. These covariances are linear combinations of the covariances of $q^{\H^{\ges i}}$ and $q^{\H^{\ges j}}$ at points $(x,y-1)$, $(x-1,y)$, $(x-1,y-1)$. At this point we can iterate the previous procedure and use Theorems \ref{Theorem_4_point_colors} and \ref{Theorem_discrete_PDE_through_integrals} to write the covariance as a sum similar to \eqref{eq_x18}. For instance, the covariance of $q^{\H^{\ges j}(x-1,y-1)}$ and $q^{\H^{\ges i}(x-1,y-1)}$ is written as
\begin{multline}
\label{eq_x26}
\mathrm{Cov} (q^{\H^{\ges j}(x-1,y-1)},q^{\H^{\ges i}(x-1,y-1)})\\ =
 \sum_{\tilde x=3/2}^{x-1} \, \sum_{\tilde y=3/2}^{y-1}  \E \bigl[ \xi^{\ges i}(\tilde x,\tilde y) \xi^{\ges j} (\tilde x,\tilde y) \bigr] \cdot \Rd(x-1, y-1; \tilde x,\tilde y) \cdot \Rd( x-1, y-1; \tilde x,\tilde y).
\end{multline}
Proceeding as above, we express $ \E \bigl[ \xi^{\ges i}(\tilde x,\tilde y) \xi^{\ges j} (\tilde x,\tilde y) \bigr]$ using \eqref{eq_4_point_covariance_colors} as the sum of the expressions \eqref{eq_x19}, \eqref{eq_x20} with $(x,y)$ replaced by $(\tilde x, \tilde y)$, and further split each expectation of the product into the product of expectations and covariance. Next, we again use Lemma \ref{Lemma_covariance_through_boundaries}.
As a result, the part of \eqref{eq_x18} involving the covariance in \eqref{eq_x26} and the product of expectations in the expansion of this covariance,  becomes transformed into
\begin{multline}
\label{eq_x28}
 \sum_{x=3/2}^X \, \sum_{y=3/2}^{B_j-\Delta}  \Rd(X, B_i; x,y) \cdot \Rd(X, B_j-\Delta; x,y) \\ \times  \sum_{\tilde x=3/2}^{x-1} \, \sum_{\tilde y=3/2}^{y-1} \Rd(x-1, y-1; \tilde x,\tilde y) \cdot \Rd( x-1, y-1; \tilde x,\tilde y)
 \\ \times \left[\sum_{y'=A-\Delta +1/2}^{\tilde y} \psi_1^\Delta(y') \Rd\left(\tilde x,\tilde y; \frac 32, y'\right)\right]
 \\ \times \left[\sum_{y''=\frac 12}^{A-\Delta-\frac12} \psi_3(y'') \Rd\left(x,y; \frac 32, y''\right) + \sum_{x'=\frac 12}^x \chi_3(x') \Rd\left(x,y; x', \frac32 \right)\right].
\end{multline}
\begin{figure}[t]
\begin{center}
{\scalebox{1.0}{\includegraphics{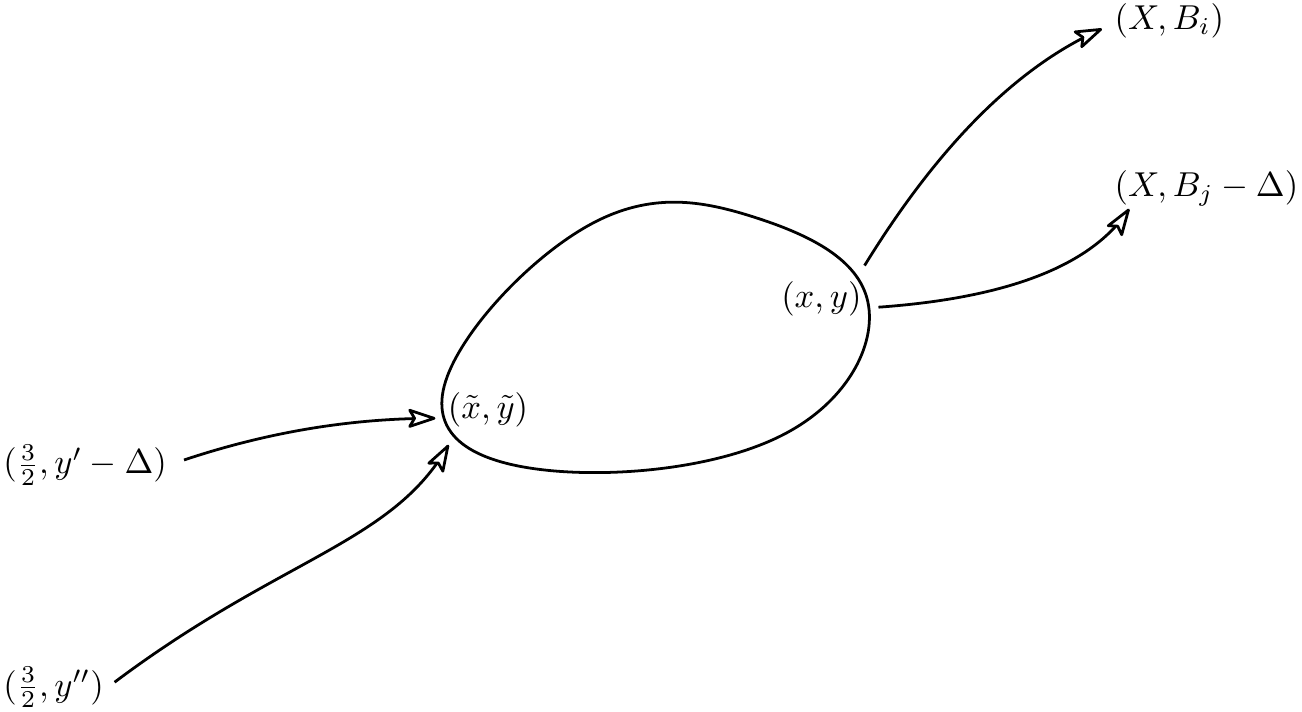}}}
 \caption{The relevant points (arguments of $\Rd$--functions) in summation \eqref{eq_x27}. For the subsequent iterations, the middle part becomes more complicated, but only these four pairs of points are relevant for the $\Delta$--independence.
 \label{Fig_sum_scheme}}
\end{center}
\end{figure}
The $\Delta$--independence of the last expression reduces to the $\Delta$--independence of the sum of products of six $\Rd$--functions, cf.\ Figure \ref{Fig_sum_scheme}:
\begin{multline}
\label{eq_x27}
 \sum_{x,y,\tilde x,\tilde y}  \Rd\left(\tilde x,\tilde y; \frac 32, y'-\Delta \right) \Rd\left(\tilde x,\tilde y; \frac 32, y''\right)\\ \times  \Rd(x-1, y-1; \tilde x,\tilde y) \cdot \Rd(x-1, y-1;\tilde x,\tilde y)\\ \times \Rd(X, B_i; x,y) \cdot \Rd(X, B_j-\Delta; x,y)
\end{multline}
and a similar expression involving $x'$. In order to see the $\Delta$--independence of \eqref{eq_x27}, we do the summation in two steps: first, we fix the differences $\delta x=x-\tilde x$, $\delta y=y-\tilde y$ and sum over all four-tuples with such differences, then we sum over all possible values for $\delta x$, $\delta y$.
Due to the invariance of the $\Rd$--function with respect to simultaneous shifts of the first and third arguments (as well as the second and the fourth arguments), the first summation already leads to $\Delta$--invariant expression by Corollary \ref{Corollary_Rd_sum_shift}. Hence, \eqref{eq_x27} is $\Delta$--invariant, and so is \eqref{eq_x28}.

It remains to deal with the covariance of $q^{\H^{\ges i}}$ and $q^{\H^{\ges j}}$ at points  $(\tilde x,\tilde y-1)$, $(\tilde x-1,\tilde y)$, $(\tilde x-1,\tilde y-1)$, and we can again iterate the previous arguments. In general, when we do $k$ iterations by computing $k$ times covariance until reaching the product of expectations, we get a combination of $(2k+3)$--fold summations (as in \eqref{eq_x28} for $k=1$), and its $\Delta$--invariance reduces to the $\Delta$--invariance of the sum of products of $(2k+2)$ $\Rd$--functions (as in \eqref{eq_x27} for $k=1$). This is schematically shown in Figure \ref{Fig_sum_scheme}. For the $\Delta$--invariance, only four of these $\Rd$--functions matters, as the rest is a translationally-invariant factor (as in the second line of \eqref{eq_x27} for $k=1$). We conclude that the $\Delta$--invariance follows from Corollary \ref{Corollary_Rd_sum_shift}.

It remains to show that we only need finitely many iterations. For that note that when we passed from $(x,y)$ to $(\tilde x, \tilde y)$,  at least one of the coordinates decreased; the same will happen in each next iteration. Hence, we will reach the boundary of the domain in finite number of steps, and there the statement becomes obvious, as the height functions become deterministic and the covariances vanish.

This finishes the proof of Theorem \ref{Theorem_6v_invariance}.

\subsection{Degeneration of the six-vertex model into SHE}

\label{Section_6v_degenerations}

In this section we outline how the shift-invariance results  for the colored stochastic six-vertex model degenerate into those for the stochastic heat equation in Section \ref{Section_SHE}.

The first step is to degenerate the six-vertex model into the asymmetric simple exclusion process (ASEP). For that we set $b_1=\eps \mathbf p$, $b_2=\eps \mathbf q$,  and observe the system in a finite neighborhood of the diagonal $(\eps^{-1} t, \eps^{-1} t)$, where $t$ plays the role of time. For the colorless ($N=1$) version in the limit $\eps\to 0$ one encounters the ASEP which is an interacting particle system on the integer lattice $\mathbb Z$ with each particle jumping to the right with intensity $\mathbf p$ and to the left with intensity $\mathbf q$; jumps to the sites already occupied by the particles are prohibited. We refer to \cite{BCG} and \cite{Ag} for the details on this limit transition.

The next step is to remove asymmetry by setting $\mathbf p=\mathbf q=1$. The result is the \emph{symmetric} simple exclusion process (SSEP). If our initial model had particles of $N+1$ different colors, then so does the SSEP. The evolution in this particle system on $\mathbb Z$ is quite simple: at each time each lattice spot $m\in\mathbb Z$ has one of the colors $0,1,\dots,N$, so that the configuration space of the model is $\{0,1,\dots,N\}^{\mathbb Z}$. Each lattice edge $(m,m+1)$ has an exponential clock of rate $1$ attached to it. Whenever the clock rings, the colors at $m$ and at $m+1$ are swapped. One can treat color $0$ as the absence of particles, and in this interpretation the $N=1$ case gives rise to the usual colorless SSEP. Note that our colored version is \emph{different} from that of \cite{Quastel}.

Finally, one can consider the large time diffusive scaling limit in the colored SSEP. The functions $\rho^i$ of Section \ref{Section_SHE} then arise as asymptotic (deterministic, but still evolving in time) densities of each of the colors $i=0,1,\dots,N$. The fields $\eta^i$ are Gaussian fluctuations of these densities. In particular, the conditions $\rho^0+\rho^1+\dots+\rho^N=1$ and $\eta^0+\eta^1+\dots+\eta^N=0$ are direct consequences of the fact that in the colored SSEP each lattice spot is necessarily occupied by exactly one of $N+1$ colors.

The limit transitions mentioned above are straightforward to prove in the model with one path, i.e.\ for the persistent random walks of Section \ref{Section_local_times}. For the full system in the quadrant and with infinitely many paths additional efforts are needed. The full proof of the transition from the stochastic six-vertex model to the colorless ASEP can be found in \cite{Ag}. The limit transition from another colored version of the SSEP to another colored version of SHE was proven in \cite{Quastel}. We expect that these proofs can be generalized to provide a rigorous justification of the limit transition from our colored stochastic six-vertex model to our version of the stochastic heat equation of Section \ref{Section_SHE}. Therefore, we treat Theorem \ref{Theorem_SHE_invariance} and Corollary \ref{Corollary_SHE_invariance} as simplified Gaussian versions of Theorem \ref{Theorem_6v_invariance} and Corollary \ref{Corollary_6v_invariance}.

\section{Inhomogeneous colored stochastic six-vertex model and Shift Theorem}

\label{Section_inhom_shift}

The aim of this section is to introduce the notations and give the statement of the general shift theorem for the colored six-vertex model, generalizing Theorems \ref{Theorem_6v_intro} and \ref{Theorem_6v_invariance}

\subsection{Stochastic $U_q(\wh{\mathfrak{sl}_{N+1}})$ vertex model}
\label{ssec:fundamental}

To begin, we recall the form of the $U_q(\wh{\mathfrak{sl}_{N+1}})$ $R$-matrix \cite{Jimbo1,Jimbo2,Bazhanov,FRT}. It acts in a tensor product $W_a \otimes W_b$ of two $(N+1)$-dimensional vector spaces, and takes the form
\begin{align}
\label{Rmat}
R_{ab}(z)
&=
\sum_{i=0}^{N}
\left(
R_z(i,i;i,i)
E^{(ii)}_a \otimes E^{(ii)}_b
\right)
\\
\nonumber
&+
\sum_{0 \leq i < j \leq N}
\left(
R_z(j,i;j,i)
E^{(ii)}_a \otimes E^{(jj)}_b
+
R_z(i,j;i,j)
E^{(jj)}_a \otimes E^{(ii)}_b
\right)
\\
\nonumber
&+
\sum_{0 \leq i < j \leq N}
\left(
R_z(j,i;i,j)
E^{(ij)}_a \otimes E^{(ji)}_b
+
R_z(i,j;j,i)
E^{(ji)}_a \otimes E^{(ij)}_b
\right)
\end{align}
where $E^{(ij)}_c \in {\rm End}(W_c)$ denotes the $(N+1) \times (N+1)$ elementary matrix with a $1$ at position $(i,j)$ and $0$ everywhere else, acting in $W_c \cong \mathbb{C}^{N+1}$. The matrix entries are rational functions of the spectral parameter $z$ and the quantization parameter $q$; they are given by
\begin{align}
\label{R-weights-a}
&
\left.
R_z(i,i;i,i)
=
1,
\quad
i \in \{0,1,\dots,N\},
\right.
\\
\nonumber
\\
\label{R-weights-bc}
&
\left.
\begin{array}{ll}
R_z(j,i;j,i)
=
\dfrac{q(1-z)}{1-qz},
&
\quad
R_z(i,j;i,j)
=
\dfrac{1-z}{1-qz}
\\ \\
R_z(j,i;i,j)
=
\dfrac{1-q}{1-qz},
&
\quad
R_z(i,j;j,i)
=
\dfrac{(1-q)z}{1-qz}
\end{array}
\right\}
\quad
i,j \in \{0,1,\dots,N\},
\quad  i<j.
\end{align}
All other matrix entries $R_z(i,j;k,\ell)$ which do not fall into a category listed above are by definition equal to $0$. The model described above differs slightly from the one listed in \cite{Jimbo2}, since its entries $R_{z}(j,i;j,i)$ and $R_{z}(i,j;i,j)$ are not symmetric for $i \not= j$. The asymmetric form that we use preserves the integrability of the model, and makes it stochastic\footnote{The weights that we use are non-negative for $0<q<1$, $0<z<1$, but might be negative for other values of parameters. In particular, in \eqref{unitarity} one of the matrices necessarily has negative matrix elements, as otherwise we would have a contradiction with \eqref{stoch}: two non-degenerate stochastic matrices can not be inverse to each other.}:
\begin{prop}
\label{prop:YB}
The $R$-matrix \eqref{Rmat} satisfies the Yang--Baxter equation and unitarity relations
\begin{align}
\label{YB}
&
R_{ab}(y/x) R_{ac}(z/x) R_{bc}(z/y)
=
R_{bc}(z/y) R_{ac}(z/x) R_{ab}(y/x),
\\
\label{unitarity}
&
R_{ab}(y/x) R_{ba}(x/y) = \mathrm{Id},
\end{align}
which hold as identities in ${\rm End}(W_a \otimes W_b \otimes W_c)$ and ${\rm End}(W_a \otimes W_b)$, respectively.
\end{prop}

\begin{prop}
\label{prop:Rstoch}
For any fixed $i,j \in \{0,1,\dots,N\}$ there holds
\begin{align}
\label{stoch}
\sum_{k=0}^{N}
\sum_{\ell=0}^{N}
R_z(i,j;k,\ell)
=
1.
\end{align}
Equivalently, all rows of the matrix \eqref{Rmat} sum to $1$.
\end{prop}
The proofs of the above two propositions are by a direct computation. Their equivalent forms are also contained in \cite{Jimbo1,Jimbo2,Bazhanov,FRT}.

We shall denote the entries of the $R$-matrix pictorially using vertices. A vertex is the intersection of an oriented horizontal and vertical line, with a state variable $i \in \{0,1,\dots,N\}$ assigned to each of the connected horizontal and vertical line segments. The $R$-matrix entries are identified with such vertices as shown below:
\begin{align}
\label{R-vert}
R_z(i,j; k,\ell)
=
\tikz{0.5}{
%\draw[densely dotted] (0.5,0) arc (0:90:0.5);
\draw[lgray,line width=1.5pt,->] (-1,0) -- (1,0);
\draw[lgray,line width=1.5pt,->] (0,-1) -- (0,1);
\node[left] at (-1,0) {\tiny $j$};\node[right] at (1,0) {\tiny $\ell$};
\node[below] at (0,-1) {\tiny $i$};\node[above] at (0,1) {\tiny $k$};
},
\quad
i,j,k,\ell \in \{0,1,\dots,N\},
\end{align}
where the dependence on the spectral parameter\footnote{We will use both of the terms {\it spectral parameter} and {\it rapidity} in this work, but for slightly different purposes. The variable attached to a lattice line will be termed ``rapidity'' whereas the argument of an $R$-matrix, which is the ratio of the vertical and horizontal rapidities passing through that vertex, will be termed ``spectral parameter''.} $z$ is implicit on the right hand side. One can interpret the above figure as the propagation of colored lattice paths through a vertex: each edge label $i \geq 1$ represents a colored path superimposed over that edge, while the case $i=0$ indicates that no path is present. The {\it incoming} paths are those situated at the left and bottom edges of the vertex; those at the right and top are called {\it outgoing}. The weight of the vertex, $R_z(i,j; k,\ell)$, vanishes identically unless the total flux of colors through the vertex is preserved, \ie\ unless the ensemble of incoming colors is the same as the ensemble of outgoing colors:
\begin{align}
\label{conserve}
R_{z}(i,j;k,\ell) = 0,
\quad
\text{unless}
\
i = k,\ j=\ell
\quad
\text{and/or}
\quad
i = \ell,\ j=k.
\end{align}
This gives rise to five categories of non-vanishing vertices, as shown in Figure \ref{fund-vert}.
\begin{figure}[t]
\begin{align*}
\begin{tabular}{|c|c|c|}
\hline
\quad
\tikz{0.5}{
    \draw[lgray,line width=1.5pt,->] (-1,0) -- (1,0);
    \draw[lgray,line width=1.5pt,->] (0,-1) -- (0,1);
    \node[left] at (-1,0) {\tiny $i$};\node[right] at (1,0) {\tiny $i$};
    \node[below] at (0,-1) {\tiny $i$};\node[above] at (0,1) {\tiny $i$};
}
\quad
&
\quad
\tikz{0.5}{
    \draw[lgray,line width=1.5pt,->] (-1,0) -- (1,0);
    \draw[lgray,line width=1.5pt,->] (0,-1) -- (0,1);
    \node[left] at (-1,0) {\tiny $i$};\node[right] at (1,0) {\tiny $i$};
    \node[below] at (0,-1) {\tiny $j$};\node[above] at (0,1) {\tiny $j$};
}
\quad
&
\quad
\tikz{0.5}{
    \draw[lgray,line width=1.5pt,->] (-1,0) -- (1,0);
    \draw[lgray,line width=1.5pt,->] (0,-1) -- (0,1);
    \node[left] at (-1,0) {\tiny $i$};\node[right] at (1,0) {\tiny $j$};
    \node[below] at (0,-1) {\tiny $j$};\node[above] at (0,1) {\tiny $i$};
}
\quad
\\[1.3cm]
\quad
$1$
\quad
&
\quad
$\dfrac{q(1-z)}{1-qz}$
\quad
&
\quad
$\dfrac{1-q}{1-qz}$
\quad
\\[0.7cm]
\hline
&
\quad
\tikz{0.5}{
    \draw[lgray,line width=1.5pt,->] (-1,0) -- (1,0);
    \draw[lgray,line width=1.5pt,->] (0,-1) -- (0,1);
    \node[left] at (-1,0) {\tiny $j$};\node[right] at (1,0) {\tiny $j$};
    \node[below] at (0,-1) {\tiny $i$};\node[above] at (0,1) {\tiny $i$};
}
\quad
&
\quad
\tikz{0.5}{
    \draw[lgray,line width=1.5pt,->] (-1,0) -- (1,0);
    \draw[lgray,line width=1.5pt,->] (0,-1) -- (0,1);
    \node[left] at (-1,0) {\tiny $j$};\node[right] at (1,0) {\tiny $i$};
    \node[below] at (0,-1) {\tiny $i$};\node[above] at (0,1) {\tiny $j$};
}
\quad
\\[1.3cm]
&
\quad
$\dfrac{1-z}{1-qz}$
\quad
&
\quad
$\dfrac{(1-q)z}{1-qz}$
\quad
\\[0.7cm]
\hline
\end{tabular}
\end{align*}

\caption{Five types of vertices and corresponding weights. We assume that $0 \leq i < j \leq N$.
These are the pictorial representations of the five types of weights in \eqref{R-weights-a}, \eqref{R-weights-bc}. \label{fund-vert}
}
\end{figure}

Having set up these vertex notations, the relations of Propositions \ref{prop:YB} and \ref{prop:Rstoch} then have simple graphical interpretations. The Yang--Baxter equation \eqref{YB} becomes
\begin{align}
\label{graph-YB}
\sum_{0 \leq k_1,k_2,k_3 \leq N}
\tikz{0.6}{
\draw[lgray,line width=1.5pt,->]
(-2,1) node[above,scale=0.6] {\color{black} $i_1$} -- (-1,0) node[below,scale=0.6] {\color{black} $k_1$} -- (1,0) node[right,scale=0.6] {\color{black} $j_1$};
\draw[lgray,line width=1.5pt,->]
(-2,0) node[below,scale=0.6] {\color{black} $i_2$} -- (-1,1) node[above,scale=0.6] {\color{black} $k_2$} -- (1,1) node[right,scale=0.6] {\color{black} $j_2$};
\draw[lgray,line width=1.5pt,->]
(0,-1) node[below,scale=0.6] {\color{black} $i_3$} -- (0,0.5) node[scale=0.6] {\color{black} $k_3$} -- (0,2) node[above,scale=0.6] {\color{black} $j_3$};
\node[left] at (-2.2,1) {$x \rightarrow$};
\node[left] at (-2.2,0) {$y \rightarrow$};
\node[below] at (0,-1.4) {$\uparrow$};
\node[below] at (0,-2) {$z$};
}
\quad
=
\quad
\sum_{0 \leq k_1,k_2,k_3 \leq N}
\tikz{0.6}{
\draw[lgray,line width=1.5pt,->]
(-1,1) node[left,scale=0.6] {\color{black} $i_1$} -- (1,1) node[above,scale=0.6] {\color{black} $k_1$} -- (2,0) node[below,scale=0.6] {\color{black} $j_1$};
\draw[lgray,line width=1.5pt,->]
(-1,0) node[left,scale=0.6] {\color{black} $i_2$} -- (1,0) node[below,scale=0.6] {\color{black} $k_2$} -- (2,1) node[above,scale=0.6] {\color{black} $j_2$};
\draw[lgray,line width=1.5pt,->]
(0,-1) node[below,scale=0.6] {\color{black} $i_3$} -- (0,0.5) node[scale=0.6] {\color{black} $k_3$} -- (0,2) node[above,scale=0.6] {\color{black} $j_3$};
\node[left] at (-1.5,1) {$x \rightarrow$};
\node[left] at (-1.5,0) {$y \rightarrow$};
\node[below] at (0,-1.4) {$\uparrow$};
\node[below] at (0,-2) {$z$};
}
\end{align}
for all fixed indices $i_1,i_2,i_3,j_1,j_2,j_3 \in \{0,1,\dots,N\}$; the unitarity relation \eqref{unitarity} becomes
\begin{align}
\label{graph-unitarity}
\sum_{0 \leq k_1,k_2 \leq N}
\tikz{0.6}{
\draw[lgray,line width=1.5pt,->,rounded corners] (-1,0) node[left,scale=0.6] {\color{black} $i_1$} -- (1,0) node[below,scale=0.6] {\color{black} $k_1$} -- (1,2) node[above,scale=0.6] {\color{black} $j_1$};
\draw[lgray,line width=1.5pt,->,rounded corners] (0,-1) node[below,scale=0.6] {\color{black} $i_2$} -- (0,1) node[above,scale=0.6] {\color{black} $k_2$} -- (2,1) node[right,scale=0.6] {\color{black} $j_2$};
\node[left] at (-1.5,0) {$x \rightarrow$};
\node[below] at (0,-1.4) {$\uparrow$};
\node[below] at (0,-2) {$y$};
}
=
{\bm 1}_{i_1=j_1} {\bm 1}_{i_2=j_2},
\end{align}
for all fixed indices $i_1,i_2,j_1,j_2 \in \{0,1,\dots,N\}$; and the stochasticity relation \eqref{stoch} reads
\begin{align}
\label{stoch-graph}
\sum_{k=0}^{N}
\sum_{\ell=0}^{N}
\tikz{0.5}{
\draw[lgray,line width=1.5pt,->] (-1,0) -- (1,0);
\draw[lgray,line width=1.5pt,->] (0,-1) -- (0,1);
\node[left] at (-1,0) {\tiny $j$};\node[right] at (1,0) {\tiny $\ell$};
\node[below] at (0,-1) {\tiny $i$};\node[above] at (0,1) {\tiny $k$};
}
=
1,
\end{align}
for all fixed indices $i,j \in \{0,1,\dots,N\}$. We will use these graphical identities frequently throughout the text.

\subsection{Vertex splitting}
\label{ssec:split}

A basic property of the $R$-matrix \eqref{Rmat} is that when its spectral parameter is set to $1$, it reduces to a permutation matrix. More precisely, analyzing the matrix entries \eqref{R-weights-a} and \eqref{R-weights-bc}, we see that both $R_z(j,i;j,i)$ and $R_z(i,j;i,j)$ vanish at $z=1$, while all remaining entries assume the value $1$ at $z=1$. From a graphical point of view, this can be understood as a ``splitting'' of the vertex:
\begin{align}
\label{R-split}
R_{1}(i,j; k,\ell)
=
{\bm 1}_{i=\ell}
\cdot
{\bm 1}_{j=k}
=
\tikz{0.5}{
\draw[lgray,line width=1.5pt,->] plot coordinates {(-1,0)(-0.25,0)(0,0.25)(0,1)};
\draw[lgray,line width=1.5pt,->] plot coordinates {(0,-1)(0,-0.25)(0.25,0)(1,0)};
\node[left] at (-1,0) {\tiny $j$};\node[right] at (1,0) {\tiny $\ell$};
\node[below] at (0,-1) {\tiny $i$};\node[above] at (0,1) {\tiny $k$};
},
\quad
i,j,k,\ell \in \{0,1,\dots,N\}.
\end{align}

\subsection{Reflection symmetry of vertex weights}
\label{ssec:reflect}

Another elementary property of the $U_q(\widehat{\mathfrak{sl}_{N+1}})$ vertex model is the invariance of its Boltzmann weights under the combined operation of (i) reflecting the vertex about its SW--NE diagonal, and (ii) replacing all colors $i$ that appear by their conjugate $N-i$. More precisely, one has the symmetry
\begin{align}
\label{reflect}
\tikz{0.5}{
\draw[lgray,line width=1.5pt,->] (-1,0) -- (1,0);
\draw[lgray,line width=1.5pt,->] (0,-1) -- (0,1);
\node[left] at (-1,0) {\tiny $j$};\node[right] at (1,0) {\tiny $\ell$};
\node[below] at (0,-1) {\tiny $i$};\node[above] at (0,1) {\tiny $k$};
}
=
\tikz{0.5}{
\draw[lgray,line width=1.5pt,->] (-1,0) -- (1,0);
\draw[lgray,line width=1.5pt,->] (0,-1) -- (0,1);
\node[left] at (-1,0) {\tiny $N-i$};\node[right] at (1,0) {\tiny $N-k$};
\node[below] at (0,-1) {\tiny $N-j$};\node[above] at (0,1) {\tiny $N-\ell$};
},
\quad
i,j,k,\ell \in \{0,1,\dots,N\}.
\end{align}
This is immediately verified by consulting the table of Figure \ref{fund-vert} of vertices that have non-vanishing weights.

\subsection{Color-merging}

\label{Section_merging}

Let us now discuss an important property of the colored six-vertex model with weights of Figure \ref{fund-vert} which will be indispensable in our future proofs. We will refer to it as {\it color-merging}; it can be viewed as a refinement of the stochasticity property \eqref{stoch-graph}.

In what follows, let $\mathcal{C} = \{c,c+1,\dots\} \subset \{0,1,\dots,N\}$ denote a contiguous subset of colors, where $0 \leq c \leq N$ and the number of elements in $\mathcal{C}$ is arbitrary. We denote by $\bar{\mathcal{C}}$ the complementary set:
\begin{align*}
\bar{\mathcal{C}} = \{0,1,\dots,N\} \backslash \mathcal{C}.
\end{align*}
Given such a subset $\mathcal{C}$, we define an associated color projection:
\begin{align}
\label{merge-A}
[k]_{\mathcal{C}} =
\left\{
\begin{array}{ll}
c, & \qquad k \in \mathcal{C},
\\ \\
k, & \qquad k \in \bar{\mathcal{C}},
\end{array}
\right.
\end{align}
for all $k \in \{0,1,\dots,N\}$.
\begin{prop}
\label{prop-merge}
We have the following vertex relations, obtained by constraining the value of the color at one of the outgoing edges to lie in the complement of $\mathcal{C}$, while summing the color at the remaining outgoing edge over all values in $\mathcal{C}$:
\begin{align}
\label{merge1}
\sum_{k \in \mathcal{C}}
\tikz{0.5}{
\draw[lgray,line width=1.5pt,->] (-1,0) -- (1,0);
\draw[lgray,line width=1.5pt,->] (0,-1) -- (0,1);
\node[left] at (-1,0) {\tiny $j$};\node[right] at (1,0) {\tiny $\ell \in \bar{\mathcal{C}}$};
\node[below] at (0,-1) {\tiny $i$};\node[above] at (0,1) {\tiny $k$};
}
=
\tikz{0.5}{
\draw[lgray,line width=1.5pt,->] (-1,0) -- (1,0);
\draw[lgray,line width=1.5pt,->] (0,-1) -- (0,1);
\node[left] at (-1,0) {\tiny $[j]_{\mathcal{C}}$};\node[right] at (1,0) {\tiny $\ell$};
\node[below] at (0,-1) {\tiny $[i]_{\mathcal{C}}$};\node[above] at (0,1) {\tiny $c$};
},
\quad\quad
\sum_{\ell \in \mathcal{C}}
\tikz{0.5}{
\draw[lgray,line width=1.5pt,->] (-1,0) -- (1,0);
\draw[lgray,line width=1.5pt,->] (0,-1) -- (0,1);
\node[left] at (-1,0) {\tiny $j$};\node[right] at (1,0) {\tiny $\ell$};
\node[below] at (0,-1) {\tiny $i$};\node[above] at (0,1)
{\tiny $k \in \bar{\mathcal{C}}$};
}
=
\tikz{0.5}{
\draw[lgray,line width=1.5pt,->] (-1,0) -- (1,0);
\draw[lgray,line width=1.5pt,->] (0,-1) -- (0,1);
\node[left] at (-1,0) {\tiny $[j]_{\mathcal{C}}$};\node[right] at (1,0) {\tiny $c$};
\node[below] at (0,-1) {\tiny $[i]_{\mathcal{C}}$};\node[above] at (0,1) {\tiny $k$};
}.
\end{align}
By summing both of the outgoing edges over values in the set $\mathcal{C}$, one has
\begin{align}
\label{merge2}
\sum_{k \in \mathcal{C}}
\sum_{\ell \in \mathcal{C}}
\tikz{0.5}{
\draw[lgray,line width=1.5pt,->] (-1,0) -- (1,0);
\draw[lgray,line width=1.5pt,->] (0,-1) -- (0,1);
\node[left] at (-1,0) {\tiny $j$};\node[right] at (1,0) {\tiny $\ell$};
\node[below] at (0,-1) {\tiny $i$};\node[above] at (0,1) {\tiny $k$};
}
=
\tikz{0.5}{
\draw[lgray,line width=1.5pt,->] (-1,0) -- (1,0);
\draw[lgray,line width=1.5pt,->] (0,-1) -- (0,1);
\node[left] at (-1,0) {\tiny $[j]_{\mathcal{C}}$};\node[right] at (1,0) {\tiny $c$};
\node[below] at (0,-1) {\tiny $[i]_{\mathcal{C}}$};\node[above] at (0,1) {\tiny $c$};
}
=
\mathbf{1}_{i \in \mathcal{C}}
\cdot
\mathbf{1}_{j \in \mathcal{C}}.
\end{align}
\end{prop}
For the proof it suffices to notice that in each sum in \eqref{merge1}  there is a single non-zero element in the left-hand side. Similarly, in \eqref{merge2} there are at most two non-zero elements in the left-hand side.

In what follows, we make use of two variants of the notation \eqref{merge-A}. Fix an integer $0 \leq m \leq N$. When $\mathcal{C} = \{m,m+1,\dots,N\}$ we define
\begin{align}
\label{merge-up}
[k]_m
=
\left\{
\begin{array}{ll}
k, \qquad\qquad & k \in \{0,1,\dots,m-1\},
\\
\\
m, \qquad\qquad & k \in \{m,m+1,\dots,N\}.
\end{array}
\right.
\end{align}
In a similar vein, when $\mathcal{C} = \{0,1,\dots,m\}$ we define
\begin{align}
\label{merge-down}
[k]^m
=
\left\{
\begin{array}{ll}
0, \qquad\qquad & k \in \{0,1,\dots,m\},
\\
\\
k, \qquad\qquad & k \in \{m+1,\dots,N\}.
\end{array}
\right.
\end{align}

\subsection{Down-right paths and domains}

\begin{defn}[Down-right paths]
A {\it down-right path} $P$ of length $M$ is a sequence of lattice points $P = (a_1,b_1) \rightarrow (a_2,b_2) \rightarrow \cdots \rightarrow (a_{M+1},b_{M+1})$, such that for all $1 \leq k \leq M$ one has
\begin{align}
(a_{k+1},b_{k+1}) = (a_k+1,b_k),
\qquad
\text{or}
\qquad
(a_{k+1},b_{k+1}) = (a_k,b_k-1).
\end{align}
\end{defn}

In any down-right path, an edge of the form $(a_k,b_k) \rightarrow (a_k+1,b_k)$ is called {\it horizontal}, while edges of the form $(a_k,b_k) \rightarrow (a_k,b_k-1)$ are called {\it vertical}.

For each coordinate $(a_k,b_k)$ of a down-right path, the quantity $a_k+b_k$ tells us on which anti-diagonal of the lattice we are situated; see the left panel of Figure \ref{fig:domains}. Given two down-right paths
\begin{align*}
P &=
(a_1,b_1) \rightarrow
(a_2,b_2) \rightarrow
\cdots
\rightarrow (a_{M+1},b_{M+1}),
\\
\tilde{P} &=
(\tilde{a}_1,\tilde{b}_1) \rightarrow
(\tilde{a}_2,\tilde{b}_2) \rightarrow
\cdots \rightarrow (\tilde{a}_{M+1},\tilde{b}_{M+1})
\end{align*}
with the same starting point, $(a_1,b_1)=(\tilde a_1,\tilde b_1)$,
we say that $P \geq \tilde{P}$ if $a_k + b_k \geq \tilde{a}_k + \tilde{b}_k$ for all $1 < k < M+1$; that is, path $P$ can never move to an anti-diagonal which is strictly below that of path $\tilde{P}$. Equivalently, $P$ and $\tilde{P}$ are allowed to touch or overlap along edges of the lattice, but never to cross each other. See the right panel of Figure \ref{fig:domains} for an illustration.

\begin{figure}[t]
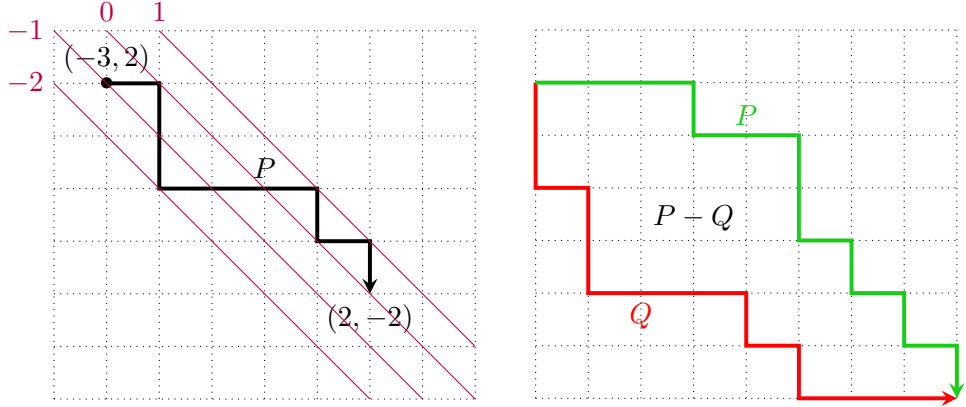

\begin{align*}
\tikz{0.7}{
\foreach\x in {0,...,8}{
\draw[dotted] (\x,-4) -- (\x,3);
}
\foreach\y in {-4,...,3}{
\draw[dotted] (0,\y) -- (8,\y);
}
\node at (1,2) {$\bullet$};
\node[above] at (1,2) {$(-3,2)$};
\node[below] at (6,-2) {$(2,-2)$};
\node[above] at (4,0) {$P$};
\draw[line width=1.5pt,->] (1,2) -- (2,2) -- (2,0) -- (5,0) -- (5,-1) -- (6,-1) -- (6,-2);
%anti-diagonals
\draw[purple] (0,2) -- (6,-4); \node[left,purple] at (0,2) {$-2$};
\draw[purple] (0,3) -- (7,-4); \node[left,purple] at (0,3) {$-1$};
\draw[purple] (1,3) -- (8,-4); \node[above,purple] at (1,3) {$0$};
\draw[purple] (2,3) -- (8,-3); \node[above,purple] at (2,3) {$1$};
}
\qquad
\tikz{0.7}{
\foreach\x in {0,...,8}{
\draw[dotted] (\x,-4) -- (\x,3);
}
\foreach\y in {-4,...,3}{
\draw[dotted] (0,\y) -- (8,\y);
}
\node[above] at (3,-1) {$P-Q$};
\node[above,green] at (4,1) {$P$};
\node[below,red] at (2,-2) {$Q$};
\draw[line width=1.5pt,->,red] (0,2) -- (0,0) -- (1,0) -- (1,-2) -- (4,-2) -- (4,-3) -- (5,-3) -- (5,-4) -- (8,-4);
\draw[line width=1.5pt,->,green] (0,2) -- (3,2) -- (3,1) -- (5,1) -- (5,-1) -- (6,-1) -- (6,-2) -- (7,-2) -- (7,-3) -- (8,-3) -- (8,-4);
%phantom
\phantom{
\node[above,purple] at (1,3) {$1$};
}
}
\end{align*}
\caption{Left panel: a down-right path $P$ of length $9$, given by the sequence $P = (-3,2) \rightarrow (-2,2) \rightarrow (-2,1) \rightarrow (-2,0) \rightarrow (-1,0) \rightarrow (0,0) \rightarrow (1,0) \rightarrow (1,-1) \rightarrow (2,-1) \rightarrow (2,-2)$. The anti-diagonals of the lattice are indicated on the picture. Right panel: two down-right paths, $P$ (shown in green) and $Q$ (shown in red), beginning and ending at the same points, and satisfying $P\geq Q$. The region cut out by the two paths is the down-right domain $P-Q$.}
\label{fig:domains}
\end{figure}

\begin{defn}[Down-right domains]
Let $M \geq 2$ be a positive integer. Fix two points in the lattice $(a,b)$ and $(c,d)$ such that $a < c$, $b > d$, and $c-a+b-d = M$. We denote the set of all down-right paths beginning at $(a,b)$ and terminating at $(c,d)$ by $\mathfrak{P}_M\{ (a,b) \rightarrow (c,d) \}$; the index ``$M$'' is due to the fact that any such path has length $M$.

Choose two down-right paths $P,Q \in \mathfrak{P}_M\{ (a,b) \rightarrow (c,d) \}$ such that $P \geq Q$. The corresponding {\it down-right domain} is the region framed by paths $P$ and $Q$, we denote it by $P-Q$; see the right panel of Figure \ref{fig:domains}.
\end{defn}

\begin{defn}[Concatenation]
\label{defn:concat}
Let $P_1 \in \mathfrak{P}_{M_1}\{ (a,b) \rightarrow (c,d) \}$ and $P_2 \in \mathfrak{P}_{M_2}\{ (c,d) \rightarrow (e,f) \}$ be two down-right paths of length $M_1$ and $M_2$, respectively, which share $(c,d)$ as a common end/start point. We define a down-right path of length $M_1+M_2$ by concatenating the two paths, and denote this by $P_1 \cup P_2$.
\end{defn}

\subsection{Partition functions on down-right domains}

Given a down-right domain $P-Q$, we now put some decorations on the paths $P$ and $Q$ which will allow us to view the cut-out region as a partition function in the colored six-vertex model with weights of Figure \ref{fund-vert}.

\begin{defn}[Colored down-right paths]
Let $P$ be a down-right path of length $M$. A {\it coloring} of $P$ is an assignment of a nonnegative integer $i_k$, $1 \leq k \leq M$, to each of the $M$ edges traced out by $P$. We denote such a coloring by
\begin{align}
\Big[P; (i_1,\dots,i_M) \Big]
=
(a_1,b_1) \xrightarrow{i_1}
(a_2,b_2) \xrightarrow{i_2}
\cdots
\xrightarrow{i_M} (a_{M+1},b_{M+1}).
\end{align}
\end{defn}

\begin{defn}[Rapidity assignments]
\label{def:rapidity}
Let $P$ be the down-right path
\begin{align*}
P =
(a_1,b_1) \rightarrow
(a_2,b_2) \rightarrow
\cdots
\rightarrow (a_{M+1},b_{M+1}),
\end{align*}
and fix another down-right path $Q \in \mathfrak{P}_{M}\{(a_1,b_1) \rightarrow (a_{M+1},b_{M+1})\}$ such that $P > Q$.

For each vertical edge $(a_i,b_i) \rightarrow (a_{i+1},b_{i+1}) \in Q$, there is a well-defined row of squares in $P-Q$ which lie to the right of this edge (with the same vertical coordinate). We assign that row a complex rapidity $x_i$. Similarly, for each horizontal edge $(a_j,b_j) \rightarrow (a_{j+1},b_{j+1}) \in Q$, there is a well-defined column of squares in $P-Q$ which lie above this edge (with the same horizontal coordinate). We assign this column a complex rapidity $y_j$.

Performing this assignment for all edges in $Q$, every square in $P-Q$ acquires a unique horizontal and vertical rapidity. We call the resulting labelled domain the {\it rapidity assignment} of $P-Q$; see the right panel of Figure \ref{fig:PF} for an illustration.
\end{defn}

\begin{figure}[t]
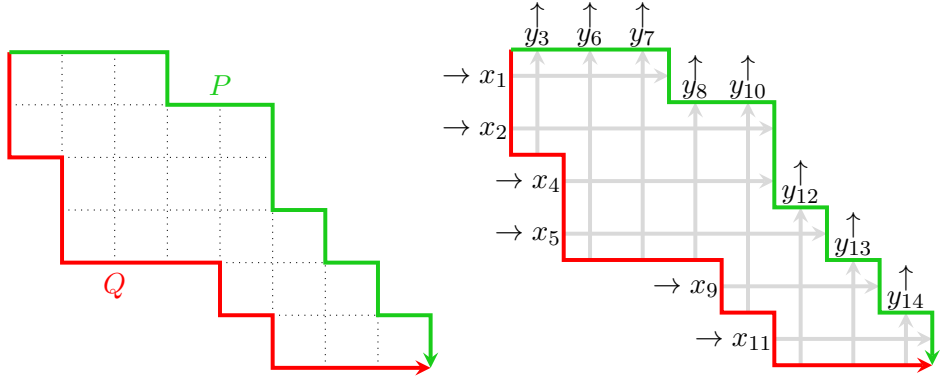

\begin{align*}
\tikz{0.7}{
%hdots
\draw[dotted] (0,1) -- (3,1);
\draw[dotted] (1,0) -- (5,0);
\draw[dotted] (1,-1) -- (5,-1);
\draw[dotted] (4,-2) -- (6,-2);
\draw[dotted] (5,-3) -- (7,-3);
%vdots
\draw[dotted] (1,0) -- (1,2);
\draw[dotted] (2,-2) -- (2,2);
\draw[dotted] (3,-2) -- (3,1);
\draw[dotted] (4,-2) -- (4,1);
\draw[dotted] (5,-3) -- (5,-1);
\draw[dotted] (6,-4) -- (6,-2);
\draw[dotted] (7,-4) -- (7,-3);
\node[above,green] at (4,1) {$P$};
\node[below,red] at (2,-2) {$Q$};
\draw[line width=1.5pt,->,red] (0,2) -- (0,0) -- (1,0) -- (1,-2) -- (4,-2) -- (4,-3) -- (5,-3) -- (5,-4) -- (8,-4);
\draw[line width=1.5pt,->,green] (0,2) -- (3,2) -- (3,1) -- (5,1) -- (5,-1) -- (6,-1) -- (6,-2) -- (7,-2) -- (7,-3) -- (8,-3) -- (8,-4);
%phantom
\phantom{
\node[below] at (0.5,3.2) {$\uparrow$}; \node[below] at (0.5,2.7) {$y_1$};
}
}
\tikz{0.7}{
%%hdots
%\draw[dotted] (0,1) -- (3,1);
%\draw[dotted] (1,0) -- (5,0);
%\draw[dotted] (1,-1) -- (5,-1);
%\draw[dotted] (4,-2) -- (6,-2);
%\draw[dotted] (5,-3) -- (7,-3);
%%vdots
%\draw[dotted] (1,0) -- (1,2);
%\draw[dotted] (2,-2) -- (2,2);
%\draw[dotted] (3,-2) -- (3,1);
%\draw[dotted] (4,-2) -- (4,1);
%\draw[dotted] (5,-3) -- (5,-1);
%\draw[dotted] (6,-4) -- (6,-2);
%\draw[dotted] (7,-4) -- (7,-3);
%hlines
\draw[lgray,line width=1.5pt,->] (0,1.5) -- (3,1.5);
\draw[lgray,line width=1.5pt,->] (0,0.5) -- (5,0.5);
\draw[lgray,line width=1.5pt,->] (1,-0.5) -- (5,-0.5);
\draw[lgray,line width=1.5pt,->] (1,-1.5) -- (6,-1.5);
\draw[lgray,line width=1.5pt,->] (4,-2.5) -- (7,-2.5);
\draw[lgray,line width=1.5pt,->] (5,-3.5) -- (8,-3.5);
%vlines
\draw[lgray,line width=1.5pt,->] (0.5,0) -- (0.5,2);
\draw[lgray,line width=1.5pt,->] (1.5,-2) -- (1.5,2);
\draw[lgray,line width=1.5pt,->] (2.5,-2) -- (2.5,2);
\draw[lgray,line width=1.5pt,->] (3.5,-2) -- (3.5,1);
\draw[lgray,line width=1.5pt,->] (4.5,-3) -- (4.5,1);
\draw[lgray,line width=1.5pt,->] (5.5,-4) -- (5.5,-1);
\draw[lgray,line width=1.5pt,->] (6.5,-4) -- (6.5,-2);
\draw[lgray,line width=1.5pt,->] (7.5,-4) -- (7.5,-3);
%y spectral
\node[below] at (0.5,3.1) {$\uparrow$}; \node[below] at (0.5,2.6) {$y_3$};
\node[below] at (1.5,3.1) {$\uparrow$}; \node[below] at (1.5,2.6) {$y_6$};
\node[below] at (2.5,3.1) {$\uparrow$}; \node[below] at (2.5,2.6) {$y_7$};
\node[below] at (3.5,2.1) {$\uparrow$}; \node[below] at (3.5,1.6) {$y_8$};
\node[below] at (4.5,2.1) {$\uparrow$}; \node[below] at (4.5,1.6) {$y_{10}$};
\node[below] at (5.5,0.1) {$\uparrow$}; \node[below] at (5.5,-0.4) {$y_{12}$};
\node[below] at (6.5,-0.9) {$\uparrow$}; \node[below] at (6.5,-1.4) {$y_{13}$};
\node[below] at (7.5,-1.9) {$\uparrow$}; \node[below] at (7.5,-2.4) {$y_{14}$};
%x spectral
\node[right] at (-1.5,1.45) {$\rightarrow x_1$};
\node[right] at (-1.5,0.45) {$\rightarrow x_2$};
\node[right] at (-0.5,-0.55) {$\rightarrow x_4$};
\node[right] at (-0.5,-1.55) {$\rightarrow x_5$};
\node[right] at (2.5,-2.55) {$\rightarrow x_9$};
\node[right] at (3.3,-3.55) {$\rightarrow x_{11}$};
\draw[line width=1.5pt,->,red] (0,2) -- (0,0) -- (1,0) -- (1,-2) -- (4,-2) -- (4,-3) -- (5,-3) -- (5,-4) -- (8,-4);
\draw[line width=1.5pt,->,green] (0,2) -- (3,2) -- (3,1) -- (5,1) -- (5,-1) -- (6,-1) -- (6,-2) -- (7,-2) -- (7,-3) -- (8,-3) -- (8,-4);
%\node at (0,2) {$\bullet$};
}
\end{align*}
\caption{A down-right domain $P-Q$ (on the left), and its conversion to a collection of vertices (on the right) with the rapidity assignment described in Definition \ref{def:rapidity}. The spectral parameter at any given vertex in $P-Q$ is the ratio $y_j/x_i$ of the rapidity variables passing through it.}
\label{fig:PF}
\end{figure}

\begin{defn}[Partition functions]
\label{def:PF}
Let $P$ and $Q$ be two down-right paths of length $M$ which form a down-right domain $P-Q$. We place vertices in the centers of the elementary squares of the lattice on which $P$ and $Q$ are defined, thus forming a subset of the dual square lattice (see the right panel of Figure \ref{fig:PF}). Assume that $P$ and $Q$ carry the colorings $[P; (i_1,\dots,i_M)]$ and $[Q;(j_1,\dots,j_M)]$. The domain $P-Q$ (viewed from the dual lattice perspective) then bears a color $i_k$ at its $k$-th outgoing edge and a color $j_k$ at its
$k$-th incoming edge, for all $1 \leq k \leq M$, and has a well-defined rapidity assignment for each horizontal and vertical line; it can thus be viewed as a partition function in the colored six-vertex model by choosing in the weights of Figure \ref{fund-vert}, $z=y_i/x_j$ for the intersection of the lines with rapidities $x_j$ and $y_i$. We denote this partition function by
\begin{align}
\label{pf-def}
Z_{P/Q}\Big[(x),(y); (i_1,\dots,i_M) \Big| (j_1,\dots,j_M) \Big]
\equiv
Z_{P/Q}\Big[ i_1,\dots,i_M \Big| j_1,\dots,j_M \Big].
\end{align}
\end{defn}

%\begin{defn}[$k$-factorization]
%Let $P$ and $Q$ be two down-right paths of length $M$, such that $P \geq Q$, and consider the partition function $Z_{P/Q}$ as defined in equation \eqref{pf-def}. Fix an integer $1 \leq k \leq M$. We say that $Z_{P/Q}$ has a $k$-factorization if
%%
%\begin{multline*}
%Z_{P/Q}\Big[ i_1,\dots,i_M \Big| j_1,\dots,j_M \Big]
%=
%\\
%Z_{P_1/Q_1}\Big[ i_1,\dots,i_k \Big| j_1,\dots,j_k \Big]
%Z_{P_2/Q_2}\Big[ i_{k+1},\dots,i_M \Big| j_{k+1},\dots,j_M \Big],
%\end{multline*}
%%
%for some down-right paths $P_1,Q_1$ of length $k$ and $P_2,Q_2$ of length $M-k$. %such that $P = P_1 \cup P_2$ and $Q = Q_1 \cup Q_2$.
%\end{defn}

\subsection{Probabilistic interpretation}

In view of the stochasticity \eqref{stoch-graph} of the weights in Figure \ref{fund-vert}, in the case when $q$ and rapidities are so that all the weights are positive, we can interpret $Z_{P/Q}[i_1,\dots,i_M | j_1,\dots,j_M]$ as the probability of seeing the color sequence $(i_1,\dots,i_M)$ as one walks along the down-right path $P$, given that the color sequence $(j_1,\dots,j_M)$ was previously observed by performing a similar walk along $Q$. We have
\begin{align}
\label{sum-to-1}
\sum_{(i_1,\dots,i_M)}
Z_{P/Q}\Big[
i_1,\dots,i_M \Big| j_1,\dots,j_M
\Big]
=1,
\end{align}
where the sum is taken over all vectors $(i_1,\dots,i_M) \in
(\mathbb{Z}_{\geq 0})^M$. On the other hand, it follows from the local conservation property \eqref{conserve} that $Z_{P/Q}[ i_1,\dots,i_M | j_1,\dots,j_M ]$ is zero unless the outgoing color sequence $(i_1,\dots,i_M)$ is a permutation of the incoming sequence $(j_1,\dots,j_M)$. We can therefore simplify \eqref{sum-to-1} as follows:
\begin{align*}
\sum_{
(i_1,\dots,i_M) \in \mathfrak{S}(j_1,\dots,j_M)}
Z_{P/Q}\Big[
i_1,\dots,i_M \Big| j_1,\dots,j_M
\Big]
=1,
\end{align*}
where the summation is over all elements in the set
$\mathfrak{S}(j_1,\dots,j_M)$, defined as
\begin{align}
\label{perms}
\mathfrak{S}(j_1,\dots,j_M)
=
\Big\{ (i_1,\dots,i_M) \in (\mathbb{Z}_{\geq 0})^M:
\exists\ \sigma \in \mathfrak{S}_M\
\text{such that}\ \sigma(i_1,\dots,i_M) = (j_1,\dots,j_M) \Big\}.
\end{align}

\subsection{Colored height functions}

The key random variables in the setting of the colored six-vertex model with weights of Figure \ref{fund-vert} are the {\it colored height functions}:
\begin{defn}[Colored height functions]
\label{def:height}
Let $[P;(i_1,\dots,i_M)]$ be a colored down-right path
\begin{align*}
\Big[P; (i_1,\dots,i_M) \Big]
=
(a_1,b_1) \xrightarrow{i_1}
(a_2,b_2) \xrightarrow{i_2}
\cdots
\xrightarrow{i_M} (a_{M+1},b_{M+1}).
\end{align*}
Fix two integers $1 \leq k \leq M$ and $m \geq 0$. We define the height function $\mathcal{H}^{\ges m}(P;k)$ of {\it level} $m$, situated at the lattice point $(a_k,b_k) \in P$, as follows:
\begin{align}
\label{height}
\mathcal{H}^{\ges m}(P;k)
=
|\{ \ell \geq k : i_{\ell} \geq m \}|.
\end{align}
In other words, $\mathcal{H}^{\ges m}(P;k)$ counts the number of colors $i_{\ell}$ with value $m$ or greater assigned to the edges $(a_{\ell},b_{\ell}) \rightarrow (a_{\ell+1},b_{\ell+1}) \in P$, where $\ell$ ranges over all values
$k \leq \ell \leq M$.
\end{defn}

Given an initial, colored down-right path $[Q;(j_1,\dots,j_M)]$ and another down-right path $P > Q$, we will be interested in the probability that $\mathcal{H}^{\ges m}(P;k) = h$, where $h \geq 0$ is some fixed nonnegative integer. This is computed as
\begin{align}
\mathbb{P}
\left[
\mathcal{H}^{\ges m}(P;k)
=
h
\right]
=
\sum_{(i_1,\dots,i_M) \in \mathfrak{S}(j_1,\dots,j_M)}
Z_{P/Q}\Big[ i_1,\dots,i_M \Big| j_1,\dots,j_M \Big]
{\bm 1}\Big( \mathcal{H}^{\ges m}(P;k) = h \Big).
\end{align}
More generally, one can consider multi-point versions of such probabilities:
\begin{multline}
\mathbb{P}
\left[
\mathcal{H}^{\ges m_1}(P;k_1)
=
h_1,
\dots,
\mathcal{H}^{\ges m_p}(P;k_p)
=
h_p
\right]
\\
=
\sum_{(i_1,\dots,i_M) \in \mathfrak{S}(j_1,\dots,j_M)}
Z_{P/Q}\Big[ i_1,\dots,i_M \Big| j_1,\dots,j_M \Big]
\prod_{\ell=1}^{p}
{\bm 1}\Big( \mathcal{H}^{\ges m_{\ell}}(P;k_{\ell}) = h_{\ell} \Big).
\end{multline}

\subsection{Color-merging of partition functions}

The result of Proposition \ref{prop-merge} naturally extends to partition functions on down-right domains. Consider the $U_q(\widehat{\mathfrak{sl}_{N+1}})$ vertex model on a domain $P-Q$, where $P$ and $Q$ are down-right paths of length $M$. Fix a subset $\mathcal{A} \subset \{1,\dots,M\}$; this subset will indicate the positions along the path $P$ where its associated colors assume some definite value. Write $\bar{\mathcal{A}} = \{1,\dots,M\} \backslash \mathcal{A}$ for the complementary set. Let $\mathcal{C} = \{c,c+1,\dots\} \subset \{0,1,\dots,N\}$ be a contiguous subset of the colors, and $\bar{\mathcal{C}} = \{0,1,\dots,N\} \backslash \mathcal{C}$ its complement, as before.
\begin{prop}
\label{prop:merge-PQ}
For all $\alpha \in \mathcal{A}$, let $i_{\alpha} \in \bar{\mathcal{C}}$. One then has the equality
\begin{align}
\label{PQ-merge}
\sum_{
(i_{\alpha \in \bar{\mathcal{A}}}) \in \mathcal{C}
}
Z_{P/Q}
\Big[
i_1,\dots,i_M
\Big|
j_1,\dots,j_M
\Big]
=
Z_{P/Q}
\Big[
a_1,\dots,a_M
\Big|
[j_1]_{\mathcal{C}},\dots,[j_M]_{\mathcal{C}}
\Big],
\end{align}
where (on the left hand side) each color $i_{\alpha}$ such that $\alpha \in \bar{\mathcal{A}}$ is summed over all values in $\mathcal{C}$, and where (on the right hand side) we have defined
\begin{align}
\label{ap}
a_p
=
\left\{
\begin{array}{ll}
i_{p},
\qquad\qquad
&
p \in \mathcal{A},
\\
\\
c,
\qquad\qquad
&
p \in \bar{\mathcal{A}}.
\end{array}
\right.
\end{align}
\end{prop}

\begin{proof}
Straightforward induction on the number of vertices in $P-Q$.
\end{proof}

\begin{rmk}
When $\mathcal{A}$ is chosen to be equal to the empty set, we have
$\bar{\mathcal{A}} = \{1,\dots,M\}$ and \eqref{PQ-merge} becomes
\begin{align}
\label{Z-stoch}
\sum_{(i_1,\dots,i_M) \in \mathcal{C}}
Z_{P/Q}
\Big[
i_1,\dots,i_M
\Big|
j_1,\dots,j_M
\Big]
=
Z_{P/Q}
\Big[
c,\dots,c
\Big|
[j_1]_{\mathcal{C}},\dots,[j_M]_{\mathcal{C}}
\Big]
=
\prod_{\alpha=1}^{M}
\bm{1}_{j_{\alpha} \in \mathcal{C}}.
\end{align}
If, in addition, $\mathcal{C} = \{0,1,\dots,N\}$, this reproduces the stochasticity result \eqref{sum-to-1}. In this way, one can view \eqref{PQ-merge} as a generalization of the sum-to-unity property \eqref{sum-to-1}, where we restrict which outgoing indices are summed (via the set $\bar{\mathcal{A}}$) and the colors that the sum is taken over (via the set $\mathcal{C}$).
\end{rmk}

\subsection{Shift Theorem}

\label{sec:shift}

The Shift Theorem is about the equality of two different partition functions in the colored six-vertex model. These partition functions are built from the following data:
\begin{itemize}
\item A nonnegative integer $m \leq N$, which labels a distinguished color in the vertex model with weights of Figure \eqref{fund-vert};

\item A nonnegative integer $h$, used to specify the value of certain height functions that appear;

\item Two down-right paths $P$ and $Q$ of length $M$ (with the same beginning/ending points) such that $P > Q$, used to specify the down-right domain on which the partition functions live;

\item Two integers $1 \leq k,\ell \leq M$ such that the $k$-th step of path $P$ and the $\ell$-th step of $Q$ are both downward\footnote{We could also consider the scenario where the $k$-th step of path $P$ and the $\ell$-th step of $Q$ are both rightward, as we discuss in Remark \ref{rmk:reflect}.};

\item Two sets $\mathcal{A} \subset \{1,\dots,k-1\}$ and
$\mathcal{B} \subset \{k+1,\dots,M\}$, used to label positions along the path $P$ where the outgoing colors assume definite values.

\item Their complements, denoted by $\bar{\mathcal{A}} = \{1,\dots,k-1\} \backslash \mathcal{A}$ and $\bar{\mathcal{B}} = \{k+1,\dots,M\} \backslash \mathcal{B}$;

\item Two vectors $(i_1,\dots,i_M)$ and $(j_1,\dots,j_M)$ which specify the outgoing and incoming colors of the down-right domain, respectively.
\end{itemize}
The outgoing and incoming colors, $(i_1,\dots,i_M)$ and $(j_1,\dots,j_M)$, require further specification. In both partition functions that we consider, we will assume that the outgoing colors satisfy
\begin{align}
\label{out-conditions}
i_{\alpha}
\left\{
\begin{array}{ll}
\fixin \{0,1,\dots,m-1\},
& \quad \alpha \in \mathcal{A},
\\ \\
\sumin \{m,m+1,\dots,N\},
& \quad \alpha \in \bar{\mathcal{A}},
\end{array}
\right.
\qquad\quad
i_{\beta}
\left\{
\begin{array}{ll}
\fixin \{m+1,m+2,\dots,N\},
& \quad \beta \in \mathcal{B},
\\ \\
\sumin \{0,1,\dots,m\},
& \quad \beta \in \bar{\mathcal{B}},
\end{array}
\right.
\end{align}
with the remaining color $i_k$ satisfying $i_k \sumin \{0,1,\dots,N\}$, where we write ``f'' or ``s'' over the ``$\in$'' symbol to keep track of which indices will be kept fixed, and which ones will be summed over the corresponding set, respectively. In a similar vein, in both partition functions to be considered we assume that the incoming colors are chosen such that
\begin{align}
\label{in-conditions}
j_{\alpha}
\left\{
\begin{array}{ll}
\fixin \{m+1,m+2,\dots,N\},
& \quad \alpha \in \{1,\dots,\ell-1\},
\\ \\
= m, & \quad \alpha = \ell,
\\ \\
\fixin \{0,1,\dots,m-1\},
& \quad \alpha \in \{\ell+1,\dots,M\}.
\end{array}
\right.
\end{align}
We now introduce our two partition functions.\footnote{To save space when writing these functions, we will not write the explicit dependence on $m$, $h$, $k$, or $\ell$, treating these integers as fixed in all quantities.} They are defined formally below; for a pictorial representation, see Figure \ref{fig:phi}. The first partition function is denoted $\Phi_{P/Q}$; it depends on a vector of incoming colors $(j_1,\dots,j_M)$ which satisfy the constraints \eqref{in-conditions}, two sets
$\mathcal{A} \subset \{1,\dots,k-1\}$ and $\mathcal{B} \subset \{k+1,\dots,M\}$ which label positions along $P$ where outgoing colors are fixed, and two vectors $(i_{\alpha \in \mathcal{A}})$ and $(i_{\beta \in \mathcal{B}})$ satisfying \eqref{out-conditions} which specify the colors at those locations. We define
\begin{multline}
\label{phi}
\Phi_{P/Q}
\Big(
\mathcal{A},\mathcal{B};
(i_{\alpha \in \mathcal{A}}),
(i_{\beta \in \mathcal{B}});
(j_1,\dots,j_M)
\Big)
\\
=
\sum_{i_k}
\sum_{
\substack{(i_{\alpha \in \bar{\mathcal{A}}})
\\ \vspace{-0.1cm} \\
(i_{\beta \in \bar{\mathcal{B}}})
}}
Z_{P/Q}\Big[ i_1,\dots,i_M \Big| j_1,\dots,j_M \Big]
{\bm 1}\Big( \mathcal{H}^{\ges m}(P;k+1) = h \Big)
\end{multline}
where $i_k$ is summed over all values $\{0,1,\dots,N\}$, while
$(i_{\alpha \in \bar{\mathcal{A}}})$ and $(i_{\beta \in \bar{\mathcal{B}}})$ are summed according to the constraints \eqref{out-conditions}. In the summand,
$Z_{P/Q}$ denotes the partition function of Definition \ref{def:height}, while $\mathcal{H}^{\ges m}(P;k+1)$ is the level $m$ height function situated at the
$(k+1)$-th lattice point of $P$.

\begin{figure}[t]
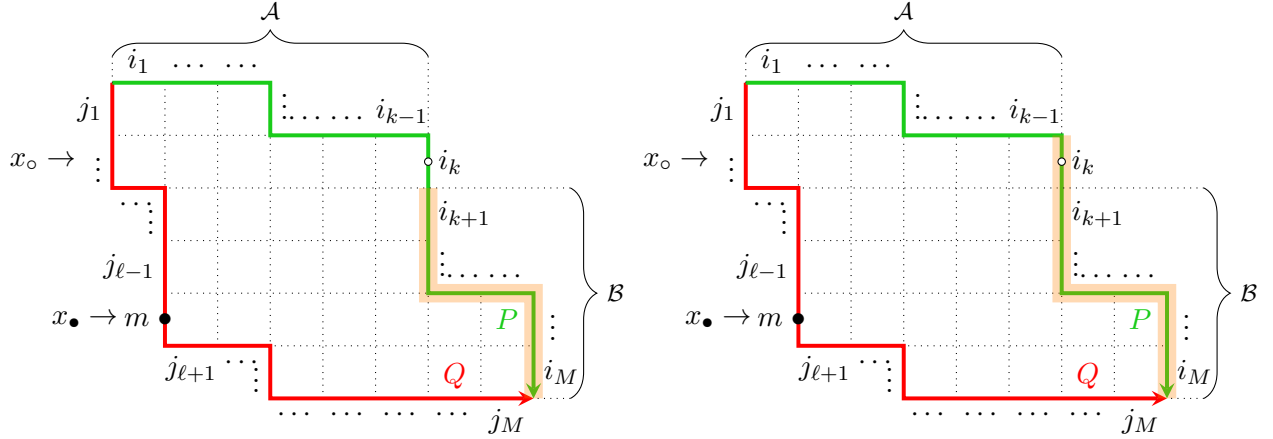

\begin{align*}
\tikz{0.7}{
%hdots
\draw[dotted] (0,1) -- (3,1);
\draw[dotted] (0,0) -- (6,0);
\draw[dotted] (1,-1) -- (6,-1);
\draw[dotted] (1,-2) -- (8,-2);
\draw[dotted] (2,-3) -- (8,-3);
%vdots
\draw[dotted] (1,-2) -- (1,2);
\draw[dotted] (2,-3) -- (2,2);
\draw[dotted] (3,-4) -- (3,1);
\draw[dotted] (4,-4) -- (4,1);
\draw[dotted] (5,-4) -- (5,1);
\draw[dotted] (6,-4) -- (6,0);
\draw[dotted] (7,-4) -- (7,-2);
%paths
\node[left,green] at (7.9,-2.5) {$P$};
\node[above,red] at (6.5,-4) {$Q$};
\draw[line width=1.5pt,->,red] (0,2) -- (0,0) -- (1,0) -- (1,-2) -- (1,-3) -- (3,-3) -- (3,-4) -- (8,-4);
\draw[line width=1.5pt,->,green] (0,2) -- (3,2) -- (3,1) -- (6,1) -- (6,-2) -- (8,-2) -- (8,-4);
%bottom labels
\node[left] at (0,1.5) {$j_1$};
\node[left] at (0,0.5) {$\vdots$};
\node[below] at (0.5,0) {$\cdots$};
\node[left] at (1,-0.5) {$\vdots$};
\node[left] at (1,-1.5) {$j_{\ell-1}$};
\node[left] at (0.9,-2.5) {$m$}; \node at (1,-2.5) {$\bullet$};
\node[below] at (1.5,-3) {$j_{\ell+1}$};
\node[left] at (3,-3.5) {$\vdots$};
\node[below] at (2.5,-3) {$\cdots$};
\node[below] at (3.5,-4) {$\cdots$};
\node[below] at (4.5,-4) {$\cdots$};
\node[below] at (5.5,-4) {$\cdots$};
\node[below] at (6.5,-4) {$\cdots$};
\node[below] at (7.5,-4) {$j_M$};
%top labels
\node[above] at (0.5,2) {$i_1$};
\node[above] at (1.5,2) {$\cdots$};
\node[above] at (2.5,2) {$\cdots$};
\node[right] at (3,1.7) {$\vdots$};
\node[above] at (3.7,1) {$\cdots$};
\node[above] at (4.5,1) {$\cdots$};
\node[above] at (5.5,1) {$i_{k-1}$};
\node[right] at (6,0.5) {$i_k$}; \filldraw[draw=black,fill=white] (6,0.5) circle (2pt);
\node[right] at (6,-0.5) {$i_{k+1}$};
\node[right] at (6,-1.3) {$\vdots$};
\node[above] at (6.7,-2) {$\cdots$};
\node[above] at (7.5,-2) {$\cdots$};
\node[right] at (8.1,-2.5) {$\vdots$};
\node[right] at (8,-3.5) {$i_M$};
%braces
\draw[decorate,decoration={brace,amplitude=10pt}] (0,2.5) -- (6,2.5)
node [black,midway,yshift=0.6cm] {\footnotesize $\mathcal{A}$};
\draw[dotted] (0,2.5) -- (0,2);
\draw[dotted] (6,2.5) -- (6,1);
\draw[decorate,decoration={brace,amplitude=10pt}] (8.7,0) -- (8.7,-4)
node [black,midway,xshift=0.6cm] {\footnotesize $\mathcal{B}$};
\draw[dotted] (8.7,0) -- (6,0);
\draw[dotted] (8.7,-4) -- (8,-4);
%height function
\draw[line width=7,orange,opacity=0.3] (6,0) -- (6,-2) -- (8,-2) -- (8,-4);
%rapidities
\node[left] at (-0.5,0.5) {$x_{\circ} \rightarrow$};
\node[left] at (0.3,-2.5) {$x_{\bullet} \rightarrow$};
}
%%%%
\tikz{0.7}{
%hdots
\draw[dotted] (0,1) -- (3,1);
\draw[dotted] (0,0) -- (6,0);
\draw[dotted] (1,-1) -- (6,-1);
\draw[dotted] (1,-2) -- (8,-2);
\draw[dotted] (2,-3) -- (8,-3);
%vdots
\draw[dotted] (1,-2) -- (1,2);
\draw[dotted] (2,-3) -- (2,2);
\draw[dotted] (3,-4) -- (3,1);
\draw[dotted] (4,-4) -- (4,1);
\draw[dotted] (5,-4) -- (5,1);
\draw[dotted] (6,-4) -- (6,0);
\draw[dotted] (7,-4) -- (7,-2);
%paths
\node[left,green] at (7.9,-2.5) {$P$};
\node[above,red] at (6.5,-4) {$Q$};
\draw[line width=1.5pt,->,red] (0,2) -- (0,0) -- (1,0) -- (1,-2) -- (1,-3) -- (3,-3) -- (3,-4) -- (8,-4);
\draw[line width=1.5pt,->,green] (0,2) -- (3,2) -- (3,1) -- (6,1) -- (6,-2) -- (8,-2) -- (8,-4);
%height function
\draw[line width=7,orange,opacity=0.3] (6,1) -- (6,-2) -- (8,-2) -- (8,-4);
%bottom labels
\node[left] at (0,1.5) {$j_1$};
\node[left] at (0,0.5) {$\vdots$};
\node[below] at (0.5,0) {$\cdots$};
\node[left] at (1,-0.5) {$\vdots$};
\node[left] at (1,-1.5) {$j_{\ell-1}$};
\node[left] at (0.9,-2.5) {$m$}; \node at (1,-2.5) {$\bullet$};
\node[below] at (1.5,-3) {$j_{\ell+1}$};
\node[left] at (3,-3.5) {$\vdots$};
\node[below] at (2.5,-3) {$\cdots$};
\node[below] at (3.5,-4) {$\cdots$};
\node[below] at (4.5,-4) {$\cdots$};
\node[below] at (5.5,-4) {$\cdots$};
\node[below] at (6.5,-4) {$\cdots$};
\node[below] at (7.5,-4) {$j_M$};
%top labels
\node[above] at (0.5,2) {$i_1$};
\node[above] at (1.5,2) {$\cdots$};
\node[above] at (2.5,2) {$\cdots$};
\node[right] at (3,1.7) {$\vdots$};
\node[above] at (3.7,1) {$\cdots$};
\node[above] at (4.5,1) {$\cdots$};
\node[above] at (5.5,1) {$i_{k-1}$};
\node[right] at (6,0.5) {$i_k$}; \filldraw[draw=black,fill=white] (6,0.5) circle (2pt);
\node[right] at (6,-0.5) {$i_{k+1}$};
\node[right] at (6,-1.3) {$\vdots$};
\node[above] at (6.7,-2) {$\cdots$};
\node[above] at (7.5,-2) {$\cdots$};
\node[right] at (8.1,-2.5) {$\vdots$};
\node[right] at (8,-3.5) {$i_M$};
%braces
\draw[decorate,decoration={brace,amplitude=10pt}] (0,2.5) -- (6,2.5)
node [black,midway,yshift=0.6cm] {\footnotesize $\mathcal{A}$};
\draw[dotted] (0,2.5) -- (0,2);
\draw[dotted] (6,2.5) -- (6,1);
\draw[decorate,decoration={brace,amplitude=10pt}] (8.7,0) -- (8.7,-4)
node [black,midway,xshift=0.6cm] {\footnotesize $\mathcal{B}$};
\draw[dotted] (8.7,0) -- (6,0);
\draw[dotted] (8.7,-4) -- (8,-4);
%rapidities
\node[left] at (-0.5,0.5) {$x_{\circ} \rightarrow$};
\node[left] at (0.3,-2.5) {$x_{\bullet} \rightarrow$};
}
\end{align*}
\caption{Left panel: the partition function $\Phi_{P/Q}$. A vector of colors $(j_1,\dots,j_M)$, satisfying the constraints \eqref{in-conditions}, enter the lattice along the edges traced out by $Q$. Colors $(i_1,\dots,i_M)$, satisfying the constraints \eqref{out-conditions}, exit the lattice along the edges traced out by $P$; some of these colors, namely $(i_{\alpha \in \mathcal{A}})$ and $(i_{\beta \in \mathcal{B}})$, assume fixed values, while the rest are summed over. The shaded part of the path $P$ indicates the extra height function constraint that we impose; there must be exactly $h$ colors of value $m$ or greater passing through these edges. Right panel: the partition function $\Psi_{P/Q}$.
%We assume almost identical boundary conditions to those used to define $\Phi_{P/Q}$. The same vector of colors $(j_1,\dots,j_M)$ enters along $Q$; we again fix the colors $(i_{\alpha \in \mathcal{A}})$ and $(i_{\beta \in \mathcal{B}})$ exiting along $P$, and sum over the remaining edges.
The sole difference is that now the shaded part of $P$ is extended by one additional vertical step, and we require that there are exactly $h$ colors of value $m+1$ or greater passing through these edges.}
\label{fig:phi}
\end{figure}

The second partition function, denoted $\Psi_{P/Q}$, is defined in a very similar manner. In fact, it differs from $\Phi_{P/Q}$ only in the way that we condition on the value of the height function in the summand. We define
\begin{multline}
\label{psi}
\Psi_{P/Q}
\Big(
\mathcal{A},\mathcal{B};
(i_{\alpha \in \mathcal{A}}),
(i_{\beta \in \mathcal{B}});
(j_1,\dots,j_M)
\Big)
\\
=
\sum_{i_k}
\sum_{
\substack{(i_{\alpha \in \bar{\mathcal{A}}})
\\ \vspace{-0.1cm} \\
(i_{\beta \in \bar{\mathcal{B}}})
}}
Z_{P/Q}\Big[ i_1,\dots,i_M \Big| j_1,\dots,j_M \Big]
{\bm 1}\Big( \mathcal{H}^{\ges m+1}(P;k) = h \Big),
\end{multline}
where $i_k$ is summed over all values $\{0,1,\dots,N\}$, while
$(i_{\alpha \in \bar{\mathcal{A}}})$ and $(i_{\beta \in \bar{\mathcal{B}}})$ are once again summed according to the constraints \eqref{out-conditions}. This time, in contrast, $\mathcal{H}^{\ges m+1}(P;k)$ is the level $m+1$ height function situated at the $k$-th lattice point of $P$ (both the position and the level were shifted).

Both quantities $\Phi_{P/Q}$ and $\Psi_{P/Q}$ depend implicitly on a collection of horizontal $(x)$ and vertical rapidities $(y)$, and we will usually suppress this dependence. Two of these parameters, however, will play a distinguished role in the sequel. Since (by assumption) the $k$-th step of path $P$ and $\ell$-th step of path $Q$ are both downward, it follows that $\Phi_{P/Q}$ and $\Psi_{P/Q}$ both depend on the horizontal rapidities $x_{\circ}$ and $x_{\bullet}$, where $x_{\circ}$ is the rapidity that passes transversally through the $k$-th step of $P$, and $x_{\bullet}$ is the rapidity passing transversally through the $\ell$-th step of $Q$. The rows where these variables appear are indicated in Figure \ref{fig:phi}.

%We are ready to state the main result of this section.
%
\begin{theorem}
\label{theorema-egr}
Let $P$ and $Q$ be two down-right paths of length $M$, which frame a down-right domain $P-Q$. Fix two integers $1 \leq k,\ell \leq M$. Then for all sets $\mathcal{A} \subset \{1,\dots,k-1\}$ and $\mathcal{B} \subset \{k+1,\dots,M\}$, fixed outgoing colors $(i_{\alpha \in \mathcal{A}})$ and $(i_{\beta \in \mathcal{B}})$ satisfying \eqref{out-conditions}, and fixed incoming colors $(j_1,\dots,j_M)$ satisfying \eqref{in-conditions}, we have
\begin{align}
\label{main-result}
\Phi_{P/Q}
\Big(
\mathcal{A},\mathcal{B};
(i_{\alpha \in \mathcal{A}}),
(i_{\beta \in \mathcal{B}});
(j_1,\dots,j_M)
\Big)
=
\Psi_{P/Q}
\Big(
\mathcal{A},\mathcal{B};
(i_{\alpha \in \mathcal{A}}),
(i_{\beta \in \mathcal{B}});
(j_1,\dots,j_M)
\Big)
\Big|_{x_{\circ} \leftrightarrow x_{\bullet}}
\end{align}
where we have permuted the two rapidity variables $x_{\circ}$ and $x_{\bullet}$ on the right hand side of the equation.
\end{theorem}

\begin{rmk}
\label{rmk:reflect}
There is a direct analogue of Theorem \ref{theorema-egr} which applies in the case where the $k$-th step of $P$ and the $\ell$-th step of $Q$ are both rightward (instead of downward). In that situation, the rapidity variables that get switched are no longer associated to horizontal lattice lines, but rather vertical lattice lines. In particular, letting $y_{\circ}$ be the rapidity that passes transversally through the $k$-th step of $P$ and $y_{\bullet}$ the rapidity that passes transversally through the $\ell$-th step of $Q$, we have
\begin{align}
\label{main-result'}
\Phi_{P/Q}
\Big(
\mathcal{A},\mathcal{B};
(i_{\alpha \in \mathcal{A}}),
(i_{\beta \in \mathcal{B}});
(j_1,\dots,j_M)
\Big)
=
\Psi_{P/Q}
\Big(
\mathcal{A},\mathcal{B};
(i_{\alpha \in \mathcal{A}}),
(i_{\beta \in \mathcal{B}});
(j_1,\dots,j_M)
\Big)
\Big|_{y_{\circ} \leftrightarrow y_{\bullet}},
\end{align}
where $\Phi_{P/Q}$ and $\Psi_{P/Q}$ are still given by \eqref{phi} and \eqref{psi}, respectively.

The statement \eqref{main-result'} about rightward steps is an easy corollary of the statement \eqref{main-result} about downward steps. Indeed, one may easily check that the statement \eqref{main-result'} transforms under the reflection symmetry \eqref{reflect} to a statement of the form \eqref{main-result}. Therefore, in what follows, we will focus on proving the version of Theorem \ref{theorema-egr} that applies for downward steps; we may do so without any loss of generality.
\end{rmk}

\begin{rmk}
Theorem \ref{theorema-egr} is our ``Shift Theorem'', referring to the shift in the position where we measure the height function (from the $(k+1)$-th to the $k$-th position along $P$) combined with the shift in its level (from $m$ to $m+1$), as we compare $\Phi_{P/Q}$ and $\Psi_{P/Q}$.

The small difference between \eqref{phi} and \eqref{psi} appears innocuous at first glance, and it seems as though it could be understood via some local transformation mapping $\Phi_{P/Q}$ to $\Psi_{P/Q}$. In fact this turns out not to be the case, and that the match of Theorem \ref{theorema-egr} is a fundamentally {\it non-local} phenomenon. One way of seeing this is to note the switch of rapidities $x_{\circ}\leftrightarrow x_{\bullet}$ that accompanies the statement, since these variables can in practice belong to two rows which are widely separated in the lattice.
\end{rmk}

\subsection{Shift invariance of height function joint distributions}

Before moving on to the proof of Theorem \ref{theorema-egr}, we examine one important corollary of it which will be essential in Section \ref{Section_higher_spin} and the shift-invariance results for polymers in Section \ref{Section_polymers}. Up to a slight change in notations, this is the same result as Theorem \ref{Theorem_6v_intro}.

\begin{theorem}
\label{Theorem_6v_main_text_ver}
    Consider the stochastic six-vertex model with inhomogeneous weights of Figure \ref{fund-vert} in the top-right quadrant, choose an index $\iota\ge 1$, color cutoff levels $m_1\dots,m_n\ge 1$, and a collection of points  $\{\mathcal U_j\}_{j=1}^n$. Set
    $$
    m_j'=\begin{cases} m_j,& j\ne \iota,\\ m_\iota+1, & j=\iota,\end{cases} \qquad \qquad \mathcal U'_j=\begin{cases}\mathcal U_j, & j\ne \iota, \\ \mathcal U_\iota + (0,1), & j=\iota. \end{cases}
    $$
    Assume that
    $$
    0\le m_1\le m_2\le \dots\le m_n, \qquad 0\le m'_1\le m'_2\le \dots\le m'_n,
    $$
    and
    $$\mathcal U_1,\dots,\mathcal U_{\iota-1} \succeq \mathcal U_\iota \succeq\mathcal U_{\iota+1},\dots, \mathcal U_n, \qquad \mathcal U'_1,\dots,\mathcal U'_{\iota-1} \succeq \mathcal U'_\iota \succeq \mathcal U'_{\iota+1},\dots, \mathcal U'_n,
    $$ where we recall that for two points $\mathcal U=(a^{\mathcal U},b^{\mathcal U})$, $\mathcal V=(a^{\mathcal V},b^{\mathcal V})$ in the quadrant, we write $\mathcal U\succeq \mathcal V$ if $a^{\mathcal U}\le a^{\mathcal V}$ and $b^{\mathcal U}\ge b^{\mathcal V}$. Then the distribution of the vector of colored height functions
    \begin{align}
    \label{vector2}
    \bigl(\H^{\ges m_1}(\mathcal U_1),\, \H^{\ges m_2}(\mathcal U_2),\, \dots, \H^{\ges m_n}(\mathcal U_n)\bigr)
    \end{align}
    coincides with the distribution of a similar vector with shifted $\iota$-th point and cutoff
    \begin{align}
    \label{shift-vector2}
    \bigl(\H^{\ges m'_1}(\mathcal U'_1),\, \H^{\ges m'_2}(\mathcal U'_2),\, \dots, \H^{\ges m'_n}(\mathcal U'_n)\bigr),
    \end{align}
under the condition that in the vertex model used to define the second vector one swaps the row rapidities $x_{m_\iota}$ and $x_{p}$, where $p$ is the vertical coordinate of $\mathcal U'_\iota$.
\end{theorem}

\begin{proof}
Throughout the proof, assume that $P$ and $Q$ are down-right paths that both begin at the lattice point $(0,N)$ and end at $(M-N,0)$, for some $M \geq N$.

We examine the constraints imposed on colored height functions which are measured at each of the $M$ edges of the down-right path $P$ in Theorem \ref{theorema-egr}.  In both of the quantities $\Phi_{P/Q}$ and $\Psi_{P/Q}$, by virtue of the conditions \eqref{out-conditions} that we impose on the colors exiting through $P$, we see that all of the height functions
\begin{align*}
\mathcal{H}^{\ges a}(P;1),
\dots,
\mathcal{H}^{\ges a}(P;k),
\qquad
0 \leq a \leq m,
\end{align*}
and
\begin{align*}
\mathcal{H}^{\ges b}(P;k+1),
\dots,
\mathcal{H}^{\ges b}(P;M),
\qquad
m+1 \leq b \leq N,
\end{align*}
have their values completely specified. Further, they assume identical values in both $\Phi_{P/Q}$ and $\Psi_{P/Q}$. In $\Phi_{P/Q}$ we additionally condition on the value of $\mathcal{H}^{\ges m}(P;k+1)$ (via the indicator function present in \eqref{phi}), while in $\Psi_{P/Q}$ we additionally condition on the value of $\mathcal{H}^{\ges m+1}(P;k)$ (via the indicator function present in \eqref{psi}).

Now choosing the lower (colored) down-right path $Q$ to be of the form
\begin{align*}
(0,N)
\xrightarrow{N}
(0,N-1)
\xrightarrow{N-1}
\cdots
\xrightarrow {1}
(0,0)
\xrightarrow{0}
(1,0)
\xrightarrow{0}
\cdots
\xrightarrow{0}
(M-N,0),
\end{align*}
(so that $Q$ frames a truncated version of the top-right quadrant) the match \eqref{main-result} from Theorem \ref{theorema-egr} implies that the distribution of the vector
\begin{align}
\label{vector}
(\mathcal{H}^{\ges a_1}(P;1),
\dots,
\mathcal{H}^{\ges a_k}(P;k),
\mathcal{H}^{\ges m}(P;k+1),
\mathcal{H}^{\ges a_{k+1}}(P;k+1),
\dots,
\mathcal{H}^{\ges a_M}(P;M))
\end{align}
coincides with the distribution of the vector
\begin{align}
\label{shift-vector}
(\mathcal{H}^{\ges a_1}(P;1),
\dots,
\mathcal{H}^{\ges a_k}(P;k),
\mathcal{H}^{\ges m+1}(P;k),
\mathcal{H}^{\ges a_{k+1}}(P;k+1),
\dots,
\mathcal{H}^{\ges a_M}(P;M)),
\end{align}
for any $0 \leq a_1 \leq \cdots \leq a_M$ and $m$ such that $a_k \leq m < a_{k+1}$, where in the vertex model used to produce the vector \eqref{shift-vector} we have made the switch of rapidities $x_{\circ} \leftrightarrow x_{\bullet}$.

To conclude the proof, we wish to upgrade the distribution match of \eqref{vector} and \eqref{shift-vector} to a more general collection of points $\{\mathcal{U}_j\}_{j=1}^{n}$, which do not all necessarily lie on a down-right path. By assumption, we know that for all $j \leq \iota-1$ we have $\mathcal{U}_j \succeq \mathcal{U}_{\iota}$ and for all $j \geq \iota+1$ we have $\mathcal{U}_{\iota} \succeq \mathcal{U}_j$. Furthermore:
\begin{itemize}
\item All of the points $\mathcal{U}_1,\dots,\mathcal{U}_{\iota-1}$ must have a vertical coordinate strictly greater than the vertical coordinate of $\mathcal{U}_{\iota}$, for if not, the assumption $\mathcal{U}'_1,\dots,\mathcal{U}'_{\iota-1} \succeq \mathcal{U}'_{\iota}$ would be violated;

\medskip

\item Since $m'_{\iota} \leq m'_{\iota+1}$, we know that $m_{\iota} < m_{\iota+1} \leq m_{\iota+2} \leq \cdots \leq m_n$. Then without loss of generality, each of $\mathcal{U}_{\iota+1}, \dots, \mathcal{U}_n$ should have vertical coordinate $m_{\iota}+1$ or greater, as any point $\mathcal{U}_j$ which violates this condition leads to the deterministic result $\mathcal{H}^{\ges m_j}(\mathcal{U}_j)=0$.
\end{itemize}
The collection of points $\{\mathcal{U}_j\}_{j=1}^{n}$ can then be illustrated by the following picture:
\begin{align*}
\tikz{0.7}{
\foreach\x in {0,...,8}{
\draw[dotted] (\x,0) -- (\x,8);
}
\foreach\y in {0,...,8}{
\draw[dotted] (0,\y) -- (8,\y);
}
%P
\draw[line width=1.5pt,green,->] (0,7) -- (1,7) -- (1,6) -- (4,6) -- (4,3) -- (8,3) -- (8,0);
%U_1 --- U_{i-1}
\node at (2,8) {$\star$};
\node at (2,7) {$\star$};
\node at (3,7) {$\star$};
%U_i
\node at (4,5) {$\bullet$};
%U_{i+1} --- U_n
\node at (5,5) {$\ast$};
\node at (7,4) {$\ast$};
\node at (6,4) {$\ast$};
%labels
\node[left] at (0,0.5) {$1$};
\node[left] at (0,1.5) {$\vdots$};
\node[left] at (0,2.5) {$m_{\iota}$};
\node[left] at (0,3.5) {$m_{\iota}+1$};
\node[left] at (0,4.5) {$\vdots$};
\node[left] at (0,5.5) {$p$};
\node[left] at (0,6.5) {$\vdots$};
\node[left] at (0,7.5) {$\vdots$};
\node[below] at (0.5,0) {$0$};
\node[below] at (1.5,0) {$0$};
\node[below] at (2.5,0) {$0$};
\node[below] at (3.5,0) {$\cdots$};
\node[below] at (4.5,0) {$\cdots$};
\node[below] at (5.5,0) {$\cdots$};
\draw[ultra thick,->] (0,0) -- (9,0);
\draw[ultra thick,->] (0,0) -- (0,9);
%
%color range
\node at (13,6) {$\star$: points $\mathcal{U}_1,\dots,\mathcal{U}_{\iota-1}$};
\node at (13,5) {$\bullet$: point $\mathcal{U}_{\iota}$};
\node at (13,4) {$\ast$: points $\mathcal{U}_{\iota+1},\dots,\mathcal{U}_n$};
}
\end{align*}
We draw a down-right path $P$ (shown above) which can be realised as the concatenation of three shorter paths: $P = P_1 \cup \{\downarrow\} \cup P_2$. Here
$\downarrow$ is the path of length $1$ which passes through both of the points $\mathcal{U}_{\iota}$ and $\mathcal{U}'_{\iota}$; $P_1$ is any down-right path which begins on the vertical axis of the quadrant and ends at $\mathcal{U}'_{\iota}$ while passing under each of the points $\mathcal{U}_1,\dots,\mathcal{U}_{\iota-1}$, and $P_2$ is any down-right path which begins at $\mathcal{U}_{\iota}$ and ends on the horizontal axis of the quadrant, while passing under each of the points $\mathcal{U}_{\iota+1},\dots,\mathcal{U}_n$ and not attaining a vertical coordinate less than $m_{\iota}$ until the path has passed to the right of all of these points. Assume, without loss of generality, that $P_1$ has length $k-1$ and $P$ itself has length $M$.

Now (by the color conservation property of the underlying vertex model) it is easy to see that the distribution of $\mathcal{H}^{\ges m_j}(\mathcal{U}_j)$, $1 \leq j \leq \iota-1$ is completely determined by the distribution of $\mathcal{H}^{\ges a_j}(P;j)$, $a_j \leq m_{\iota}$, $1 \leq j \leq k$ (which lie along $P_1$), while the distribution of $\mathcal{H}^{\ges m_j}(\mathcal{U}_j)$, $\iota+1 \leq j \leq n$ is completely determined by the distribution of $\mathcal{H}^{\ges a_j}(P;j)$, $a_j > m_{\iota}$, $k+1 \leq j \leq M$ (which lie along $P_2$). Furthermore, once one conditions in this way on the values of the height functions along $P_1$ and $P_2$, the distribution of the height functions $\mathcal{H}^{\ges m_j}(\mathcal{U}_j)$, $1 \leq j \leq \iota-1$ and $\mathcal{H}^{\ges m_j}(\mathcal{U}_j)$, $\iota+1 \leq j \leq n$ is unaffected by switching the rapidities $x_{m_{\iota}}$ and $x_p$. It follows immediately that the distribution match of \eqref{vector} and \eqref{shift-vector} implies the distribution match of \eqref{vector2} and \eqref{shift-vector2}.
\end{proof}

\section{Proof of the Shift Theorem \ref{theorema-egr}}

\label{Section_Shift_inhom_proof}

Our proof of the Shift Theorem is long, and so we divide it into three parts.

\subsection{Part one: proof for two-row partition functions}
\label{ssec:two-row}

In this subsection we prove Theorem \ref{theorema-egr} in the case of down-right domains that consist of two rows. The result remains highly nontrivial even at this simplified level; in fact, the proof of the two-row case is essentially the ``kernel'' of the proof for generic down-right domains.

As previously, fix an integer $M$ and two other integers $k$, $\ell$ such that $1 \leq k, \ell \leq M$. Introduce two subsets $\mathcal{A} \subset \{1,\dots,k-1\}$ and $\mathcal{B} \subset \{k+1,\dots,M\}$. Our two-row partition functions are inherited from the previous definitions. We write
\begin{multline}
\label{phi-2}
\Phi_{{\rm two-row}}
\Big(\mathcal{A},\mathcal{B};
(i_{\alpha \in \mathcal{A}}),
(i_{\beta \in \mathcal{B}});
(j_1,\dots,j_M)
\Big)
=
\\
\tikz{0.9}{
%hdots
\draw[dotted] (3,0) -- (6,0);
%vdots
\draw[dotted] (1,0) -- (1,1);
\draw[dotted] (2,0) -- (2,1);
\draw[dotted] (3,0) -- (3,1);
\draw[dotted] (4,-1) -- (4,1);
\draw[dotted] (5,-1) -- (5,1);
\draw[dotted] (6,-1) -- (6,0);
\draw[dotted] (7,-1) -- (7,0);
\draw[dotted] (8,-1) -- (8,0);
\draw[dotted] (9,-1) -- (9,0);
%paths
\node[right,green] at (10,-1) {$P$};
\node[below,red] at (10,-1) {$Q$};
\draw[line width=1.5pt,->,red] (0,1) -- (0,0) -- (3,0) -- (3,-1) -- (10,-1);
\draw[line width=1.5pt,->,green] (0,1) -- (6,1) -- (6,0) -- (10,0) -- (10,-1);
%height function
\draw[line width=7,orange,opacity=0.3]  (6,0) -- (10,0) -- (10,-1);
%bottom labels
\node[left] at (0,0.5) {$j_1$};
\node[below] at (0.5,0) {$\cdots$};
\node[below] at (1.5,0) {$\cdots$};
\node[below] at (2.5,0) {$\cdots$};
\node[left] at (3,-0.5) {$m$}; \node at (3,-0.5) {$\bullet$};
\node[below] at (3.5,-1) {$j_{\ell+1}$};
\node[below] at (4.5,-1) {$\cdots$};
\node[below] at (5.5,-1) {$\cdots$};
\node[below] at (6.5,-1) {$\cdots$};
\node[below] at (7.5,-1) {$\cdots$};
\node[below] at (8.5,-1) {$\cdots$};
\node[below] at (9.5,-1) {$j_M$};
%top labels
\node[above] at (0.5,1) {$i_1$};
\node[above] at (1.5,1) {$\cdots$};
\node[above] at (2.5,1) {$\cdots$};
\node[above] at (3.5,1) {$\cdots$};
\node[above] at (4.5,1) {$\cdots$};
\node[above] at (5.5,1) {$i_{k-1}$};
\node[right] at (6,0.6) {$i_k$}; \filldraw[draw=black,fill=white] (6,0.5) circle (2pt);
\node[above] at (6.5,0) {$\cdots$};
\node[above] at (7.5,0) {$\cdots$};
\node[above] at (8.5,0) {$\cdots$};
\node[above] at (9.5,0) {$\cdots$};
\node[right] at (10,-0.5) {$i_M$};
%rapidities
\node[left] at (-0.5,0.5) {$x_1 \rightarrow$};
\node[left] at (-0.5,-0.5) {$x_{\ell} \rightarrow$};
\node[below] at (0.5,2) {$y_2$}; \node at (0.5,2.3) {$\uparrow$};
\node[below] at (1.5,2) {$\cdots$};
\node[below] at (2.5,2) {$y_{\ell-1}$}; \node at (2.5,2.3) {$\uparrow$};
\node[below] at (3.5,2) {$y_{\ell+1}$}; \node at (3.5,2.3) {$\uparrow$};
\node[below] at (4.5,2) {$\cdots$};
\node[below] at (5.5,2) {$y_{k+1}$}; \node at (5.5,2.3) {$\uparrow$};
\node[below] at (6.5,2) {$y_{k+2}$}; \node at (6.5,2.3) {$\uparrow$};
\node[below] at (7.5,2) {$\cdots$};
\node[below] at (8.5,2) {$\cdots$};
\node[below] at (9.5,2) {$y_{M}$}; \node at (9.5,2.3) {$\uparrow$};
}
\end{multline}
where $(j_1,\dots,j_M)$ is a fixed vector of colors satisfying \eqref{in-conditions}, and the outgoing colors $(i_1,\dots,i_M)$ are either summed or fixed, as specified by \eqref{out-conditions}. We also require that the number of colors of value $m$ or greater passing through the shaded edges is equal to $h$.

Similarly, we write
\begin{multline}
\label{psi-2}
\Psi_{{\rm two-row}}
\Big(\mathcal{A},\mathcal{B};
(i_{\alpha \in \mathcal{A}}),
(i_{\beta \in \mathcal{B}});
(j_1,\dots,j_M)
\Big)
=
\\
\tikz{0.9}{
%hdots
\draw[dotted] (3,0) -- (6,0);
%vdots
\draw[dotted] (1,0) -- (1,1);
\draw[dotted] (2,0) -- (2,1);
\draw[dotted] (3,0) -- (3,1);
\draw[dotted] (4,-1) -- (4,1);
\draw[dotted] (5,-1) -- (5,1);
\draw[dotted] (6,-1) -- (6,0);
\draw[dotted] (7,-1) -- (7,0);
\draw[dotted] (8,-1) -- (8,0);
\draw[dotted] (9,-1) -- (9,0);
%paths
\node[right,green] at (10,-1) {$P$};
\node[below,red] at (10,-1) {$Q$};
\draw[line width=1.5pt,->,red] (0,1) -- (0,0) -- (3,0) -- (3,-1) -- (10,-1);
\draw[line width=1.5pt,->,green] (0,1) -- (6,1) -- (6,0) -- (10,0) -- (10,-1);
%height function
\draw[line width=7,orange,opacity=0.3]  (6,1) -- (6,0) -- (10,0) -- (10,-1);
%bottom labels
\node[left] at (0,0.5) {$j_1$};
\node[below] at (0.5,0) {$\cdots$};
\node[below] at (1.5,0) {$\cdots$};
\node[below] at (2.5,0) {$\cdots$};
\node[left] at (3,-0.5) {$m$}; \node at (3,-0.5) {$\bullet$};
\node[below] at (3.5,-1) {$j_{\ell+1}$};
\node[below] at (4.5,-1) {$\cdots$};
\node[below] at (5.5,-1) {$\cdots$};
\node[below] at (6.5,-1) {$\cdots$};
\node[below] at (7.5,-1) {$\cdots$};
\node[below] at (8.5,-1) {$\cdots$};
\node[below] at (9.5,-1) {$j_M$};
%top labels
\node[above] at (0.5,1) {$i_1$};
\node[above] at (1.5,1) {$\cdots$};
\node[above] at (2.5,1) {$\cdots$};
\node[above] at (3.5,1) {$\cdots$};
\node[above] at (4.5,1) {$\cdots$};
\node[above] at (5.5,1) {$i_{k-1}$};
\node[right] at (6,0.6) {$i_k$}; \filldraw[draw=black,fill=white] (6,0.5) circle (2pt);
\node[above] at (6.5,0) {$\cdots$};
\node[above] at (7.5,0) {$\cdots$};
\node[above] at (8.5,0) {$\cdots$};
\node[above] at (9.5,0) {$\cdots$};
\node[right] at (10,-0.5) {$i_M$};
%rapidities
\node[left] at (-0.5,0.5) {$x_1 \rightarrow$};
\node[left] at (-0.5,-0.5) {$x_{\ell} \rightarrow$};
\node[below] at (0.5,2) {$y_2$}; \node at (0.5,2.3) {$\uparrow$};
\node[below] at (1.5,2) {$\cdots$};
\node[below] at (2.5,2) {$y_{\ell-1}$}; \node at (2.5,2.3) {$\uparrow$};
\node[below] at (3.5,2) {$y_{\ell+1}$}; \node at (3.5,2.3) {$\uparrow$};
\node[below] at (4.5,2) {$\cdots$};
\node[below] at (5.5,2) {$y_{k+1}$}; \node at (5.5,2.3) {$\uparrow$};
\node[below] at (6.5,2) {$y_{k+2}$}; \node at (6.5,2.3) {$\uparrow$};
\node[below] at (7.5,2) {$\cdots$};
\node[below] at (8.5,2) {$\cdots$};
\node[below] at (9.5,2) {$y_{M}$}; \node at (9.5,2.3) {$\uparrow$};
}
\end{multline}
where $(j_1,\dots,j_M)$ is a fixed vector of colors satisfying \eqref{in-conditions}, and the outgoing colors $(i_1,\dots,i_M)$ are either summed or fixed, as specified by \eqref{out-conditions}. This time we require that the number of colors of value $m+1$ or greater passing through the shaded edges is equal to $h$.

\begin{theorem}
\label{lem:2row}
There holds
\begin{align*}
\Phi_{{\rm two-row}}
\Big(\mathcal{A},\mathcal{B};
(i_{\alpha \in \mathcal{A}}),
(i_{\beta \in \mathcal{B}});
(j_1,\dots,j_M)
\Big)
=
\Psi_{{\rm two-row}}
\Big(\mathcal{A},\mathcal{B};
(i_{\alpha \in \mathcal{A}}),
(i_{\beta \in \mathcal{B}});
(j_1,\dots,j_M)
\Big)
\Big|_{x_1 \leftrightarrow x_{\ell}}.
\end{align*}
\end{theorem}

\begin{proof}
Let us consider the outgoing colors $(i_{k+1},\dots,i_M)$ (in both partition functions). Color conservation stipulates that these colors can only assume values in the set $\{m\} \cup \{j_{\ell+1},\dots,j_M\}$. Given the constraints \eqref{in-conditions} on $(j_{\ell+1},\dots,j_M)$, it follows that $(i_{k+1},\dots,i_M)$ can only assume values in $\{0,1,\dots,m\}$. Therefore, both $\Phi_{{\rm two-row}}$ and $\Psi_{{\rm two-row}}$ vanish identically unless $\mathcal{B} = \varnothing$; we will assume this throughout the rest of the proof.

Now given that $\mathcal{B} = \varnothing$ we have
$\bar{\mathcal{B}} = \{k+1,\dots,M\}$, and consulting \eqref{out-conditions}, the outgoing colors $(i_{k+1},\dots,i_M)$ are all summed over values in
$\{0,1,\dots,m\}$. Of these colors, only color $m$ can contribute to the height function
$\mathcal{H}^{\ges m}(P;k+1)$ in $\Phi_{{\rm two-row}}$, while none of these colors contribute to the height function $\mathcal{H}^{\ges m+1}(P;k)$ in $\Psi_{{\rm two-row}}$. It follows that both quantities vanish identically unless $h=1$ or $h=0$; we now focus on these cases separately.

\bigskip
\underline{The case $h=1$.} We begin with $\Phi_{{\rm two-row}}$, assuming that
$\mathcal{H}^{\ges m}(P;k+1)=1$. With the height function thus fixed, we can more precisely specify the values of the outgoing colors:
\begin{align}
\label{i-values_h=1}
i_{\alpha}
\left\{
\begin{array}{ll}
\fixin \{0,1,\dots,m-1\},
& \quad \alpha \in \mathcal{A},
\\ \\
\sumin \{m+1,m+2,\dots,N\},
& \quad \alpha \in \bar{\mathcal{A}},
\end{array}
\right.
\qquad\quad
i_{\beta} \sumin \{0,1,\dots,m\},
\quad \beta \in \{k+1,\dots,M\},
\end{align}
and $i_k \sumin \{0,1,\dots,m-1\} \cup \{m+1,m+2,\dots,N\}$. Several simplifications of $\Phi_{{\rm two-row}}$ then ensue. Noting the possible values of the indices \eqref{i-values_h=1}, the color $m$ can only exit the partition function \eqref{phi-2} via one of its shaded edges. This leads to the freezing of a portion of the partition function \eqref{phi-2}, as shown below:
\begin{multline*}
\Phi_{{\rm two-row}}
\Big(\mathcal{A},\varnothing;
(i_{\alpha \in \mathcal{A}});
(j_1,\dots,j_M)
\Big)
=
\\
\tikz{0.9}{
\filldraw[fill=lgray,draw=lgray] (3.05,-0.95) -- (3.05,-0.05) -- (5.95,-0.05) -- (5.95,-0.95) -- (3.05,-0.95);
%hdots
\draw[dotted] (3,0) -- (6,0);
%vdots
\draw[dotted] (1,0) -- (1,1);
\draw[dotted] (2,0) -- (2,1);
\draw[dotted] (3,0) -- (3,1);
\draw[dotted] (4,-1) -- (4,1);
\draw[dotted] (5,-1) -- (5,1);
\draw[dotted] (6,-1) -- (6,0);
\draw[dotted] (7,-1) -- (7,0);
\draw[dotted] (8,-1) -- (8,0);
\draw[dotted] (9,-1) -- (9,0);
%paths
\node[right,green] at (10,-1) {$P$};
\node[below,red] at (10,-1) {$Q$};
\draw[line width=1.5pt,->,red] (0,1) -- (0,0) -- (3,0) -- (3,-1) -- (10,-1);
\draw[line width=1.5pt,->,green] (0,1) -- (6,1) -- (6,0) -- (10,0) -- (10,-1);
%height function
\draw[line width=7,orange,opacity=0.3]  (6,0) -- (10,0) -- (10,-1);
%bottom labels
\node[left] at (0,0.5) {$j_1$};
\node[below] at (0.5,0) {$\cdots$};
\node[below] at (1.5,0) {$\cdots$};
\node[below] at (2.5,0) {$\cdots$};
\node[left] at (3,-0.5) {$m$}; \node at (3,-0.5) {$\bullet$};
\node[below] at (3.5,-1) {$j_{\ell+1}$};
\node[below] at (4.5,-1) {$\cdots$};
\node[below] at (5.5,-1) {$j_{k+1}$};
\node[below] at (6.5,-1) {$\cdots$};
\node[below] at (7.5,-1) {$\cdots$};
\node[below] at (8.5,-1) {$\cdots$};
\node[below] at (9.5,-1) {$j_M$};
\node[left] at (6,-0.5) {$m$};
\node[above] at (3.5,0) {$j_{\ell+1}$};
\node[above] at (4.5,0) {$\cdots$};
\node[above] at (5.5,0) {$j_{k+1}$};
%top labels
\node[above] at (0.5,1) {$i_1$};
\node[above] at (1.5,1) {$\cdots$};
\node[above] at (2.5,1) {$\cdots$};
\node[above] at (3.5,1) {$\cdots$};
\node[above] at (4.5,1) {$\cdots$};
\node[above] at (5.5,1) {$i_{k-1}$};
\node[right] at (6,0.6) {$i_k$}; \filldraw[draw=black,fill=white] (6,0.5) circle (2pt);
\node[above] at (6.5,0) {$\cdots$};
\node[above] at (7.5,0) {$\cdots$};
\node[above] at (8.5,0) {$\cdots$};
\node[above] at (9.5,0) {$\cdots$};
\node[right] at (10,-0.5) {$i_M$};
%rapidities
\node[left] at (-0.5,0.5) {$x_1 \rightarrow$};
\node[left] at (-0.5,-0.5) {$x_{\ell} \rightarrow$};
}.
\end{multline*}
In all the vertices of the frozen (gray) area the edges of color $m$ are necessarily horizontal.
Computing the weight of the frozen portion (it is simply a product of bottom middle weights in Figure \ref{fund-vert}, since $m > j_{\alpha}$ for all $\alpha \in \{\ell+1,\dots,k+1\}$), \eqref{phi-2} factorizes into two disjoint one-row partition functions:
\begin{multline}
\label{xy-pieces}
\Phi_{{\rm two-row}}
\Big(\mathcal{A},\varnothing;
(i_{\alpha})_{\alpha \in \mathcal{A}};
(j_1,\dots,j_M)
\Big)
=
\prod_{\alpha=\ell+1}^{k+1}
\frac{x_{\ell}-y_{\alpha}}{x_{\ell}-qy_{\alpha}}
\times
\\
\left(
\tikz{0.9}{
%hdots
\draw[dotted] (3,0) -- (6,0);
%vdots
\draw[dotted] (1,0) -- (1,1);
\draw[dotted] (2,0) -- (2,1);
\draw[dotted] (3,0) -- (3,1);
\draw[dotted] (4,0) -- (4,1);
\draw[dotted] (5,0) -- (5,1);
%paths
\draw[line width=1.5pt,->,red] (0,1) -- (0,0) -- (6,0);
\draw[line width=1.5pt,->,green] (0,1) -- (6,1) -- (6,0);
%bottom labels
\node[left] at (0,0.5) {$j_1$};
\node[below] at (0.5,0) {$\cdots$};
\node[below] at (1.5,0) {$\cdots$};
\node[below] at (2.5,0) {$j_{\ell-1}$};
\node[below] at (3.5,0) {$j_{\ell+1}$};
\node[below] at (4.5,0) {$\cdots$};
\node[below] at (5.5,0) {$j_{k+1}$};
%top labels
\node[above] at (0.5,1) {$i_1$};
\node[above] at (1.5,1) {$\cdots$};
\node[above] at (2.5,1) {$\cdots$};
\node[above] at (3.5,1) {$\cdots$};
\node[above] at (4.5,1) {$\cdots$};
\node[above] at (5.5,1) {$i_{k-1}$};
\node[right] at (6,0.5) {$i_k$};
%rapidities
\node[left] at (-0.5,0.5) {$x_1 \rightarrow$};
}
\right)
\left(
\tikz{0.9}{
%hdots
\draw[dotted] (3,0) -- (6,0);
%vdots
\draw[dotted] (3,0) -- (3,1);
\draw[dotted] (4,0) -- (4,1);
\draw[dotted] (5,0) -- (5,1);
%paths
\draw[line width=1.5pt,->,red] (2,1) -- (2,0) -- (6,0);
\draw[line width=1.5pt,->,green] (2,1) -- (6,1) -- (6,0);
%bottom labels
\node[left] at (2,0.5) {$m$};
\node[below] at (2.5,0) {$j_{k+2}$};
\node[below] at (3.5,0) {$\cdots$};
\node[below] at (4.5,0) {$\cdots$};
\node[below] at (5.5,0) {$j_M$};
%top labels
\node[above] at (2.5,1) {$i_{k+1}$};
\node[above] at (3.5,1) {$\cdots$};
\node[above] at (4.5,1) {$\cdots$};
\node[above] at (5.5,1) {$\cdots$};
\node[right] at (6,0.5) {$i_M$};
%rapidities
\node[left] at (1.5,0.5) {$x_{\ell} \rightarrow$};
}
\right)
\end{multline}
where the indices $(i_1,\dots,i_M)$ are specified by \eqref{i-values_h=1}. Let us refer to these one-row partition functions as the {\it $x_1$ piece} and the {\it $x_{\ell}$ piece}. Given that $i_{k+1},\dots,i_M \sumin \{0,1,\dots,m\}$, we can invoke equation \eqref{Z-stoch} with $\mathcal{C} = \{0,1,\dots,m\}$ to conclude that the $x_{\ell}$ piece is equal to $1$, irrespective of the values of $j_{k+2},\dots,j_M \fixin \{0,1,\dots,m-1\}$. Considering the outgoing indices $i_1,\dots,i_k$ of the $x_1$ piece we see that they satisfy the necessary criteria to invoke Proposition \ref{prop:merge-PQ} with $\mathcal{C} = \{m+1,m+2,\dots,N\}$. We thus have, after performing everywhere the color relabelling $m+1 \rightarrow m$,
\begin{align}
\label{y-piece1}
\tikz{0.9}{
%hdots
\draw[dotted] (3,0) -- (6,0);
%vdots
\draw[dotted] (1,0) -- (1,1);
\draw[dotted] (2,0) -- (2,1);
\draw[dotted] (3,0) -- (3,1);
\draw[dotted] (4,0) -- (4,1);
\draw[dotted] (5,0) -- (5,1);
%paths
\draw[line width=1.5pt,->,red] (0,1) -- (0,0) -- (6,0);
\draw[line width=1.5pt,->,green] (0,1) -- (6,1) -- (6,0);
%bottom labels
\node[left] at (0,0.5) {$j_1$};
\node[below] at (0.5,0) {$\cdots$};
\node[below] at (1.5,0) {$\cdots$};
\node[below] at (2.5,0) {$j_{\ell-1}$};
\node[below] at (3.5,0) {$j_{\ell+1}$};
\node[below] at (4.5,0) {$\cdots$};
\node[below] at (5.5,0) {$j_{k+1}$};
%top labels
\node[above] at (0.5,1) {$i_1$};
\node[above] at (1.5,1) {$\cdots$};
\node[above] at (2.5,1) {$\cdots$};
\node[above] at (3.5,1) {$\cdots$};
\node[above] at (4.5,1) {$\cdots$};
\node[above] at (5.5,1) {$i_{k-1}$};
\node[right] at (6,0.5) {$i_k$};
%rapidities
\node[left] at (-0.5,0.5) {$x_1 \rightarrow$};
}
=
\tikz{0.9}{
%hdots
\draw[dotted] (3,0) -- (6,0);
%vdots
\draw[dotted] (1,0) -- (1,1);
\draw[dotted] (2,0) -- (2,1);
\draw[dotted] (3,0) -- (3,1);
\draw[dotted] (4,0) -- (4,1);
\draw[dotted] (5,0) -- (5,1);
%paths
\draw[line width=1.5pt,->,red] (0,1) -- (0,0) -- (6,0);
\draw[line width=1.5pt,->,green] (0,1) -- (6,1) -- (6,0);
%bottom labels
\node[left] at (0,0.5) {$m$};
\node[below] at (0.5,0) {$\cdots$};
\node[below] at (1.5,0) {$\cdots$};
\node[below] at (2.5,0) {$m$};
\node[below] at (3.5,0) {$j_{\ell+1}$};
\node[below] at (4.5,0) {$\cdots$};
\node[below] at (5.5,0) {$j_{k+1}$};
%top labels
\node[above] at (0.5,1) {$p_1$};
\node[above] at (1.5,1) {$\cdots$};
\node[above] at (2.5,1) {$\cdots$};
\node[above] at (3.5,1) {$\cdots$};
\node[above] at (4.5,1) {$\cdots$};
\node[above] at (5.5,1) {$p_{k-1}$};
\node[right] at (6,0.5) {$p_k$};
%rapidities
\node[left] at (-0.5,0.5) {$x_1 \rightarrow$};
}
\end{align}
where we have replaced each of the incoming indices $j_1,\dots,j_{\ell-1} \fixin \{m+1,m+2,\dots,N\}$ by $m$ and have defined
\begin{align}
\label{p-indices}
p_{\alpha} =
\left\{
\begin{array}{ll}
i_{\alpha}, & \qquad \alpha \in \mathcal{A},
\\ \\
m, & \qquad \alpha \not\in \b{\mathcal{A}},
\end{array}
\right.
\qquad
\text{and}
\qquad
p_k \sumin \{0,1,\dots,m\}.
\end{align}
Examining the right hand side of \eqref{y-piece1}, by color-conservation arguments we see that it vanishes identically unless $p_1 = \cdots = p_{\ell-2} = m$ (equivalently, $1,\dots,\ell-2 \in \b{\mathcal{A}}$). Provided that these criteria are met, we have
\begin{align}
\nonumber
\tikz{0.9}{
%hdots
\draw[dotted] (3,0) -- (6,0);
%vdots
\draw[dotted] (1,0) -- (1,1);
\draw[dotted] (2,0) -- (2,1);
\draw[dotted] (3,0) -- (3,1);
\draw[dotted] (4,0) -- (4,1);
\draw[dotted] (5,0) -- (5,1);
%paths
\draw[line width=1.5pt,->,red] (0,1) -- (0,0) -- (6,0);
\draw[line width=1.5pt,->,green] (0,1) -- (6,1) -- (6,0);
%bottom labels
\node[left] at (0,0.5) {$j_1$};
\node[below] at (0.5,0) {$\cdots$};
\node[below] at (1.5,0) {$\cdots$};
\node[below] at (2.5,0) {$j_{\ell-1}$};
\node[below] at (3.5,0) {$j_{\ell+1}$};
\node[below] at (4.5,0) {$\cdots$};
\node[below] at (5.5,0) {$j_{k+1}$};
%top labels
\node[above] at (0.5,1) {$i_1$};
\node[above] at (1.5,1) {$\cdots$};
\node[above] at (2.5,1) {$\cdots$};
\node[above] at (3.5,1) {$\cdots$};
\node[above] at (4.5,1) {$\cdots$};
\node[above] at (5.5,1) {$i_{k-1}$};
\node[right] at (6,0.5) {$i_k$};
%rapidities
\node[left] at (-0.5,0.5) {$x_1 \rightarrow$};
}
&=
\tikz{0.9}{
\filldraw[fill=lgray,draw=lgray] (0.05,0.05) -- (0.05,0.95) -- (2.95,0.95) -- (2.95,0.05) -- (0.05,0.05);
%hdots
\draw[dotted] (3,0) -- (6,0);
%vdots
\draw[dotted] (1,0) -- (1,1);
\draw[dotted] (2,0) -- (2,1);
\draw[dotted] (3,0) -- (3,1);
\draw[dotted] (4,0) -- (4,1);
\draw[dotted] (5,0) -- (5,1);
%paths
\draw[line width=1.5pt,->,red] (0,1) -- (0,0) -- (6,0);
\draw[line width=1.5pt,->,green] (0,1) -- (6,1) -- (6,0);
%bottom labels
\node[left] at (0,0.5) {$m$};
\node[below] at (0.5,0) {$\cdots$};
\node[below] at (1.5,0) {$\cdots$};
\node[below] at (2.5,0) {$m$};
\node[below] at (3.5,0) {$j_{\ell+1}$};
\node[below] at (4.5,0) {$\cdots$};
\node[below] at (5.5,0) {$j_{k+1}$};
%top labels
\node[above] at (0.5,1) {$m$};
\node[above] at (1.5,1) {$\cdots$};
\node[above] at (2.5,1) {$m$};
\node[above] at (3.5,1) {$p_{\ell-1}$};
\node[above] at (4.5,1) {$\cdots$};
\node[above] at (5.5,1) {$p_{k-1}$};
\node[right] at (6,0.5) {$p_k$};
\node[left] at (3.5,0.5) {$m$};
%rapidities
\node[left] at (-0.5,0.5) {$x_1 \rightarrow$};
}
\\
\label{y-piece2}
&=
\tikz{0.9}{
%vdots
\draw[dotted] (3,-1) -- (3,0);
\draw[dotted] (4,-1) -- (4,0);
\draw[dotted] (5,-1) -- (5,0);
%paths
\draw[line width=1.5pt,->,red] (2,0) -- (2,-1) -- (5,-1);
\draw[line width=1.5pt,->,green] (2,0) -- (5,0) -- (5,-1);
%bottom labels
\node[left] at (2,-0.5) {$m$};
\node[below] at (2.5,-1) {$j_{\ell+1}$};
\node[below] at (3.5,-1) {$\cdots$};
\node[below] at (4.5,-1) {$j_{k+1}$};
%top labels
\node[above] at (2.5,0) {$p_{\ell-1}$};
\node[above] at (3.5,0) {$\cdots$};
\node[above] at (4.5,0) {$p_{k-1}$};
\node[right] at (5,-0.5) {$p_{k}$};
%rapidities
\node[left] at (1.5,-0.5) {$x_1 \rightarrow$};
}
\end{align}
Combining all of these results in \eqref{xy-pieces}, we have shown that
\begin{align}
\label{phi-ans-h1}
\Phi_{{\rm two-row}}
\Big(\mathcal{A},\varnothing;
(i_{\alpha \in \mathcal{A}});
(j_1,\dots,j_M)
\Big)
=
\prod_{\alpha=1}^{\ell-2}
\bm{1}_{\alpha \in \bar{\mathcal{A}}}
\prod_{\alpha=\ell+1}^{k+1}
\frac{x_{\ell}-y_{\alpha}}{x_{\ell}-qy_{\alpha}}
\left(
\tikz{0.9}{
%vdots
\draw[dotted] (3,-1) -- (3,0);
\draw[dotted] (4,-1) -- (4,0);
\draw[dotted] (5,-1) -- (5,0);
%paths
\draw[line width=1.5pt,->,red] (2,0) -- (2,-1) -- (5,-1);
\draw[line width=1.5pt,->,green] (2,0) -- (5,0) -- (5,-1);
%bottom labels
\node[left] at (2,-0.5) {$m$};
\node[below] at (2.5,-1) {$j_{\ell+1}$};
\node[below] at (3.5,-1) {$\cdots$};
\node[below] at (4.5,-1) {$j_{k+1}$};
%top labels
\node[above] at (2.5,0) {$p_{\ell-1}$};
\node[above] at (3.5,0) {$\cdots$};
\node[above] at (4.5,0) {$p_{k-1}$};
\node[right] at (5,-0.5) {$p_{k}$};
%rapidities
\node[left] at (1.5,-0.5) {$x_1 \rightarrow$};
}
\right).
\end{align}

Now let us compare this against $\Psi_{{\rm two-row}}$, assuming that
$\mathcal{H}^{\ges m+1}(P;k)=1$. This leads to the following specification of the outgoing colors:
\begin{align*}
i_{\alpha}
\left\{
\begin{array}{ll}
\fixin \{0,1,\dots,m-1\},
& \quad \alpha \in \mathcal{A},
\\ \\
\sumin \{m,m+1,\dots,N\},
& \quad \alpha \in \bar{\mathcal{A}},
\end{array}
\right.
\qquad\qquad
i_{\beta} \sumin \{0,1,\dots,m\},
\quad \beta \in \{k+1,\dots,M\},
\end{align*}
and $i_k \sumin \{m+1,m+2,\dots,N\}$ (note that, in particular, we cannot allow $i_k \sumin \{0,1,\dots,m\}$, since for such values of $i_k$ the height function constraint $\mathcal{H}^{\ges m+1}(P;k)=1$ is violated). We observe that the outgoing indices $i_1,\dots,i_M$ satisfy the necessary criteria to invoke Proposition \ref{prop:merge-PQ} with $\mathcal{C} = \{m+1,m+2,\dots,N\}$, and we thus have
\begin{multline}
\label{psi-22}
\Psi_{{\rm two-row}}
\Big(\mathcal{A},\varnothing;
(i_{\alpha \in \mathcal{A}});
(j_1,\dots,j_M)
\Big)
=
\\
\tikz{0.9}{
%hdots
\draw[dotted] (3,0) -- (6,0);
%vdots
\draw[dotted] (1,0) -- (1,1);
\draw[dotted] (2,0) -- (2,1);
\draw[dotted] (3,0) -- (3,1);
\draw[dotted] (4,-1) -- (4,1);
\draw[dotted] (5,-1) -- (5,1);
\draw[dotted] (6,-1) -- (6,0);
\draw[dotted] (7,-1) -- (7,0);
\draw[dotted] (8,-1) -- (8,0);
\draw[dotted] (9,-1) -- (9,0);
%paths
\node[right,green] at (10,-1) {$P$};
\node[below,red] at (10,-1) {$Q$};
\draw[line width=1.5pt,->,red] (0,1) -- (0,0) -- (3,0) -- (3,-1) -- (10,-1);
\draw[line width=1.5pt,->,green] (0,1) -- (6,1) -- (6,0) -- (10,0) -- (10,-1);
%height function
\draw[line width=7,orange,opacity=0.3]  (6,1) -- (6,0) -- (10,0) -- (10,-1);
%bottom labels
\node[left] at (0,0.5) {\tiny $m+1$};
\node[below] at (0.5,0) {$\cdots$};
\node[below] at (1.5,0) {$\cdots$};
\node[below] at (2.5,0) {\tiny $m+1$};
\node[right] at (3,-0.5) {$m$}; \node at (3,-0.5) {$\bullet$};
\node[below] at (3.5,-1) {$j_{\ell+1}$};
\node[below] at (4.5,-1) {$\cdots$};
\node[below] at (5.5,-1) {$\cdots$};
\node[below] at (6.5,-1) {$\cdots$};
\node[below] at (7.5,-1) {$\cdots$};
\node[below] at (8.5,-1) {$\cdots$};
\node[below] at (9.5,-1) {$j_M$};
%top labels
\node[above] at (0.5,1) {$r_1$};
\node[above] at (1.5,1) {$\cdots$};
\node[above] at (2.5,1) {$\cdots$};
\node[above] at (3.5,1) {$\cdots$};
\node[above] at (4.5,1) {$\cdots$};
\node[above] at (5.5,1) {$r_{k-1}$};
\node[right] at (6,0.5) {\tiny $m+1$};
\filldraw[draw=black,fill=white] (6,0.5) circle (2pt);
\node[below] at (6.5,0) {$i_{k+1}$};
\node[above] at (7.5,0) {$\cdots$};
\node[above] at (8.5,0) {$\cdots$};
\node[above] at (9.5,0) {$\cdots$};
\node[right] at (10,-0.5) {$i_M$};
%rapidities
\node[left] at (-1,0.5) {$x_1 \rightarrow$};
\node[left] at (-1,-0.5) {$x_{\ell} \rightarrow$};
}
\end{multline}
where we have replaced each of the incoming indices $j_1,\dots,j_{\ell-1} \fixin \{m+1,m+2,\dots,N\}$ by $m+1$ and have defined
\begin{align}
\label{r-indices}
r_{\alpha}
\left\{
\begin{array}{ll}
=
i_{\alpha}, & \qquad \alpha \in \mathcal{A},
\\ \\
\sumin \{m,m+1\}, & \qquad \alpha \not\in \b{\mathcal{A}}.
\end{array}
\right.
\end{align}
By color conservation arguments, we see that \eqref{psi-22} vanishes unless $r_1 = \cdots = r_{\ell-2} = m+1$ (which stipulates that $1,\dots,\ell-2 \in \b{\mathcal{A}}$). After performing this replacement, $\ell-1$ of the outgoing edges have already been assigned the value $m+1$; on the other hand, exactly $\ell-1$ of the incoming edges assume the value $m+1$. It follows that none of the indices $r_{\ell-1},\dots,r_{k-1}$ can take the value $m+1$ (this would violate color conservation). Taking into account the preceding requirements, the $x_1$-dependent row of \eqref{psi-22} freezes, as shown below:
\begin{multline}
\Psi_{{\rm two-row}}
\Big(\mathcal{A},\varnothing;
(i_{\alpha \in \mathcal{A}});
(j_1,\dots,j_M)
\Big)
=
\\
\tikz{0.9}{
\filldraw[fill=lgray,draw=lgray] (0.05,0.05) -- (0.05,0.95) -- (5.95,0.95) -- (5.95,0.05) -- (0.05,0.05);
%hdots
\draw[dotted] (3,0) -- (6,0);
%vdots
\draw[dotted] (1,0) -- (1,1);
\draw[dotted] (2,0) -- (2,1);
\draw[dotted] (3,0) -- (3,1);
\draw[dotted] (4,-1) -- (4,1);
\draw[dotted] (5,-1) -- (5,1);
\draw[dotted] (6,-1) -- (6,0);
\draw[dotted] (7,-1) -- (7,0);
\draw[dotted] (8,-1) -- (8,0);
\draw[dotted] (9,-1) -- (9,0);
%paths
\node[right,green] at (10,-1) {$P$};
\node[below,red] at (10,-1) {$Q$};
\draw[line width=1.5pt,->,red] (0,1) -- (0,0) -- (3,0) -- (3,-1) -- (10,-1);
\draw[line width=1.5pt,->,green] (0,1) -- (6,1) -- (6,0) -- (10,0) -- (10,-1);
%height function
\draw[line width=7,orange,opacity=0.3]  (6,1) -- (6,0) -- (10,0) -- (10,-1);
%bottom labels
\node[left] at (0,0.5) {\tiny $m+1$};
\node[below] at (0.5,0) {$\cdots$};
\node[below] at (1.5,0) {$\cdots$};
\node[below] at (2.5,0) {\tiny $m+1$};
\node[right] at (3,-0.5) {$m$}; \node at (3,-0.5) {$\bullet$};
\node[below] at (3.5,-1) {$j_{\ell+1}$};
\node[below] at (4.5,-1) {$\cdots$};
\node[below] at (5.5,-1) {$\cdots$};
\node[below] at (6.5,-1) {$\cdots$};
\node[below] at (7.5,-1) {$\cdots$};
\node[below] at (8.5,-1) {$\cdots$};
\node[below] at (9.5,-1) {$j_M$};
%top labels
\node[above] at (0.5,1) {\tiny $m+1$};
\node[above] at (1.5,1) {$\cdots$};
\node[above] at (2.5,1) {\tiny $m+1$};
\node[above] at (3.5,1) {$p_{\ell-1}$};
\node[above] at (4.5,1) {$\cdots$};
\node[above] at (5.5,1) {$p_{k-1}$};
\node[above] at (3.5,0) {$p_{\ell-1}$};
\node[above] at (4.5,0) {$\cdots$};
\node[above] at (5.5,0) {$p_{k-1}$};
\node[right] at (6,0.5) {\tiny $m+1$};
\filldraw[draw=black,fill=white] (6,0.5) circle (2pt);
\node[below] at (6.5,0) {$i_{k+1}$};
\node[above] at (7.5,0) {$\cdots$};
\node[above] at (8.5,0) {$\cdots$};
\node[above] at (9.5,0) {$\cdots$};
\node[right] at (10,-0.5) {$i_M$};
%rapidities
\node[left] at (-1,0.5) {$x_1 \rightarrow$};
\node[left] at (-1,-0.5) {$x_{\ell} \rightarrow$};
}
\end{multline}
where the indices $p_{\ell-1},\dots,p_{k-1}$ have the same definition as in \eqref{p-indices}. The gray area splits into the left part where both horizontal and vertical edges are of color $m+1$ and right part where horizontal edges are of color $m+1$. The weight of the frozen row can be easily computed (it is a product of top left and bottom middle weights in Figure \ref{fund-vert}, since $m+1 > p_{\alpha}$ for all $\alpha \in \{\ell-1,\dots,k-1\}$), and we thus obtain the equation
\begin{multline}
\Psi_{{\rm two-row}}
\Big(\mathcal{A},\varnothing;
(i_{\alpha \in \mathcal{A}});
(j_1,\dots,j_M)
\Big)
=
\prod_{\alpha=1}^{\ell-2}
\bm{1}_{\alpha \in \bar{\mathcal{A}}}
\prod_{\alpha=\ell+1}^{k+1}
\frac{x_1-y_{\alpha}}{x_1-qy_{\alpha}}
\times
\\
\left(
\tikz{0.9}{
%hdots
\draw[dotted] (3,0) -- (6,0);
%vdots
\draw[dotted] (4,-1) -- (4,0);
\draw[dotted] (5,-1) -- (5,0);
\draw[dotted] (6,-1) -- (6,0);
\draw[dotted] (7,-1) -- (7,0);
\draw[dotted] (8,-1) -- (8,0);
\draw[dotted] (9,-1) -- (9,0);
%paths
\draw[line width=1.5pt,->,red]  (3,0) -- (3,-1) -- (10,-1);
\draw[line width=1.5pt,->,green] (3,0) -- (10,0) -- (10,-1);
%bottom labels
\node[left] at (3,-0.5) {$m$};
\node[below] at (3.5,-1) {$j_{\ell+1}$};
\node[below] at (4.5,-1) {$\cdots$};
\node[below] at (5.5,-1) {$\cdots$};
\node[below] at (6.5,-1) {$\cdots$};
\node[below] at (7.5,-1) {$\cdots$};
\node[below] at (8.5,-1) {$\cdots$};
\node[below] at (9.5,-1) {$j_M$};
%top labels
\node[above] at (3.5,0) {$p_{\ell-1}$};
\node[above] at (4.5,0) {$\cdots$};
\node[above] at (5.5,0) {$p_{k-1}$};
\node[above] at (6.5,0) {$i_{k+1}$};
\node[above] at (7.5,0) {$\cdots$};
\node[above] at (8.5,0) {$\cdots$};
\node[above] at (9.5,0) {$\cdots$};
\node[right] at (10,-0.5) {$i_M$};
%rapidities
\node[left] at (2.5,-0.5) {$x_{\ell} \rightarrow$};
}
\right).
\end{multline}
We conclude by dividing the remaining one-row partition function into two segments:
\begin{multline}
\Psi_{{\rm two-row}}
\Big(\mathcal{A},\varnothing;
(i_{\alpha \in \mathcal{A}});
(j_1,\dots,j_M)
\Big)
=
\prod_{\alpha=1}^{\ell-2}
\bm{1}_{\alpha \in \bar{\mathcal{A}}}
\prod_{\alpha=\ell+1}^{k+1}
\frac{x_1-y_{\alpha}}{x_1-qy_{\alpha}}
\times
\\
\sum_{p_k = 0}^{m}
\left(
\tikz{0.9}{
%hdots
\draw[dotted] (3,0) -- (6,0);
%vdots
\draw[dotted] (4,-1) -- (4,0);
\draw[dotted] (5,-1) -- (5,0);
\draw[dotted] (8,-1) -- (8,0);
\draw[dotted] (9,-1) -- (9,0);
\draw[dotted] (10,-1) -- (10,0);
\draw[dotted] (11,-1) -- (11,0);
%paths
\draw[line width=1.5pt,->,red]  (3,0) -- (3,-1) -- (6,-1);
\draw[line width=1.5pt,->,green] (3,0) -- (6,0) -- (6,-1);
\draw[line width=1.5pt,->,red] (8,0) -- (8,-1) -- (12,-1);
\draw[line width=1.5pt,->,green] (8,0) -- (12,0) -- (12,-1);
%bottom labels
\node[left] at (3,-0.5) {$m$};
\node[below] at (3.5,-1) {$j_{\ell+1}$};
\node[below] at (4.5,-1) {$\cdots$};
\node[below] at (5.5,-1) {$j_{k+1}$};
\node[below] at (8.5,-1) {$j_{k+2}$};
\node[below] at (9.5,-1) {$\cdots$};
\node[below] at (10.5,-1) {$\cdots$};
\node[below] at (11.5,-1) {$j_M$};
%top labels
\node[above] at (3.5,0) {$p_{\ell-1}$};
\node[above] at (4.5,0) {$\cdots$};
\node[above] at (5.5,0) {$p_{k-1}$};
\node[above] at (8.5,0) {$i_{k+1}$};
\node[above] at (9.5,0) {$\cdots$};
\node[above] at (10.5,0) {$\cdots$};
\node[above] at (11.5,0) {$\cdots$};
\node[right] at (12,-0.5) {$i_M$};
\node[right] at (6,-0.5) {$p_k$};
\node[left] at (8,-0.5) {$p_k$};
%rapidities
\node[left] at (2.5,-0.5) {$x_{\ell} \rightarrow$};
%times
\node at (7,-0.5) {$\times$};
}
\right).
\end{multline}
Now since the indices $i_{k+1},\dots,i_M$ are summed over all values $\{0,1,\dots,m\}$, we can use \eqref{Z-stoch} with $\mathcal{C} = \{0,1,\dots,m\}$ to conclude that the right segment is equal to $1$ irrespective of the values of $p_k$ and $j_{k+2},\dots,j_M$. We have thus shown that
\begin{align}
\label{psi-ans-h1}
\Psi_{{\rm two-row}}
\Big(\mathcal{A},\varnothing;
(i_{\alpha \in \mathcal{A}});
(j_1,\dots,j_M)
\Big)
=
\prod_{\alpha=1}^{\ell-2}
\bm{1}_{\alpha \in \bar{\mathcal{A}}}
\prod_{\alpha=\ell+1}^{k+1}
\frac{x_1-y_{\alpha}}{x_1-qy_{\alpha}}
\left(
\tikz{0.9}{
%vdots
\draw[dotted] (3,-1) -- (3,0);
\draw[dotted] (4,-1) -- (4,0);
\draw[dotted] (5,-1) -- (5,0);
%paths
\draw[line width=1.5pt,->,red] (2,0) -- (2,-1) -- (5,-1);
\draw[line width=1.5pt,->,green] (2,0) -- (5,0) -- (5,-1);
%bottom labels
\node[left] at (2,-0.5) {$m$};
\node[below] at (2.5,-1) {$j_{\ell+1}$};
\node[below] at (3.5,-1) {$\cdots$};
\node[below] at (4.5,-1) {$j_{k+1}$};
%top labels
\node[above] at (2.5,0) {$p_{\ell-1}$};
\node[above] at (3.5,0) {$\cdots$};
\node[above] at (4.5,0) {$p_{k-1}$};
\node[right] at (5,-0.5) {$p_{k}$};
%rapidities
\node[left] at (1.5,-0.5) {$x_{\ell} \rightarrow$};
}
\right),
\end{align}
where $p_{\ell-1},\dots,p_k$ are given by \eqref{p-indices}. Comparing \eqref{phi-ans-h1} and \eqref{psi-ans-h1}, we see that they are indeed equal up the exchange $x_1 \leftrightarrow x_{\ell}$.

\bigskip
\underline{The case $h=0$.} We move on to study $\Phi_{\rm two-row}$ of \eqref{phi-2}, assuming that $\mathcal{H}^{\ges m}(P;k+1)=0$. Fixing the height function to this value leads to the following specification of the outgoing colors:
\begin{align}
\label{i-values_h=0}
i_{\alpha}
\left\{
\begin{array}{ll}
\fixin \{0,1,\dots,m-1\},
& \quad \alpha \in \mathcal{A},
\\ \\
\sumin \{m,m+1,\dots,N\},
& \quad \alpha \in \bar{\mathcal{A}},
\end{array}
\right.
\qquad\qquad
i_{\beta} \sumin \{0,1,\dots,m-1\},
\quad \beta \in \{k+1,\dots,M\},
\end{align}
and $i_k \sumin \{0,1,\dots,N\}$. (Note that we have used the fact that  $\mathcal{H}^{\ges m}(P;k+1)=0$ to eliminate the possibility that $i_{\beta} = m$ for $\beta \in \{k+1,\dots,M\}$.) We begin by subdividing the bottom row of $\Phi_{\rm two-row}$ as follows:
\begin{multline*}
\Phi_{{\rm two-row}}
\Big(\mathcal{A},\varnothing;
(i_{\alpha \in \mathcal{A}});
(j_1,\dots,j_M)
\Big)
=
\\
\sum_{r_{k+1} = 0}^{m-1}
\left(
\tikz{0.9}{
%hdots
\draw[dotted] (3,0) -- (6,0);
%vdots
\draw[dotted] (1,0) -- (1,1);
\draw[dotted] (2,0) -- (2,1);
\draw[dotted] (3,0) -- (3,1);
\draw[dotted] (4,-1) -- (4,1);
\draw[dotted] (5,-1) -- (5,1);
\draw[dotted] (8.5,-1) -- (8.5,0);
\draw[dotted] (9.5,-1) -- (9.5,0);
\draw[dotted] (10.5,-1) -- (10.5,0);
\draw[dotted] (11.5,-1) -- (11.5,0);
%paths
\draw[line width=1.5pt,->,red] (0,1) -- (0,0) -- (3,0) -- (3,-1) -- (6,-1);
\draw[line width=1.5pt,->,green] (0,1) -- (6,1) -- (6,-1);
\draw[line width=1.5pt,->,red] (8.5,0) -- (8.5,-1) -- (12.5,-1);
\draw[line width=1.5pt,->,green] (8.5,0) -- (12.5,0) -- (12.5,-1);
%bottom labels
\node[left] at (0,0.5) {$j_1$};
\node[below] at (0.5,0) {$\cdots$};
\node[below] at (1.5,0) {$\cdots$};
\node[below] at (2.5,0) {$\cdots$};
\node[left] at (3,-0.5) {$m$};
\node[below] at (3.5,-1) {$j_{\ell+1}$};
\node[below] at (4.5,-1) {$\cdots$};
\node[below] at (5.5,-1) {$j_{k+1}$};
\node[below] at (9,-1) {$j_{k+2}$};
\node[below] at (10,-1) {$\cdots$};
\node[below] at (11,-1) {$\cdots$};
\node[below] at (12,-1) {$j_M$};
%top labels
\node[above] at (0.5,1) {$i_1$};
\node[above] at (1.5,1) {$\cdots$};
\node[above] at (2.5,1) {$\cdots$};
\node[above] at (3.5,1) {$\cdots$};
\node[above] at (4.5,1) {$\cdots$};
\node[above] at (5.5,1) {$i_{k-1}$};
\node[right] at (6,0.5) {$i_k$};
\node[above] at (9,0) {$i_{k+1}$};
\node[above] at (10,0) {$\cdots$};
\node[above] at (11,0) {$\cdots$};
\node[above] at (12,0) {$\cdots$};
\node[right] at (12.5,-0.5) {$i_M$};
\node[right] at (6,-0.5) {$r_{k+1}$};
\node[left] at (8.5,-0.5) {$r_{k+1}$};
%rapidities
\node[left] at (-0.5,0.5) {$x_1 \rightarrow$};
\node[left] at (-0.5,-0.5) {$x_{\ell} \rightarrow$};
%times
\node at (7.25,-0.5) {$\times$};
}
\right).
\end{multline*}
The rightmost segment of this expression has outgoing colors $i_{k+1},\dots,i_M \sumin \{0,1,\dots,m-1\}$; we can thus employ \eqref{Z-stoch} with $\mathcal{C} = \{0,1,\dots,m-1\}$ to show that it is equal to $1$ regardless of the values of the incoming colors $r_{k+1},j_{k+2},\dots,j_M \fixin \{0,1,\dots,m-1\}$. Deleting this segment, we are led to the equation
\begin{align*}
\Phi_{{\rm two-row}}
\Big(\mathcal{A},\varnothing;
(i_{\alpha \in \mathcal{A}});
(j_1,\dots,j_M)
\Big)
=
%\sum_{r_{k+1} = 0}^{m-1}
%\left(
\tikz{0.9}{
%hdots
\draw[dotted] (3,0) -- (6,0);
%vdots
\draw[dotted] (1,0) -- (1,1);
\draw[dotted] (2,0) -- (2,1);
\draw[dotted] (3,0) -- (3,1);
\draw[dotted] (4,-1) -- (4,1);
\draw[dotted] (5,-1) -- (5,1);
%paths
\draw[line width=1.5pt,->,red] (0,1) -- (0,0) -- (3,0) -- (3,-1) -- (6,-1);
\draw[line width=1.5pt,->,green] (0,1) -- (6,1) -- (6,-1);
%bottom labels
\node[left] at (0,0.5) {$j_1$};
\node[below] at (0.5,0) {$\cdots$};
\node[below] at (1.5,0) {$\cdots$};
\node[below] at (2.5,0) {$\cdots$};
\node[left] at (3,-0.5) {$m$};
\node[below] at (3.5,-1) {$j_{\ell+1}$};
\node[below] at (4.5,-1) {$\cdots$};
\node[below] at (5.5,-1) {$j_{k+1}$};
%top labels
\node[above] at (0.5,1) {$i_1$};
\node[above] at (1.5,1) {$\cdots$};
\node[above] at (2.5,1) {$\cdots$};
\node[above] at (3.5,1) {$\cdots$};
\node[above] at (4.5,1) {$\cdots$};
\node[above] at (5.5,1) {$i_{k-1}$};
\node[right] at (6,0.5) {$i_k$};
\node[right] at (6,-0.5) {$r_{k+1}$};
%rapidities
\node[left] at (-0.5,0.5) {$x_1 \rightarrow$};
\node[left] at (-0.5,-0.5) {$x_{\ell} \rightarrow$};
}
%\right)
\end{align*}
where we note that the outgoing colors $i_1,\dots,i_k$ and $r_{k+1} \sumin \{0,1,\dots,m-1\}$ allow us to invoke Proposition \ref{prop:merge-PQ} with $\mathcal{C} = \{m,m+1,\dots,N\}$. We obtain
\begin{align}
\label{phi-25}
\Phi_{{\rm two-row}}
\Big(\mathcal{A},\varnothing;
(i_{\alpha \in \mathcal{A}});
(j_1,\dots,j_M)
\Big)
=
%\sum_{r_{k+1} = 0}^{m-1}
%\left(
\tikz{0.9}{
%hdots
\draw[dotted] (3,0) -- (6,0);
%vdots
\draw[dotted] (1,0) -- (1,1);
\draw[dotted] (2,0) -- (2,1);
\draw[dotted] (3,0) -- (3,1);
\draw[dotted] (4,-1) -- (4,1);
\draw[dotted] (5,-1) -- (5,1);
%paths
\draw[line width=1.5pt,->,red] (0,1) -- (0,0) -- (3,0) -- (3,-1) -- (6,-1);
\draw[line width=1.5pt,->,green] (0,1) -- (6,1) -- (6,-1);
%bottom labels
\node[left] at (0,0.5) {$m$};
\node[above] at (0.5,0) {$\cdots$};
\node[above] at (1.5,0) {$\cdots$};
\node[above] at (2.5,0) {$m$};
\node[left] at (3,-0.5) {$m$};
\node[below] at (3.5,-1) {$j_{\ell+1}$};
\node[below] at (4.5,-1) {$\cdots$};
\node[below] at (5.5,-1) {$j_{k+1}$};
%top labels
\node[above] at (0.5,1) {$p_1$};
\node[above] at (1.5,1) {$\cdots$};
\node[above] at (2.5,1) {$p_{\ell-2}$};
\node[above] at (3.5,1) {$p_{\ell-1}$};
\node[above] at (4.5,1) {$\cdots$};
\node[above] at (5.5,1) {$p_{k-1}$};
\node[right] at (6,0.5) {$p_k$};
\node[right] at (6,-0.5) {$r_{k+1}$};
%rapidities
\node[left] at (-0.5,0.5) {$x_1 \rightarrow$};
\node[left] at (-0.5,-0.5) {$x_{\ell} \rightarrow$};
}
%\right)
\end{align}
where we have replaced each of the incoming indices $j_1,\dots,j_{\ell-1} \fixin \{m+1,m+2,\dots,N\}$ by $m$, and have defined $p_1,\dots,p_k$ in the same way as in \eqref{p-indices}. Examining the right hand side of \eqref{phi-25}, by color-conservation arguments we see that it vanishes identically unless $p_1 = \cdots = p_{\ell-2} = m$ (equivalently, $1,\dots,\ell-2 \in \b{\mathcal{A}}$). Provided that these criteria are met, we have
\begin{align*}
\Phi_{{\rm two-row}}
\Big(\mathcal{A},\varnothing;
(i_{\alpha \in \mathcal{A}});
(j_1,\dots,j_M)
\Big)
&=
%\sum_{r_{k+1} = 0}^{m-1}
%\left(
\tikz{0.9}{
\filldraw[fill=lgray,draw=lgray] (0.05,0.05) -- (0.05,0.95) -- (2.95,0.95) -- (2.95,0.05) -- (0.05,0.05);
%hdots
\draw[dotted] (3,0) -- (6,0);
%vdots
\draw[dotted] (1,0) -- (1,1);
\draw[dotted] (2,0) -- (2,1);
\draw[dotted] (3,0) -- (3,1);
\draw[dotted] (4,-1) -- (4,1);
\draw[dotted] (5,-1) -- (5,1);
%paths
\draw[line width=1.5pt,->,red] (0,1) -- (0,0) -- (3,0) -- (3,-1) -- (6,-1);
\draw[line width=1.5pt,->,green] (0,1) -- (6,1) -- (6,-1);
%bottom labels
\node[left] at (0,0.5) {$m$};
\node[above] at (0.5,0) {$\cdots$};
\node[above] at (1.5,0) {$\cdots$};
\node[above] at (2.5,0) {$m$};
\node[left] at (3,-0.5) {$m$};
\node[below] at (3.5,-1) {$j_{\ell+1}$};
\node[below] at (4.5,-1) {$\cdots$};
\node[below] at (5.5,-1) {$j_{k+1}$};
\node[right] at (2.9,0.5) {$m$};
%top labels
\node[above] at (0.5,1) {$m$};
\node[above] at (1.5,1) {$\cdots$};
\node[above] at (2.5,1) {$m$};
\node[above] at (3.5,1) {$p_{\ell-1}$};
\node[above] at (4.5,1) {$\cdots$};
\node[above] at (5.5,1) {$p_{k-1}$};
\node[right] at (6,0.5) {$p_k$};
\node[right] at (6,-0.5) {$r_{k+1}$};
%rapidities
\node[left] at (-0.5,0.5) {$x_1 \rightarrow$};
\node[left] at (-0.5,-0.5) {$x_{\ell} \rightarrow$};
}
%\right)
\\
&
=
\tikz{0.9}{
%hdots
\draw[dotted] (3,0) -- (6,0);
%vdots
\draw[dotted] (3,0) -- (3,1);
\draw[dotted] (4,-1) -- (4,1);
\draw[dotted] (5,-1) -- (5,1);
%paths
\draw[line width=1.5pt,->,red] (3,1) -- (3,-1) -- (6,-1);
\draw[line width=1.5pt,->,green] (3,1) -- (6,1) -- (6,-1);
%bottom labels
\node[left] at (3,-0.5) {$m$};
\node[below] at (3.5,-1) {$j_{\ell+1}$};
\node[below] at (4.5,-1) {$\cdots$};
\node[below] at (5.5,-1) {$j_{k+1}$};
\node[left] at (3,0.5) {$m$};
%top labels
\node[above] at (3.5,1) {$p_{\ell-1}$};
\node[above] at (4.5,1) {$\cdots$};
\node[above] at (5.5,1) {$p_{k-1}$};
\node[right] at (6,0.5) {$p_k$};
\node[right] at (6,-0.5) {$r_{k+1}$};
%rapidities
\node[left] at (2.5,0.5) {$x_1 \rightarrow$};
\node[left] at (2.5,-0.5) {$x_{\ell} \rightarrow$};
}.
\end{align*}
At this stage, we observe that $p_k \sumin \{0,1,\dots,m\}$, while $r_{k+1} \sumin \{0,1,\dots,m-1\}$. We will replace $r_{k+1}$ by a new summation index $p_{k+1} \sumin \{0,1,\dots,m\}$, at the expense of subtracting away the $p_{k+1} = m$ term that was not present before:
\begin{multline}
\label{phi-27}
\Phi_{{\rm two-row}}
\Big(\mathcal{A},\varnothing;
(i_{\alpha \in \mathcal{A}});
(j_1,\dots,j_M)
\Big)
=
\\
\left(
\tikz{0.9}{
%hdots
\draw[dotted] (3,0) -- (6,0);
%vdots
\draw[dotted] (3,0) -- (3,1);
\draw[dotted] (4,-1) -- (4,1);
\draw[dotted] (5,-1) -- (5,1);
%paths
\draw[line width=1.5pt,->,red] (3,1) -- (3,-1) -- (6,-1);
\draw[line width=1.5pt,->,green] (3,1) -- (6,1) -- (6,-1);
%bottom labels
\node[left] at (3,-0.5) {$m$};
\node[below] at (3.5,-1) {$j_{\ell+1}$};
\node[below] at (4.5,-1) {$\cdots$};
\node[below] at (5.5,-1) {$j_{k+1}$};
\node[left] at (3,0.5) {$m$};
%top labels
\node[above] at (3.5,1) {$p_{\ell-1}$};
\node[above] at (4.5,1) {$\cdots$};
\node[above] at (5.5,1) {$p_{k-1}$};
\node[right] at (6,0.5) {$p_k$};
\node[right] at (6,-0.5) {$p_{k+1}$};
%rapidities
\node[left] at (2.5,0.5) {$x_1 \rightarrow$};
\node[left] at (2.5,-0.5) {$x_{\ell} \rightarrow$};
}
\right)
-
\left(
\tikz{0.9}{
\filldraw[fill=lgray,draw=lgray] (3.05,-0.95) -- (3.05,-0.05) -- (5.95,-0.05) -- (5.95,-0.95) -- (3.05,-0.95);
%hdots
\draw[dotted] (3,0) -- (6,0);
%vdots
\draw[dotted] (3,0) -- (3,1);
\draw[dotted] (4,-1) -- (4,1);
\draw[dotted] (5,-1) -- (5,1);
%paths
\draw[line width=1.5pt,->,red] (3,1) -- (3,-1) -- (6,-1);
\draw[line width=1.5pt,->,green] (3,1) -- (6,1) -- (6,-1);
%bottom labels
\node[left] at (3,-0.5) {$m$};
\node[below] at (3.5,-1) {$j_{\ell+1}$};
\node[below] at (4.5,-1) {$\cdots$};
\node[below] at (5.5,-1) {$j_{k+1}$};
\node[left] at (3,0.5) {$m$};
\node[below] at (3.5,0) {$j_{\ell+1}$};
\node[below] at (4.5,0) {$\cdots$};
\node[below] at (5.5,0) {$j_{k+1}$};
%top labels
\node[above] at (3.5,1) {$p_{\ell-1}$};
\node[above] at (4.5,1) {$\cdots$};
\node[above] at (5.5,1) {$p_{k-1}$};
\node[right] at (6,0.5) {$p_k$};
\node[right] at (6,-0.5) {$m$};
%rapidities
\node[left] at (2.5,0.5) {$x_1 \rightarrow$};
\node[left] at (2.5,-0.5) {$x_{\ell} \rightarrow$};
}
\right)
\end{multline}
where the bottom row of the subtracted term is frozen (it gives rise to a product of bottom middle weights in Figure \ref{fund-vert}, since $m > j_{\alpha}$ for all $\alpha \in \{\ell+1,\dots,k+1\}$). Since both $p_k,p_{k+1} \sumin \{0,1,\dots,m\}$, the first term on the right hand side of \eqref{phi-27} is symmetric under the interchange $x_1 \leftrightarrow x_{\ell}$,\footnote{The symmetry in $x_1,x_{\ell}$ follows from a standard idea in integrable lattice models, namely, the Yang--Baxter equation \eqref{graph-YB} applied to rows of vertices:
\begin{align}
\label{x-sym}
\tikz{0.6}{
\draw[lgray,line width=1.5pt,->] (2,0.5) -- (3,-0.5);
\draw[lgray,line width=1.5pt,->] (2,-0.5) -- (3,0.5);
\draw[lgray,line width=1.5pt,->] (3,0.5) -- (6,0.5);
\draw[lgray,line width=1.5pt,->] (3,-0.5) -- (6,-0.5);
\draw[lgray,line width=1.5pt,->] (3.5,-1) -- (3.5,1);
\draw[lgray,line width=1.5pt,->] (4.5,-1) -- (4.5,1);
\draw[lgray,line width=1.5pt,->] (5.5,-1) -- (5.5,1);
%bottom labels
\node[left] at (2,-0.5) {$m$};
\node[below] at (3.5,-1) {$j_{\ell+1}$};
\node[below] at (4.5,-1) {$\cdots$};
\node[below] at (5.5,-1) {$j_{k+1}$};
\node[left] at (2,0.5) {$m$};
%top labels
\node[above] at (3.5,1) {$p_{\ell-1}$};
\node[above] at (4.5,1) {$\cdots$};
\node[above] at (5.5,1) {$p_{k-1}$};
\node[right] at (6,0.5) {$p_k$};
\node[right] at (6,-0.5) {$p_{k+1}$};
%rapidities
\node[left] at (1.5,0.5) {$x_1 \rightarrow$};
\node[left] at (1.5,-0.5) {$x_{\ell} \rightarrow$};
}
=
\tikz{0.6}{
\draw[lgray,line width=1.5pt,->] (6,0.5) -- (7,-0.5);
\draw[lgray,line width=1.5pt,->] (6,-0.5) -- (7,0.5);
\draw[lgray,line width=1.5pt,->] (3,0.5) -- (6,0.5);
\draw[lgray,line width=1.5pt,->] (3,-0.5) -- (6,-0.5);
\draw[lgray,line width=1.5pt,->] (3.5,-1) -- (3.5,1);
\draw[lgray,line width=1.5pt,->] (4.5,-1) -- (4.5,1);
\draw[lgray,line width=1.5pt,->] (5.5,-1) -- (5.5,1);
%bottom labels
\node[left] at (3,-0.5) {$m$};
\node[below] at (3.5,-1) {$j_{\ell+1}$};
\node[below] at (4.5,-1) {$\cdots$};
\node[below] at (5.5,-1) {$j_{k+1}$};
\node[left] at (3,0.5) {$m$};
%top labels
\node[above] at (3.5,1) {$p_{\ell-1}$};
\node[above] at (4.5,1) {$\cdots$};
\node[above] at (5.5,1) {$p_{k-1}$};
\node[right] at (7,0.5) {$p_k$};
\node[right] at (7,-0.5) {$p_{k+1}$};
%rapidities
\node[left] at (2.5,0.5) {$x_1 \rightarrow$};
\node[left] at (2.5,-0.5) {$x_{\ell} \rightarrow$};
}
\end{align}
where on both sides of the equation we take $p_k,p_{k+1} \sumin \{0,1,\dots,m\}$. The diagonally placed vertices in \eqref{x-sym} can both be dropped from the equation; on the left hand side this is possible because this vertex is frozen to the top left type in Figure \ref{fund-vert}, while on the right hand side this is possible by invoking the stochasticity property \eqref{stoch-graph}. After dropping the diagonally placed vertices, the claimed symmetry in $x_1,x_{\ell}$ is immediate.
} and we thus have
\begin{multline}
\label{phi-277}
\Phi_{{\rm two-row}}
\Big(\mathcal{A},\varnothing;
(i_{\alpha \in \mathcal{A}});
(j_1,\dots,j_M)
\Big)
=
\\
\left(
\tikz{0.9}{
%hdots
\draw[dotted] (3,0) -- (6,0);
%vdots
\draw[dotted] (3,0) -- (3,1);
\draw[dotted] (4,-1) -- (4,1);
\draw[dotted] (5,-1) -- (5,1);
%paths
\draw[line width=1.5pt,->,red] (3,1) -- (3,-1) -- (6,-1);
\draw[line width=1.5pt,->,green] (3,1) -- (6,1) -- (6,-1);
%bottom labels
\node[left] at (3,-0.5) {$m$};
\node[below] at (3.5,-1) {$j_{\ell+1}$};
\node[below] at (4.5,-1) {$\cdots$};
\node[below] at (5.5,-1) {$j_{k+1}$};
\node[left] at (3,0.5) {$m$};
%top labels
\node[above] at (3.5,1) {$p_{\ell-1}$};
\node[above] at (4.5,1) {$\cdots$};
\node[above] at (5.5,1) {$p_{k-1}$};
\node[right] at (6,0.5) {$p_k$};
\node[right] at (6,-0.5) {$p_{k+1}$};
%rapidities
\node[left] at (2.5,0.5) {$x_{\ell} \rightarrow$};
\node[left] at (2.5,-0.5) {$x_1 \rightarrow$};
}
\right)
-
\prod_{\alpha=\ell+1}^{k+1}
\frac{x_{\ell}-y_{\alpha}}{x_{\ell}-qy_{\alpha}}
\left(
\tikz{0.9}{
%vdots
\draw[dotted] (3,-1) -- (3,0);
\draw[dotted] (4,-1) -- (4,0);
\draw[dotted] (5,-1) -- (5,0);
%paths
\draw[line width=1.5pt,->,red] (2,0) -- (2,-1) -- (5,-1);
\draw[line width=1.5pt,->,green] (2,0) -- (5,0) -- (5,-1);
%bottom labels
\node[left] at (2,-0.5) {$m$};
\node[below] at (2.5,-1) {$j_{\ell+1}$};
\node[below] at (3.5,-1) {$\cdots$};
\node[below] at (4.5,-1) {$j_{k+1}$};
%top labels
\node[above] at (2.5,0) {$p_{\ell-1}$};
\node[above] at (3.5,0) {$\cdots$};
\node[above] at (4.5,0) {$p_{k-1}$};
\node[right] at (5,-0.5) {$p_{k}$};
%rapidities
\node[left] at (1.5,-0.5) {$x_1 \rightarrow$};
}
\right).
\end{multline}

Finally, let us compare this against $\Psi_{\rm two-row}$ of \eqref{psi-2}, assuming that $\mathcal{H}^{\ges m+1}(P;k)=0$. In this case the outgoing colors are given by
\begin{align*}
i_{\alpha}
\left\{
\begin{array}{ll}
\fixin \{0,1,\dots,m-1\},
& \quad \alpha \in \mathcal{A},
\\ \\
\sumin \{m,m+1,\dots,N\},
& \quad \alpha \in \bar{\mathcal{A}},
\end{array}
\right.
\qquad\qquad
i_{\beta} \sumin \{0,1,\dots,m\},
\quad \beta \in \{k+1,\dots,M\},
\end{align*}
and $i_k \sumin \{0,1,\dots,m\}$, noting that $i_k$ cannot assume any value greater than $m$ (this would violate the assumption that $\mathcal{H}^{\ges m+1}(P;k)=0$). Similarly to the calculation that we just performed, we begin by subdividing the bottom row of $\Psi_{\rm two-row}$ as follows:
\begin{multline*}
\Psi_{{\rm two-row}}
\Big(\mathcal{A},\varnothing;
(i_{\alpha \in \mathcal{A}});
(j_1,\dots,j_M)
\Big)
=
\\
\sum_{p_{k+1} = 0}^{m}
\left(
\tikz{0.9}{
%hdots
\draw[dotted] (3,0) -- (6,0);
%vdots
\draw[dotted] (1,0) -- (1,1);
\draw[dotted] (2,0) -- (2,1);
\draw[dotted] (3,0) -- (3,1);
\draw[dotted] (4,-1) -- (4,1);
\draw[dotted] (5,-1) -- (5,1);
\draw[dotted] (8.5,-1) -- (8.5,0);
\draw[dotted] (9.5,-1) -- (9.5,0);
\draw[dotted] (10.5,-1) -- (10.5,0);
\draw[dotted] (11.5,-1) -- (11.5,0);
%paths
\draw[line width=1.5pt,->,red] (0,1) -- (0,0) -- (3,0) -- (3,-1) -- (6,-1);
\draw[line width=1.5pt,->,green] (0,1) -- (6,1) -- (6,-1);
\draw[line width=1.5pt,->,red] (8.5,0) -- (8.5,-1) -- (12.5,-1);
\draw[line width=1.5pt,->,green] (8.5,0) -- (12.5,0) -- (12.5,-1);
%bottom labels
\node[left] at (0,0.5) {$j_1$};
\node[below] at (0.5,0) {$\cdots$};
\node[below] at (1.5,0) {$\cdots$};
\node[below] at (2.5,0) {$\cdots$};
\node[left] at (3,-0.5) {$m$};
\node[below] at (3.5,-1) {$j_{\ell+1}$};
\node[below] at (4.5,-1) {$\cdots$};
\node[below] at (5.5,-1) {$j_{k+1}$};
\node[below] at (9,-1) {$j_{k+2}$};
\node[below] at (10,-1) {$\cdots$};
\node[below] at (11,-1) {$\cdots$};
\node[below] at (12,-1) {$j_M$};
%top labels
\node[above] at (0.5,1) {$i_1$};
\node[above] at (1.5,1) {$\cdots$};
\node[above] at (2.5,1) {$\cdots$};
\node[above] at (3.5,1) {$\cdots$};
\node[above] at (4.5,1) {$\cdots$};
\node[above] at (5.5,1) {$i_{k-1}$};
\node[right] at (6,0.5) {$i_k$};
\node[above] at (9,0) {$i_{k+1}$};
\node[above] at (10,0) {$\cdots$};
\node[above] at (11,0) {$\cdots$};
\node[above] at (12,0) {$\cdots$};
\node[right] at (12.5,-0.5) {$i_M$};
\node[right] at (6,-0.5) {$p_{k+1}$};
\node[left] at (8.5,-0.5) {$p_{k+1}$};
%rapidities
\node[left] at (-0.5,0.5) {$x_1 \rightarrow$};
\node[left] at (-0.5,-0.5) {$x_{\ell} \rightarrow$};
%times
\node at (7.25,-0.5) {$\times$};
}
\right).
\end{multline*}
The rightmost segment of this expression has outgoing colors $i_{k+1},\dots,i_M \sumin \{0,1,\dots,m\}$; we can therefore use \eqref{Z-stoch} with $\mathcal{C} = \{0,1,\dots,m\}$ to compute it as $1$ regardless of the values of the incoming colors $p_{k+1} \fixin \{0,1,\dots,m\}$ and $j_{k+2},\dots,j_M \fixin \{0,1,\dots,m-1\}$. Deleting this segment, we have
\begin{align*}
\Psi_{{\rm two-row}}
\Big(\mathcal{A},\varnothing;
(i_{\alpha \in \mathcal{A}});
(j_1,\dots,j_M)
\Big)
=
%\sum_{r_{k+1} = 0}^{m-1}
%\left(
\tikz{0.9}{
%hdots
\draw[dotted] (3,0) -- (6,0);
%vdots
\draw[dotted] (1,0) -- (1,1);
\draw[dotted] (2,0) -- (2,1);
\draw[dotted] (3,0) -- (3,1);
\draw[dotted] (4,-1) -- (4,1);
\draw[dotted] (5,-1) -- (5,1);
%paths
\draw[line width=1.5pt,->,red] (0,1) -- (0,0) -- (3,0) -- (3,-1) -- (6,-1);
\draw[line width=1.5pt,->,green] (0,1) -- (6,1) -- (6,-1);
%bottom labels
\node[left] at (0,0.5) {$j_1$};
\node[below] at (0.5,0) {$\cdots$};
\node[below] at (1.5,0) {$\cdots$};
\node[below] at (2.5,0) {$\cdots$};
\node[left] at (3,-0.5) {$m$};
\node[below] at (3.5,-1) {$j_{\ell+1}$};
\node[below] at (4.5,-1) {$\cdots$};
\node[below] at (5.5,-1) {$j_{k+1}$};
%top labels
\node[above] at (0.5,1) {$i_1$};
\node[above] at (1.5,1) {$\cdots$};
\node[above] at (2.5,1) {$\cdots$};
\node[above] at (3.5,1) {$\cdots$};
\node[above] at (4.5,1) {$\cdots$};
\node[above] at (5.5,1) {$i_{k-1}$};
\node[right] at (6,0.5) {$i_k$};
\node[right] at (6,-0.5) {$p_{k+1}$};
%rapidities
\node[left] at (-0.5,0.5) {$x_1 \rightarrow$};
\node[left] at (-0.5,-0.5) {$x_{\ell} \rightarrow$};
}
%\right)
\end{align*}
where we note that the outgoing colors $i_1,\dots,i_k$ and $p_{k+1} \sumin \{0,1,\dots,m\}$ allow us to invoke Proposition \ref{prop:merge-PQ} with $\mathcal{C} = \{m+1,m+2,\dots,N\}$. We thus have
\begin{align}
\label{psi-25}
\Psi_{{\rm two-row}}
\Big(\mathcal{A},\varnothing;
(i_{\alpha \in \mathcal{A}});
(j_1,\dots,j_M)
\Big)
=
%\sum_{r_{k+1} = 0}^{m-1}
%\left(
\tikz{0.9}{
%hdots
\draw[dotted] (3,0) -- (6,0);
%vdots
\draw[dotted] (1,0) -- (1,1);
\draw[dotted] (2,0) -- (2,1);
\draw[dotted] (3,0) -- (3,1);
\draw[dotted] (4,-1) -- (4,1);
\draw[dotted] (5,-1) -- (5,1);
%paths
\draw[line width=1.5pt,->,red] (0,1) -- (0,0) -- (3,0) -- (3,-1) -- (6,-1);
\draw[line width=1.5pt,->,green] (0,1) -- (6,1) -- (6,-1);
%bottom labels
\node[left] at (0,0.5) {\tiny $m+1$};
\node[above] at (0.5,0) {$\cdots$};
\node[above] at (1.5,0) {$\cdots$};
\node[above] at (2.5,0) {\tiny $m+1$};
\node[left] at (3,-0.5) {$m$};
\node[below] at (3.5,-1) {$j_{\ell+1}$};
\node[below] at (4.5,-1) {$\cdots$};
\node[below] at (5.5,-1) {$j_{k+1}$};
%top labels
\node[above] at (0.5,1) {$r_1$};
\node[above] at (1.5,1) {$\cdots$};
\node[above] at (2.5,1) {$r_{\ell-2}$};
\node[above] at (3.5,1) {$r_{\ell-1}$};
\node[above] at (4.5,1) {$\cdots$};
\node[above] at (5.5,1) {$r_{k-1}$};
\node[right] at (6,0.5) {$i_k$};
\node[right] at (6,-0.5) {$p_{k+1}$};
%rapidities
\node[left] at (-1,0.5) {$x_1 \rightarrow$};
\node[left] at (-1,-0.5) {$x_{\ell} \rightarrow$};
}
%\right)
\end{align}
where we have replaced all incoming colors $j_1,\dots,j_{\ell-1} \fixin \{m+1,m+2,\dots,N\}$ by $m+1$, and where the outgoing indices $r_1,\dots,r_{k-1}$ are given by \eqref{r-indices}. Examining the right hand side of \eqref{psi-25}, by color-conservation arguments we see that it vanishes identically unless $r_1 = \cdots = r_{\ell-2} = m+1$ (equivalently $1,\dots,\ell-2 \in \b{\mathcal{A}}$). If these criteria are met, we have
\begin{align*}
\Psi_{{\rm two-row}}
\Big(\mathcal{A},\varnothing;
(i_{\alpha \in \mathcal{A}});
(j_1,\dots,j_M)
\Big)
&=
%\sum_{r_{k+1} = 0}^{m-1}
%\left(
\tikz{0.9}{
\filldraw[fill=lgray,draw=lgray] (0.05,0.05) -- (0.05,0.95) -- (2.95,0.95) -- (2.95,0.05) -- (0.05,0.05);
%hdots
\draw[dotted] (3,0) -- (6,0);
%vdots
\draw[dotted] (1,0) -- (1,1);
\draw[dotted] (2,0) -- (2,1);
\draw[dotted] (3,0) -- (3,1);
\draw[dotted] (4,-1) -- (4,1);
\draw[dotted] (5,-1) -- (5,1);
%paths
\draw[line width=1.5pt,->,red] (0,1) -- (0,0) -- (3,0) -- (3,-1) -- (6,-1);
\draw[line width=1.5pt,->,green] (0,1) -- (6,1) -- (6,-1);
%bottom labels
\node[left] at (0,0.5) {\tiny $m+1$};
\node[below] at (0.5,0) {$\cdots$};
\node[below] at (1.5,0) {$\cdots$};
\node[below] at (2.5,0) {\tiny $m+1$};
\node[right] at (3,-0.5) {$m$};
\node[below] at (3.5,-1) {$j_{\ell+1}$};
\node[below] at (4.5,-1) {$\cdots$};
\node[below] at (5.5,-1) {$j_{k+1}$};
\node[right] at (2.7,0.5) {\tiny $m+1$};
%top labels
\node[above] at (0.5,1) {\tiny $m+1$};
\node[above] at (1.5,1) {$\cdots$};
\node[above] at (2.5,1) {\tiny $m+1$};
\node[above] at (3.5,1) {$r_{\ell-1}$};
\node[above] at (4.5,1) {$\cdots$};
\node[above] at (5.5,1) {$r_{k-1}$};
\node[right] at (6,0.5) {$i_k$};
\node[right] at (6,-0.5) {$p_{k+1}$};
%rapidities
\node[left] at (-1,0.5) {$x_1 \rightarrow$};
\node[left] at (-1,-0.5) {$x_{\ell} \rightarrow$};
}
%\right)
\\
&=
\tikz{0.9}{
%hdots
\draw[dotted] (3,0) -- (6,0);
%vdots
\draw[dotted] (3,0) -- (3,1);
\draw[dotted] (4,-1) -- (4,1);
\draw[dotted] (5,-1) -- (5,1);
%paths
\draw[line width=1.5pt,->,red] (3,1) -- (3,-1) -- (6,-1);
\draw[line width=1.5pt,->,green] (3,1) -- (6,1) -- (6,-1);
%bottom labels
\node[left] at (3,-0.5) {$m$};
\node[below] at (3.5,-1) {$j_{\ell+1}$};
\node[below] at (4.5,-1) {$\cdots$};
\node[below] at (5.5,-1) {$j_{k+1}$};
\node[left] at (3,0.5) {\tiny $m+1$};
%top labels
\node[above] at (3.5,1) {$r_{\ell-1}$};
\node[above] at (4.5,1) {$\cdots$};
\node[above] at (5.5,1) {$r_{k-1}$};
\node[right] at (6,0.5) {$i_k$};
\node[right] at (6,-0.5) {$p_{k+1}$};
%rapidities
\node[left] at (2,0.5) {$x_1 \rightarrow$};
\node[left] at (2,-0.5) {$x_{\ell} \rightarrow$};
}.
\end{align*}
Now let us replace the index $i_k \sumin \{0,1,\dots,m\}$ by $r_k \sumin \{0,1,\dots,m+1\}$, at the expense of subtracting away the $m+1$ term:
\begin{multline}
\label{psi-27}
\Psi_{{\rm two-row}}
\Big(\mathcal{A},\varnothing;
(i_{\alpha \in \mathcal{A}});
(j_1,\dots,j_M)
\Big)
=
\\
\left(
\tikz{0.9}{
%hdots
\draw[dotted] (3,0) -- (6,0);
%vdots
\draw[dotted] (3,0) -- (3,1);
\draw[dotted] (4,-1) -- (4,1);
\draw[dotted] (5,-1) -- (5,1);
%paths
\draw[line width=1.5pt,->,red] (3,1) -- (3,-1) -- (6,-1);
\draw[line width=1.5pt,->,green] (3,1) -- (6,1) -- (6,-1);
%bottom labels
\node[left] at (3,-0.5) {$m$};
\node[below] at (3.5,-1) {$j_{\ell+1}$};
\node[below] at (4.5,-1) {$\cdots$};
\node[below] at (5.5,-1) {$j_{k+1}$};
\node[left] at (3,0.5) {\tiny $m+1$};
%top labels
\node[above] at (3.5,1) {$r_{\ell-1}$};
\node[above] at (4.5,1) {$\cdots$};
\node[above] at (5.5,1) {$r_{k-1}$};
\node[right] at (6,0.5) {$r_k$};
\node[right] at (6,-0.5) {$p_{k+1}$};
%rapidities
\node[left] at (2,0.5) {$x_1 \rightarrow$};
\node[left] at (2,-0.5) {$x_{\ell} \rightarrow$};
}
\right)
-
\left(
\tikz{0.9}{
%\filldraw[fill=lgray,draw=lgray] (3.05,-0.95) -- (3.05,-0.05) -- (5.95,-0.05) -- (5.95,-0.95) -- (3.05,-0.95);
%hdots
\draw[dotted] (3,0) -- (6,0);
%vdots
\draw[dotted] (3,0) -- (3,1);
\draw[dotted] (4,-1) -- (4,1);
\draw[dotted] (5,-1) -- (5,1);
%paths
\draw[line width=1.5pt,->,red] (3,1) -- (3,-1) -- (6,-1);
\draw[line width=1.5pt,->,green] (3,1) -- (6,1) -- (6,-1);
%bottom labels
\node[left] at (3,-0.5) {$m$};
\node[below] at (3.5,-1) {$j_{\ell+1}$};
\node[below] at (4.5,-1) {$\cdots$};
\node[below] at (5.5,-1) {$j_{k+1}$};
\node[left] at (3,0.5) {\tiny $m+1$};
%\node[below] at (3.5,0) {$j_{\ell+1}$};
%\node[below] at (4.5,0) {$\cdots$};
%\node[below] at (5.5,0) {$j_{k+1}$};
%top labels
\node[above] at (3.5,1) {$r_{\ell-1}$};
\node[above] at (4.5,1) {$\cdots$};
\node[above] at (5.5,1) {$r_{k-1}$};
\node[right] at (6,0.5) {\tiny $m+1$};
\node[right] at (6,-0.5) {$p_{k+1}$};
%rapidities
\node[left] at (2,0.5) {$x_1 \rightarrow$};
\node[left] at (2,-0.5) {$x_{\ell} \rightarrow$};
}
\right).
\end{multline}
Both objects appearing in \eqref{psi-27} can then be simplified further. The object on the left has outgoing indices which allow us to invoke Proposition \ref{prop:merge-PQ} with $\mathcal{C} = \{m,m+1\}$. The object on the right has a lone $m+1$ among its incoming colors, and it follows that none of $r_{\ell-1},\dots,r_{k-1}$ can assume this value (since $m+1$ already exits the partition function via the right edge of the top row). We conclude that the whole top row of the second term is frozen, and its weight is computed as a product of bottom middle weights in Figure \ref{fund-vert}:
\begin{multline*}
\Psi_{{\rm two-row}}
\Big(\mathcal{A},\varnothing;
(i_{\alpha \in \mathcal{A}});
(j_1,\dots,j_M)
\Big)
=
\\
\left(
\tikz{0.9}{
%hdots
\draw[dotted] (3,0) -- (6,0);
%vdots
\draw[dotted] (3,0) -- (3,1);
\draw[dotted] (4,-1) -- (4,1);
\draw[dotted] (5,-1) -- (5,1);
%paths
\draw[line width=1.5pt,->,red] (3,1) -- (3,-1) -- (6,-1);
\draw[line width=1.5pt,->,green] (3,1) -- (6,1) -- (6,-1);
%bottom labels
\node[left] at (3,-0.5) {$m$};
\node[below] at (3.5,-1) {$j_{\ell+1}$};
\node[below] at (4.5,-1) {$\cdots$};
\node[below] at (5.5,-1) {$j_{k+1}$};
\node[left] at (3,0.5) {$m$};
%top labels
\node[above] at (3.5,1) {$p_{\ell-1}$};
\node[above] at (4.5,1) {$\cdots$};
\node[above] at (5.5,1) {$p_{k-1}$};
\node[right] at (6,0.5) {$p_k$};
\node[right] at (6,-0.5) {$p_{k+1}$};
%rapidities
\node[left] at (2.5,0.5) {$x_1 \rightarrow$};
\node[left] at (2.5,-0.5) {$x_{\ell} \rightarrow$};
}
\right)
-
\left(
\tikz{0.9}{
\filldraw[fill=lgray,draw=lgray] (3.05,0.05) -- (3.05,0.95) -- (5.95,0.95) -- (5.95,0.05) -- (3.05,0.05);
%hdots
\draw[dotted] (3,0) -- (6,0);
%vdots
\draw[dotted] (3,0) -- (3,1);
\draw[dotted] (4,-1) -- (4,1);
\draw[dotted] (5,-1) -- (5,1);
%paths
\draw[line width=1.5pt,->,red] (3,1) -- (3,-1) -- (6,-1);
\draw[line width=1.5pt,->,green] (3,1) -- (6,1) -- (6,-1);
%bottom labels
\node[left] at (3,-0.5) {$m$};
\node[below] at (3.5,-1) {$j_{\ell+1}$};
\node[below] at (4.5,-1) {$\cdots$};
\node[below] at (5.5,-1) {$j_{k+1}$};
\node[left] at (3,0.5) {\tiny $m+1$};
\node[above] at (3.5,0) {$p_{\ell-1}$};
\node[above] at (4.5,0) {$\cdots$};
\node[above] at (5.5,0) {$p_{k-1}$};
%top labels
\node[above] at (3.5,1) {$p_{\ell-1}$};
\node[above] at (4.5,1) {$\cdots$};
\node[above] at (5.5,1) {$p_{k-1}$};
\node[right] at (6,0.5) {\tiny $m+1$};
\node[right] at (6,-0.5) {$p_{k+1}$};
%rapidities
\node[left] at (2,0.5) {$x_1 \rightarrow$};
\node[left] at (2,-0.5) {$x_{\ell} \rightarrow$};
}
\right),
\end{multline*}
where the indices $p_{\ell-1},\dots,p_k$ are given by \eqref{p-indices}. More explicitly, we now read the identity
\begin{multline}
\label{psi-277}
\Psi_{{\rm two-row}}
\Big(\mathcal{A},\varnothing;
(i_{\alpha \in \mathcal{A}});
(j_1,\dots,j_M)
\Big)
=
\\
\left(
\tikz{0.9}{
%hdots
\draw[dotted] (3,0) -- (6,0);
%vdots
\draw[dotted] (3,0) -- (3,1);
\draw[dotted] (4,-1) -- (4,1);
\draw[dotted] (5,-1) -- (5,1);
%paths
\draw[line width=1.5pt,->,red] (3,1) -- (3,-1) -- (6,-1);
\draw[line width=1.5pt,->,green] (3,1) -- (6,1) -- (6,-1);
%bottom labels
\node[left] at (3,-0.5) {$m$};
\node[below] at (3.5,-1) {$j_{\ell+1}$};
\node[below] at (4.5,-1) {$\cdots$};
\node[below] at (5.5,-1) {$j_{k+1}$};
\node[left] at (3,0.5) {$m$};
%top labels
\node[above] at (3.5,1) {$p_{\ell-1}$};
\node[above] at (4.5,1) {$\cdots$};
\node[above] at (5.5,1) {$p_{k-1}$};
\node[right] at (6,0.5) {$p_k$};
\node[right] at (6,-0.5) {$p_{k+1}$};
%rapidities
\node[left] at (2.5,0.5) {$x_1 \rightarrow$};
\node[left] at (2.5,-0.5) {$x_{\ell} \rightarrow$};
}
\right)
-
\prod_{\alpha=\ell+1}^{k+1}
\frac{x_1-y_{\alpha}}{x_1-qy_{\alpha}}
\left(
\tikz{0.9}{
%vdots
\draw[dotted] (3,-1) -- (3,0);
\draw[dotted] (4,-1) -- (4,0);
\draw[dotted] (5,-1) -- (5,0);
%paths
\draw[line width=1.5pt,->,red] (2,0) -- (2,-1) -- (5,-1);
\draw[line width=1.5pt,->,green] (2,0) -- (5,0) -- (5,-1);
%bottom labels
\node[left] at (2,-0.5) {$m$};
\node[below] at (2.5,-1) {$j_{\ell+1}$};
\node[below] at (3.5,-1) {$\cdots$};
\node[below] at (4.5,-1) {$j_{k+1}$};
%top labels
\node[above] at (2.5,0) {$p_{\ell-1}$};
\node[above] at (3.5,0) {$\cdots$};
\node[above] at (4.5,0) {$p_{k-1}$};
\node[right] at (5,-0.5) {$p_{k+1}$};
%rapidities
\node[left] at (1.5,-0.5) {$x_{\ell} \rightarrow$};
}
\right),
\end{multline}
and we see that \eqref{phi-277} and \eqref{psi-277} are indeed equal up to the exchange $x_1 \leftrightarrow x_{\ell}$ (note that both $p_k \sumin \{0,1,\dots,m\}$ and $p_{k+1} \sumin \{0,1,\dots,m\}$). This completes the proof of Theorem \ref{lem:2row}.
\end{proof}

\subsection{Part two: proof for Z-shaped domains}
\label{ssec:z-shape}

We now elevate our proof of Theorem \ref{theorema-egr} to a more general class of domains. We call these domains ``Z-shaped'' because of their appearance; one may view them as coming from the partition functions in Section \ref{ssec:two-row} by inserting ``internal'' rows that now drive apart the two rows that were initially present.

Recall that the two-row partition functions in Section \ref{ssec:two-row} depend implicitly on two integers $k,\ell$; knowing these integers, as well as $M$ (the number of steps of the two down-right paths $P$ and $Q$), allows one to draw the precise down-right domain that one is dealing with. The Z-shaped domains in this subsection will be specified in terms of $k,\ell$ and a supplementary integer $2 \leq c \leq \ell$. The extra integer allows us to keep track of the number of ``internal'' rows in our domain; we assume that there are $\ell-c$ such rows. When $c=\ell$, we return to the case studied in Section \ref{ssec:two-row}.

The two quantities that we study are phrased in terms of the following class of partition functions (a subclass of the partition functions in Definition \ref{def:PF}):
\begin{align}
\label{Zc}
Z^{c}_{k,\ell}
\Big[
i_1,\dots,i_M
\Big|
j_1,\dots,j_M
\Big]
:=
%\\
\tikz{0.8}{
%hdots
\draw[dotted] (3,0) -- (7,0);
\draw[dotted] (3,-1) -- (7,-1);
\draw[dotted] (3,-2) -- (7,-2);
\draw[dotted] (3,-3) -- (7,-3);
\draw[dotted] (3,-4) -- (7,-4);
%vdots
\draw[dotted] (1,0) -- (1,1);
\draw[dotted] (2,0) -- (2,1);
\draw[dotted] (3,0) -- (3,1);
\draw[dotted] (4,-5) -- (4,1);
\draw[dotted] (5,-5) -- (5,1);
\draw[dotted] (6,-5) -- (6,1);
\draw[dotted] (7,-5) -- (7,-4);
\draw[dotted] (8,-5) -- (8,-4);
\draw[dotted] (9,-5) -- (9,-4);
\draw[dotted] (10,-5) -- (10,-4);
%paths
\node[right,green] at (11,-5) {$P$};
\node[below,red] at (11,-5) {$Q$};
\draw[line width=1.5pt,->,red] (0,1) -- (0,0) -- (3,0) -- (3,-5) -- (11,-5);
\draw[line width=1.5pt,->,green] (0,1) -- (7,1) -- (7,-4) -- (11,-4) -- (11,-5);
%bottom labels
\node[left] at (0,0.5) {$j_1$};
\node[below] at (0.5,0) {$\cdots$};
\node[below] at (1.5,0) {$\cdots$};
\node[below] at (2.5,0) {$\cdots$};
\node[left] at (3,-0.6) {$j_c$};
\node[left] at (3,-1.5) {$\vdots$};
\node[left] at (3,-2.5) {$\vdots$};
\node[left] at (3,-3.5) {$j_{\ell-1}$};
\node[left] at (3,-4.5) {$j_{\ell}$}; \node at (3,-4.5) {$\bullet$};
\node[below] at (3.5,-5) {$j_{\ell+1}$};
\node[below] at (4.5,-5) {$\cdots$};
\node[below] at (5.5,-5) {$\cdots$};
\node[below] at (6.5,-5) {$\cdots$};
\node[below] at (7.5,-5) {$\cdots$};
\node[below] at (8.5,-5) {$\cdots$};
\node[below] at (9.5,-5) {$\cdots$};
\node[below] at (10.5,-5) {$j_M$};
%top labels
\node[above] at (0.5,1) {$i_1$};
\node[above] at (1.5,1) {$\cdots$};
\node[above] at (2.5,1) {$\cdots$};
\node[above] at (3.5,1) {$\cdots$};
\node[above] at (4.5,1) {$\cdots$};
\node[above] at (5.5,1) {$\cdots$};
\node[above] at (6.5,1) {$i_{k-1}$};
\node[right] at (7,0.5) {$i_k$}; \filldraw[draw=black,fill=white] (7,0.5) circle (2pt);
\node[right] at (7,-0.5) {$i_{k+1}$};
\node[right] at (7,-1.3) {$\vdots$};
\node[right] at (7,-2.3) {$\vdots$};
\node[right] at (7,-3.3) {$\vdots$};
\node[above] at (7.5,-4) {$\cdots$};
\node[above] at (8.5,-4) {$\cdots$};
\node[above] at (9.5,-4) {$\cdots$};
\node[above] at (10.5,-4) {$\cdots$};
\node[right] at (11,-4.5) {$i_M$};
%rapidities
\node[left] at (-0.5,0.5) {$x_1 \rightarrow$};
\node[left] at (-0.5,-0.5) {$x_c \rightarrow$};
\node[left] at (-0.5,-3.5) {$x_{\ell-1} \rightarrow$};
\node[left] at (-0.5,-4.5) {$x_{\ell} \rightarrow$};
\node[below] at (0.5,2.1) {$y_2$}; \node at (0.5,2.4) {$\uparrow$};
\node[below] at (1.5,2.1) {$\cdots$};
\node[below] at (2.5,2.1) {$y_{c-1}$}; \node at (2.5,2.4) {$\uparrow$};
\node[below] at (3.5,2.1) {$y_{\ell+1}$}; \node at (3.5,2.4) {$\uparrow$};
\node[below] at (4.5,2.1) {$\cdots$};
\node[below] at (5.5,2.1) {$\cdots$};
\node[below] at (6.5,2.1) {$y_{d}$}; \node at (6.5,2.4) {$\uparrow$};
\node[below] at (7.5,2.1) {$y_{d+1}$}; \node at (7.5,2.4) {$\uparrow$};
\node[below] at (8.5,2.1) {$\cdots$};
\node[below] at (9.5,2.1) {$\cdots$};
\node[below] at (10.5,2.1) {$y_{M}$}; \node at (10.5,2.4) {$\uparrow$};
}
\end{align}
where $(i_1,\dots,i_M)$ and $(j_1,\dots,j_M)$ are colorings of the down-right paths $P$ and $Q$, respectively; we have also introduced the shorthand
$d = k+\ell-c+1$.

For the proof of the Shift Theorem we use induction in $\ell-c$. The main idea is to treat the equality of two partition functions as a polynomial identity in rapidities, which then holds true as soon as we can prove it in a large enough collection of values for the rapidities. For the values we are going to choose the points $x_i=y_j$, for which the vertex splitting of Section \ref{ssec:split} occurs. The partition functions for the split domains can be then modified using the Yang-Baxter relations to the ones for which the Shift Theorem is already known by the induction assumption.

We begin the argument by proving some properties of the partition function \eqref{Zc}.

\begin{prop}
\label{prop:properties}
Fix an integer $c$ such that $2 \leq c \leq \ell-1$. The partition function $Z^{c}_{k,\ell}$ satisfies the following properties:
\begin{enumerate}[label=(\roman*)]

\item The quantity
\begin{align*}
\bar{Z}^{c}_{k,\ell}\Big[ i_1,\dots,i_M \Big| j_1,\dots,j_M \Big]
:=
\prod_{\alpha=\ell+1}^{d} (x_c-q y_\alpha)
Z^{c}_{k,\ell}\Big[ i_1,\dots,i_M \Big| j_1,\dots,j_M \Big]
\end{align*}
is a polynomial in $x_c$.

\medskip

\item The indices $j_c$ and $i_{k+1}$ are the left and right edge states of the row of the partition function that carries the parameter $x_c$. If $j_c \leq i_{k+1}$, the polynomial $\bar{Z}^{c}_{k,\ell}[ i_1,\dots,i_M | j_1,\dots,j_M ]$ has degree $d-\ell$ in $x_c$. If $j_c > i_{k+1}$, $\bar{Z}^{c}_{k,\ell}[ i_1,\dots,i_M | j_1,\dots,j_M ]$ has degree $d-\ell-1$ in $x_c$.

\medskip

\item Setting $x_c = y_\gamma$, $\gamma \in \{\ell+1,\dots,d\}$ produces the following recursion:
\begin{multline}
\label{prop3}
Z^{c}_{k,\ell}\Big[ i_1,\dots,i_M \Big| j_1,\dots,j_M \Big]
\Big|_{x_c = y_{\gamma}}
=
\sum_{(i'_1,\dots,i'_M)}
\sum_{(j'_1,\dots,j'_M)}
\xi_{\gamma}
\Big[
j'_1,\dots,j'_M \Big| j_1,\dots,j_M
\Big]
\\
\times
\eta_{\gamma}
\Big[
i_1,\dots,i_M \Big| i'_1,\dots,i'_M
\Big]
%\\
%\times
%\mathfrak{s}^{(y)}_{\gamma,d}
%\dots
%\mathfrak{s}^{(y)}_{\gamma,\gamma+1}
Z^{c+1}_{k,\ell}
\Big[ i'_1,\dots,i'_M
\Big| j'_1,\dots,j'_M \Big]
%\Big|_{y_c=y_{\gamma}}
\end{multline}
where $Z^{c+1}_{k,\ell}$ is a partition function of the type \eqref{Zc}, whose rapidity assignment needs to be modified from the canonical conventions; see Figure \ref{fig:recursion-Ztilde}.
%for an illustration of the rapidity assignment after applying the permutation operators $\mathfrak{s}^{(y)}_{\gamma,d} \dots \mathfrak{s}^{(y)}_{\gamma,\gamma+1}$ and setting $y_c=y_{\gamma}$.
The coefficients $\xi_{\gamma}$ and $\eta_{\gamma}$ are independent of $x_c$ and satisfy certain color conservation properties, detailed in \eqref{xi-conserve1}--\eqref{eta-conserve2} below.
%
%, and satisfy the factorization properties
%%
%\begin{align}
%\label{xi-factor}
%\xi_{\gamma}
%\Big[
%j'_1,\dots,j'_M \Big| j_1,\dots,j_M
%\Big]
%=
%\xi_{\gamma}
%\Big[
%j'_1,\dots,j'_{\ell-1} \Big| j_1,\dots,j_{\ell-1}
%\Big]
%\cdot
%\left(
%{\bm 1}_{j'_{\ell}=j_{\ell}=m}
%\right)
%\cdot
%\xi_{\gamma}
%\Big[
%j'_{\ell+1},\dots,j'_M \Big| j_{\ell+1},\dots,j_M
%\Big],
%\end{align}
%%
%\begin{align}
%\label{eta-factor}
%\eta_{\gamma}
%\Big[
%i_1,\dots,i_M \Big| i'_1,\dots,i'_M
%\Big]
%=
%\eta_{\gamma}
%\Big[
%i_1,\dots,i_{k-1} \Big| i'_1,\dots,i'_{k-1}
%\Big]
%\cdot
%\left(
%{\bm 1}_{i'_k=i_k}
%\right)
%\cdot
%\eta_{\gamma}
%\Big[
%i_{k+1},\dots,i_M \Big| i'_{k+1},\dots,i'_M
%\Big].
%\end{align}

\begin{figure}[t]
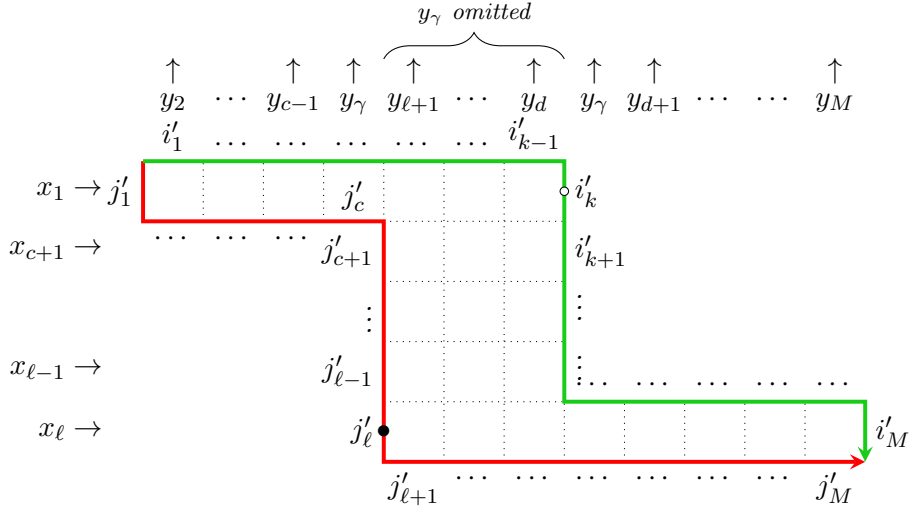

\tikz{0.8}{
%hdots
\draw[dotted] (4,0) -- (7,0);
\draw[dotted] (4,-1) -- (7,-1);
\draw[dotted] (4,-2) -- (7,-2);
\draw[dotted] (4,-3) -- (7,-3);
\draw[dotted] (4,-4) -- (7,-4);
%vdots
\draw[dotted] (1,0) -- (1,1);
\draw[dotted] (2,0) -- (2,1);
\draw[dotted] (3,0) -- (3,1);
\draw[dotted] (4,-4) -- (4,1);
\draw[dotted] (5,-4) -- (5,1);
\draw[dotted] (6,-4) -- (6,1);
\draw[dotted] (7,-4) -- (7,-3);
\draw[dotted] (8,-4) -- (8,-3);
\draw[dotted] (9,-4) -- (9,-3);
\draw[dotted] (10,-4) -- (10,-3);
\draw[dotted] (11,-4) -- (11,-3);
%paths
%\node[right,green] at (11,-5) {$P$};
%\node[below,red] at (11,-5) {$Q$};
\draw[line width=1.5pt,->,red] (0,1) -- (0,0) -- (4,0) -- (4,-4) -- (12,-4);
\draw[line width=1.5pt,->,green] (0,1) -- (7,1) -- (7,-3) -- (12,-3) -- (12,-4);
%bottom labels
\node[left] at (0,0.5) {$j'_1$};
\node[below] at (0.5,0) {$\cdots$};
\node[below] at (1.5,0) {$\cdots$};
\node[below] at (2.5,0) {$\cdots$};
\node[above] at (3.5,0) {$j'_c$};
\node[left] at (4,-0.5) {$j'_{c+1}$};
\node[left] at (4,-1.5) {$\vdots$};
\node[left] at (4,-2.5) {$j'_{\ell-1}$};
\node[left] at (4,-3.5) {$j'_{\ell}$}; \node at (4,-3.5) {$\bullet$};
\node[below] at (4.5,-4) {$j'_{\ell+1}$};
\node[below] at (5.5,-4) {$\cdots$};
\node[below] at (6.5,-4) {$\cdots$};
\node[below] at (7.5,-4) {$\cdots$};
\node[below] at (8.5,-4) {$\cdots$};
\node[below] at (9.5,-4) {$\cdots$};
\node[below] at (10.5,-4) {$\cdots$};
\node[below] at (11.5,-4) {$j'_M$};
%top labels
\node[above] at (0.5,1) {$i'_1$};
\node[above] at (1.5,1) {$\cdots$};
\node[above] at (2.5,1) {$\cdots$};
\node[above] at (3.5,1) {$\cdots$};
\node[above] at (4.5,1) {$\cdots$};
\node[above] at (5.5,1) {$\cdots$};
\node[above] at (6.5,1) {$i'_{k-1}$};
\node[right] at (7,0.5) {$i'_k$}; \filldraw[draw=black,fill=white] (7,0.5) circle (2pt);
\node[right] at (7,-0.5) {$i'_{k+1}$};
\node[right] at (7,-1.3) {$\vdots$};
\node[right] at (7,-2.3) {$\vdots$};
\node[above] at (7.5,-3) {$\cdots$};
\node[above] at (8.5,-3) {$\cdots$};
\node[above] at (9.5,-3) {$\cdots$};
\node[above] at (10.5,-3) {$\cdots$};
\node[above] at (11.5,-3) {$\cdots$};
\node[right] at (12,-3.5) {$i'_M$};
%rapidities
\node[left] at (-0.5,0.5) {$x_1 \rightarrow$};
\node[left] at (-0.5,-0.5) {$x_{c+1} \rightarrow$};
\node[left] at (-0.5,-2.5) {$x_{\ell-1} \rightarrow$};
\node[left] at (-0.5,-3.5) {$x_{\ell} \rightarrow$};
\node[below] at (0.5,2.3) {$y_2$}; \node at (0.5,2.5) {$\uparrow$};
\node[below] at (1.5,2.3) {$\cdots$};
\node[below] at (2.5,2.3) {$y_{c-1}$}; \node at (2.5,2.5) {$\uparrow$};
\node[below] at (3.5,2.3) {$y_{\gamma}$}; \node at (3.5,2.5) {$\uparrow$};
\node[below] at (4.5,2.3) {$y_{\ell+1}$}; \node at (4.5,2.5) {$\uparrow$};
\node[below] at (5.5,2.3) {$\cdots$};
\node[below] at (6.5,2.3) {$y_{d}$}; \node at (6.5,2.5) {$\uparrow$};
\node[below] at (7.5,2.3) {$y_{\gamma}$}; \node at (7.5,2.5) {$\uparrow$};
\node[below] at (8.5,2.3) {$y_{d+1}$}; \node at (8.5,2.5) {$\uparrow$};
\node[below] at (9.5,2.3) {$\cdots$};
\node[below] at (10.5,2.3) {$\cdots$};
\node[below] at (11.5,2.3) {$y_{M}$}; \node at (11.5,2.5) {$\uparrow$};
\draw[decorate,decoration={brace,amplitude=10pt}] (4,2.7) -- (7,2.7)
node [black,midway,yshift=0.6cm] {\footnotesize $y_{\gamma}$\ omitted};
}
\caption{The quantity $Z^{c+1}_{k,\ell}[i'_1,\dots,i'_M | j'_1,\dots,j'_M]$, in which the vertical rapidities have a modified labelling.
%after acting with the permutation operators $\mathfrak{s}^{(y)}_{\gamma,d} \dots \mathfrak{s}^{(y)}_{\gamma,\gamma+1}$ and setting $y_c=y_{\gamma}$.
}
\label{fig:recursion-Ztilde}
\end{figure}

\item If $j_c < i_{k+1}$, then setting $x_c=0$ we have
%
%\begin{align*}
$
Z^{c}_{k,\ell}\Big[ i_1,\dots,i_M \Big| j_1,\dots,j_M \Big]
\Big|_{x_c = 0}
=
0.
$
%\end{align*}

\medskip

\item If $j_c = i_{k+1}$, then sending $x_c \rightarrow \infty$ we have
\begin{multline}
\label{x=0}
\lim_{x_c \rightarrow \infty}
Z^{c}_{k,\ell}\Big[ i_1,\dots,i_M \Big| j_1,\dots,j_M \Big]
=
q^{[\#(i_1,\dots,i_k) > j_c]-[\#(j_1,\dots,j_{c-1}) > j_c]}
\\
\times
%\mathfrak{s}^{(x)}_{c,\dots,\ell-1}
%\mathfrak{s}^{(y)}_{\ell,\dots,M-1}
Z^{c}_{k,\ell-1}
\Big[ i_1,\dots,i_k,i_{k+2},\dots,i_M
\Big| j_1,\dots,j_{c-1},j_{c+1},\dots,j_M \Big],
\end{multline}
where the right hand side of \eqref{x=0} denotes the partition function obtained by deleting the $x_c$ row from $Z^{c}_{k,\ell}[ i_1,\dots,i_M | j_1,\dots,j_M ]$; see Figure \ref{fig:recursion-Zhat}.

\end{enumerate}

\begin{figure}[t]
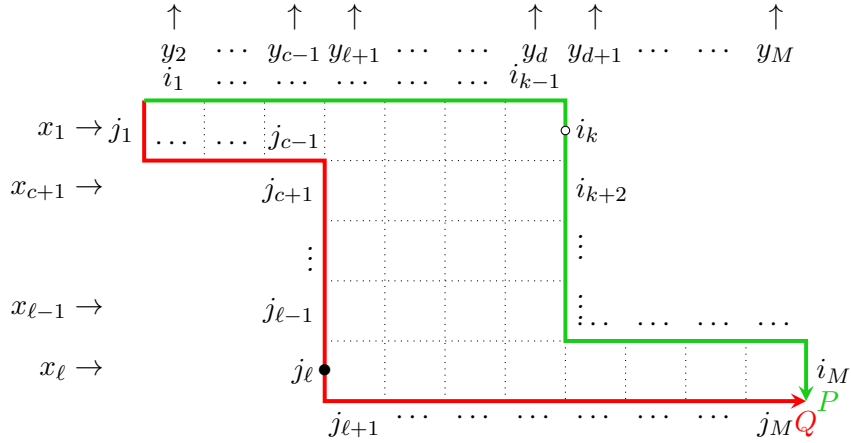

\begin{align*}
\tikz{0.8}{
%hdots
\draw[dotted] (3,0) -- (7,0);
\draw[dotted] (3,-1) -- (7,-1);
\draw[dotted] (3,-2) -- (7,-2);
\draw[dotted] (3,-3) -- (7,-3);
%vdots
\draw[dotted] (1,0) -- (1,1);
\draw[dotted] (2,0) -- (2,1);
\draw[dotted] (3,0) -- (3,1);
\draw[dotted] (4,-4) -- (4,1);
\draw[dotted] (5,-4) -- (5,1);
\draw[dotted] (6,-4) -- (6,1);
\draw[dotted] (7,-3) -- (7,-4);
\draw[dotted] (8,-3) -- (8,-4);
\draw[dotted] (9,-3) -- (9,-4);
\draw[dotted] (10,-3) -- (10,-4);
%paths
\node[right,green] at (11,-4) {$P$};
\node[below,red] at (11,-4) {$Q$};
\draw[line width=1.5pt,->,red] (0,1) -- (0,0) -- (3,0) -- (3,-4) -- (11,-4);
\draw[line width=1.5pt,->,green] (0,1) -- (7,1) -- (7,-3) -- (11,-3) -- (11,-4);
%bottom labels
\node[left] at (0,0.5) {$j_1$};
\node[above] at (0.5,0) {$\cdots$};
\node[above] at (1.5,0) {$\cdots$};
\node[above] at (2.5,0) {$j_{c-1}$};
\node[left] at (3,-0.5) {$j_{c+1}$};
\node[left] at (3,-1.5) {$\vdots$};
\node[left] at (3,-2.5) {$j_{\ell-1}$};
\node[left] at (3,-3.5) {$j_{\ell}$}; \node at (3,-3.5) {$\bullet$};
\node[below] at (3.5,-4) {$j_{\ell+1}$};
\node[below] at (4.5,-4) {$\cdots$};
\node[below] at (5.5,-4) {$\cdots$};
\node[below] at (6.5,-4) {$\cdots$};
\node[below] at (7.5,-4) {$\cdots$};
\node[below] at (8.5,-4) {$\cdots$};
\node[below] at (9.5,-4) {$\cdots$};
\node[below] at (10.5,-4) {$j_M$};
%top labels
\node[above] at (0.5,1) {$i_1$};
\node[above] at (1.5,1) {$\cdots$};
\node[above] at (2.5,1) {$\cdots$};
\node[above] at (3.5,1) {$\cdots$};
\node[above] at (4.5,1) {$\cdots$};
\node[above] at (5.5,1) {$\cdots$};
\node[above] at (6.5,1) {$i_{k-1}$};
\node[right] at (7,0.5) {$i_k$}; \filldraw[draw=black,fill=white] (7,0.5) circle (2pt);
\node[right] at (7,-0.5) {$i_{k+2}$};
\node[right] at (7,-1.3) {$\vdots$};
\node[right] at (7,-2.3) {$\vdots$};
\node[above] at (7.5,-3) {$\cdots$};
\node[above] at (8.5,-3) {$\cdots$};
\node[above] at (9.5,-3) {$\cdots$};
\node[above] at (10.5,-3) {$\cdots$};
\node[right] at (11,-3.5) {$i_M$};
%rapidities
\node[left] at (-0.5,0.5) {$x_1 \rightarrow$};
\node[left] at (-0.5,-0.5) {$x_{c+1} \rightarrow$};
\node[left] at (-0.5,-2.5) {$x_{\ell-1} \rightarrow$};
\node[left] at (-0.5,-3.5) {$x_{\ell} \rightarrow$};
\node[below] at (0.5,2.1) {$y_2$}; \node at (0.5,2.4) {$\uparrow$};
\node[below] at (1.5,2.1) {$\cdots$};
\node[below] at (2.5,2.1) {$y_{c-1}$}; \node at (2.5,2.4) {$\uparrow$};
\node[below] at (3.5,2.1) {$y_{\ell+1}$}; \node at (3.5,2.4) {$\uparrow$};
\node[below] at (4.5,2.1) {$\cdots$};
\node[below] at (5.5,2.1) {$\cdots$};
\node[below] at (6.5,2.1) {$y_{d}$}; \node at (6.5,2.4) {$\uparrow$};
\node[below] at (7.5,2.1) {$y_{d+1}$}; \node at (7.5,2.4) {$\uparrow$};
\node[below] at (8.5,2.1) {$\cdots$};
\node[below] at (9.5,2.1) {$\cdots$};
\node[below] at (10.5,2.1) {$y_{M}$}; \node at (10.5,2.4) {$\uparrow$};
}
\end{align*}
\caption{The quantity
$Z^{c}_{k,\ell-1} [i_1,\dots,i_k,i_{k+2},\dots,i_M
| j_1,\dots,j_{c-1},j_{c+1},\dots,j_M]$;
%under the action of the variable-shifting operators
%$\mathfrak{s}^{(x)}_{c,\dots,\ell-1}$ and $\mathfrak{s}^{(y)}_{\ell,\dots,M-1}$
one may view this as the partition function obtained by deleting the $x_c$-bearing row from $Z^{c}_{k,\ell}[ i_1,\dots,i_M | j_1,\dots,j_M ]$.}
\label{fig:recursion-Zhat}
\end{figure}
\end{prop}

\begin{proof}
We establish the five properties using the vertex description of the Z-shaped domain; see Figure \ref{fig:Z-vertex}.
%In what follows, the listed items correspond with the five listed properties of the Proposition.
%
\begin{figure}[t]
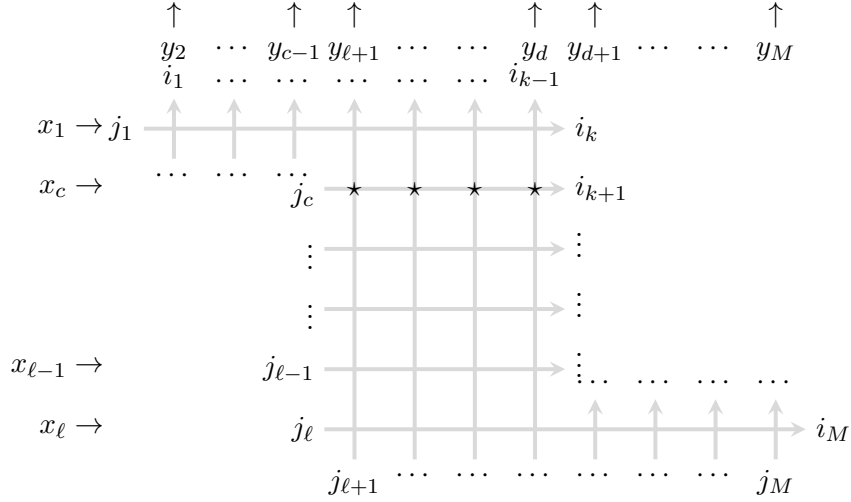

\tikz{0.8}{
%hdots
\draw[lgray,line width=1.5pt,->] (0,0.5) -- (7,0.5);
\draw[lgray,line width=1.5pt,->] (3,-0.5) -- (7,-0.5);
\draw[lgray,line width=1.5pt,->] (3,-1.5) -- (7,-1.5);
\draw[lgray,line width=1.5pt,->] (3,-2.5) -- (7,-2.5);
\draw[lgray,line width=1.5pt,->] (3,-3.5) -- (7,-3.5);
\draw[lgray,line width=1.5pt,->] (3,-4.5) -- (11,-4.5);
%vdots
\draw[lgray,line width=1.5pt,->] (0.5,0) -- (0.5,1);
\draw[lgray,line width=1.5pt,->] (1.5,0) -- (1.5,1);
\draw[lgray,line width=1.5pt,->] (2.5,0) -- (2.5,1);
\draw[lgray,line width=1.5pt,->] (3.5,-5) -- (3.5,1);
\draw[lgray,line width=1.5pt,->] (4.5,-5) -- (4.5,1);
\draw[lgray,line width=1.5pt,->] (5.5,-5) -- (5.5,1);
\draw[lgray,line width=1.5pt,->] (6.5,-5) -- (6.5,1);
\draw[lgray,line width=1.5pt,->] (7.5,-5) -- (7.5,-4);
\draw[lgray,line width=1.5pt,->] (8.5,-5) -- (8.5,-4);
\draw[lgray,line width=1.5pt,->] (9.5,-5) -- (9.5,-4);
\draw[lgray,line width=1.5pt,->] (10.5,-5) -- (10.5,-4);
%bottom labels
\node[left] at (0,0.5) {$j_1$};
\node[below] at (0.5,0) {$\cdots$};
\node[below] at (1.5,0) {$\cdots$};
\node[below] at (2.5,0) {$\cdots$};
\node[left] at (3,-0.6) {$j_c$};
\node[left] at (3,-1.5) {$\vdots$};
\node[left] at (3,-2.5) {$\vdots$};
\node[left] at (3,-3.5) {$j_{\ell-1}$};
\node[left] at (3,-4.5) {$j_{\ell}$}; %\node at (3,-4.5) {$\bullet$};
\node[below] at (3.5,-5) {$j_{\ell+1}$};
\node[below] at (4.5,-5) {$\cdots$};
\node[below] at (5.5,-5) {$\cdots$};
\node[below] at (6.5,-5) {$\cdots$};
\node[below] at (7.5,-5) {$\cdots$};
\node[below] at (8.5,-5) {$\cdots$};
\node[below] at (9.5,-5) {$\cdots$};
\node[below] at (10.5,-5) {$j_M$};
%top labels
\node[above] at (0.5,1) {$i_1$};
\node[above] at (1.5,1) {$\cdots$};
\node[above] at (2.5,1) {$\cdots$};
\node[above] at (3.5,1) {$\cdots$};
\node[above] at (4.5,1) {$\cdots$};
\node[above] at (5.5,1) {$\cdots$};
\node[above] at (6.5,1) {$i_{k-1}$};
\node[right] at (7,0.5) {$i_k$}; %\filldraw[draw=black,fill=white] (7,0.5) circle (2pt);
\node[right] at (7,-0.5) {$i_{k+1}$};
\node[right] at (7,-1.3) {$\vdots$};
\node[right] at (7,-2.3) {$\vdots$};
\node[right] at (7,-3.3) {$\vdots$};
\node[above] at (7.5,-4) {$\cdots$};
\node[above] at (8.5,-4) {$\cdots$};
\node[above] at (9.5,-4) {$\cdots$};
\node[above] at (10.5,-4) {$\cdots$};
\node[right] at (11,-4.5) {$i_M$};
%markers
\node at (3.5,-0.5) {$\star$};
\node at (4.5,-0.5) {$\star$};
\node at (5.5,-0.5) {$\star$};
\node at (6.5,-0.5) {$\star$};
%rapidities
\node[left] at (-0.5,0.5) {$x_1 \rightarrow$};
\node[left] at (-0.5,-0.5) {$x_c \rightarrow$};
\node[left] at (-0.5,-3.5) {$x_{\ell-1} \rightarrow$};
\node[left] at (-0.5,-4.5) {$x_{\ell} \rightarrow$};
\node[below] at (0.5,2.1) {$y_2$}; \node at (0.5,2.4) {$\uparrow$};
\node[below] at (1.5,2.1) {$\cdots$};
\node[below] at (2.5,2.1) {$y_{c-1}$}; \node at (2.5,2.4) {$\uparrow$};
\node[below] at (3.5,2.1) {$y_{\ell+1}$}; \node at (3.5,2.4) {$\uparrow$};
\node[below] at (4.5,2.1) {$\cdots$};
\node[below] at (5.5,2.1) {$\cdots$};
\node[below] at (6.5,2.1) {$y_{d}$}; \node at (6.5,2.4) {$\uparrow$};
\node[below] at (7.5,2.1) {$y_{d+1}$}; \node at (7.5,2.4) {$\uparrow$};
\node[below] at (8.5,2.1) {$\cdots$};
\node[below] at (9.5,2.1) {$\cdots$};
\node[below] at (10.5,2.1) {$y_{M}$}; \node at (10.5,2.4) {$\uparrow$};
}
\caption{The partition function $Z^{c}_{k,\ell}[i_1,\dots,i_M | j_1,\dots,j_M]$, expressed in terms of vertices. The vertices marked by $\star$ are the ones which depend on the parameter $x_c$.}
\label{fig:Z-vertex}
\end{figure}

\begin{enumerate}[label=(\roman*)]

\item We note that the only vertices in $Z^{c}_{k,\ell}[i_1,\dots,i_M | j_1,\dots,j_M]$ which depend on $x_c$ are the $d-\ell$ vertices marked by $\star$ in Figure \ref{fig:Z-vertex}. From Figure \ref{fund-vert} (recalling that $z=y/x$, where $x$ and $y$ are the horizontal and vertical rapidities passing through a vertex, respectively) we see that the Boltzmann weights of these vertices are given by the table
\begin{align}
\label{rational-wt}
\begin{tabular}{|c|c|c|}
\hline
\quad
\tikz{0.5}{
    \draw[lgray,line width=1.5pt,->] (-1,0) -- (1,0);
    \draw[lgray,line width=1.5pt,->] (0,-1) -- (0,1);
    \node[left] at (-1,0) {\tiny $i$};\node[right] at (1,0) {\tiny $i$};
    \node[below] at (0,-1) {\tiny $i$};\node[above] at (0,1) {\tiny $i$};
}
\quad
&
\quad
\tikz{0.5}{
    \draw[lgray,line width=1.5pt,->] (-1,0) -- (1,0);
    \draw[lgray,line width=1.5pt,->] (0,-1) -- (0,1);
    \node[left] at (-1,0) {\tiny $i$};\node[right] at (1,0) {\tiny $i$};
    \node[below] at (0,-1) {\tiny $j$};\node[above] at (0,1) {\tiny $j$};
}
\quad
&
\quad
\tikz{0.5}{
    \draw[lgray,line width=1.5pt,->] (-1,0) -- (1,0);
    \draw[lgray,line width=1.5pt,->] (0,-1) -- (0,1);
    \node[left] at (-1,0) {\tiny $i$};\node[right] at (1,0) {\tiny $j$};
    \node[below] at (0,-1) {\tiny $j$};\node[above] at (0,1) {\tiny $i$};
}
\quad
\\[1.3cm]
\quad
$1$
\quad
&
\quad
$\dfrac{q(x_c-y_{\alpha})}{x_c-q y_{\alpha}}$
\quad
&
\quad
$\dfrac{(1-q) x_c}{x_c-q y_{\alpha}}$
\quad
\\[0.7cm]
\hline
&
\quad
\tikz{0.5}{
    \draw[lgray,line width=1.5pt,->] (-1,0) -- (1,0);
    \draw[lgray,line width=1.5pt,->] (0,-1) -- (0,1);
    \node[left] at (-1,0) {\tiny $j$};\node[right] at (1,0) {\tiny $j$};
    \node[below] at (0,-1) {\tiny $i$};\node[above] at (0,1) {\tiny $i$};
}
\quad
&
\quad
\tikz{0.5}{
    \draw[lgray,line width=1.5pt,->] (-1,0) -- (1,0);
    \draw[lgray,line width=1.5pt,->] (0,-1) -- (0,1);
    \node[left] at (-1,0) {\tiny $j$};\node[right] at (1,0) {\tiny $i$};
    \node[below] at (0,-1) {\tiny $i$};\node[above] at (0,1) {\tiny $j$};
}
\quad
\\[1.3cm]
&
\quad
$\dfrac{x_c-y_{\alpha}}{x_c-q y_{\alpha}}$
\quad
&
\quad
$\dfrac{(1-q)y_{\alpha}}{x_c-q y_{\alpha}}$
\quad
\\[0.7cm]
\hline
\end{tabular}
\end{align}
for all $0 \leq i < j$, and where $\alpha \in \{\ell+1,\dots,d\}$. After multiplying
$Z^{c}_{k,\ell}[i_1,\dots,i_M | j_1,\dots,j_M]$ by
$\prod_{\alpha=\ell+1}^{d} (x_c-q y_{\alpha})$ and distributing the $d-\ell$ terms in this product over the $d-\ell$ vertices marked by $\star$ in Figure \ref{fig:Z-vertex}, we effectively clear the denominators present in the weights \eqref{rational-wt}, and it is then clear that $\bar{Z}^{c}_{k,\ell}[i_1,\dots,i_M | j_1,\dots,j_M]$ is a polynomial in $x_c$.

\medskip

\item To compute the degree of $\bar{Z}^{c}_{k,\ell}[i_1,\dots,i_M | j_1,\dots,j_M]$ in $x_c$, we simply count the number of starred vertices in Figure \ref{fig:Z-vertex}, since each of these vertices has a weight which (after the above change of normalization) is a polynomial in $x_c$ of degree at most $1$. We conclude that $\bar{Z}^{c}_{k,\ell}[i_1,\dots,i_M | j_1,\dots,j_M]$ has degree at most $d-\ell$ in $x_c$, which is the best bound we can obtain for $j_c \leq i_{k+1}$. We can obtain a slightly sharper bound for $j_c > i_{k+1}$, since in that situation at least one of the starred vertices must take the form
\tikz{0.3}{
\draw[lgray,line width=1.5pt,->] (-1,0) -- (1,0);
\draw[lgray,line width=1.5pt,->] (0,-1) -- (0,1);
\node[left] at (-1,0) {\tiny $j$};\node[right] at (1,0) {\tiny $i$};
\node[below] at (0,-1) {\tiny $i$};\node[above] at (0,1) {\tiny $j$};
}
with $j > i$, and these vertices (after the change in normalization) are constant with respect to $x_c$. Hence for $j_c > i_{k+1}$ we see that $\bar{Z}^{c}_{k,\ell}[i_1,\dots,i_M | j_1,\dots,j_M]$ has degree at most $d-\ell-1$ in $x_c$.

\medskip

\item %The proof of this property is given in Appendix \ref{app:prop-3}.

We begin from Figure \ref{fig:Z-vertex} and observe that after setting $x_c = y_{\gamma}$, $\gamma \in \{\ell+1,\dots,d\}$, we produce a vertex splitting of the form \eqref{R-split} at the position where the $x_c$ and $y_{\gamma}$ parameters cross; see Figure \ref{fig:Z-split}.

\begin{figure}[t]
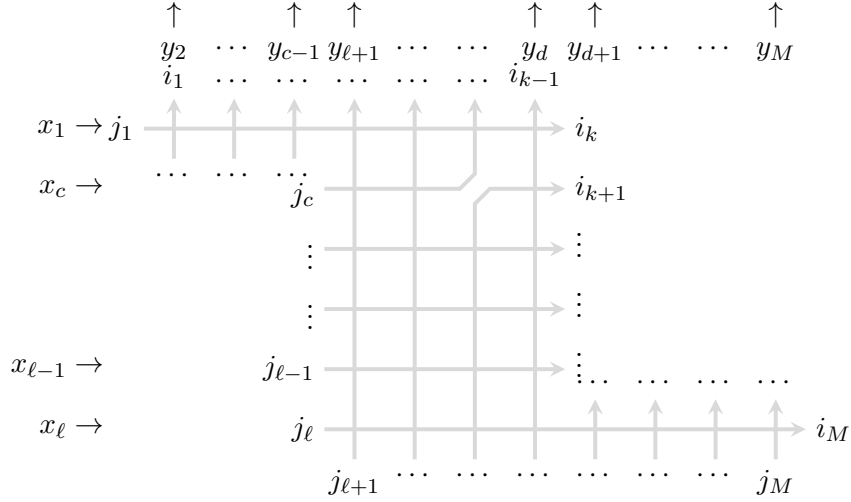

\tikz{0.8}{
%hdots
\draw[lgray,line width=1.5pt,->] (0,0.5) -- (7,0.5);
\draw[lgray,line width=1.5pt,->] (3,-0.5) -- (5.25,-0.5) -- (5.5,-0.25) -- (5.5,1);
\draw[lgray,line width=1.5pt,->] (3,-1.5) -- (7,-1.5);
\draw[lgray,line width=1.5pt,->] (3,-2.5) -- (7,-2.5);
\draw[lgray,line width=1.5pt,->] (3,-3.5) -- (7,-3.5);
\draw[lgray,line width=1.5pt,->] (3,-4.5) -- (11,-4.5);
%vdots
\draw[lgray,line width=1.5pt,->] (0.5,0) -- (0.5,1);
\draw[lgray,line width=1.5pt,->] (1.5,0) -- (1.5,1);
\draw[lgray,line width=1.5pt,->] (2.5,0) -- (2.5,1);
\draw[lgray,line width=1.5pt,->] (3.5,-5) -- (3.5,1);
\draw[lgray,line width=1.5pt,->] (4.5,-5) -- (4.5,1);
\draw[lgray,line width=1.5pt,->] (5.5,-5) -- (5.5,-0.75) -- (5.75,-0.5) -- (7,-0.5);
\draw[lgray,line width=1.5pt,->] (6.5,-5) -- (6.5,1);
\draw[lgray,line width=1.5pt,->] (7.5,-5) -- (7.5,-4);
\draw[lgray,line width=1.5pt,->] (8.5,-5) -- (8.5,-4);
\draw[lgray,line width=1.5pt,->] (9.5,-5) -- (9.5,-4);
\draw[lgray,line width=1.5pt,->] (10.5,-5) -- (10.5,-4);
%bottom labels
\node[left] at (0,0.5) {$j_1$};
\node[below] at (0.5,0) {$\cdots$};
\node[below] at (1.5,0) {$\cdots$};
\node[below] at (2.5,0) {$\cdots$};
\node[left] at (3,-0.6) {$j_c$};
\node[left] at (3,-1.5) {$\vdots$};
\node[left] at (3,-2.5) {$\vdots$};
\node[left] at (3,-3.5) {$j_{\ell-1}$};
\node[left] at (3,-4.5) {$j_{\ell}$}; %\node at (3,-4.5) {$\bullet$};
\node[below] at (3.5,-5) {$j_{\ell+1}$};
\node[below] at (4.5,-5) {$\cdots$};
\node[below] at (5.5,-5) {$\cdots$};
\node[below] at (6.5,-5) {$\cdots$};
\node[below] at (7.5,-5) {$\cdots$};
\node[below] at (8.5,-5) {$\cdots$};
\node[below] at (9.5,-5) {$\cdots$};
\node[below] at (10.5,-5) {$j_M$};
%top labels
\node[above] at (0.5,1) {$i_1$};
\node[above] at (1.5,1) {$\cdots$};
\node[above] at (2.5,1) {$\cdots$};
\node[above] at (3.5,1) {$\cdots$};
\node[above] at (4.5,1) {$\cdots$};
\node[above] at (5.5,1) {$\cdots$};
\node[above] at (6.5,1) {$i_{k-1}$};
\node[right] at (7,0.5) {$i_k$}; %\filldraw[draw=black,fill=white] (7,0.5) circle (2pt);
\node[right] at (7,-0.5) {$i_{k+1}$};
\node[right] at (7,-1.3) {$\vdots$};
\node[right] at (7,-2.3) {$\vdots$};
\node[right] at (7,-3.3) {$\vdots$};
\node[above] at (7.5,-4) {$\cdots$};
\node[above] at (8.5,-4) {$\cdots$};
\node[above] at (9.5,-4) {$\cdots$};
\node[above] at (10.5,-4) {$\cdots$};
\node[right] at (11,-4.5) {$i_M$};
%rapidities
\node[left] at (-0.5,0.5) {$x_1 \rightarrow$};
\node[left] at (-0.5,-0.5) {$x_c \rightarrow$};
\node[left] at (-0.5,-3.5) {$x_{\ell-1} \rightarrow$};
\node[left] at (-0.5,-4.5) {$x_{\ell} \rightarrow$};
\node[below] at (0.5,2.1) {$y_2$}; \node at (0.5,2.4) {$\uparrow$};
\node[below] at (1.5,2.1) {$\cdots$};
\node[below] at (2.5,2.1) {$y_{c-1}$}; \node at (2.5,2.4) {$\uparrow$};
\node[below] at (3.5,2.1) {$y_{\ell+1}$}; \node at (3.5,2.4) {$\uparrow$};
\node[below] at (4.5,2.1) {$\cdots$};
\node[below] at (5.5,2.1) {$\cdots$};
\node[below] at (6.5,2.1) {$y_{d}$}; \node at (6.5,2.4) {$\uparrow$};
\node[below] at (7.5,2.1) {$y_{d+1}$}; \node at (7.5,2.4) {$\uparrow$};
\node[below] at (8.5,2.1) {$\cdots$};
\node[below] at (9.5,2.1) {$\cdots$};
\node[below] at (10.5,2.1) {$y_{M}$}; \node at (10.5,2.4) {$\uparrow$};
}
\caption{The partition function $Z^{c}_{k,\ell}[i_1,\dots,i_M | j_1,\dots,j_M]$ after setting $x_c=y_{\gamma}$. The vertex at the intersection of the $x_c$ and $y_{\gamma}$ lines gets split, in the same sense as in Section \ref{ssec:split}.}
\label{fig:Z-split}
\end{figure}

After creating this splitting we reposition the affected lattice lines; by the Yang--Baxter \eqref{graph-YB} and unitarity \eqref{graph-unitarity} relations of the model, these lines may be repositioned in any way that preserves their starting and ending points. One such repositioning is illustrated in Figure \ref{fig:Z-reposition}; this repositioning is obtained by dragging one of the split lines as far upwards as possible, and the other line as far downwards as possible.

\begin{figure}[t]
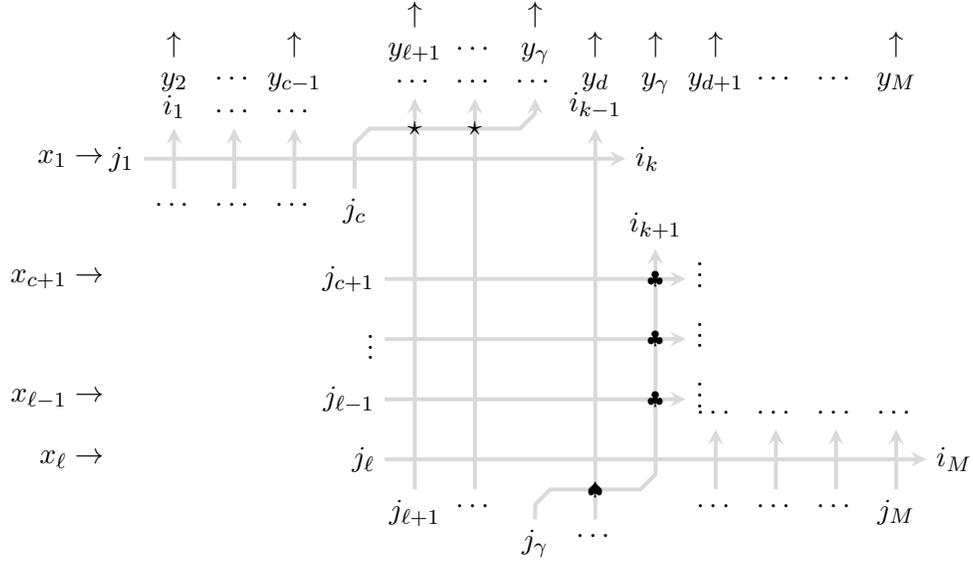

\tikz{0.8}{
%hdots
\draw[lgray,line width=1.5pt,->] (-1,0.5) -- (7,0.5);
\draw[lgray,line width=1.5pt,->] (2.5,0) -- (2.5,0.75) -- (2.75,1) -- (5.25,1) -- (5.5,1.25) -- (5.5,1.5);
\draw[lgray,line width=1.5pt,->] (3,-1.5) -- (8,-1.5);
\draw[lgray,line width=1.5pt,->] (3,-2.5) -- (8,-2.5);
\draw[lgray,line width=1.5pt,->] (3,-3.5) -- (8,-3.5);
\draw[lgray,line width=1.5pt,->] (3,-4.5) -- (12,-4.5);
%vdots
\draw[lgray,line width=1.5pt,->] (-0.5,0) -- (-0.5,1);
\draw[lgray,line width=1.5pt,->] (0.5,0) -- (0.5,1);
\draw[lgray,line width=1.5pt,->] (1.5,0) -- (1.5,1);
\draw[lgray,line width=1.5pt,->] (3.5,-5) -- (3.5,1.5);
\draw[lgray,line width=1.5pt,->] (4.5,-5) -- (4.5,1.5);
\draw[lgray,line width=1.5pt,->] (5.5,-5.5) -- (5.5,-5.25) -- (5.75,-5) -- (7.25,-5) -- (7.5,-4.75) -- (7.5,-1);
\draw[lgray,line width=1.5pt,->] (6.5,-5.5) -- (6.5,1);
\draw[lgray,line width=1.5pt,->] (8.5,-5) -- (8.5,-4);
\draw[lgray,line width=1.5pt,->] (9.5,-5) -- (9.5,-4);
\draw[lgray,line width=1.5pt,->] (10.5,-5) -- (10.5,-4);
\draw[lgray,line width=1.5pt,->] (11.5,-5) -- (11.5,-4);
%bottom labels
\node[left] at (-1,0.5) {$j_1$};
\node[below] at (-0.5,0) {$\cdots$};
\node[below] at (0.5,0) {$\cdots$};
\node[below] at (1.5,0) {$\cdots$};
\node[below] at (2.5,0) {$j_c$};
\node[left] at (3,-1.5) {$j_{c+1}$};
\node[left] at (3,-2.5) {$\vdots$};
\node[left] at (3,-3.5) {$j_{\ell-1}$};
\node[left] at (3,-4.5) {$j_{\ell}$}; %\node at (3,-4.5) {$\bullet$};
\node[below] at (3.5,-5) {$j_{\ell+1}$};
\node[below] at (4.5,-5) {$\cdots$};
\node[below] at (5.5,-5.5) {$j_{\gamma}$};
\node[below] at (6.5,-5.5) {$\cdots$};
\node[below] at (8.5,-5) {$\cdots$};
\node[below] at (9.5,-5) {$\cdots$};
\node[below] at (10.5,-5) {$\cdots$};
\node[below] at (11.5,-5) {$j_M$};
%top labels
\node[above] at (-0.5,1) {$i_1$};
\node[above] at (0.5,1) {$\cdots$};
\node[above] at (1.5,1) {$\cdots$};
\node[above] at (3.5,1.5) {$\cdots$};
\node[above] at (4.5,1.5) {$\cdots$};
\node[above] at (5.5,1.5) {$\cdots$};
\node[above] at (6.5,1) {$i_{k-1}$};
\node[right] at (7,0.5) {$i_k$}; %\filldraw[draw=black,fill=white] (7,0.5) circle (2pt);
\node[above] at (7.5,-1) {$i_{k+1}$};
\node[right] at (8,-1.3) {$\vdots$};
\node[right] at (8,-2.3) {$\vdots$};
\node[right] at (8,-3.3) {$\vdots$};
\node[above] at (8.5,-4) {$\cdots$};
\node[above] at (9.5,-4) {$\cdots$};
\node[above] at (10.5,-4) {$\cdots$};
\node[above] at (11.5,-4) {$\cdots$};
\node[right] at (12,-4.5) {$i_M$};
%markers
\node at (3.5,1) {$\star$};
\node at (4.5,1) {$\star$};
\node at (6.5,-5) {\tiny$\spadesuit$};
\node at (7.5,-1.5) {\tiny$\clubsuit$};
\node at (7.5,-2.5) {\tiny$\clubsuit$};
\node at (7.5,-3.5) {\tiny$\clubsuit$};
%rapidities
\node[left] at (-1.5,0.5) {$x_1 \rightarrow$};
\node[left] at (-1.5,-1.5) {$x_{c+1} \rightarrow$};
\node[left] at (-1.5,-3.5) {$x_{\ell-1} \rightarrow$};
\node[left] at (-1.5,-4.5) {$x_{\ell} \rightarrow$};
\node[below] at (-0.5,2.1) {$y_2$}; \node at (-0.5,2.4) {$\uparrow$};
\node[below] at (0.5,2.1) {$\cdots$};
\node[below] at (1.5,2.1) {$y_{c-1}$}; \node at (1.5,2.4) {$\uparrow$};
\node[below] at (3.5,2.6) {$y_{\ell+1}$}; \node at (3.5,2.9) {$\uparrow$};
\node[below] at (4.5,2.6) {$\cdots$};
\node[below] at (5.5,2.6) {$y_{\gamma}$};  \node at (5.5,2.9) {$\uparrow$};
\node[below] at (6.5,2.1) {$y_{d}$}; \node at (6.5,2.4) {$\uparrow$};
\node[below] at (7.5,2.1) {$y_{\gamma}$}; \node at (7.5,2.4) {$\uparrow$};
\node[below] at (8.5,2.1) {$y_{d+1}$}; \node at (8.5,2.4) {$\uparrow$};
\node[below] at (9.5,2.1) {$\cdots$};
\node[below] at (10.5,2.1) {$\cdots$};
\node[below] at (11.5,2.1) {$y_{M}$}; \node at (11.5,2.4) {$\uparrow$};
}
\caption{The partition function $Z^{c}_{k,\ell}[i_1,\dots,i_M | j_1,\dots,j_M]$ after setting $x_c=y_{\gamma}$, followed by repositioning the lines affected by the splitting. If we neglect all vertices marked by $\star$, {\tiny$\spadesuit$} and {\tiny$\clubsuit$}, we see that this effectively leads to a Z-shaped domain with $\ell-c-1$ internal rows; one less than we started with.}
\label{fig:Z-reposition}
\end{figure}

Examining the partition function in Figure \ref{fig:Z-reposition}, it is clear that an expansion of the form \eqref{prop3} must exist; it remains only to compute the coefficients $\xi_{\gamma}$ and $\eta_{\gamma}$. These coefficients come from expanding over all possible configurations of the vertices marked by $\star$, {\tiny$\spadesuit$} and {\tiny$\clubsuit$} in Figure \ref{fig:Z-reposition}; we may compute them explicitly in terms of single-row or single-column partition functions. For our purposes, it will be useful to note the following factorized form of the coefficients:
\begin{multline}
\label{xi-factor}
\xi_{\gamma}
\Big[
j'_1,\dots,j'_M \Big| j_1,\dots,j_M
\Big]
=
\\
\xi_{\gamma}
\Big[
j'_1,\dots,j'_{\ell-1} \Big| j_1,\dots,j_{\ell-1}
\Big]
\cdot
\left(
{\bm 1}_{j_{\ell}=j'_{\ell}=m}
\right)
\cdot
\xi_{\gamma}
\Big[
j'_{\ell+1},\dots,j'_M \Big| j_{\ell+1},\dots,j_M
\Big],
\end{multline}
where
\begin{align}
\label{xi-left}
\xi_{\gamma}
\Big[
j'_1,\dots,j'_{\ell-1} \Big| j_1,\dots,j_{\ell-1}
\Big]
=
\prod_{\alpha=1}^{\ell-1}
{\bm 1}_{j_{\alpha} = j'_{\alpha}},
\end{align}
and
\begin{align}
\label{xi-right}
\xi_{\gamma}
\Big[
j'_{\ell+1},\dots,j'_M \Big| j_{\ell+1},\dots,j_M
\Big]
=
\prod_{\alpha=\ell+1}^{\gamma-1}
{\bm 1}_{j_{\alpha} = j'_{\alpha}}
\cdot
\left[
\tikz{0.9}{
%paths
\draw[lgray,line width=1.5pt,->] (8,-1) -- (8,-0.75) -- (8.25,-0.5) -- (12.75,-0.5) -- (13,-0.25) -- (13,0);
\foreach\x in {9,10,11,12}{
\draw[lgray,line width=1.5pt,->] (\x,-1) -- (\x,0);
}
%bottom labels
\node[below] at (9,-1) {$\cdots$};
\node[below] at (10,-1) {$\cdots$};
\node[below] at (11,-1) {$\cdots$};
\node[below] at (12,-1) {$j_d$};
%top labels
\node[above] at (9,0) {$j'_{\gamma}$};
\node[above] at (10,0) {$\cdots$};
\node[above] at (11,0) {$\cdots$};
\node[above] at (12,0) {$\cdots$};
\node[above] at (13,0) {$j'_d$};
\node[below] at (8,-1) {$j_{\gamma}$};
%markers
\node at (9,-0.5) {\tiny$\spadesuit$};
\node at (10,-0.5) {\tiny$\spadesuit$};
\node at (11,-0.5) {\tiny$\spadesuit$};
\node at (12,-0.5) {\tiny$\spadesuit$};
%h rapidities
\node[left] at (8,-0.5) {$y_{\gamma} \rightarrow$};
%v rapidities
\node[below] at (9,1.3) {$y_{\gamma+1}$}; \node at (9,1.6) {$\uparrow$};
\node[below] at (10,1.3) {$\cdots$};
\node[below] at (11,1.3) {$\cdots$};
\node[below] at (12,1.3) {$y_{d}$}; \node at (12,1.6) {$\uparrow$};
}
\right]
\cdot
\prod_{\alpha=d+1}^{M}
{\bm 1}_{j_{\alpha} = j'_{\alpha}}.
\end{align}
In a similar vein, we have
\begin{multline}
\label{eta-factor}
\eta_{\gamma}
\Big[
i_1,\dots,i_M \Big| i'_1,\dots,i'_M
\Big]
=
\\
\eta_{\gamma}
\Big[
i_1,\dots,i_{k-1} \Big| i'_1,\dots,i'_{k-1}
\Big]
\cdot
\left(
{\bm 1}_{i_k=i'_k}
\right)
\cdot
\eta_{\gamma}
\Big[
i_{k+1},\dots,i_M \Big| i'_{k+1},\dots,i'_M
\Big],
\end{multline}
where
\begin{align}
\label{eta-left}
\eta_{\gamma}
\Big[
i_1,\dots,i_{k-1} \Big| i'_1,\dots,i'_{k-1}
\Big]
=
\prod_{\alpha=1}^{c-2}
{\bm 1}_{i_{\alpha} = i'_{\alpha}}
\cdot
\left[
\tikz{0.9}{
%paths
\draw[lgray,line width=1.5pt,->] (8,-1) -- (8,-0.75) -- (8.25,-0.5) -- (12.75,-0.5) -- (13,-0.25) -- (13,0);
\foreach\x in {9,10,11,12}{
\draw[lgray,line width=1.5pt,->] (\x,-1) -- (\x,0);
}
%bottom labels
\node[below] at (9,-1) {$\cdots$};
\node[below] at (10,-1) {$\cdots$};
\node[below] at (11,-1) {$\cdots$};
\node[below] at (12.5,-1) {$i'_{k-d+\gamma-1}$};
%top labels
\node[above] at (9,0) {$i_{c-1}$};
\node[above] at (10,0) {$\cdots$};
\node[above] at (11,0) {$\cdots$};
\node[above] at (12,0) {$\cdots$};
\node[above] at (13.2,0) {$i_{k-d+\gamma-1}$};
\node[below] at (8,-1) {$i'_{c-1}$};
%markers
\node at (9,-0.5) {$\star$};
\node at (10,-0.5) {$\star$};
\node at (11,-0.5) {$\star$};
\node at (12,-0.5) {$\star$};
%h rapidities
\node[left] at (7.7,-0.5) {$y_{\gamma} \rightarrow$};
%v rapidities
\node[below] at (9,1.3) {$y_{\ell+1}$}; \node at (9,1.6) {$\uparrow$};
\node[below] at (10,1.3) {$\cdots$};
\node[below] at (11,1.3) {$\cdots$};
\node[below] at (12,1.3) {$y_{\gamma-1}$}; \node at (12,1.6) {$\uparrow$};
}
\right]
\cdot
\prod_{\alpha=k-d+\gamma}^{k-1}
{\bm 1}_{i_{\alpha} = i'_{\alpha}}
\end{align}
and
\begin{align}
\label{eta-right}
\eta_{\gamma}
\Big[
i_{k+1},\dots,i_M \Big| i'_{k+1},\dots,i'_M
\Big]
=
\left[
\tikz{0.9}{
%paths
\foreach\x in {0.5,-0.5,-1.5,-2.5}{
\draw[lgray,line width=1.5pt,->] (3,\x) -- (4,\x);
}
\draw[lgray,line width=1.5pt,->] (3.5,-3) -- (3.5,1);
%bottom labels
\node[left] at (3,0.5) {$i'_{k+1}$};
\node[left] at (3,-0.5) {$\vdots$};
\node[left] at (3,-1.5) {$\vdots$};
\node[left] at (3,-2.5) {$\vdots$};
\node[below] at (3.5,-3) {$i'_{d-1}$};
%top labels
\node[above] at (3.5,1) {$i_{k+1}$};
\node[right] at (4,0.5) {$\vdots$};
\node[right] at (4,-0.5) {$\vdots$};
\node[right] at (4,-1.5) {$\vdots$};
\node[right] at (4,-2.5) {$i_{d-1}$};
%markers
\node at (3.5,0.5) {\tiny $\clubsuit$};
\node at (3.5,-0.5) {\tiny $\clubsuit$};
\node at (3.5,-1.5) {\tiny $\clubsuit$};
\node at (3.5,-2.5) {\tiny $\clubsuit$};
%rapidities
\node[left] at (2,0.5) {$x_{c+1} \rightarrow$};
\node[left] at (2,-0.5) {$\vdots$};
\node[left] at (2,-1.5) {$\vdots$};
\node[left] at (2,-2.5) {$x_{\ell-1} \rightarrow$};
\node[below] at (3.5,2.2) {$y_{\gamma}$}; \node at (3.5,2.5) {$\uparrow$};
}
\right]
\cdot
\prod_{\alpha=d}^{M}
{\bm 1}_{i_{\alpha} = i'_{\alpha}}.
\end{align}
From the explicit form of the coefficients \eqref{xi-left}, \eqref{xi-right}, \eqref{eta-left} and \eqref{eta-right}, we immediately see that they satisfy the basic conservation properties
\begin{align}
\label{xi-conserve1}
\xi_{\gamma}
\Big[
j'_1,\dots,j'_{\ell-1} \Big| j_1,\dots,j_{\ell-1}
\Big]
&=
0,
\qquad
\text{unless}
\quad
(j'_1,\dots,j'_{\ell-1}) \in \mathfrak{S}(j_1,\dots,j_{\ell-1}),
\\
\label{xi-conserve2}
\xi_{\gamma}
\Big[
j'_{\ell+1},\dots,j'_M \Big| j_{\ell+1},\dots,j_M
\Big]
&=
0,
\qquad
\text{unless}
\quad
(j'_{\ell+1},\dots,j'_M) \in \mathfrak{S}(j_{\ell+1},\dots,j_M),
\\
\label{eta-conserve1}
\eta_{\gamma}
\Big[
i_1,\dots,i_{k-1} \Big| i'_1,\dots,i'_{k-1}
\Big]
&=
0,
\qquad
\text{unless}
\quad
(i'_1,\dots,i'_{k-1}) \in \mathfrak{S}(i_1,\dots,i_{k-1}),
\\
\label{eta-conserve2}
\eta_{\gamma}
\Big[
i_{k+1},\dots,i_M \Big| i'_{k+1},\dots,i'_M
\Big]
&=
0,
\qquad
\text{unless}
\quad
(i'_{k+1},\dots,i'_M) \in \mathfrak{S}(i_{k+1},\dots,i_M).
\end{align}
Note that all coefficients $\xi_{\gamma}$ and $\eta_{\gamma}$ are also manifestly independent of $x_c$.

\medskip

\item When $j_c < i_{k+1}$, at least one of the starred vertices in Figure \ref{fig:Z-vertex} must take the form
\tikz{0.3}{
\draw[lgray,line width=1.5pt,->] (-1,0) -- (1,0);
\draw[lgray,line width=1.5pt,->] (0,-1) -- (0,1);
\node[left] at (-1,0) {\tiny $i$};\node[right] at (1,0) {\tiny $j$};
\node[below] at (0,-1) {\tiny $j$};\node[above] at (0,1) {\tiny $i$};
}
with $i < j$. As can be seen from \eqref{rational-wt}, the weight of this vertex vanishes when $x_c=0$; since no other weight is singular under the limit $x_c \rightarrow 0$, we immediately conclude that $Z^{c}_{k,\ell} |_{x_c=0} = 0$.

\medskip

\item When $j_c = i_{k+1}$, there are two possibilities for the ensemble of starred vertices in Figure \ref{fig:Z-vertex}. The first is that a vertex of the form
\tikz{0.3}{
\draw[lgray,line width=1.5pt,->] (-1,0) -- (1,0);
\draw[lgray,line width=1.5pt,->] (0,-1) -- (0,1);
\node[left] at (-1,0) {\tiny $j$};\node[right] at (1,0) {\tiny $i$};
\node[below] at (0,-1) {\tiny $i$};\node[above] at (0,1) {\tiny $j$};
}
appears among them, for some $j>i$. The weight of this vertex vanishes after taking the limit $x_c \rightarrow \infty$; since all other weights in the table \eqref{rational-wt} have a well-defined and non-vanishing limit as $x_c \rightarrow \infty$, we can neglect all lattice configurations which feature this vertex among the ensemble of starred vertices. We are left with the second possibility, that the vertex
\tikz{0.3}{
\draw[lgray,line width=1.5pt,->] (-1,0) -- (1,0);
\draw[lgray,line width=1.5pt,->] (0,-1) -- (0,1);
\node[left] at (-1,0) {\tiny $j$};\node[right] at (1,0) {\tiny $i$};
\node[below] at (0,-1) {\tiny $i$};\node[above] at (0,1) {\tiny $j$};
}
and (consequently) the vertex
\tikz{0.3}{
\draw[lgray,line width=1.5pt,->] (-1,0) -- (1,0);
\draw[lgray,line width=1.5pt,->] (0,-1) -- (0,1);
\node[left] at (-1,0) {\tiny $i$};\node[right] at (1,0) {\tiny $j$};
\node[below] at (0,-1) {\tiny $j$};\node[above] at (0,1) {\tiny $i$};
}
do not appear among the starred vertices, for any $i<j$. The absence of these vertices causes the row of starred vertices to have a frozen configuration; if we label the colors of the $d-\ell$ bottom and top edges of this row by $(j'_{\ell+1},\dots,j'_d)$ and $(i'_{\ell+1},\dots,i'_d)$ respectively, we find that
\begin{align}
\label{eq-430}
\lim_{x_c \rightarrow \infty}
\tikz{0.9}{
%hdots
\draw[lgray,line width=1.5pt,->] (3,-0.5) -- (7,-0.5);
%vdots
\draw[lgray,line width=1.5pt,->] (3.5,-1) -- (3.5,0);
\draw[lgray,line width=1.5pt,->] (4.5,-1) -- (4.5,0);
\draw[lgray,line width=1.5pt,->] (5.5,-1) -- (5.5,0);
\draw[lgray,line width=1.5pt,->] (6.5,-1) -- (6.5,0);
%bottom labels
\node[left] at (3,-0.5) {$j_c$};
\node[below] at (3.5,-1) {$j'_{\ell+1}$};
\node[below] at (4.5,-1) {$\cdots$};
\node[below] at (5.5,-1) {$\cdots$};
\node[below] at (6.5,-1) {$j'_d$};
%top labels
\node[above] at (3.5,0) {$i'_{\ell+1}$};
\node[above] at (4.5,0) {$\cdots$};
\node[above] at (5.5,0) {$\cdots$};
\node[above] at (6.5,0) {$i'_d$};
\node[right] at (7,-0.5) {$j_c$};
}
=
\prod_{\alpha = \ell+1}^{d}
\bm{1}_{i'_{\alpha}=j'_{\alpha}}
\cdot
\lim_{x_c \rightarrow \infty}
\tikz{0.9}{
%hdots
\draw[lgray,line width=1.5pt,->] (3,-0.5) -- (7,-0.5);
%vdots
\draw[lgray,line width=1.5pt,->] (3.5,-1) -- (3.5,0);
\draw[lgray,line width=1.5pt,->] (4.5,-1) -- (4.5,0);
\draw[lgray,line width=1.5pt,->] (5.5,-1) -- (5.5,0);
\draw[lgray,line width=1.5pt,->] (6.5,-1) -- (6.5,0);
%bottom labels
\node[left] at (3,-0.5) {$j_c$};
\node at (4,-0.5) {$j_c$};
\node at (5,-0.5) {$j_c$};
\node at (6,-0.5) {$j_c$};
\node[below] at (3.5,-1) {$i'_{\ell+1}$};
\node[below] at (4.5,-1) {$\cdots$};
\node[below] at (5.5,-1) {$\cdots$};
\node[below] at (6.5,-1) {$i'_d$};
%top labels
\node[above] at (3.5,0) {$i'_{\ell+1}$};
\node[above] at (4.5,0) {$\cdots$};
\node[above] at (5.5,0) {$\cdots$};
\node[above] at (6.5,0) {$i'_d$};
\node[right] at (7,-0.5) {$j_c$};
}.
\end{align}
We see that every vertex that appears on the right hand side of \eqref{eq-430} is of the form
\tikz{0.3}{
\draw[lgray,line width=1.5pt,->] (-1,0) -- (1,0);
\draw[lgray,line width=1.5pt,->] (0,-1) -- (0,1);
\node[left] at (-1,0) {\tiny $j$};\node[right] at (1,0) {\tiny $j$};
\node[below] at (0,-1) {\tiny $j$};\node[above] at (0,1) {\tiny $j$};
},
\tikz{0.3}{
\draw[lgray,line width=1.5pt,->] (-1,0) -- (1,0);
\draw[lgray,line width=1.5pt,->] (0,-1) -- (0,1);
\node[left] at (-1,0) {\tiny $j$};\node[right] at (1,0) {\tiny $j$};
\node[below] at (0,-1) {\tiny $i$};\node[above] at (0,1) {\tiny $i$};
} or
\tikz{0.3}{
\draw[lgray,line width=1.5pt,->] (-1,0) -- (1,0);
\draw[lgray,line width=1.5pt,->] (0,-1) -- (0,1);
\node[left] at (-1,0) {\tiny $i$};\node[right] at (1,0) {\tiny $i$};
\node[below] at (0,-1) {\tiny $j$};\node[above] at (0,1) {\tiny $j$};
}, with $i<j$. As $x_c \rightarrow \infty$, the weights of these vertices have the limits $1$, $1$ and $q$ respectively, and we thus have
\begin{align}
\label{eq-431}
\lim_{x_c \rightarrow \infty}
\tikz{0.9}{
%hdots
\draw[lgray,line width=1.5pt,->] (3,-0.5) -- (7,-0.5);
%vdots
\draw[lgray,line width=1.5pt,->] (3.5,-1) -- (3.5,0);
\draw[lgray,line width=1.5pt,->] (4.5,-1) -- (4.5,0);
\draw[lgray,line width=1.5pt,->] (5.5,-1) -- (5.5,0);
\draw[lgray,line width=1.5pt,->] (6.5,-1) -- (6.5,0);
%bottom labels
\node[left] at (3,-0.5) {$j_c$};
\node[below] at (3.5,-1) {$j'_{\ell+1}$};
\node[below] at (4.5,-1) {$\cdots$};
\node[below] at (5.5,-1) {$\cdots$};
\node[below] at (6.5,-1) {$j'_d$};
%top labels
\node[above] at (3.5,0) {$i'_{\ell+1}$};
\node[above] at (4.5,0) {$\cdots$};
\node[above] at (5.5,0) {$\cdots$};
\node[above] at (6.5,0) {$i'_d$};
\node[right] at (7,-0.5) {$j_c$};
}
=
\prod_{\alpha = \ell+1}^{d}
\bm{1}_{i'_{\alpha}=j'_{\alpha}}
\cdot
q^{\#[(i'_{\ell+1},\dots,i'_d) > j_c]}.
\end{align}
The net effect of taking the limit $x_c \rightarrow \infty$ is thus the deletion of the starred vertices from the lattice in Figure \ref{fig:Z-vertex}, at the expense of an overall power of $q$. To match with the power of $q$ appearing in \eqref{x=0}, we note that because of color-conservation through the top row of vertices in Figure \ref{fig:Z-vertex}, we must have
\begin{align*}
\#[(i'_{\ell+1},\dots,i'_d) > j_c]
=
[\#(i_1,\dots,i_k) > j_c]
-
[\#(j_1,\dots,j_{c-1}) > j_c];
\end{align*}
substituting this into \eqref{eq-431}, we then immediately recover \eqref{x=0}.
\end{enumerate}
\end{proof}

Next, we state a simple auxiliary result that we need at various stages in the rest of the argument.

\begin{lem}
\label{lem:diff}
Consider a quantity $\Delta_{k}[i_1,\dots,i_M]$ which depends on a set of colors
$(i_1,\dots,i_M)$. Suppose that, for any partitioning $\mathcal{A}\cup\bar{\mathcal{A}} = \{1,\dots,k-1\}$ and $\mathcal{B}\cup\bar{\mathcal{B}} = \{k+1,\dots,M\}$ there holds
\begin{align}
\label{sum-to-0}
\sum_{i_k}
\sum_{
\substack{(i_{\alpha \in \bar{\mathcal{A}}})
\\ \vspace{-0.1cm} \\
(i_{\beta \in \bar{\mathcal{B}}})
}}
\Delta_{k}[i_1,\dots,i_M]
=
0,
\end{align}
where the summation being performed satisfies
$$
i_{\alpha} \sumin \{m,m+1,\dots,N\},\quad \text{for}\ \alpha \in \bar{\mathcal{A}},
\quad\quad
i_{\beta} \sumin \{0,1,\dots,m\},\quad \text{for}\ \beta \in \bar{\mathcal{B}},
\quad\quad
i_k \sumin \{0,1,\dots,N\};
$$
all other colors are fixed, with
$$
i_{\alpha} \fixin \{0,1,\dots,m-1\},\quad \text{for}\ \alpha \in \mathcal{A},
\quad\quad
i_{\beta} \fixin \{m+1,m+2,\dots,N\},\quad \text{for}\ \beta \in \mathcal{B}.
$$
Also suppose that the coefficients $C(i_1,\dots,i_M )$ satisfy
\begin{align}
\label{color-blind}
C(i_1,\dots,i_M)
=
C([i_1]_m,\dots,[i_{k-1}]_m,0,[i_{k+1}]^m,\dots,[i_M]^m)
\end{align}
for all $i_1,\dots,i_M \geq 0$, where we use the notations \eqref{merge-up}, \eqref{merge-down}. Then
\begin{align}
\label{COB}
\sum_{i_1,\dots,i_M}
C(i_1,\dots,i_M)
\Delta_{k}[i_1,\dots,i_M]
=
0,
\end{align}
with unrestricted summation over $(i_1,\dots,i_M)$.
\end{lem}

\begin{proof}
The summation over $(i_1,\dots,i_M)$ in \eqref{COB} can be broken up in the following way. Focusing on the indices $(i_1,\dots,i_{k-1})$, we choose a subset of these (with labels in $\mathcal{A} \subset \{1,\dots,k-1\}$) to be summed over $\{0,1,\dots,m-1\}$, while the remainder are summed over $\{m,m+1,\dots,N\}$; we also sum over all ways to choose such subsets $\mathcal{A}$. Likewise, focusing on the indices $(i_{k+1},\dots,i_M)$, we choose a subset of these (with labels in $\mathcal{B} \subset \{k+1,\dots,M\}$) to be summed over $\{m+1,m+2,\dots,N\}$, while the remainder get summed over $\{0,1,\dots,m\}$; again, we sum over all ways to choose such subsets $\mathcal{B}$. This leads us to the following equation:
\begin{multline*}
\sum_{i_1,\dots,i_M}
C(i_1,\dots,i_M)
\Delta_{k}[i_1,\dots,i_M]
=
\\
\sum_{\mathcal{A} \subset \{1,\dots,k-1\}}
\sum_{\mathcal{B} \subset \{k+1,\dots,M\}}
\sum_{\substack{
0 \\ (i_{\alpha\in\mathcal{A}})
}}^{m-1}
\sum_{\substack{
m \\ (i_{\alpha\in\bar{\mathcal{A}}})
}}^{N}
\sum_{\substack{
m+1 \\ (i_{\beta\in\mathcal{B}})
}}^{N}
\sum_{\substack{
0 \\ (i_{\beta\in\bar{\mathcal{B}}})
}}^{m}
\sum_{i_k}
C(i_1,\dots,i_M)
\Delta_{k}[i_1,\dots,i_M],
\end{multline*}
where we do not refine the sum over $i_k$ in any way. Making use of the property \eqref{color-blind}, we find that it is effectively possible to interchange the order of the coefficient $C$ and the sums over $i_k$, $(i_{\alpha\in\bar{\mathcal{A}}})$, $(i_{\beta\in\bar{\mathcal{B}}})$. More precisely, we have
\begin{multline}
\label{interchange-C}
\sum_{i_1,\dots,i_M}
C(i_1,\dots,i_M)
\Delta_{k}[i_1,\dots,i_M]
=
\\
\sum_{\mathcal{A} \subset \{1,\dots,k-1\}}
\sum_{\mathcal{B} \subset \{k+1,\dots,M\}}
\sum_{\substack{
0 \\ (i_{\alpha\in\mathcal{A}})
}}^{m-1}
\sum_{\substack{
m+1 \\ (i_{\beta\in\mathcal{B}})
}}^{N}
C(a_1,\dots,a_M)
\sum_{\substack{
m \\ (i_{\alpha\in\bar{\mathcal{A}}})
}}^{N}
\sum_{\substack{
0 \\ (i_{\beta\in\bar{\mathcal{B}}})
}}^{m}
\sum_{i_k}
\Delta_{k}[i_1,\dots,i_M],
\end{multline}
where we have defined
\begin{align*}
a_{p}
=
\left\{
\begin{array}{ll}
i_p,
\qquad\quad
&
p \in \mathcal{A} \cup \mathcal{B},
\\ \\
m, \qquad\quad
&
p \in \bar{\mathcal{A}},
\\ \\
0, \qquad\quad
&
p \in \{k\} \cup \bar{\mathcal{B}}.
\end{array}
\right.
\end{align*}
By the assumption \eqref{sum-to-0}, the final three sums in \eqref{interchange-C} satisfy
\begin{align*}
\sum_{\substack{
m \\ (i_{\alpha\in\bar{\mathcal{A}}})
}}^{N}
\sum_{\substack{
0 \\ (i_{\beta\in\bar{\mathcal{B}}})
}}^{m}
\sum_{i_k}
\Delta_{k}[i_1,\dots,i_M]
=
0,
\end{align*}
for all subsets $\bar{\mathcal{A}}$ and $\bar{\mathcal{B}}$. Hence the right hand side of \eqref{interchange-C} vanishes term-by-term, and the proof of \eqref{COB} is complete.
\end{proof}

We are now ready to prove Theorem \ref{theorema-egr} in the case of Z-shaped domains. Let us first restate it in terms of the notation being used in this section. Fixing two sets $\mathcal{A} \subset \{1,\dots,k-1\}$ and $\mathcal{B} \subset \{k+1,\dots,M\}$, and a vector of incoming colors $(j_1,\dots,j_M)$ satisfying \eqref{in-conditions}, we quote from equation \eqref{phi} to define
\begin{multline}
\label{phiZ}
\Phi^{c}_{k,\ell}
\Big(
\mathcal{A},\mathcal{B};
(i_{\alpha \in \mathcal{A}}),
(i_{\beta \in \mathcal{B}});
(j_1,\dots,j_M)
\Big)
\\
=
\sum_{i_k}
\sum_{
\substack{(i_{\alpha \in \bar{\mathcal{A}}})
\\ \vspace{-0.1cm} \\
(i_{\beta \in \bar{\mathcal{B}}})
}}
Z^{c}_{k,\ell}\Big[ i_1,\dots,i_M \Big| j_1,\dots,j_M \Big]
{\bm 1}\Big( \mathcal{H}^{\ges m}(P;k+1) = h \Big)
\end{multline}
where $i_k$ is summed over all values $\{0,1,\dots,N\}$, while
$(i_{\alpha \in \bar{\mathcal{A}}})$ and $(i_{\beta \in \bar{\mathcal{B}}})$ are summed according to the constraints \eqref{out-conditions}.
%In the summand, $\mathcal{H}^{\ges m}(P;k+1)$ is the level $m$ height function situated at the
%$(k+1)$-th edge of $P$.
Similarly, quoting from \eqref{psi}, we define
\begin{multline}
\label{psiZ}
\Psi^{c}_{k,\ell}
\Big(
\mathcal{A},\mathcal{B};
(i_{\alpha \in \mathcal{A}}),
(i_{\beta \in \mathcal{B}});
(j_1,\dots,j_M)
\Big)
\\
=
\sum_{i_k}
\sum_{
\substack{(i_{\alpha \in \bar{\mathcal{A}}})
\\ \vspace{-0.1cm} \\
(i_{\beta \in \bar{\mathcal{B}}})
}}
Z^{c}_{k,\ell}\Big[ i_1,\dots,i_M \Big| j_1,\dots,j_M \Big]
{\bm 1}\Big( \mathcal{H}^{\ges m+1}(P;k) = h \Big),
\end{multline}
where $i_k$ is summed over all values $\{0,1,\dots,N\}$, while
$(i_{\alpha \in \bar{\mathcal{A}}})$ and $(i_{\beta \in \bar{\mathcal{B}}})$ are summed according to the constraints \eqref{out-conditions}.

\begin{theorem}
\label{thm:Z}
Let $\Phi^{c}_{k,\ell}$ and $\Psi^{c}_{k,\ell}$ be defined as in \eqref{phiZ} and \eqref{psiZ}. For all $2 \leq c \leq \ell$, we have
\begin{align}
\label{z-shaped-claim}
\Phi^{c}_{k,\ell}
\Big(
\mathcal{A},\mathcal{B};
(i_{\alpha \in \mathcal{A}}),
(i_{\beta \in \mathcal{B}});
(j_1,\dots,j_M)
\Big)
=
\Psi^{c}_{k,\ell}
\Big(
\mathcal{A},\mathcal{B};
(i_{\alpha \in \mathcal{A}}),
(i_{\beta \in \mathcal{B}});
(j_1,\dots,j_M)
\Big)
\Big|_{x_1 \leftrightarrow x_{\ell}}.
\end{align}
\end{theorem}

\begin{proof}
We proceed by induction on the number of internal rows in the Z-shaped domain. In Section \ref{ssec:two-row}, we proved that \eqref{z-shaped-claim} is true when $c=\ell$; that is, when there are no internal rows.

Now assume that the claim is true for all Z-shaped domains with $\ell-c-1$ internal rows. We shall consider the quantity
\begin{align*}
\Upsilon^{c}_{k,\ell}
:=
\Phi^{c}_{k,\ell}
\Big(
\mathcal{A},\mathcal{B};
(i_{\alpha \in \mathcal{A}}),
(i_{\beta \in \mathcal{B}});
(j_1,\dots,j_M)
\Big)
-
\Psi^{c}_{k,\ell}
\Big(
\mathcal{A},\mathcal{B};
(i_{\alpha \in \mathcal{A}}),
(i_{\beta \in \mathcal{B}});
(j_1,\dots,j_M)
\Big)
\Big|_{x_1 \leftrightarrow x_{\ell}}
\end{align*}
with the goal of showing that it vanishes. Using the definitions \eqref{phiZ}, \eqref{psiZ} and Property (i) of Proposition \ref{prop:properties}, we see that
\begin{align*}
\bar{\Upsilon}^{c}_{k,\ell}
:=
\prod_{\alpha=\ell+1}^{d} (x_c-q y_{\alpha})
\Upsilon^{c}_{k,\ell}
\Big(
\mathcal{A},\mathcal{B};
(i_{\alpha \in \mathcal{A}}),
(i_{\beta \in \mathcal{B}});
(j_1,\dots,j_M)
\Big)
\end{align*}
is a polynomial in $x_c$. We will show that this quantity vanishes at sufficiently many interpolating points, and is therefore identically zero. To determine how many points we require, there are two separate cases to consider:
\begin{enumerate} \item $k+1 \not\in \mathcal{B}$, which implies that $i_{k+1} \sumin \{0,1,\dots,m\}$ is summed;
\item $k+1 \in \mathcal{B}$, which implies that $i_{k+1} \fixin \{m+1,m+2,\dots,N\}$ is fixed.
\end{enumerate}

\medskip
\underline{{\bf Case 1:} $k+1 \not\in \mathcal{B}$.}

\medskip

In this situation, $i_{k+1}$ is summed over the values $\{0,1,\dots,m\}$. Using Property (ii) of Proposition \ref{prop:properties}, we see that in this case (since $j_c > m \geq i_{k+1}$, irrespective of the value of $i_{k+1}$) $\bar{\Upsilon}^{c}_{k,\ell}$ has degree $d-\ell-1$ in $x_c$. We therefore need $d-\ell$ interpolating values for $x_c$, and these are provided by Property (iii) of Proposition \ref{prop:properties}, which allows us to compute $\bar{\Upsilon}^{c}_{k,\ell}$ at each of the points $x_c = y_{\ell+1},\dots,y_d$:
\begin{multline}
\label{d-l-points}
\Upsilon^{c}_{k,\ell}
\Big|_{x_c = y_{\gamma}}
=
\sum_{(i'_1,\dots,i'_M)}
\sum_{(j'_1,\dots,j'_M)}
\sum_{i_k}
\sum_{
\substack{(i_{\alpha \in \bar{\mathcal{A}}})
\\ \vspace{-0.1cm} \\
(i_{\beta \in \bar{\mathcal{B}}})
}}
\xi_{\gamma}
\Big[
j'_1,\dots,j'_M \Big| j_1,\dots,j_M
\Big]
\eta_{\gamma}
\Big[
i_1,\dots,i_M \Big| i'_1,\dots,i'_M
\Big]
\\
\times
%\mathfrak{s}^{(y)}_{\gamma,d}
%\dots
%\mathfrak{s}^{(y)}_{\gamma,\gamma+1}
\left\{
Z^{c+1}_{k,\ell}
\Big[ i'_1,\dots,i'_M
\Big| j'_1,\dots,j'_M \Big]
%\Big|_{y_c=y_{\gamma}}
{\bm 1}\Big( [\#(i_{k+1},\dots,i_M) \geq m ] = h \Big)
\right.
\\
\left.
-
Z^{c+1}_{k,\ell}
\Big[ i'_1,\dots,i'_M
\Big| j'_1,\dots,j'_M \Big]
%\Big|_{y_c=y_{\gamma}}
\Big|_{x_1\leftrightarrow x_{\ell}}
{\bm 1}\Big( [\#(i_{k},\dots,i_M) \geq m+1 ] = h \Big)
\right\}.
\end{multline}
Now we use several of the properties of the coefficients $\xi_{\gamma}$, $\eta_{\gamma}$. First, in view of the factorization \eqref{eta-factor} of $\eta_{\gamma}$ and the conservation property \eqref{eta-conserve2}, we are able to replace the two indicator functions appearing in \eqref{d-l-points} by
\begin{equation}
\label{indicator-ii'}
\begin{split}
{\bm 1}\Big( [\#(i_{k+1},\dots,i_M) \geq m ] = h \Big)
&=
{\bm 1}\Big( [\#(i'_{k+1},\dots,i'_M) \geq m ] = h \Big),
\\
{\bm 1}\Big( [\#(i_{k},\dots,i_M) \geq m+1 ] = h \Big)
&=
{\bm 1}\Big( [\#(i'_{k},\dots,i'_M) \geq m+1 ] = h \Big).
\end{split}
\end{equation}
Hence the sums over $i_k$, $(i_{\alpha \in \bar{\mathcal{A}}})$, $(i_{\beta \in \bar{\mathcal{B}}})$ that appear on the right hand side of \eqref{d-l-points} localize over the coefficients $\eta_{\gamma}[i_1,\dots,i_M | i'_1,\dots,i'_M]$. Second, using again the factorized form \eqref{eta-factor} of these coefficients, we have
\begin{multline*}
\sum_{i_k}
\sum_{
\substack{(i_{\alpha \in \bar{\mathcal{A}}})
\\ \vspace{-0.1cm} \\
(i_{\beta \in \bar{\mathcal{B}}})
}}
\eta_{\gamma}
\Big[
i_1,\dots,i_M \Big| i'_1,\dots,i'_M
\Big]
=
\\
\sum_{(i_{\alpha \in \bar{\mathcal{A}}})}
\eta_{\gamma}
\Big[
i_1,\dots,i_{k-1} \Big| i'_1,\dots,i'_{k-1}
\Big]
\sum_{(i_{\beta \in \bar{\mathcal{B}}})}
\eta_{\gamma}
\Big[
i_{k+1},\dots,i_M \Big| i'_{k+1},\dots,i'_M
\Big]
\sum_{i_k}
{\bm 1}_{i_k=i'_k}.
\end{multline*}
These sums can be computed using the color-merging result of Proposition \ref{prop:merge-PQ}. We find that
\begin{multline}
\label{eta-collapse}
\sum_{i_k}
\sum_{
\substack{(i_{\alpha \in \bar{\mathcal{A}}})
\\ \vspace{-0.1cm} \\
(i_{\beta \in \bar{\mathcal{B}}})
}}
\eta_{\gamma}
\Big[
i_1,\dots,i_M \Big| i'_1,\dots,i'_M
\Big]
=
\\
\eta_{\gamma}
\Big[
a_1,\dots,a_{k-1} \Big| [i'_1]_m,\dots,[i'_{k-1}]_m
\Big]
\eta_{\gamma}
\Big[
a_{k+1},\dots,a_M \Big| [i'_{k+1}]^m,\dots,[i'_M]^m
\Big],
\end{multline}
where
\begin{align*}
a_p
=
\left\{
\begin{array}{ll}
i_p, & \qquad\quad p \in \mathcal{A} \cup \mathcal{B},
\\ \\
m, & \qquad\quad p \in \bar{\mathcal{A}},
\\ \\
0, & \qquad\quad p \in \bar{\mathcal{B}}.
\end{array}
\right.
\end{align*}
Substituting \eqref{indicator-ii'} and \eqref{eta-collapse} into \eqref{d-l-points}, we obtain
\begin{multline}
\label{d-l-points2}
\Upsilon^{c}_{k,\ell}
\Big|_{x_c = y_{\gamma}}
=
\sum_{(i'_1,\dots,i'_M)}
\eta_{\gamma}
\Big[
a_1,\dots,a_{k-1} \Big| [i'_1]_m,\dots,[i'_{k-1}]_m
\Big]
\eta_{\gamma}
\Big[
a_{k+1},\dots,a_M \Big| [i'_{k+1}]^m,\dots,[i'_M]^m
\Big]
\\
\times
\sum_{(j'_1,\dots,j'_M)}
\xi_{\gamma}
\Big[
j'_1,\dots,j'_M \Big| j_1,\dots,j_M
\Big]
%\mathfrak{s}^{(y)}_{\gamma,d}
%\dots
%\mathfrak{s}^{(y)}_{\gamma,\gamma+1}
%\\
\Bigg\{
Z^{c+1}_{k,\ell}
\Big[ i'_1,\dots,i'_M
\Big| j'_1,\dots,j'_M \Big]
%\Big|_{y_c=y_{\gamma}}
{\bm 1}\Big( [\#(i'_{k+1},\dots,i'_M) \geq m ] = h \Big)
\\
-
Z^{c+1}_{k,\ell}
\Big[ i'_1,\dots,i'_M
\Big| j'_1,\dots,j'_M \Big]
%\Big|_{y_c=y_{\gamma}}
\Big|_{x_1\leftrightarrow x_{\ell}}
{\bm 1}\Big( [\#(i'_k,\dots,i'_M) \geq m+1 ] = h \Big)
\Bigg\}.
\end{multline}
We may now apply Lemma \ref{lem:diff} (with all indices appearing in that result replaced by primed ones) to the right hand side of \eqref{d-l-points2}. In order to do this, we need to make some identification with the quantities appearing in the statement of Lemma \ref{lem:diff}.

We identify the terms in the second and third lines of \eqref{d-l-points2} with the quantity $\Delta_k[i'_1,\dots,i'_M]$ from Lemma \ref{lem:diff}; this is justified by noting that for any partitioning $\mathcal{A} \cup \bar{\mathcal{A}} = \{1,\dots,k-1\}$ and $\mathcal{B} \cup \bar{\mathcal{B}} = \{k+1,\dots,M\}$, one has
\begin{multline}
\label{d-l-points3}
\sum_{i'_k}
\sum_{
\substack{(i'_{\alpha \in \bar{\mathcal{A}}})
\\ \vspace{-0.1cm} \\
(i'_{\beta \in \bar{\mathcal{B}}})
}}
\Delta_k[i'_1,\dots,i'_M]
=
\sum_{(j'_1,\dots,j'_M)}
\xi_{\gamma}
\Big[
j'_1,\dots,j'_M \Big| j_1,\dots,j_M
\Big]
%\mathfrak{s}^{(y)}_{\gamma,d}
%\dots
%\mathfrak{s}^{(y)}_{\gamma,\gamma+1}
\\
\times
\sum_{i'_k}
\sum_{
\substack{(i'_{\alpha \in \bar{\mathcal{A}}})
\\ \vspace{-0.1cm} \\
(i'_{\beta \in \bar{\mathcal{B}}})
}}
\Bigg\{
Z^{c+1}_{k,\ell}
\Big[ i'_1,\dots,i'_M
\Big| j'_1,\dots,j'_M \Big]
%\Big|_{y_c=y_{\gamma}}
{\bm 1}\Big( [\#(i'_{k+1},\dots,i'_M) \geq m ] = h \Big)
\\
-
Z^{c+1}_{k,\ell}
\Big[ i'_1,\dots,i'_M
\Big| j'_1,\dots,j'_M \Big]
%\Big|_{y_c=y_{\gamma}}
\Big|_{x_1\leftrightarrow x_{\ell}}
{\bm 1}\Big( [\#(i'_k,\dots,i'_M) \geq m+1 ] = h \Big)
\Bigg\}.
\end{multline}
Factorization \eqref{xi-factor} and conservation properties \eqref{xi-conserve1}, \eqref{xi-conserve2} of the $\xi_{\gamma}$ coefficients then ensure that the non-zero terms in the sum over $(j'_1,\dots,j'_M)$ satisfy $j'_1,\dots,j'_{\ell-1} \in \{m+1,m+2,\dots,N\}$, $j'_{\ell} = m$, and $j'_{\ell+1},\dots,j'_M \in \{0,1,\dots,m-1\}$; the $(j'_1,\dots,j'_M)$ values are therefore constrained to the form \eqref{out-conditions}. We can thus write \eqref{d-l-points3} as
\begin{multline*}
\sum_{i'_k}
\sum_{
\substack{(i'_{\alpha \in \bar{\mathcal{A}}})
\\ \vspace{-0.1cm} \\
(i'_{\beta \in \bar{\mathcal{B}}})
}}
\Delta_k[i'_1,\dots,i'_M]
=
\sum_{(j'_1,\dots,j'_M)}
\xi_{\gamma}
\Big[
j'_1,\dots,j'_M \Big| j_1,\dots,j_M
\Big]
\times
\\
%\mathfrak{s}^{(y)}_{\gamma,d}
%\dots
%\mathfrak{s}^{(y)}_{\gamma,\gamma+1}
\Bigg\{
\Phi^{c+1}_{k,\ell}
\Big(
\mathcal{A},\mathcal{B};
(i'_{\alpha \in \mathcal{A}}),
(i'_{\beta \in \mathcal{B}});
(j'_1,\dots,j'_M)
\Big)
%\Big|_{y_c=y_{\gamma}}
-
%\\
\Psi^{c+1}_{k,\ell}
\Big(
\mathcal{A},\mathcal{B};
(i'_{\alpha \in \mathcal{A}}),
(i'_{\beta \in \mathcal{B}});
(j'_1,\dots,j'_M)
\Big)
%\Big|_{y_c=y_{\gamma}}
\Big|_{x_1 \leftrightarrow x_{\ell}}
\Bigg\}
=
0,
\end{multline*}
where the vanishing of the sum is ensured by the inductive assumption (for Z-shaped domains with $\ell-c-1$ internal rows). Hence the quantity $\Delta_k[i'_1,\dots,i'_M]$ satisfies the required sum-to-zero property \eqref{sum-to-0}.

Similarly, we identify the terms in the first line of \eqref{d-l-points2} with the coefficients $C(i'_1,\dots,i'_M)$ from Lemma \ref{lem:diff}; that is, we define
\begin{align*}
C(i'_1,\dots,i'_M)
=
\eta_{\gamma}
\Big[
a_1,\dots,a_{k-1} \Big| [i'_1]_m,\dots,[i'_{k-1}]_m
\Big]
\eta_{\gamma}
\Big[
a_{k+1},\dots,a_M \Big| [i'_{k+1}]^m,\dots,[i'_M]^m
\Big].
\end{align*}
These coefficients clearly obey the requirement \eqref{color-blind}. With these identifications, we may thus apply Lemma \ref{lem:diff} to the right hand side of \eqref{d-l-points2}, yielding
\begin{align*}
\Upsilon^{c}_{k,\ell}
\Big|_{x_c = y_{\gamma}}
=
0,
\qquad\quad
\forall\
\gamma \in \{\ell+1,\dots,d\}.
\end{align*}
Hence $\Upsilon^{c}_{k,\ell}=0$ at $d-\ell$ independent values of $x_c$, and we conclude that $\Upsilon^{c}_{k,\ell} = 0$ identically.

\medskip
\underline{\bf Case 2: $k+1 \in \mathcal{B}$.}

\medskip

In this situation, $i_{k+1}$ assumes a fixed value in the set $\{m+1,m+2,\dots,N\}$. Depending on which value it takes, all three outcomes $j_c > i_{k+1}$, $j_c = i_{k+1}$ or $j_c < i_{k+1}$ are possible. As we saw in the analysis of {\bf Case  1}, $j_c > i_{k+1}$ means that $\bar{\Upsilon}^{c}_{k,\ell}$ has degree $d-\ell-1$ in $x_c$, and we may proceed in exactly the same way to show that $\bar{\Upsilon}^{c}_{k,\ell}$ vanishes at the $d-\ell$ values $x_c \in \{y_{\ell+1},\dots,y_d\}$.

On the other hand, recalling Property (ii) of Proposition \ref{prop:properties}, for $j_c \leq i_{k+1}$ we conclude that $\bar{\Upsilon}^{c}_{k,\ell}$ has degree $d-\ell$ in $x_c$; we therefore need one additional interpolating point. This extra point is either $x_c=0$ or $x_c \rightarrow \infty$, supplied to us by the Properties (iv) and (v) of Proposition \ref{prop:properties}. When $j_c < i_{k+1}$, from Property (iv) we see directly that $\Upsilon^{c}_{k,\ell} |_{x_c=0} = 0$. When $j_c = i_{k+1}$, from Property (v) we see that
\begin{multline}
\label{eq-439}
\lim_{x_c \rightarrow \infty}
\Upsilon^{c}_{k,\ell}
=
%\mathfrak{s}^{(x)}_{c,\dots,\ell-1}
%\mathfrak{s}^{(y)}_{\ell,\dots,M-1}
\sum_{i_k}
\sum_{
\substack{(i_{\alpha \in \bar{\mathcal{A}}})
\\ \vspace{-0.1cm} \\
(i_{\beta \in \bar{\mathcal{B}}})
}}
q^{[\#(i_1,\dots,i_k) > j_c]-[\#(j_1,\dots,j_{c-1}) > j_c]}
\\
\times
\Big\{
Z^{c}_{k,\ell-1}
\Big[ i_1,\dots,\widehat{i_{k+1}},\dots,i_M
\Big| j_1,\dots,\widehat{j_c},\dots,j_M \Big]
{\bm 1}\Big( [\#(i_{k+1},\dots,i_M) \geq m ] = h \Big)
\\
-
Z^{c}_{k,\ell-1}
\Big[ i_1,\dots,\widehat{i_{k+1}},\dots,i_M
\Big| j_1,\dots,\widehat{j_c},\dots,j_M \Big]
\Big|_{x_1\leftrightarrow x_{\ell-1}}
{\bm 1}\Big( [\#(i_{k},\dots,i_M) \geq m+1 ] = h \Big)
\Big\}.
\end{multline}
It turns out that the $q$-dependent factor which stands in the summand of \eqref{eq-439} is actually a constant with respect to the summation. To demonstrate this, we note that
\begin{align}
\label{count>}
[\#(i_1,\dots,i_k) > j_c]
=
[\#(j_1,\dots,j_M) > j_c]
-
[\#(i_{k+1},\dots,i_M) > j_c],
\end{align}
by color-conservation through the original Z-shaped domain. Furthermore, since $i_{\beta} \sumin \{0,1,\dots,m\}$ for all $\beta \in \bar{\mathcal{B}}$ and $j_{\ell+1},\dots,j_M \fixin \{0,1,\dots,m-1\}$, none of these colors can exceed the value $j_c$ (because $j_c \geq m+1$); we may therefore exclude them from the count in \eqref{count>}, yielding
\begin{align}
\label{count>2}
[\#(i_1,\dots,i_k) > j_c]
=
[\#(j_1,\dots,j_{\ell-1}) > j_c]
-
[\#(i_{\beta \in \mathcal{B}}) > j_c],
\end{align}
which is now clearly independent of the summation being taken in \eqref{eq-439}. We conclude that
\begin{multline}
\label{eq-444}
\lim_{x_c \rightarrow \infty}
\Upsilon^{c}_{k,\ell}
=
q^{[\#(j_{c+1},\dots,j_{\ell-1}) > j_c]-[\#(i_{\beta \in \mathcal{B}}) > j_c]}
%\times
%\mathfrak{s}^{(x)}_{c,\dots,\ell-1}
%\mathfrak{s}^{(y)}_{\ell,\dots,M-1}
\\
\times
\sum_{i_k}
\sum_{
\substack{(i_{\alpha \in \bar{\mathcal{A}}})
\\ \vspace{-0.1cm} \\
(i_{\beta \in \bar{\mathcal{B}}})
}}
\Big\{
Z^{c}_{k,\ell-1}
\Big[ i_1,\dots,\widehat{i_{k+1}},\dots,i_M
\Big| j_1,\dots,\widehat{j_c},\dots,j_M \Big]
{\bm 1}\Big( [\#(i_{k+1},\dots,i_M) \geq m ] = h \Big)
\\
-
Z^{c}_{k,\ell-1}
\Big[ i_1,\dots,\widehat{i_{k+1}},\dots,i_M
\Big| j_1,\dots,\widehat{j_c},\dots,j_M \Big]
\Big|_{x_1\leftrightarrow x_{\ell-1}}
{\bm 1}\Big( [\#(i_{k},\dots,i_M) \geq m+1 ] = h \Big)
\Big\}.
\end{multline}
Color $i_{k+1}$ only appears in the indicator functions in \eqref{eq-444}, and since $i_{k+1} \fixin \{m+1,m+2,\dots,N\}$, it contributes equally to the two height parameters appearing in that expression. We deduce that
\begin{multline*}
\lim_{x_c \rightarrow \infty}
\Upsilon^{c}_{k,\ell}
=
q^{[\#(j_{c+1},\dots,j_{\ell-1}) > j_c]-[\#(i_{\beta \in \mathcal{B}}) > j_c]}
%\times
%\mathfrak{s}^{(x)}_{c,\dots,\ell-1}
%\mathfrak{s}^{(y)}_{\ell,\dots,M-1}
\\
\times
\sum_{i_k}
\sum_{
\substack{(i_{\alpha \in \bar{\mathcal{A}}})
\\ \vspace{-0.1cm} \\
(i_{\beta \in \bar{\mathcal{B}}})
}}
\Big\{
Z^{c}_{k,\ell-1}
\Big[ i_1,\dots,\widehat{i_{k+1}},\dots,i_M
\Big| j_1,\dots,\widehat{j_c},\dots,j_M \Big]
{\bm 1}\Big( [\#(i_{k+2},\dots,i_M) \geq m ] = h-1 \Big)
\\
-
Z^{c}_{k,\ell-1}
\Big[ i_1,\dots,\widehat{i_{k+1}},\dots,i_M
\Big| j_1,\dots,\widehat{j_c},\dots,j_M \Big]
\Big|_{x_1\leftrightarrow x_{\ell-1}}
{\bm 1}\Big( [\#(i_k,i_{k+2},\dots,i_M) \geq m+1 ] = h-1 \Big)
\Big\},
\end{multline*}
or equivalently,
\begin{align*}
\lim_{x_c \rightarrow \infty}
\Upsilon^{c}_{k,\ell}
&=
q^{[\#(j_{c+1},\dots,j_{\ell-1}) > j_c]-[\#(i_{\beta \in \mathcal{B}}) > j_c]}
%\times
%\mathfrak{s}^{(x)}_{c,\dots,\ell-1}
%\mathfrak{s}^{(y)}_{\ell,\dots,M-1}
%\\
\times
\Big\{
\Phi^{c}_{k,\ell-1}
\Big(
\mathcal{A},\mathcal{B}';
(i_{\alpha \in \mathcal{A}}),
(i_{\beta \in \mathcal{B}'});
( j_1,\dots,\widehat{j_c},\dots,j_M)
\Big)
\\
&-
\Psi^{c}_{k,\ell-1}
\Big(
\mathcal{A},\mathcal{B}';
(i_{\alpha \in \mathcal{A}}),
(i_{\beta \in \mathcal{B}'});
( j_1,\dots,\widehat{j_c},\dots,j_M)
\Big)
\Big|_{x_1 \leftrightarrow x_{\ell-1}}
\Big\},
\end{align*}
where $\mathcal{B}' = \mathcal{B} \backslash \{k+1\}$ and with $h$ replaced by $h-1$. The final quantity vanishes by the inductive assumption (about Z-shaped domains with $\ell-c-1$ internal rows), and therefore $\lim_{x_c \rightarrow \infty} \Upsilon^{c}_{k,\ell} = 0$. Hence $\Upsilon^{c}_{k,\ell} = 0$ at $d-\ell+1$ independent values of $x_c$, and we conclude that $\Upsilon^{c}_{k,\ell} = 0$ identically.

This concludes the proof that \eqref{z-shaped-claim} holds in the case of $\ell-c$ internal rows, and therefore generically, by induction on the number of internal rows.
\end{proof}

\subsection{Part three: proof for generic down-right domains}

We conclude the proof of Theorem \ref{theorema-egr} by upgrading from Z-shaped to generic down-right domains. Generalizing the statement turns out to be relatively straightforward; the key is to recognise an underlying Z-shaped domain embedded within a generic down-right domain, as illustrated in Figure \ref{fig:join-domains}. We are going to show that the distinction between $\Phi_{P/Q}$ and $\Psi_{P/Q}$ can be attributed to the changes inside this Z-shaped domain.

\begin{figure}
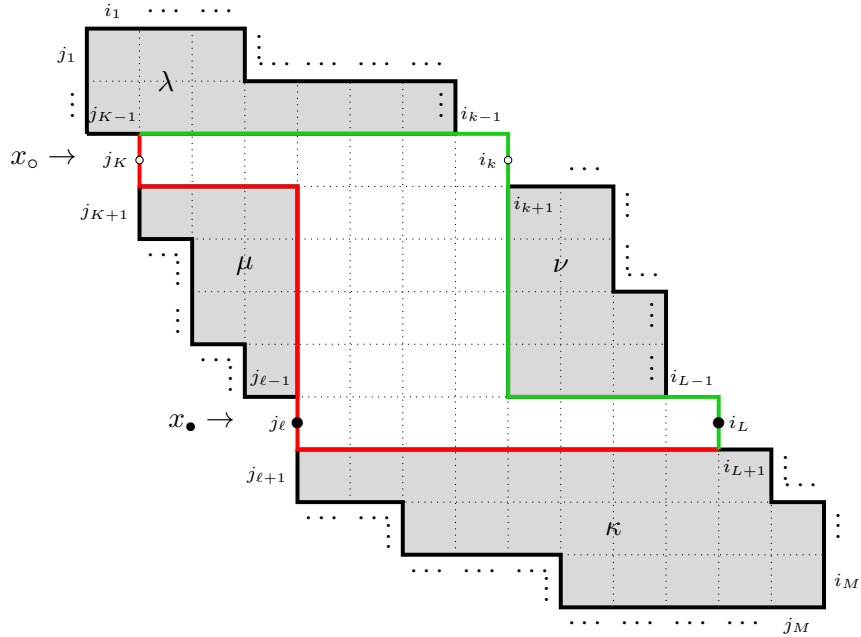

\begin{center}
\tikz{0.7}{
%young diagrams
\filldraw[draw=black,line width=1.5pt,fill=lgray]
(7,0) -- (9,0) -- (9,-2) -- (10,-2) -- (10,-4) -- (7,-4) -- (7,0);
\filldraw[draw=black,line width=1.5pt,fill=lgray]
(0,0) -- (0,-1) -- (1,-1) -- (1,-3) -- (2,-3) -- (2,-4) -- (3,-4) -- (3,0) -- (0,0);
\filldraw[draw=black,line width=1.5pt,fill=lgray]
(3,-5) -- (3,-6) -- (5,-6) -- (5,-7) -- (8,-7) -- (8,-8) -- (13,-8) -- (13,-6) -- (12,-6) -- (12,-5) -- (3,-5);
\filldraw[draw=black,line width=1.5pt,fill=lgray]
(-1,1) -- (-1,3) -- (2,3) -- (2,2) -- (6,2) -- (6,1) -- (-1,1);
%hdots
\draw[dotted] (-1,2) -- (2,2);
\draw[dotted] (0,1) -- (6,1);
\draw[dotted] (0,0) -- (9,0);
\draw[dotted] (0,-1) -- (9,-1);
\draw[dotted] (1,-2) -- (9,-2);
\draw[dotted] (1,-3) -- (10,-3);
\draw[dotted] (2,-4) -- (10,-4);
\draw[dotted] (3,-5) -- (11,-5);
\draw[dotted] (3,-6) -- (12,-6);
\draw[dotted] (5,-7) -- (13,-7);
%vdots
\draw[dotted] (0,1) -- (0,3);
\draw[dotted] (1,-1) -- (1,3);
\draw[dotted] (2,-3) -- (2,2);
\draw[dotted] (3,-4) -- (3,2);
\draw[dotted] (4,-6) -- (4,2);
\draw[dotted] (5,-6) -- (5,2);
\draw[dotted] (6,-7) -- (6,1);
\draw[dotted] (7,-7) -- (7,0);
\draw[dotted] (8,-7) -- (8,0);
\draw[dotted] (9,-8) -- (9,-2);
\draw[dotted] (10,-8) -- (10,-3);
\draw[dotted] (11,-8) -- (11,-5);
\draw[dotted] (12,-8) -- (12,-6);
%paths
\draw[line width=1.5pt,red] (0,1) -- (0,0) -- (3,0) -- (3,-5) -- (11,-5);
\draw[line width=1.5pt,green] (0,1) -- (7,1) -- (7,-4) -- (11,-4) -- (11,-5);
%region labels
\node at (0.5,2) {$\lambda$};
\node at (2,-1.5) {$\mu$};
\node at (8,-1.5) {$\nu$};
\node at (9,-6.5) {$\kappa$};
%bottom labels
\node[left] at (-1,2.5) {\tiny $j_1$};
\node[left] at (-1,1.7) {$\vdots$};
\node[above] at (-0.5,1) {\tiny $j_{K-1}$};
\node[left] at (-1,0.5) {$x_{\circ} \rightarrow$};
\node[left] at (0,0.5) {\tiny $j_K$}; \filldraw[draw=black,fill=white] (0,0.5) circle (2pt);
\node[left] at (0,-0.5) {\tiny $j_{K+1}$};
\node[below] at (0.5,-1) {$\cdots$};
\node[left] at (1,-1.5) {$\vdots$};
\node[left] at (1,-2.4) {$\vdots$};
\node[below] at (1.5,-3) {$\cdots$};
\node[left] at (2,-3.5) {$\vdots$};
\node[above] at (2.5,-4) {\tiny $j_{\ell-1}$};
\node[left] at (2,-4.5) {$x_{\bullet} \rightarrow$};
\node[left] at (3,-4.5) {\tiny $j_{\ell}$};
\node at (3,-4.5) {$\bullet$};
\node[left] at (3,-5.5) {\tiny $j_{\ell+1}$};
\node[below] at (3.5,-6) {$\cdots$};
\node[below] at (4.5,-6) {$\cdots$};
\node[left] at (5,-6.5) {$\vdots$};
\node[below] at (5.5,-7) {$\cdots$};
\node[below] at (6.5,-7) {$\cdots$};
\node[below] at (7.5,-7) {$\cdots$};
\node[left] at (8,-7.5) {$\vdots$};
\node[below] at (8.5,-8) {$\cdots$};
\node[below] at (9.5,-8) {$\cdots$};
\node[below] at (10.5,-8) {$\cdots$};
\node[below] at (11.5,-8) {$\cdots$};
\node[below] at (12.5,-8) {\tiny $j_M$};
\node[above] at (-0.5,3) {\tiny $i_1$};
\node[above] at (0.5,3) {$\cdots$};
\node[above] at (1.5,3) {$\cdots$};
\node[right] at (2,2.8) {$\vdots$};
\node[above] at (2.6,2) {$\cdots$};
\node[above] at (3.5,2) {$\cdots$};
\node[above] at (4.5,2) {$\cdots$};
\node[above] at (5.5,2) {$\cdots$};
\node[left] at (6,1.7) {$\vdots$};
\node[above] at (6.5,1) {\tiny $i_{k-1}$};
\node[left] at (7,0.5) {\tiny $i_k$}; \filldraw[draw=black,fill=white] (7,0.5) circle (2pt);
\node[below] at (7.5,0) {\tiny $i_{k+1}$};
\node[above] at (8.5,0) {$\cdots$};
\node[right] at (9,-0.2) {$\vdots$};
\node[right] at (9,-1.2) {$\vdots$};
\node[above] at (9.6,-2) {$\cdots$};
\node[left] at (10,-2.3) {$\vdots$};
\node[left] at (10,-3.3) {$\vdots$};
\node[above] at (10.5,-4) {\tiny $i_{L-1}$};
\node[right] at (11,-4.5) {\tiny$i_L$}; \node at (11,-4.5) {$\bullet$};
\node[below] at (11.5,-5) {\tiny $i_{L+1}$};
\node[right] at (12,-5.2) {$\vdots$};
\node[above] at (12.6,-6) {$\cdots$};
\node[right] at (13,-6.3) {$\vdots$};
\node[right] at (13,-7.5) {\tiny $i_M$};
}
\end{center}
\caption{A generic down-right domain can always be deconstructed in terms of a Z-shaped domain and four other down-right domains $\lambda,\mu,\nu,\kappa$ (shown in grey). Starting from a generic down-right domain, one begins by identifying the subset of boxes that form the unique Z-shaped domain that (i) takes the $x_{\circ}$ and $x_{\bullet}$-bearing rows as its ``outer'' rows, and (ii) completely saturates those rows. After identifying this underlying Z-shaped domain, the remaining boxes are then split into four disjoint down-right domains. Importantly, none of the subdomains $\lambda,\mu,\nu,\kappa$ depend on $x_{\circ}$ or $x_{\bullet}$.}
\label{fig:join-domains}
\end{figure}

Consider the partition function
$Z_{P/Q} \equiv Z_{P/Q}[i_1,\dots,i_M | j_1,\dots,j_M]$, inherited from the down right domain $P-Q$, where $P,Q$ are down-right paths of length $M$. Label the outgoing edges of the domain by colors $i_1,\dots,i_M$ as usual, where the $k$-th of these, $i_k$, is assigned to a down step of $P$. Assume that the opposing edge on $Q$ carries the label $j_K$. Label the incoming edges of the domain by colors $j_1,\dots,j_M$, where the $\ell$-th of these, $j_{\ell}$, is assigned to a down step of $Q$. Assume that the opposing edge on $P$ carries the label $i_L$. The two distinguished horizontal rapidities of our domain are thus $x_{\circ} = x_K$ and $x_{\bullet} = x_{\ell}$. As usual, the incoming colors $(j_1,\dots,j_M)$ are chosen according to the conditions \eqref{in-conditions}, while the outgoing colors $(i_1,\dots,i_M)$ are fixed/summed according to the rules \eqref{out-conditions}.

Proceeding as in Figure \ref{fig:join-domains}, we then subdivide into five smaller partition functions. We do this by extracting the largest possible Z-shaped subdomain which takes the $x_{\circ}$ and $x_{\bullet}$ carrying rows as its ``outer'' rows; this then naturally isolates four disjoint down-right subdomains (labelled $\lambda,\mu,\nu,\kappa$). We end up with the following, brute-force expansion of $Z_{P/Q}$ over these five sub-partition functions:
\begin{align}
\label{brute-force}
Z_{P/Q}
=
\sum_{\substack{\text{primed}\\ \text{indices}}}
&
X_{\lambda}\Big[ i_1,\dots,i_{k-1}\Big| j_1,\dots,j_{K-1},i'_K,\dots,i'_{k-1}\Big]
\\
\nonumber
\times
&
X_{\mu}\Big[j'_{K+1},\dots,j'_{\ell-1}\Big| j_{K+1},\dots,j_{\ell-1} \Big]
%\\
%\nonumber
%\times
%&
X_{\nu}\Big[ i_{k+1},\dots,i_{L-1}\Big| i'_{k+1},\dots,i'_{L-1} \Big]
\\
\nonumber
\times
&
X_{\kappa}
\Big[ j'_{\ell+1},\dots,j'_L,i_{L+1},\dots,i_M
\Big| j_{\ell+1},\dots,j_M\Big]
\\
\nonumber
\times
&
Z\Big[ i'_K,\dots,i'_{k-1},i_k,i'_{k+1},\dots,i'_{L-1},i_L
\Big| j_K,j'_{K+1},\dots,j'_{\ell-1},j_{\ell},j'_{\ell+1},\dots,j'_L\Big],
\end{align}
where the summation is performed over all primed indices that parametrize the colors along the boundaries of the cut-out subdomains and the Z-shaped one.

\subsubsection{Analysis of $\Phi_{P/Q}$}

We begin by considering $\Phi_{P/Q}$, which is given by
\begin{align}
\Phi_{P/Q}
=
\sum_{i_k}
\sum_{
\substack{(i_{\alpha \in \bar{\mathcal{A}}})
\\ \vspace{-0.1cm} \\
(i_{\beta \in \bar{\mathcal{B}}})
}}
\sum_{\substack{\text{primed}\\ \text{indices}}}
&
{\bm 1}\Big( [\#(i_{k+1},\dots,i_M) \geq m ] = h \Big)
X_{\lambda}\Big[ i_1,\dots,i_{k-1}\Big| j_1,\dots,j_{K-1},i'_K,\dots,i'_{k-1}\Big]
\nonumber
\\
\times
&
X_{\mu}\Big[j'_{K+1},\dots,j'_{\ell-1}\Big| j_{K+1},\dots,j_{\ell-1} \Big]
%\nonumber
%\\
%\times
%&
X_{\nu}\Big[ i_{k+1},\dots,i_{L-1}\Big| i'_{k+1},\dots,i'_{L-1} \Big]
\nonumber
\\
\times
&
X_{\kappa}
\Big[ j'_{\ell+1},\dots,j'_L,i_{L+1},\dots,i_M
\Big| j_{\ell+1},\dots,j_M\Big]
\label{phi1}
\\
\times
&
Z\Big[ i'_K,\dots,i'_{k-1},i_k,i'_{k+1},\dots,i'_{L-1},i_L
\Big| j_K,j'_{K+1},\dots,j'_{\ell-1},j_{\ell},j'_{\ell+1},\dots,j'_L\Big],
\nonumber
\end{align}
where we have introduced summation over $i_k$ and $(i_{\alpha \in \bar{\mathcal{A}}})$, $(i_{\beta \in \bar{\mathcal{B}}})$, subject to the constraints \eqref{out-conditions}. As we will now see, it is possible to apply the color-merging result of Proposition \ref{prop:merge-PQ} to two of the pieces in \eqref{phi1}; namely, $X_{\lambda}$ and $X_{\nu}$.

Starting with the indices $i_{\alpha}$, $\alpha \in \{1,\dots,k-1\}$, recall that they either satisfy $i_{\alpha} \fixin \{0,1,\dots,m-1\}$ if $\alpha \in \mathcal{A}$, or $i_{\alpha} \sumin \{m,m+1,\dots,N\}$ if $\alpha \in \bar{\mathcal{A}}$. This leads to a color-merged version of $X_{\lambda}$:
\begin{align}
\nonumber
\sum_{
(i_{\alpha \in \bar{\mathcal{A}}})
}
X_{\lambda}\Big[ i_1,\dots,i_{k-1} &\Big| j_1,\dots,j_{K-1},i'_K,\dots,i'_{k-1}\Big]
=
X_{\lambda}\Big[ a_1,\dots,a_{k-1}\Big|
[j_1]_m,\dots,[j_{K-1}]_m,[i'_K]_m,\dots,[i'_{k-1}]_m\Big]
\\
&=
X_{\lambda}\Big[ a_1,\dots,a_{k-1}\Big| m,\dots,m,[i'_K]_m,\dots,[i'_{k-1}]_m\Big],
\label{collapse1}
\end{align}
where we have defined
\begin{align}
\label{a}
a_{\alpha} =
\left\{
\begin{array}{cc}
i_{\alpha}, & \quad \alpha \in \mathcal{A},
\\ \\
m, & \quad \alpha \in \bar{\mathcal{A}},
\end{array}
\right.
\end{align}
and noted that since $j_1,\dots,j_{K-1} \fixin \{m+1,m+2,\dots,N\}$, we must have
$[j_1]_m = \cdots = [j_{K-1}]_m = m$.

Next, given that $j_{\ell+1},\dots,j_M \fixin \{0,1,\dots,m-1\}$ and making use of the color-conservation of
\begin{align*}
X_{\kappa}[ j'_{\ell+1},\dots,j'_L,i_{L+1},\dots,i_M | j_{\ell+1},\dots,j_M ],
\end{align*}
we see that necessarily $i_{L+1},\dots,i_M \sumin \{0,1,\dots,m-1\}$ (which means that $L+1,\dots,M \in \bar{\mathcal{B}}$). This means that none of the colors $i_{L+1},\dots,i_M$ can contribute to the value of the height function
$\mathcal{H}^{\ges m}(P;k+1)$. Combining this with the color-conservation of
$X_{\nu}[i_{k+1},\dots,i_{L-1} | i'_{k+1},\dots,i'_{L-1}]$, we deduce that
\begin{align}
%\nonumber
{\bm 1}\Big( [\#(i_{k+1},\dots,i_M) \geq m ] = h \Big)
&=
{\bm 1}\Big( [\#(i_{k+1},\dots,i_L) \geq m ] = h \Big)
%\\
%&
=
{\bm 1}\Big( [\#(i'_{k+1},\dots,i'_{L-1},i_L) \geq m ] = h \Big).
\label{ii'}
\end{align}
Equation \eqref{ii'} turns out to be important in obtaining the color-merged version of $X_{\nu}$, as we will now show. Turning to the indices $i_{\beta}$, $\beta \in \{k+1,\dots,L-1\}$, let us recall that they can either satisfy $i_{\beta} \fixin \{m+1,m+2,\dots,N\}$ if $\beta \in \mathcal{B}$, or $i_{\beta} \sumin \{0,1,\dots,m\}$ if $\beta \in \bar{\mathcal{B}}$. Writing $\breve{\mathcal{B}} = \bar{\mathcal{B}} \cap \{k+1,\dots,L-1\}$, we make use of \eqref{ii'} and Proposition \ref{prop:merge-PQ} to deduce that
\begin{multline}
\label{collapse2}
\sum_{
(i_{\beta \in \breve{\mathcal{B}}})
}
{\bm 1}\Big( [\#(i_{k+1},\dots,i_M) \geq m ] = h \Big)
X_{\nu}\Big[ i_{k+1},\dots,i_{L-1}\Big| i'_{k+1},\dots,i'_{L-1} \Big]
\\
=
{\bm 1}\Big( [\#(i'_{k+1},\dots,i'_{L-1},i_L) \geq m ] = h \Big)
X_{\nu}\Big[ b_{k+1},\dots,b_{L-1}\Big| [i'_{k+1}]^m,\dots,[i'_{L-1}]^m \Big]
\end{multline}
where we have defined
\begin{align}
\label{b}
b_{\beta} =
\left\{
\begin{array}{cc}
i_{\beta}, & \quad \beta \in \mathcal{B},
\\ \\
0, & \quad \beta \in \bar{\mathcal{B}}.
\end{array}
\right.
\end{align}
We now combine our two color-merged results \eqref{collapse1} and \eqref{collapse2} in equation \eqref{phi1}:
\begin{align}
\label{phi2}
\Phi_{P/Q}
=
\sum_{i_k}
\sum_{
(i_{\beta \in \bar{\mathcal{D}}})
}
\sum_{\substack{\text{primed}\\ \text{indices}}}
&
X_{\lambda}\Big[ a_1,\dots,a_{k-1}\Big| m,\dots,m,[i'_K]_m,\dots,[i'_{k-1}]_m\Big]
\\
\times
&
X_{\mu}\Big[j'_{K+1},\dots,j'_{\ell-1}\Big| j_{K+1},\dots,j_{\ell-1} \Big]
\nonumber
%\\
%\times
%&
X_{\nu}\Big[ b_{k+1},\dots,b_{L-1}\Big| [i'_{k+1}]^m,\dots,[i'_{L-1}]^m \Big]
\nonumber
\\
\times
&
X_{\kappa}
\Big[ j'_{\ell+1},\dots,j'_L,i_{L+1},\dots,i_M
\Big| j_{\ell+1},\dots,j_M\Big]
%\\
%\times
%&
{\bm 1}\Big( [\#(i'_{k+1},\dots,i'_{L-1},i_L) \geq m ] = h \Big)
\nonumber
\\
\times
&
Z\Big[ i'_K,\dots,i'_{k-1},i_k,i'_{k+1},\dots,i'_{L-1},i_L
\Big| j_K,j'_{K+1},\dots,j'_{\ell-1},j_{\ell},j'_{\ell+1},\dots,j'_L\Big],
\nonumber
\end{align}
where we have defined $\bar{\mathcal{D}} = \{L+1,\dots,M\}$ in the case $L \in \mathcal{B}$, and $\bar{\mathcal{D}} = \{L,L+1,\dots,M\}$ in the case $L \in \bar{\mathcal{B}}$. Crucially, we note that in this final equation, the only dependence on $x_{\bullet}$ and $x_{\circ}$ is via the Z-shaped domain.

\subsubsection{Analysis of $\Psi_{P/Q}$}

Now we perform a similar analysis on $\Psi_{P/Q}$, which is given by
\begin{align}
\Psi_{P/Q}
=
\sum_{i_k}
\sum_{
\substack{(i_{\alpha \in \bar{\mathcal{A}}})
\\ \vspace{-0.1cm} \\
(i_{\beta \in \bar{\mathcal{B}}})
}}
\sum_{\substack{\text{primed}\\ \text{indices}}}
&
{\bm 1}\Big( [\#(i_k,\dots,i_M) \geq m+1 ] = h \Big)
X_{\lambda}\Big[ i_1,\dots,i_{k-1}\Big| j_1,\dots,j_{K-1},i'_K,\dots,i'_{k-1}\Big]
\nonumber
\\
\times
&
X_{\mu}\Big[j'_{K+1},\dots,j'_{\ell-1}\Big| j_{K+1},\dots,j_{\ell-1} \Big]
%\nonumber
%\\
%\times
%&
X_{\nu}\Big[ i_{k+1},\dots,i_{L-1}\Big| i'_{k+1},\dots,i'_{L-1} \Big]
\nonumber
\\
\times
&
X_{\kappa}
\Big[ j'_{\ell+1},\dots,j'_L,i_{L+1},\dots,i_M
\Big| j_{\ell+1},\dots,j_M\Big]
\label{psi1}
\\
\times
&
Z\Big[ i'_K,\dots,i'_{k-1},i_k,i'_{k+1},\dots,i'_{L-1},i_L
\Big| j_K,j'_{K+1},\dots,j'_{\ell-1},j_{\ell},j'_{\ell+1},\dots,j'_L\Big]
\nonumber
\end{align}
where we have introduced summation over $i_k$ and $(i_{\alpha \in \bar{\mathcal{A}}})$, $(i_{\beta \in \bar{\mathcal{B}}})$, subject to the constraints \eqref{out-conditions}. Not much changes in the analysis, except that the height function whose value we condition is now $\mathcal{H}^{\ges m+1}(P;k)$. This means that \eqref{collapse1}, \eqref{a} still hold as stated; on the other hand, the height function relation \eqref{ii'} gets replaced with
\begin{align}
\nonumber
{\bm 1}\Big( [\#(i_{k},\dots,i_M) \geq m+1 ] = h \Big)
&=
{\bm 1}\Big( [\#(i_{k},\dots,i_L) \geq m+1 ] = h \Big)
\\
&=
{\bm 1}\Big( [\#(i_k,i'_{k+1},\dots,i'_{L-1},i_L) \geq m+1 ] = h \Big),
\label{ii'2}
\end{align}
meaning that in place of the color-merged relation \eqref{collapse2}, we now have
\begin{multline}
\label{collapse3}
\sum_{
(i_{\beta \in \breve{\mathcal{B}}})
}
{\bm 1}\Big( [\#(i_{k},\dots,i_M) \geq m+1 ] = h \Big)
X_{\nu}\Big[ i_{k+1},\dots,i_{L-1}\Big| i'_{k+1},\dots,i'_{L-1} \Big]
\\
=
{\bm 1}\Big( [\#(i_k,i'_{k+1},\dots,i'_{L-1},i_L) \geq m+1 ] = h \Big)
X_{\nu}\Big[ b_{k+1},\dots,b_{L-1}\Big| [i'_{k+1}]^m,\dots,[i'_{L-1}]^m \Big]
\end{multline}
with $\breve{\mathcal{B}} = \bar{\mathcal{B}} \cap \{k+1,\dots,L-1\}$ as previously, and $b_{k+1},\dots,b_{L-1}$ as defined in \eqref{b}. We then combine the color-merged results \eqref{collapse1} and \eqref{collapse3} in equation \eqref{psi1}:
\begin{align}
\label{psi2}
\Psi_{P/Q}
=
\sum_{i_k}
\sum_{
(i_{\beta \in \bar{\mathcal{D}}})
}
\sum_{\substack{\text{primed}\\ \text{indices}}}
&
X_{\lambda}\Big[ a_1,\dots,a_{k-1}\Big| m,\dots,m,[i'_K]_m,\dots,[i'_{k-1}]_m\Big]
\\
\times
&
X_{\mu}\Big[j'_{K+1},\dots,j'_{\ell-1}\Big| j_{K+1},\dots,j_{\ell-1} \Big]
%\nonumber
%\\
%\times
%&
X_{\nu}\Big[ b_{k+1},\dots,b_{L-1}\Big| [i'_{k+1}]^m,\dots,[i'_{L-1}]^m \Big]
\nonumber
\\
\times
&
X_{\kappa}
\Big[ j'_{\ell+1},\dots,j'_L,i_{L+1},\dots,i_M
\Big| j_{\ell+1},\dots,j_M\Big]
\nonumber
%\\
%\nonumber
%\times
%&
{\bm 1}\Big( [\#(i_k,i'_{k+1},\dots,i'_{L-1},i_L) \geq m+1 ] = h \Big)
\\
\times
&
Z\Big[ i'_K,\dots,i'_{k-1},i_k,i'_{k+1},\dots,i'_{L-1},i_L
\Big| j_K,j'_{K+1},\dots,j'_{\ell-1},j_{\ell},j'_{\ell+1},\dots,j'_L\Big],
\nonumber
\end{align}
with the same definition of $\bar{\mathcal{D}}$ as in equation \eqref{phi2}. Once again, $x_{\bullet}$ and $x_{\circ}$ appear only via the final Z-shaped domain.

\subsubsection{Equality of $\Phi_{P/Q}$ and $\Psi_{P/Q}$} To conclude, we compute
\begin{align*}
\Phi_{P/Q} - \Psi_{P/Q} \Big|_{x_{\circ} \leftrightarrow x_{\bullet}}
\end{align*}
using the formulae \eqref{phi2} and \eqref{psi2}. We find that (after a change of summation indices) we can write
\begin{align}
\label{final-quantity}
\Phi_{P/Q} - \Psi_{P/Q} \Big|_{x_{\circ} \leftrightarrow x_{\bullet}}
=
&
\sum_{i_K'',\dots,i_L''}
\sum_{j_K'',\dots,j_L''}
C(i_K'',\dots,i_L'')
D(j_K'',\dots,j_L'')
\\
\nonumber
&
\times
\left\{
{\bm 1}\Big( [\#(i''_{k+1},\dots,i''_L) \geq m ] = h \Big)
\right.
\\
\nonumber
&
\times
Z\Big[ i''_K,\dots,i''_{k-1},i''_k,i''_{k+1},\dots,i''_L
\Big| j''_K,\dots,j''_{\ell-1},j''_{\ell},j''_{\ell+1},\dots,j''_L\Big]
\\
\nonumber
&
-
{\bm 1}\Big( [\#(i''_k,\dots,i''_L) \geq m+1 ] = h \Big)
\\
\nonumber
&
\times
\left.
Z\Big[ i''_K,\dots,i''_{k-1},i''_k,i''_{k+1},\dots,i''_L
\Big|j''_K,\dots,j''_{\ell-1},j''_{\ell},j''_{\ell+1},\dots,j''_L
\Big]_{x_{\circ} \leftrightarrow x_{\bullet}}
\right\},
\end{align}
where the coefficients $C(i_K'',\dots,i_L'')$ are given by
\begin{multline*}
C(i_K'',\dots,i_L'')
=
X_{\lambda}\Big[ a_1,\dots,a_{k-1}
\Big| m,\dots,m,[i''_K]_m,\dots,[i''_{k-1}]_m\Big]
\\
\times
X_{\nu}\Big[ b_{k+1},\dots,b_{L-1}
\Big| [i''_{k+1}]^m,\dots,[i''_{L-1}]^m \Big]
\left({\bm 1}_{i''_L \leq m}\right)
\end{multline*}
if $L \in \bar{\mathcal{B}}$, and by
\begin{multline*}
C(i_K'',\dots,i_L'')
=
X_{\lambda}\Big[ a_1,\dots,a_{k-1}
\Big| m,\dots,m,[i''_K]_m,\dots,[i''_{k-1}]_m\Big]
\\
\times
X_{\nu}\Big[ b_{k+1},\dots,b_{L-1}
\Big| [i''_{k+1}]^m,\dots,[i''_{L-1}]^m \Big]
\left({\bm 1}_{i''_L = i_L}\right)
\end{multline*}
if $L \in \mathcal{B}$. The coefficients $D(j_K'',\dots,j_L'')$ are given by
\begin{multline*}
D(j_K'',\dots,j_L'')
=
\left({\bm 1}_{j''_K=j_K}\right)
X_{\mu}
\Big[j''_{K+1},\dots,j''_{\ell-1}
\Big| j_{K+1},\dots,j_{\ell-1} \Big]
\\
\times
\left({\bm 1}_{j''_{\ell}=m}\right)
\sum_{i_{L+1},\dots,i_M =0}^{m-1}
X_{\kappa}
\Big[ j''_{\ell+1},\dots,j''_L,i_{L+1},\dots,i_M
\Big| j_{\ell+1},\dots,j_M\Big].
\end{multline*}
The precise form of these coefficients is not of essential importance. What chiefly concerns us is that they satisfy
\begin{align*}
C(i_K'',\dots,i_{k-1}'',i_k'',i_{k+1}'',\dots,i_L'')
=
C([i_K'']_m,\dots,[i_{k-1}'']_m,
0,
[i_{k+1}'']^m,\dots,[i_L'']^m)
\end{align*}
for all $i_K'',\dots,i_L'' \geq 0$, and
\begin{align}
\label{D-coeff}
D(j_K'',\dots,j_{\ell-1}'',j_\ell'',j_{\ell+1}'',\dots,j_L'')
=
0,
\qquad
\text{if either}
\qquad
\left\{
\begin{array}{ll}
j_K'',\dots,j_{\ell-1}'' \leq m,
\\
\\
j''_{\ell} \not= m,
\\
\\
j''_{\ell+1},\dots,j''_L \geq m.
\end{array}
\right.
\end{align}
We may therefore invoke Lemma \ref{lem:diff} with
\begin{align*}
\Delta_k[i''_K,\dots,i''_L]
=
&
\sum_{j_K'',\dots,j_L''}
D(j_K'',\dots,j_L'')
\left\{
{\bm 1}\Big( [\#(i''_{k+1},\dots,i''_L) \geq m ] = h \Big)
\right.
\\
&
\times
Z\Big[ i''_K,\dots,i''_{k-1},i''_k,i''_{k+1},\dots,i''_L
\Big| j''_K,\dots,j''_{\ell-1},j''_{\ell},j''_{\ell+1},\dots,j''_L\Big]
\\
&
-
{\bm 1}\Big( [\#(i''_k,\dots,i''_L) \geq m+1 ] = h \Big)
\\
&
\times
\left.
Z\Big[ i''_K,\dots,i''_{k-1},i''_k,i''_{k+1},\dots,i''_L
\Big|j''_K,\dots,j''_{\ell-1},j''_{\ell},j''_{\ell+1},\dots,j''_L
\Big]_{x_{\circ} \leftrightarrow x_{\bullet}}
\right\}.
\end{align*}
The fact that this choice obeys the sum-to-zero requirements \eqref{sum-to-0} is immediate from \eqref{D-coeff} (which ensures that the incoming colors of the Z-shaped domain fall into the category \eqref{in-conditions}, for which we proved all our statements) and Theorem \ref{thm:Z}. We may thus employ Lemma \ref{lem:diff} to conclude that the quantity \eqref{final-quantity} is identically zero. This completes the proof of Theorem \ref{theorema-egr} for arbitrary down-right domains.

\section{Higher spin colored six-vertex model}

\label{Section_higher_spin}

\subsection{Finite spin model}

%\textcolor{green}{[I follow the notation $(L,M)$ of Appendix C  in \url{https://arxiv.org/pdf/1808.01866.pdf} rather than $(N,M)$ of double-fusion.pdf]}

The inhomogeneous six-vertex model can be turned into its higher spin version by a procedure known as (stochastic) fusion. The basic idea is to take an array of vertices lying at the intersections of adjacent $L$ rows and adjacent $M$ columns, and to replace it by a single vertex with the same collections of incoming and outgoing colors. Before explaining the procedure in more detail, let us write down the resulting weight of the \emph{fused} vertex (cf.\ Figure \ref{Fig_higher_vertex}).

There are two groups of colored paths on the incoming edges. A collection $\A=(A_1,\dots,A_N)$ enters from below, which means that there are $A_1$ paths of color $1$, $A_2$ paths of color $2$,\dots, $A_N$ paths of color $N$. We are not keeping track of paths of color $0$ explicitly, but we assume $|\A|=A_1+\dots+A_N\le M$, which corresponds to the fact that the model arose from fusing $M$ columns with paths of colors from $\{0,1,\dots,N\}$. A second collection of colored paths $\B=(B_1,\dots,B_N)$ enters from the left, and this time $|\B|\le L$.

Similarly, there are two groups of colored paths on the outgoing edges: $\C$ exists vertically and $\D$ exits horizontally. We use the following \emph{mnemonic rule} in the notations: the colors on four edges adjacent to a vertex are always listed in the clockwise order, starting from the bottom edge. Projecting on the alphabet, we get precisely the $\A$, $\B$, $\C$, $\D$ notation.

With coordinate-wise addition operation, the conservation law says
\begin{equation}
\label{eq_conseravation_law}
 \A+\B=\C+\D.
\end{equation}
In addition, there are exactly $M$ outgoing paths in the vertical direction and $L$ outgoing paths in the horizontal direction. Since we do not keep track of paths of color $0$, this implies  $|\C|=C_1+\dots+C_N\le M$ and $|\D|\le L$.

\begin{figure}[t]
\begin{center}
{\scalebox{0.75}{\includegraphics{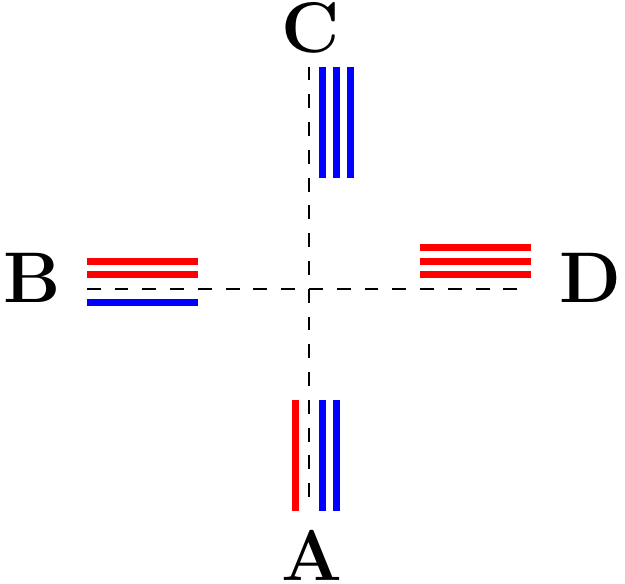}}}
 \caption{A possible vertex for the higher spin colored model with $N=2$. Here $\A=(2,1)$, $\B=(1,2)$, $\C=(3,0)$, $\D=(0,3)$.
 \label{Fig_higher_vertex}}
\end{center}
\end{figure}

The weight of a vertex of type $(\A, \B; \C, \D)$ satisfying \eqref{eq_conseravation_law} depends on two (generally speaking, complex) parameters $z$ and $q$. Following \cite{KMMO}, \cite{BM}, it is given by
\begin{multline}\label{eq_higher_spin_weight}
 W_{L,M}(z,q;  \A, \B; \C, \D)= z^{|\D|-|\B|} q^{|\A|L-|\D|M}
 \\ \times \sum_{P} \Phi(\C-P, \C+\D-P; q^{L-M} z, q^{-M} z) \Phi(P,\B; q^{-L}/z, q^{-L}),
\end{multline}
where the summation goes over $P=(P_1,\dots,P_N)$ with $0\le P_i \le \min(B_i,C_i)$, and
\begin{equation}
\label{eq_Phi}
 \Phi(\lambda,\mu; u,v)=\frac{(u;q)_{|\lambda|} (v/u;q)_{|\mu|-|\lambda|}}{(v;q)_{|\mu|}} \left(\frac{v}{u}\right)^{|\lambda|} q^{\sum_{i<j} (\mu_i-\lambda_i) \lambda_j} \prod_{i\ge 1} {\mu_i\choose \lambda_i}_q.
\end{equation}
We use the $q$-Binomial coefficient defined for $a\ge b\ge 0$ through
$$
 {a\choose b}_{q}=\frac{(1-q^a)(1-q^{a-1})\cdots (1-q^{a-b+1})}{(1-q)(1-q^2)\cdots (1-q^b)}=\frac{(q;q)_{a}}{(q;q)_b (q;q)_{a-b}}.
$$
Note that the number of colors $N$ does not enter directly into the formula \eqref{eq_higher_spin_weight}, and as long as $|\A|$, $|\B|$, $|\C|$, $|\D|$ are finite, we can assume without loss of generality that $N=\infty$.
\begin{remark}
\label{Remark_color_merging}
 An important feature of \eqref{eq_higher_spin_weight} is that if we make paths of two neighboring colors $i$ and $i+1$ indistinguishable (combine them into a single color), then the formula remains the same. More precisely,
 \begin{equation}
 \label{eq_color_projection}
  \sum_{\begin{smallmatrix} C_i=0,1,\dots,\tilde C_i, \\ C_{i+1}=\tilde C_i-C_i,\\ \D=\A+\B-\C, \end{smallmatrix}}  W_{L,M}(z,q;  \A, \B; \C, \D)=  W_{L,M}(z,q;  \tilde \A, \tilde \B; \tilde \C, \tilde \D),
 \end{equation}
 where
 $$
  \tilde \A=(A_1,A_2,\dots,A_{i-1}, A_i+A_{i+1}, A_{i+2},\dots,A_N),
 $$
 and similarly for $\tilde \B$, $\tilde \C$, $\tilde \D$. A direct proof of this fact can  be obtained by applying summation identities for the $q$-Binomial coefficients, cf.\ Lemma \ref{Lemma_q_Binomial_split}.  Another way to deduce \eqref{eq_color_projection} is by combining the following two observations:
 \begin{itemize}
  \item For the colored six-vertex model this is  immediate, cf.\ Section \ref{Section_merging}. Indeed, by the definition, whenever paths of colors $i$ and $j$ enter into a vertex, the stochastic rule of their evolution depends only on the order of $i$ and $j$, but not on the exact values of $i$ and $j$.
  \item \eqref{eq_higher_spin_weight} is obtained from the colored six-vertex model by  fusion, which commutes with combination of two colors into one.
 \end{itemize}
\end{remark}

We would like to consider the model in the quadrant with rows $1, 2, 3, \dots$, and columns $0,1,2,\dots$. Later on, column $0$ will become special, while the rest of the system will be set to be homogeneous. The boundary conditions are as follows: along the bottom of the quadrant there are no entering paths of positive colors, i.e., all incoming paths are of color $0$. Along the left boundary of the quadrant, all $L$ left paths entering in row $i$ are of color $i$, see Figure \ref{Fig_fused_bc}.  The integral parameter $L$, as well as $0<q<1$, are assumed to be fixed throughout the system. On the other hand, we allow the number of paths on vertical edges, $M$, to vary across columns. We thus choose a sequence of positive integers $M_0, M_1, M_2,\dots$ and set $M=M_x$ for the vertices in column $x$. Similarly, we fix spectral parameters $z_0, z_1, z_2,\dots$, corresponding to the columns of the quadrant.

\begin{figure}[t]
\begin{center}
{\scalebox{0.7}{\includegraphics{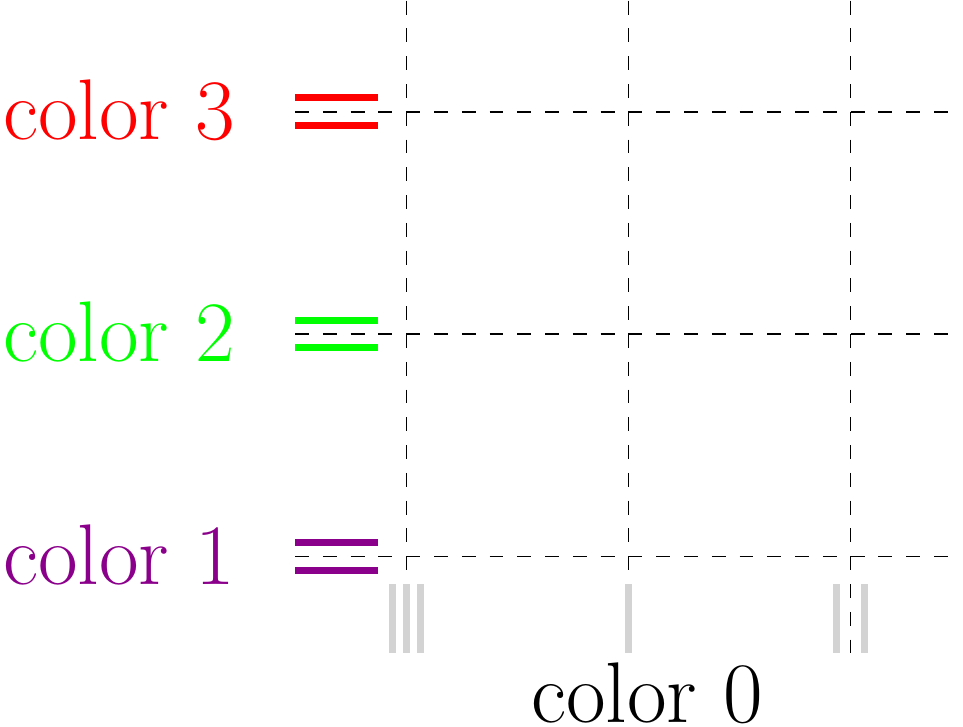}}}
 \caption{Boundary conditions for the fused vertex model in the quadrant and with $L=2$, $M_0=3$, $M_1=1$, $M_2=2$.
 \label{Fig_fused_bc}}
\end{center}
\end{figure}

 The model is sampled sequentially, starting from the vertex at $(0,1)$, then proceeding to $(0,2)$ and $(1,1)$, then to $(0,3)$, $(1,2)$, $(2,1)$, etc. At each step we use \eqref{eq_higher_spin_weight} to sample the colors of the outgoing paths given the colors of the incoming paths. For any choice of spectral parameters the probabilities sum up to $1$, see \cite{BM}, \cite[(C.1.5)]{BW},  and we tacitly assume that $z_0, z_1,\dots$ are chosen so that all probabilities are non-negative.

As for the colored six-vertex model, we use height functions to describe the configurations. For each color $i\ge 1$, the \emph{height function} $\H^{\ges i}(x,y)$ describes colors $\geqslant i$. It is defined by setting $\H^{\ges i}(-\frac12,\frac12)=0$ and
$$
\H^{\ges i}(x,y+1)-\H^{\ges i}(x,y)=\text{ number of paths of colors } \geqslant i\text{ at } (x,y+\tfrac12),
$$
$$
\H^{\ges i}(x+1,y)-\H^{\ges i}(x,y)=\text{ number of paths of colors } \geqslant i\text{ at } (x+\tfrac12,y).
$$
In other words, $\H^{\ges i}(x,y)$ counts the total number of paths of colors $\geqslant i$  below the point $(x,y)$.

The following theorem explains how the higher spin version can be obtained from the ordinary six-vertex model.

\begin{theorem} \label{Theorem_fusion}
 Consider the inhomogeneous colored six-vertex model in the quadrant $\mathbb Z_{>0}\times \mathbb Z_{>0}$ with the following specialization of the parameters:
 \begin{itemize}
 \item The row rapidites $v_1, v_2, v_3,\dots$ are periodic with period $L$ and are given by
 $$
  1,q,q^2,\dots,q^{L-1},1,q,q^2,\dots q^{L-1},\dots.
 $$
 \item The incoming paths along the left boundary are split into adjacent groups of $L$ paths of the same color. From bottom to top it reads:
 $$
  \underbrace{1,\dots,1}_L, \underbrace{2,\dots,2}_L, \dots.
 $$
 \item The column rapidities $u_1, u_2, u_3,\dots$ form subsequent geometric sequences of lengths $M_0$, $M_1$, \dots and are given by
     $$
      z_0, q z_0,\dots, q^{M_0-1} z_0, z_1, q z_1,\dots,q^{M_1-1} z_1, z_2, qz_2,\dots, q^{M_2-1} z_2,\dots
     $$
 \item Incoming paths along the bottom boundary all have color $0$.
 \end{itemize}
 Let $\H^{\ges i}_{6v}(x,y)$, $i=1,2,\dots$ denote the height functions of the resulting vertex model and let $\H^{\ges i}_{\text{fused}}(x,y)$, $i=1,2,\dots$ denote the height functions of the higher spin model defined above the theorem. Then we have an identity of finite-dimensional distributions for the height functions:
 $$
  \H^{\ges i}_{\text{fused}}\left(-\tfrac{1}{2}+x,\tfrac{1}{2}+y\right)= \H^{\ges i}_{6v}\left(\tfrac{1}{2}+(M_0+M_1+\dots+M_{x-1}),\tfrac{1}{2}+L y\right),\qquad i\in \mathbb Z_{>0}; x,y\in\mathbb Z_{\ge 0}.
 $$
\end{theorem}

%The proof of Theorem \ref{Theorem_fusion} is given in Section \ref{Section_fusion}.

\begin{proof}
We shall present the proof in the case of one-dimensional distributions, as the proof for joint distributions is the same.

The proof for one-dimensional distributions boils down to equating two partition functions; the first one in the unfused six-vertex model with weights of Figure \ref{fund-vert}, and the second one in the doubly-fused model with weighs \eqref{eq_higher_spin_weight}. By its very definition, we note that the distribution of
$\mathcal{H}^{\ges i}_{6v}$ can be computed as follows:
\begin{align}
\label{unfused-pf}
\mathbb{P}\Big[
\mathcal{H}^{\ges i}_{6v}
\left(\tfrac{1}{2}+(M_0+M_1+\dots+M_{x-1}),\tfrac{1}{2}+L y\right)
= h\Big]
=
\tikz{0.45}{
\foreach\y in {1,...,3}{
\draw[lgray,line width=1.5pt,->] (0,\y) -- (13,\y);
}
\foreach\y in {5,...,7}{
\draw[lgray,line width=1.5pt,->] (0,\y) -- (13,\y);
}
\foreach\y in {9,...,11}{
\draw[lgray,line width=1.5pt,->] (0,\y) -- (13,\y);
}
\foreach\x in {1,...,4}{
\draw[lgray,line width=1.5pt,->] (\x,0) -- (\x,12);
}
\foreach\x in {6,...,8}{
\draw[lgray,line width=1.5pt,->] (\x,0) -- (\x,12);
}
\foreach\x in {10,...,12}{
\draw[lgray,line width=1.5pt,->] (\x,0) -- (\x,12);
}
\node[left] at (0,1) {$1$}; \node[left] at (0,2) {$1$};  \node[left] at (0,3) {$1$};
\node[left] at (0,5) {$2$}; \node[left] at (0,6) {$2$};  \node[left] at (0,7) {$2$};
\node at (0,8.3) {$\vdots$};
\node[left] at (0,9) {$y$}; \node[left] at (0,10) {$y$};  \node[left] at (0,11) {$y$};
\node[right] at (13,1) {$\ell_1$};
\node[right] at (13,2) {$\ell_2$};
\node at (13.5,6.3) {$\vdots$};
\node[right] at (13,11) {$\ell_{yL}$};
\node[above] at (1,12) {$k_1$};
\node[above] at (2,12) {$k_2$};
\node[above] at (7,12) {$\cdots$};
\node[above] at (12,12) {$k_{\tilde{M}}$};
\node[below] at (1,0) {$0$}; \node[below] at (2,0) {$0$}; \node[below] at (3,0) {$0$}; \node[below] at (4,0) {$0$};
\node[below] at (6,0) {$0$}; \node[below] at (7,0) {$0$}; \node[below] at (8,0) {$0$};
\node at (9,0) {$\cdots$};
\node[below] at (10,0) {$0$}; \node[below] at (11,0) {$0$}; \node[below] at (12,0) {$0$};
}
\end{align}
where we have defined $\tilde{M} = M_0 + M_1 + \cdots + M_{x-1}$, and where the outgoing indices $(k_1,\dots,k_{\tilde{M}})$ and $(\ell_1,\dots,\ell_{yL})$ are summed over all colors, subject to the constraint $[\#(\ell_1,\dots,\ell_{yL}) \geq i] = h$. We recall that the row rapidities are periodic with period $L$ and set to a geometric progression in $q$ with base $1$; the column rapidities in the $j$-th cluster of vertical lines are given by $z_j,qz_j,\dots,q^{M_j-1} z_j$, where $0 \leq j \leq x-1$.

On the other hand, $\H^{\ges i}_{\text{fused}}$ can be computed as follows:
\begin{align}
\label{fused-pf}
\mathbb{P}\Big[
 \H^{\ges i}_{\text{fused}}\left(-\tfrac{1}{2}+x,\tfrac{1}{2}+y\right)
 =
 h
\Big]
=
\tikz{0.6}{
\foreach\x in {1,3,5}{
\draw[lgray,line width=4pt,->] (\x,0) -- (\x,6.5);
}
\foreach\y in {1,3,5}{
\draw[lgray,line width=4pt,->] (0,\y) -- (6.5,\y);
}
\node[left] at (0,1) {$L \cdot {\bm e}_1$};
\node[left] at (0,3) {$L \cdot {\bm e}_2$};
\node at (0,4.3) {$\vdots$};
\node[left] at (0,5) {$L \cdot {\bm e}_y$};
\node[right] at (6.5,1) {$\bm{L}^{(1)}$};
\node[right] at (6.5,3) {$\bm{L}^{(2)}$};
\node at (6.5,4.3) {$\vdots$};
\node[right] at (6.5,5) {$\bm{L}^{(y)}$};
\node[below] at (1,0) {${\bm 0}$};
\node[below] at (3,0) {${\bm 0}$};
\node at (4,0) {$\cdots$};
\node[below] at (5,0) {${\bm 0}$};
\node[above] at (1,6.5) {$\bm{K}^{(0)}$};
\node[above] at (3,6.5) {$\bm{K}^{(1)}$};
\node at (4,6.5) {$\cdots$};
\node[above] at (5.5,6.5) {$\bm{K}^{(x-1)}$};
}
\end{align}
where $\bm{e}_j$ denotes the $j$-th Euclidean unit vector, while each $\bm{L}^{(j)}$ is a composition whose total weight is at most $L$ and each $\bm{K}^{(j)}$ is a composition whose total weight is at most $M_j$. Here the outgoing compositions $\bm{K}^{(j)}$, $\bm{L}^{(j)}$ are assumed to be summed over all possible choices, such that the total number of colors of value $i$ or greater present in the compositions $\bm{L}^{(1)},\dots,\bm{L}^{(y)}$ is equal to $h$. The spectral parameter associated to each vertex in the $j$-th column of the lattice is $z_j$, $0 \leq j \leq x-1$.

We can now fuse each $M_x\times L$ block of vertices in \eqref{unfused-pf} into a single vertex in \eqref{fused-pf}. This is done sequentially in the order of growing $x$ and $y$ coordinates. So we take $M_x\times L$ block of vertices, fix some distribution of $M_x+L$ of colors of the incoming (from below and from the left) edges, and sample all the vertices in the block. As a result we get a new distribution on $M_x+L$ colors of the outgoing (to the right and up) edges. The procedure is now based on two ingredients:
\begin{itemize}
 \item Suppose that the distribution of colors of incoming $M_x+L$ edges is $q$--exchangeable, which means that if we interchange colors of two incoming from below edges, then the probability of such configuration of colors is multiplied by $q$ if we increase the number of inversions (from positions to color numbers) and is multiplied by $q^{-1}$ if we decrease the number of inversions; and similarly for the incoming from the left colors of edges. Then the distribution of colors of outgoing $M_x+L$ edges is also $q$--exchangeable. This is proven by applying the result of \cite[Proposition B.2.2]{BW} in both the horizontal and vertical directions.
 \item If the distribution of the incoming colors is $q$--exchangeable and we ignore the positions of edges of different colors, but only keep track of the number of incoming/outgoing edges for each color and vertical/horizontal directions, then the transition from incoming colors to outgoing colors is given by the higher spin weight \eqref{eq_higher_spin_weight}. This is a certain summation identity, which we prove in Section \ref{Section_fusion}.
\end{itemize}
Note that for our boundary conditions in the quadrant there is no way to exchange the positions of incoming colors on the border of $M_x\times L$ block in a non-trivial way. Hence, the deterministic distribution of incoming colors is $q$--exchangeable and we can proceed with sampling $M_x\times L$ blocks one by one, establishing $q$--exchangeability along the boundaries of the blocks and identifying with higher spin model on each step.
\end{proof}

\bigskip

We can now state the shift-invariance theorem for the higher spin model.
For two points $\mathcal U=(x^{\mathcal U},y^{\mathcal U})$, $\mathcal V=(x^{\mathcal V},y^{\mathcal V})$ in the quadrant, we write $\mathcal U\succeq \mathcal V$ if $x^{\mathcal U}\le x^{\mathcal V}$ and $y^{\mathcal U}\ge y^{\mathcal V}$. In other words, $\mathcal V$ is in the down--right direction from $\mathcal U$. Fix a collection of numbers $0\le k_1\le k_2\le \dots\le k_n$ and a collection of points in the quadrant $\{\mathcal U_i\}$.
\begin{theorem} \label{Theorem_6v_invariance_fused}
 In the above setting of the fused colored stochastic model in quadrant with vertex weights \eqref{eq_higher_spin_weight}, with $z$ and $M$ parameters depending on the column in $\{0,1,2,\dots\}$ (but not on the row) and with fixed $q$, choose an index $1\le \iota \le n$ and an integer $\Delta>0$. Set
 $$
  k_j'=\begin{cases} k_j,& j\ne \iota,\\ k_\iota+\Delta, & j=\iota,\end{cases} \qquad \qquad \mathcal U'_j=\begin{cases}\mathcal U_j, & j\ne \iota \\ \mathcal U_\iota + (0,\Delta), & j=\iota. \end{cases}
 $$
Suppose that
 $$
  0\le k_1\le k_2\le \dots\le k_n, \qquad 0\le k'_1\le k'_2\le \dots\le k'_n,
 $$
 $$\mathcal U_1,\dots,\mathcal U_{\iota-1} \succeq \mathcal U_\iota \succeq\mathcal U_{\iota+1},\dots, \mathcal U_n, \qquad \mathcal U'_1,\dots,\mathcal U'_{\iota-1} \succeq \mathcal U'_\iota \succeq \mathcal U'_{\iota+1},\dots, \mathcal U'_n.
 $$
 Then the distribution of the vector of the height functions
 $$
  \bigl(\H^{\ges k_1}(\mathcal U_1),\, \H^{\ges k_2}(\mathcal U_2),\, \dots, \H^{\ges k_n}(\mathcal U_n)\bigr)
 $$
 coincides with the distribution of the vector with shifted $\iota$-th coordinate
 $$
  \bigl(\H^{\ges k'_1}(\mathcal U'_1),\, \H^{\ges k'_2}(\mathcal U'_2),\, \dots, \H^{\ges k'_n}(\mathcal U'_n)\bigr).
 $$
\end{theorem}
\begin{remark}
  The chosen ordering of points guarantees that the segments $(0,k_j)-\mathcal U_j$ intersect with $(0,k_\iota)-\mathcal U_\iota$ for each $j\ne\iota$.
\end{remark}

\begin{remark}
 It is natural to ask whether one can make the statement of Theorem \ref{Theorem_6v_invariance_fused} inhomogeneous in the vertical direction in the sense that the parameter $L$ will not be fixed, but will vary with the row number. We do not know, as our proofs do not extend to such generality. However, computer simulations indicate that this should be the case.
\end{remark}
\begin{remark}
\label{Remark_analogy_binomial}
 If we restrict our attention to a rectangle, then the state space of  the fused colored stochastic model is finite, and, therefore, Theorem \ref{Theorem_6v_invariance_fused} is an identity between two finite sums.

 Let us draw a vague analogy here. If we think about the colored six-vertex model as being an analogue of a Bernoulli random variable, then the fused version is an analogue of a binomial random variable, as we obtain the fused model by taking several instances of the vertices of the six-vertex model and then ignoring a part of the data (cf.\ taking several Bernoulli random variables and looking only at their sum to get the binomial distribution). The Binomial distribution has an \emph{analytic continuation} in which the state space is no longer finite --- this is the negative binomial distribution. Similarly, in the following sections we are going to study analytic continuations of Theorem \ref{Theorem_6v_invariance_fused} leading to an infinite state space and culminating in Theorem \ref{Theorem_6v_invariance_infinite}.
\end{remark}

\begin{proof}[Proof of Theorem \ref{Theorem_6v_invariance_fused}]
 It suffices to prove the theorem for $\Delta=1$, which we do. Consider the colored six-vertex model appearing in Theorem \ref{Theorem_fusion}. It can be obtained from the six-vertex model with rainbow boundary condition $1,2,3,4,\dots$ along the left boundary of the quadrant (as in Theorems \ref{Theorem_6v_intro} and \ref{Theorem_6v_main_text_ver}) by merging the colors:
 $$
  i\mapsto \left\lfloor \frac{i-1}{L}\right\rfloor +1.
 $$
 Recall that merging the adjacent colors leads to the model of the same kind and with the same rapidities, cf.\ Section \ref{Section_merging} and Remark \ref{Remark_color_merging}. Hence, the distributional identity of Theorem \ref{Theorem_6v_invariance_fused} is obtained through the following five steps:
 \begin{enumerate}
  \item The vector  $\bigl(\H^{\ges k_1}(\mathcal U_1),\, \H^{\ges k_2}(\mathcal U_2),\, \dots, \H^{\ges k_n}(\mathcal U_n)\bigr)$ is replaced by a vector of height functions in the six-vertex model of Theorem \ref{Theorem_fusion}.
   \item The six-vertex model of Theorem \ref{Theorem_fusion} is identified with color merging of the six-vertex model with rainbow boundary condition.
   \item We apply Theorem \ref{Theorem_6v_intro} (or Theorem \ref{Theorem_6v_main_text_ver}) $L$ times to the latter. As a result, the observation point is shifted by $L$ in the vertical direction. Note that since $L$ was the length of the period for the row rapidities, after $L$ swaps of rapidities, we get the same periodic sequence of length $L$ geometric series.
   \item For the vector of height functions with shifted observation points, we again do color merging, arriving back at the six-vertex model of Theorem \ref{Theorem_fusion}.
   \item Applying Theorem \ref{Theorem_fusion} second time, we get $\bigl(\H^{\ges k'_1}(\mathcal U'_1),\, \H^{\ges k'_2}(\mathcal U'_2),\, \dots, \H^{\ges k'_n}(\mathcal U'_n)\bigr)$. \qedhere
 \end{enumerate}
\end{proof}

\subsection{Analytic continuation}

\label{Section_analytic_continuation}

In the setting of Theorem \ref{Theorem_6v_invariance_fused} the total number of colored paths on a horizontal/vertical lattice edge is bounded by $L$/$M$. It turns out that we can get rid of this restriction through an analytic continuation of the weight \eqref{eq_higher_spin_weight} in $q^L$ and $q^M$. This procedure, however, requires some care due to the boundary condition we use: densely packed collection of paths enters the quadrant from the left, which would not make sense if we naively put $L=\infty$.

\bigskip

We start from the analytic continuation in $M$. Note that the expression \eqref{eq_higher_spin_weight} is a rational function in $q^{-M}$. The restrictions $|\A|$, $|\C|\le M$ become irrelevant as long as $M$ is large enough, since our boundary conditions imply that a vertex at $(x,y)$ can have at most $y L$ paths of positive colors.

Hence, we can simply replace $q^{-M}$ by a complex number $\mathfrak m$ and get a formula for the weight:
 \begin{multline}\label{eq_higher_spin_weight_continued_1}
 W_{L,\infty,\mathfrak m}(z,q;  \A, \B; \C, \D)= z^{|\D|-|\B|} q^{|\A|L} \mathfrak m^{|\D|}
 \\ \times \sum_{P} \Phi(\C-P, \C+\D-P; \mathfrak m q^{L} z, \mathfrak m z) \Phi(P,\B; q^{-L}/z, q^{-L}),
\end{multline}
where $\infty$ in the subscript indicates that there are no restrictions on the number of vertical edges (of positive colors, with multiplicities, and we are only keeping track of those). The following statement is a direct analytic continuation of the identity of Theorem \ref{Theorem_6v_invariance_fused}, cf.\ Remark \ref{Remark_analogy_binomial}.
\begin{corollary}
\label{Corollary_higher_spin_inv}
 With weights \eqref{eq_higher_spin_weight_continued_1} and with parameter $\mathfrak m$ depeding on the column through an arbitrary sequence $\mathfrak m_0,\mathfrak m_1,\dots$, the statement of Theorem \ref{Theorem_6v_invariance_fused} remains valid.
\end{corollary}

The next step is to analytically continue in $q^L$. Let us start from the analysis of the $0$-th column, which is based on the following observation.
\begin{lemma} \label{Lemma_first_column}
 Take the weight of \eqref{eq_higher_spin_weight_continued_1}, assume that $\A$ has only colors no larger than $N$  and that $\B$ has $L$ paths of color $N$, i.e., $\B=(0,0,\dots,L)$. Then for each $0<q<1$ and $z \in\mathbb C$ we have
 \begin{multline}
 \label{eq_limit_first_column}
  \lim_{\mathfrak m \to 0} W_{L,\infty,\mathfrak m}(z/\mathfrak m,q;  \A, \B; \C, \D)\\=\begin{cases} \dfrac{(zq^L;q)_\infty  (q^{-L};q)_d}{(z;q)_\infty (q;q)_d} (zq^L)^d,& \D=(0,\dots,0,d),\, d\ge 0,\text{ and }  \A+\B=\C+\D,\\ 0,& \text{ otherwise.}\end{cases}
 \end{multline}
\end{lemma}
\begin{remark}
 The sum (over $d=0,\dots,L$) of the probabilities in the right-hand side of \eqref{eq_limit_first_column} is $1$, as follows from the $q$-Binomial theorem. For $0<q<1$, $L=1,2,\dots$, and $z<0$ the expressions in \eqref{eq_limit_first_column} are non-negative, and hence they give a \emph{bona fide} probability distribution.
\end{remark}
\begin{proof}[Proof of Lemma \ref{Lemma_first_column}] We have
\begin{multline*}
  W_{L,\infty,\mathfrak m}(z/\mathfrak m,q;  \A, \B; \C, \D)= z^{|\D|-|\B|} q^{|\A|L} \mathfrak m^{|\B|}
 \\ \times \sum_{P} \Phi(\C-P, \C+\D-P; q^{L} z ,  z) \Phi(P,\B; \mathfrak m q^{-L}/z, q^{-L}).
\end{multline*}
Since $0\le P\le \B$, the summation goes over $P=(0,\dots,0,p)$ with $0\le p \le L$.  As $\mathfrak m\to 0$, we have
\begin{multline}
 \Phi(P,\B; \mathfrak m q^{-L}/z, q^{-L}) \sim  (-1)^{|\B|-|P|}\cdot \frac{ q^{0+1+\dots+(|\B|-|P|-1)}}{(q^{-L};q)_{|\B|}} \left(z/\mathfrak m\right)^{|\B|} q^{\sum_{i<j} (B_i-P_i) P_j} \prod_{i=1}^N {B_i\choose P_i}_q
 \\=  (-1)^{L-p} \frac{ q^{0+1+\dots+(L-p-1)}}{(q^{-L};q)_{L}} \left(z/\mathfrak m\right)^{L} {L\choose p}_q.
\end{multline}
Also
\begin{multline}
 \Phi(\C-P, \C+\D-P;  q^{L} z ,  z)\\=\frac{(q^{L} z;q)_{|\C|-p} (q^{-L};q)_{|\D|}}{(z;q)_{|\C|+|\D|-p}} \left(q^{-L}\right)^{|\C|-p} q^{\sum_{i<j} (D_i) (\C-P)_j} \prod_{i=1}^N {(\C+\D-P)_i\choose (\C-P)_i}_q.
\end{multline}

We now assume that $L=1$. In this case we get
\begin{multline*}
  \lim_{\mathfrak m\to 0} W_{1,\infty,\mathfrak m}(z/\mathfrak m,q;  \A, \B; \C, \D)= z^{|\D|-1} q^{|\A|} \\ \times \Biggl[
 -\frac{(q z;q)_{|\C|} (q^{-1};q)_{|\D|}}{(z;q)_{|\C|+|\D|}} \left(q^{-1}\right)^{|\C|} q^{\sum_{i<j} D_i C_j} \prod_{i=1}^N {C_i+D_i\choose C_i}_q  \frac{1}{(1-q^{-1})} z {1\choose 0}_q
 \\+\frac{(q  z;q)_{|\C|-1} (q^{-1};q)_{|\D|}}{(z;q)_{|\C|+|\D|-1}} \left(q^{-1}\right)^{|\C|-1} q^{\sum_{i<j} D_i (\C-(0,\dots,0,1))_j} \\ \times \prod_{i=1}^N {(\C+\D-(0,\dots,0,1))_i\choose (\C-(0,\dots,0,1))_i}_q \frac{1}{(1-q^{-1})} z {1\choose 1}_q\Biggr].
\end{multline*}
The factor $(q^{-1};q)_{|\D|}$ in both terms within brackets leads to vanishing of the expression unless $|\D|=0$ or $|\D|=1$. In the former case, we get
\begin{multline*}
  \lim_{\mathfrak m\to 0} W_{1,\infty,\mathfrak m}(z/\mathfrak m,q;  \A, \B; \C, 0)= z^{-1} q^{|\A|}\biggl[
  -\frac{(q z;q)_{|\C|} }{(z;q)_{|\C|}} \left(q^{-1}\right)^{|\C|}  \frac{1}{(1-q^{-1})} z
 \\ +\frac{(q  z;q)_{|\C|-1} }{(z;q)_{|\C|-1}} \left(q^{-1}\right)^{|\C|-1}  \frac{1}{(1-q^{-1})} z\biggr]
 =
 q^{|\A|+1-|\C|}\frac{1}{1-z}.
\end{multline*}
Due to the condition $\A+\B=\C+\D$, the power of $q$ vanishes, and the result matches \eqref{eq_limit_first_column}.

For the case $|\D|=1$ we have two subcases. Either $\D=(0,\dots,0,1)$, i.e., the outgoing path has color $N$, or it has smaller color. In the former case,
\begin{multline*}
  \lim_{\mathfrak m\to 0} W_{1,\infty,\mathfrak m}(z/\mathfrak m,q;  \A, \B; \C, (0,\dots,0,1))= \frac{z q^{|\A|-|\C|}}{1-z} \Biggl[
 -   {C_N+1 \choose C_N}_q
 + q   {C_N \choose C_N-1}_q   \Biggr]\\= \frac{z q^{|\A|-|\C|}}{1-z} \Biggl[
 -  \frac{1-q^{C_N+1}}{1-q}
 + q   \frac{1-q^{C_N}}{1-q}  \Biggr]
 = -\frac{z }{1-z},
\end{multline*}
which matches \eqref{eq_limit_first_column}. In remains to study the case when the outgoing path in $\D$ has color $u<N$. In principle, since the contributions of the first two cases sum up to $1$, if we know the positivity and stochasticity of the weights, then the third case must give zero contribution. Nevertheless, let us make the computation:
\begin{multline*}
  \lim_{\mathfrak m\to 0} W_{1,\infty,\mathfrak m}(z/\mathfrak m,q;  \A, \B; \C, (0,\dots,1,0,\dots,0))
 \\ =  \frac{z}{1-z} q^{|\A|-|\C|} \left[
  - q^{\sum_{j>u}  C_j} \cdot {C_u+1 \choose C_u}_q
 + q \cdot  q^{\sum_{j>u}  C_j -1} \cdot  {C_u+1\choose C_u}_q\right]=0 .
\end{multline*}

We proceed to the $L>1$ case,\footnote{An alternative way to perform the following computation is by using the symmetry of the weight $W$ under interchange $\A\leftrightarrow \B$, $\C\leftrightarrow \D$ with simultaneous adjustment of the parameters of the distribution, see \cite[Proposition C.1.3]{BW}. This way avoids the hypergeometric identities that we need to use in the present approach.} We start by taking $\D=(0,\dots,0,d)$. In this case we have
\begin{multline} \label{eq_x30}
  \lim_{\mathfrak m\to 0} W_{L,\infty,\mathfrak m}(z/\mathfrak m,q;  \A, \B; \C, (0,\dots,0,d))
 \\ = \frac{z^{d} q^{(d-L)L}   (q^{-L};q)_{d}}{(q^{-L};q)_{L} (q;q)_d}
  \sum_{p=0}^L(-1)^{L-p} q^{\frac{(L-p)(L-p-1)}{2}} q^{pL}  \frac{(q^{L} z;q)_{|\C|-p} (q^{C_N-p+1};q)_d (q^{L-p+1};q)_p }{(z;q)_{|\C|+d-p}(q;q)_p}.
\end{multline}
We would like to use hypergeometric identities and for that we transform the sum into more standard form:
\begin{multline}\label{eq_x29}
  \sum_{p=0}^L(-1)^{L-p} q^{\frac{(L-p)(L-p-1)}{2}} q^{pL}  \frac{(q^{L} z;q)_{|\C|-p} (q^{C_N-p+1};q)_d (q^{L-p+1};q)_p }{(z;q)_{|\C|+d-p}(q;q)_p}
  \\
  =
  \frac{(q^{L} z;q)_{|\C|} (q^{C_N+1};q)_{d}}{ (z;q)_{|\C|+d}}  \sum_{p=0}^L(-1)^{L-p} q^{\frac{L(L-1)}{2} +\frac{p(p+1)}{2}}   \frac{ (q^{C_N-p+1};q)_{p}  (q^{L-p+1};q)_p (zq^{|\C|+d-p};q)_{p} }{(q^{C_N-p+1+d};q)_{p} ( q^{L+|\C|-p} z;q)_{p} (q;q)_p}
\\   =
  \frac{(q^{L} z;q)_{|\C|} (q^{C_N+1};q)_{d}}{ (z;q)_{|\C|+d}} (-1)^{L} q^{\frac{L(L-1)}{2}}  \sum_{p=0}^L   q^p \frac{ (q^{-C_N};q)_{p}  (q^{-L};q)_p (z^{-1}q^{1-|\C|-d};q)_{p} }{(q^{-C_N-d};q)_{p} ( z^{-1}q^{1-L-|\C|} ;q)_{p} (q;q)_p},
\end{multline}
where in the last identity we used
\begin{multline*}
 (x q^{-p};q)_p=(1-x q^{-p})(1-x q^{1-p})\cdots (1-x q^{-1})\\ = (-1)^p x^{p} q^{-1-2-\dots-p} (1-q/x) (1-q^2/x)\cdots (1- q^p/x)= (-1)^p x^p q^{-p(p+1)/2} (q/x;q)_p.
\end{multline*}
For the last sum in \eqref{eq_x29} we use the transformation of $_3\phi_2$ given in \cite[(III.11)]{GR}:
$$
 _3\phi_2\left(q^{-n},b,c;d,e;q,q\right)=\frac{ (de/bc;q)_n}{(e;q)_n} \left(\frac{bc}{d}\right)^n \,_3\phi_2\left(q^{-n},\frac{d}{b},\frac{d}{c}; d, \frac{de}{bc};q,q\right)
$$
with $n:=L$, $b:=q^{-C_N}$, $c:=z^{-1}q^{1-|\C|-d}$, $d:=q^{-C_N-d}$, $e:=z^{-1}q^{1-L-|\C|}$. Hence, \eqref{eq_x29} becomes
\begin{multline*}
  (-1)^{L} q^{\frac{L(L-1)}{2}}  \left( z^{-1}q^{1-|\C|}\right)^L  \frac{(q^{L} z;q)_{|\C|} (q^{\C_N+1};q)_{d}(q^{-L};q)_L }{ (z;q)_{|\C|+d} (z^{-1}q^{1-L-|\C|};q)_L}  \\ \times \,_3\phi_2\left(q^{-L},q^{-d},z q^{-C_N-1+|\C|}; q^{-C_N-d}, q^{-L};q,q\right).
\end{multline*}
Since $q^{-L}$ appears twice in the parameters of the last $_3\phi_2$, we can remove it and replace $_3\phi_2$ with $_2\phi_1$, at which point we can use the $q$-Chu-Vandermonde identity \cite[(1.5.3)]{GR}
$$
 _2\phi_1(a,q^{-d};c;q;q)=\frac{(c/a;q)_d}{(c;q)_d} a^d.
$$
Therefore, \eqref{eq_x29} transforms into
\begin{multline}
\label{eq_x31}
  (-1)^{L} q^{\frac{L(L-1)}{2}}  \left( z^{-1}q^{1-|\C|}\right)^L  \\ \times \frac{(q^{L} z;q)_{|\C|} (q^{C_N+1};q)_{d}(q^{-L};q)_L }{ (z;q)_{|\C|+d} (z^{-1}q^{1-L-|\C|};q)_L} \frac{(z^{-1}q^{-d +1-|\C|};q)_d}{(q^{-C_N-d};q)_d} \left(z q^{-C_N-1+|\C|}\right)^d,
\end{multline}
and the limit in \eqref{eq_x30} is the last expression multiplied by
\begin{equation}
\label{eq_x32}
 \frac{z^{d} q^{(d-L)L}   (q^{-L};q)_{d}}{(q^{-L};q)_{L} (q;q)_d}.
\end{equation}
Simplifying the product of \eqref{eq_x31} and \eqref{eq_x32} we arrive at
$$
 \frac{(q^{-L};q)_d}{(z;q)_L (q;q)_d} (zq^L)^d,
$$
which matches \eqref{eq_limit_first_column}.

\smallskip

It now remains to show that only $\D$ of the form $(0,\dots,0,d)$ give non-zero limit in \eqref{eq_limit_first_column}. There should be a direct formulaic way to prove it, but instead we use a shortcut. We notice that for the $L=1$ case we proved  that only $\D=(0,\dots,0)$ and $\D=(0,\dots,0,1)$ lead to non-vanishing limits. The vertex weights for general $L$ are obtained from the $L=1$ weights by collapsing $L$ rows in the fusion procedure (cf.\ Theorem \ref{Theorem_fusion} and Section \ref{Section_fusion}) . Clearly, if no paths of colors $<N$ exit to the right before fusion, then this is still true for the fused rows, and this completes the proof.
\end{proof}

For the vertices in columns $1,2,\dots$ we do not need any limit transition. Instead we will simply set $z=1$ in \eqref{eq_higher_spin_weight_continued_1}. Then the two last arguments in the second $\Phi(\cdot)$ coincide, which leads to $\mu=\lambda$ in \eqref{eq_Phi} and then $\Phi(\cdot)=1$.\footnote{In this argument it is important to restrict the values for $\mathfrak m$ so that the first $\Phi(\cdot)$ factor in \eqref{eq_higher_spin_weight_continued_3} would not explode because of the denominator $(y;q)_{|\mu|}$ in its definition. Since $y$ in the last $q$-Pochammer symbol needs to be set to $\mathfrak m$, the choice $0<\mathfrak m<1$ works well.} Hence, the weight simplifies. After noting that $\C-\B=\A-\D$ and $\C+\D-\B=\A$, we can write
\begin{multline}\label{eq_higher_spin_weight_continued_2}
 W_{L,\infty,\mathfrak m}(1,q;  \A, \B; \C, \D)= q^{|\A|L} \mathfrak m^{|\D|} \Phi(\A-\D, \A; \mathfrak m q^{L} , \mathfrak m )
 \\ =  \mathfrak m^{|\D|}  q^{L \cdot  |\D|} \frac{(\mathfrak m q^{L} ;q)_{|\A-\D|} (  q^{-L};q)_{|\D|}}{( \mathfrak m ;q)_{|\A|}} q^{\sum_{i<j} D_i (A_j-D_j)} \prod_{i=1}^N {A_i\choose A_i-D_i}_q.
\end{multline}

The next step is to notice that both formulas \eqref{eq_limit_first_column} and \eqref{eq_higher_spin_weight_continued_2} are analytic (meromorphic) functions of the argument $q^{-L}$. Hence, we can replace $q^{-L}$ with a new complex number $\mathfrak l$. We are also going to ignore completely the paths on the vertical edges in column $0$, and only keep track of what is happening in the quadrant $\mathbb Z_{>0}\times \mathbb Z_{>0}$. The final description is as follows:

\begin{itemize}

\item We deal with configurations of colored paths on the edges joining  vertices in the integral quadrant $\{(x,y)\mid x,y\in\mathbb Z_{>0}\}$.
    %All edges are oriented in the direction of either growing $x$ or growing $y$.
    Each lattice edge has finite (but allowed to be arbitrarily large) number of paths of different colors. There might be more than one path of each color and the colors are from $\{1,2,\dots\}$.
\item The local configurations at the vertices are sampled sequentially in the direction of growing $x$ and $y$. The distribution depends on the parameters $0<q<1$, $\mathfrak l$, $z$, and a sequence $\mathfrak m_x$, $x=1,2,\dots$.

\item At the bottom boundary, no paths of positive colors enter into the quadrant (i.e., into the vertices $(x,1)$, $x\in\mathbb Z_{>0}$).

\item Along the left boundary, the only paths entering into $(y,1)$, $y\in\mathbb Z_{>0}$, have color $y$. The number of such paths  $\mathfrak d=0,1,2\dots$, is distributed according to the law
\begin{equation}
\label{eq_incoming_prob}
   \mathrm{Prob}(\mathfrak d = d)=\frac{(z/\mathfrak l;q)_\infty}{(z;q)_\infty} \cdot   \frac{(\mathfrak l;q)_d}{(q;q)_d} \left(\frac{z}{\mathfrak l}\right)^d.
\end{equation}
\item At each vertex $(x,y)$, the configuration of incoming  from below paths $\A$ and from the left $\B$ is transformed into the outgoing paths going up $\C$ and going to the right $\D$. We use the notation $\A=(A_1,A_2,\dots)$, where $A_i$ stands for the number of paths of color $i$. We also set $|\A|=A_1+A_2+\dots$, and similarly for the other three groups. The transformation is done by the stochastic rule with probability of outgoing configuration $\C, \D$ given by
\begin{multline}\label{eq_higher_spin_weight_continued_3}
 \mathcal{W}_{\mathfrak l,\mathfrak m_x,q}( \A, \B; \C, \D)\\ =  \left(\frac{\mathfrak m_x}{\mathfrak l}\right)^{|\D|} \frac{\bigl(\mathfrak m_x/\mathfrak l ;q\bigr)_{|\A-\D|} ( \mathfrak l;q)_{|\D|}}{( \mathfrak m_x ;q)_{|\A|}} q^{\sum_{i<j} D_i (A_j-D_j)} \prod_{i=1}^\infty {A_i\choose A_i-D_i}_q,
\end{multline}
subject to the conservation of colors condition $\A+\B=\C+\D$.
\item The parameters $\mathfrak l$, $z$, $\mathfrak m_x$, $x=1,2,\dots$, are assumed to be chosen so that \eqref{eq_incoming_prob} and \eqref{eq_higher_spin_weight_continued_3} are non-negative and \eqref{eq_incoming_prob} is summable over $d=0,1,\dots$. For instance, this is the case when  $0<z,\mathfrak l,\mathfrak m_x, z/\mathfrak l, \mathfrak m_x/\mathfrak l<1$ for all $x=1,2,\dots$, but other choices are also possible.
\end{itemize}

One way to think about \eqref{eq_higher_spin_weight_continued_3} is that all paths entering into a vertex from the left deterministically turn up. On the other hand, the paths entering from below might either turn to the right or proceed straight up. In particular, \eqref{eq_higher_spin_weight_continued_3} implies that if $\A$ is empty, then so is $\D$. Hence, for our boundary conditions, no paths of positive colors lie below the diagonal $x=y$, so that the only non-trivial vertices in the system are $(x,y)$ with $x\le y$.

\begin{theorem} \label{Theorem_6v_invariance_infinite}
 In the above setting of the analytically continued fused colored model in quadrant with vertex weights \eqref{eq_higher_spin_weight_continued_3} and incoming probabilities \eqref{eq_incoming_prob}, choose a set of positive integers $k_j$ and a set of  points in the quadrant $\mathcal U_j$, $j=1,\dots,n$. Fix an index $1\le \iota\le n$ and an integer $\Delta>0$. Define
 $$
  k_j'=\begin{cases} k_j,& j\ne \iota,\\ k_\iota+\Delta, & j=\iota,\end{cases} \qquad \qquad \mathcal U'_j=\begin{cases}\mathcal U_j, & j\ne \iota, \\ \mathcal U_\iota + (0,\Delta), & j=\iota. \end{cases}
 $$
Suppose that
 $$
  0\le k_1\le k_2\le \dots\le k_n, \qquad 0\le k'_1\le k'_2\le \dots\le k'_n,
 $$
 $$\mathcal U_1,\dots,\mathcal U_{\iota-1} \succeq \mathcal U_\iota \succeq\mathcal U_{\iota+1},\dots, \mathcal U_n, \qquad \mathcal U'_1,\dots,\mathcal U'_{\iota-1} \succeq \mathcal U'_\iota \succeq \mathcal U'_{\iota+1},\dots, \mathcal U'_n.
 $$
 Then the distribution of the vector of the height functions
 \begin{equation}
 \label{eq_x33}
  \bigl(\H^{\ges k_1}(\mathcal U_1),\, \H^{\ges k_2}(\mathcal U_2),\, \dots, \H^{\ges k_n}(\mathcal U_n)\bigr)
 \end{equation}
 coincides with the distribution of the vector with shifted $\iota$-th coordinate
 \begin{equation}
 \label{eq_x34}
  \bigl(\H^{\ges k'_1}(\mathcal U'_1),\, \H^{\ges k'_2}(\mathcal U'_2),\, \dots, \H^{\ges k'_n}(\mathcal U'_n)\bigr).
 \end{equation}
\end{theorem}
\begin{remark} \label{Remark_vertical_inhomogeneities}
  It is natural to expect that an addition of vertical inhomogenities, i.e., replacement of $\mathfrak l$ by $\mathfrak l_y$, should preserve the statement of Theorem \ref{Theorem_6v_invariance_infinite} with proper interchange of parameters as in Theorem \ref{Theorem_6v_intro}. However, our proofs do not allow to claim this generalization.
\end{remark}
\begin{proof}[Proof of Theorem \ref{Theorem_6v_invariance_infinite}] Introduce a complex variable $u:=\mathfrak l^{-1}$. The distributions of both vectors \eqref{eq_x33} and \eqref{eq_x34} (i.e., the probabilities that they attain certain values) are functions of $u$, which are holomorphic in a complex neighborhood of $u=0$. Indeed, while the positivity conditions fail for complex $u$, yet the formulas for the distribution of the number of incoming paths \eqref{eq_incoming_prob} are still converging uniformly in $u$, while the formulas for the vertex weights \eqref{eq_higher_spin_weight_continued_3} only involve finite sums. Hence, the distributions of the vectors \eqref{eq_x33}, \eqref{eq_x34} are represented as converging sums involving products of the expressions \eqref{eq_x33} and \eqref{eq_x34}. For small $u\ne 0$ the holomorphicity is clear from the definitions. On the other hand, as $u\to 0$ (equivalently, as $\mathfrak l\to\infty$) the formulas \eqref{eq_x33}, \eqref{eq_x34} remain bounded and, in fact, converge to a finite limit. Hence, the singularity at $u=0$ is removable, and the functions of interest are holomorphic there.

For $u=q^{L}$, $L=1,2,\dots$, the coincidence of the two holomorphic functions decribing the probability of attaining some value by \eqref{eq_x33} and \eqref{eq_x34} is the content of Corollary \ref{Corollary_higher_spin_inv}. Since $\{q^{L}\}_{L=1,2,\dots}$ has a limiting point $0$ and is, therefore, a uniqueness set for holomorphic (in $u$) functions, we conclude that the coincidence of the distributions of vectors \eqref{eq_x33} and \eqref{eq_x34} extends to all complex values of $u$, such that the sum (over $d$) in \eqref{eq_incoming_prob} remains uniformly convergent.
\end{proof}

\subsection{Auxiliary $q$-identities and limit transitions.}

This section collects some results on $q$-Pochhammer symbols and $q$-Binomial coefficients, which will be useful in our asymptotic analysis later on.

\begin{lemma}
\label{lemma_q_poch_conv}
 For any $a,b\in\mathbb R$ and complex--valued function $u(\cdot)$ defined in a neighborhood of $1$ and
  such that
$$
 \lim_{q\to 1} u(q)=u
$$
with $0<u<1$, we have
 $$
  \lim_{q\to 1} \frac{(q^a u(q);q)_{\infty}}{(q^b u(q);q)_{\infty}} = (1-u)^{b-a}.
 $$
\end{lemma}
\begin{proof}
 See \cite[Theorem 10.2.4]{AAR}.
\end{proof}

We define the $q$-Gamma function through
$$
\Gamma_q(x)=\frac{(q;q)_\infty}{(q^x;q)_\infty} (1-q)^{1-x}, \quad x\ne 0,-1,-2,\dots.
$$
\begin{lemma}
\label{lemma_q_Gamma_conv}
 For any $x\in\mathbb C\setminus\{0,-1,-2,\dots\}$ we have
$$
 \lim_{q\to 1}  \Gamma_q(x)= \Gamma(x).
$$
\end{lemma}
\begin{proof}
 See \cite[Corollary 10.3.4]{AAR}.
\end{proof}

Recall the definition of dilogarithm:
$$
 Li_2(x)=\sum_{n=1}^{\infty} \frac{x^n}{n^2}.
$$
We have
$$
 \frac{\partial}{\partial x} Li_2(x)= -\frac{\ln(1-x)}{x}.
$$

\begin{lemma}  \label{Lemma_Poch_asymptotics}  For any\, $0<a<1$ as $q\to 1$ we have the following asymptotic expansion valid for each $K=1,2,\dots$:
\begin{equation}
 \ln (a;q)_\infty=\frac{1}{\ln q} Li_2(a)+\frac{1}{2}\ln(1-a) +\sum_{k=1}^{K} C_k(a) (\ln q)^{2k-1}+ o\left(\ln(q)^{2K-1}\right),
\end{equation}
where $C_k(a)$ is a polynomial in $a$ multiplied by $(a-1)^{1-2k}$, and the remainder $o(\cdot)$ is uniform as long as $a$ is bounded away from $1$.
\end{lemma}
\begin{proof}
 See \cite[Corollary 10]{K}.
\end{proof}
\begin{lemma} \label{Lemma_q_Binomial_split} For non-negative integers $M$, $N$, $K$, we have
\begin{equation}
 \sum_{k=0}^K q^{k(M-K+k)} {N \choose k}_q {M\choose K-k}_q=
 \sum_{k=0}^K q^{(N-k)(K-k)} {N \choose k}_q {M\choose K-k}_q= {N+M \choose K}_q.
\end{equation}
\end{lemma}
\begin{proof}
 The $q$-Binomial theorem reads
 \begin{equation}
 \label{eq_x35}
  \prod_{i=1}^N (1+q^{i-1} t) = \sum_{k=0}^N q^{k(k-1)/2} {N \choose k}_q t^k.
 \end{equation}
 Changing the variables and replacing $N$ by $M$, we also have
 \begin{equation}
 \label{eq_x36}
  \prod_{i=N+1}^{N+M} (1+q^{i-1} t) = \sum_{k'=0}^M q^{k'(k'-1)/2+N k'} {M \choose k'}_q t^{k'}.
  \end{equation}
  Multiplying \eqref{eq_x35} and \eqref{eq_x36} and comparing the coefficient of $t^K$ with the one arising from
   \begin{equation}
 \label{eq_x37}
  \prod_{i=1}^{N+M} (1+q^{i-1} t) = \sum_{K=0}^{N+M} q^{K(K-1)/2} {N+M \choose K}_q t^K,
 \end{equation}
we obtain the result.
\end{proof}

\subsection{A continuous limit.}

\label{Section_cont_limit}

Our next aim is to send $q\to 1$ in the vertex model of Theorem \ref{Theorem_6v_invariance_infinite}, rescaling the system so that in the limit we obtain a vertex model with continuous (rather than integral) number of paths on each lattice edge. For the colorless ($N=1$) model, such an asymptotic transition was previously performed in \cite{Bar-Cor}.

The phenomenon that we will observe is that asymptotically the $N$--dimensional distribution \eqref{eq_higher_spin_weight_continued_3} (where $N$ is the number of colors) in the leading order is concentrated on a one-dimensional subspace, and its law is related to the classical Beta distribution. On the other hand, the second order $(N-1)$--dimensional fluctuations are Gaussian. There is a certain ambiguity in formulating a precise mathematical statement, because there is no canonical choice of the aforementioned one-dimensional subspace. But as soon as we fix one, we can no longer observe the smaller order Gaussian component in the direction parallel to the subspace (the weak convergence of the random variables does not distinguish between a random variable and a sum of the same random variable and an asymptotically vanishing Gaussian correction) --- we can only track the orthogonal directions.

The choice that we make is to look at the total number of paths of all colors (equivalently, on the evolution of the underlying \emph{colorless} model), for which the Beta distribution appears asymptotically. We further observe the lower order Gaussian fluctuations as asymptotic corrections to the (deterministic in the first order) splitting of the paths between the colors.

\smallskip

In more detail, we take a small parameter $\eps>0$ and concentrate on the following limit regime:
\begin{equation} \label{eq_limit_regime}
 \eps\to 0,\quad  q=\exp(-\eps),\quad \mathfrak l=q^\rho, \quad z= q^{\sigma_0} ,\quad \mathfrak m_x=q^{\sigma_x},
\end{equation}
where $0<\rho<\sigma_x$, $x=0,1,2,\dots$, are arbitrary.

The first result of this section describes the limit of the boundary condition. Recall that the Beta distribution $B(a,b)$ with parameters $a>0$, $b>0$ is absolutely continuous with respect to the Lebesgue measure on $(0,1)$ with density
$$
 \frac{\Gamma(a) \Gamma(b)}{\Gamma(a+b)} \, x^{a-1} (1-x)^{b-1},\quad 0<x<1.
$$

\begin{proposition} Under \eqref{eq_limit_regime}, the random variable $\exp(-\eps \mathfrak d)$, with \eqref{eq_incoming_prob}--distributed $\mathfrak d$, weakly converges as $\eps\to 0$ to the Beta random variable with parameters $(\sigma_0-\rho,\rho)$.
\end{proposition}
\begin{proof}
For $d=x/\eps$, the right-hand side of \eqref{eq_incoming_prob} transforms into
\begin{multline}
\frac{(q^{\sigma_0-\rho};q)_\infty}{(q^{\sigma_0};q)_\infty} \cdot   \frac{(q^\rho;q)_{x/\eps}}{(q;q)_{x/\eps}} \exp(x(\rho-\sigma_0))\\ =
\frac{(q^{\sigma_0-\rho};q)_\infty (q^\rho;q)_{\infty}}{(q^{\sigma_0};q)_\infty (q;q)_\infty} \cdot   \frac{(q^{1+x/\eps};q)_{\infty} }{(q^{\rho+x/\eps};q)_{\infty}} \exp(x(\rho-\sigma_0)).
\end{multline}
Using Lemma \ref{lemma_q_Gamma_conv} for the first fraction (three times) and Lemma \ref{lemma_q_poch_conv} for the second fraction, we see that the asymptotic behavior of the last expression is
\begin{equation}
 \eps \cdot  \frac{\Gamma(\sigma_0)}{\Gamma(\rho)\Gamma(\sigma_0-\rho)} (1-\exp(-x))^{\rho-1} \exp(-x)^{\sigma_0-\rho}.
\end{equation}
Hence, this expression (with $\eps$--factor removed) gives the asymptotic density of $\eps \mathfrak d$. Therefore, the asymptotic density of $\exp(-\eps\mathfrak d)$ at a point $0<y<1$ is
$$
\frac{\Gamma(\sigma_0)}{\Gamma(\rho)\Gamma(\sigma_0-\rho)} (1-y)^{\rho-1} y^{\sigma_0-\rho-1}. \qedhere
$$
\end{proof}
We proceed to analysis of the distribution at a general vertex \eqref{eq_higher_spin_weight_continued_3}. We start by transforming it into a more convenient form.

\begin{lemma}
  If we fix $\A$, then the distribution of $\D$ given by \eqref{eq_higher_spin_weight_continued_3} can be sampled through a two-step procedure. We first sample an integer $|\D|$, $0\le |\D|\le |\A|$, from the distribution
 \begin{equation}
 \label{eq_x38}
  \left(\frac{\mathfrak m_x}{\mathfrak l}\right)^{|\D|} \, \frac{(\mathfrak m_x/\mathfrak l ;q)_{|\A|-|\D|} ( \mathfrak l;q)_{|\D|}}{( \mathfrak m_x ;q)_{|\A|}} {|\A|\choose |\D|}_q.
 \end{equation}
 Given the vector $\A$ and the number $|\D|$ we further sample the vector $\D$ with prescribed sum of the coordinates $|\D|$ from the distribution
 \begin{equation}
 \label{eq_x39}
{|\A|\choose |\D|}_q^{-1}\, q^{\sum_{i<j} D_i (A_j-D_j)} \prod_{i=1}^\infty {A_i\choose D_i}_q.
 \end{equation}
\end{lemma}
\begin{proof}
 We start from \eqref{eq_higher_spin_weight_continued_3} and sum over all non-negative integral  vectors $\D$ with prescribed sum of the coordinates $|\D|$ using Lemma \ref{Lemma_q_Binomial_split} to get \eqref{eq_x38}. Dividing \eqref{eq_higher_spin_weight_continued_3} by \eqref{eq_x38}, we arrive at \eqref{eq_x39}.
\end{proof}

The limit of the distribution of $|\D|$ can now be computed, cf.\ \cite[Lemma 2.4]{Bar-Cor}.

\begin{proposition} \label{Proposition_colorless_limit}
 Suppose that as $\eps\to 0$, the number $|\A|$ scales so that $|\A|=\alpha \eps^{-1}$ for some real $\alpha$ staying in a compact subset of $(0,\infty)$. Then for $|\D|$ distributed according to \eqref{eq_x38}, in the regime \eqref{eq_limit_regime}, the random variable $\eps |\D|$  weakly converges as $\eps\to 0$ to a continuous random variable $\delta$, such that
 $$
   \exp(-\delta)=\exp(-\alpha)+ B(\sigma_x-\rho,\rho) (1-\exp(-\alpha)),
 $$
 where $B(\sigma_x-\rho,\rho)$ is a Beta random variable.
\end{proposition}
\begin{proof} For $0<\delta<\alpha$, setting $|\D|=\eps^{-1}\delta$ and bringing all the $q$-Pochhammers in \eqref{eq_x38} to the form involving only $(u;q)_\infty$, we get
\begin{multline*}
   \exp \left( (\rho-\sigma_x) \delta\right) \frac{\left( q^{\sigma_x-\rho} ;q\right)_{\infty} ( q^{\rho} ;q)_{\infty} }{  (q;q)_{\infty} ( q^{\sigma_x} ;q)_{\infty}} \\ \times
   \frac{( q^{\sigma_x} \exp(-\alpha);q)_{\infty}} { (q \exp(-\alpha);q)_{\infty}} \cdot \frac{(q\exp(\delta-\alpha);q)_{\infty}}{ \left( q^{\sigma_x-\rho}
   \exp(\delta-\alpha) ;q\right)_{\infty}} \cdot  \frac{ (q \exp(-\delta);q)_{\infty}  }{ ( q^{\rho} \exp(-\delta);q)_{\infty}  }.
\end{multline*}
 We use Lemma \ref{lemma_q_Gamma_conv} for the first fraction (three times) and Lemma \ref{lemma_q_poch_conv} for the next three fractions. The resulting asymptotic behavior is
\begin{multline}
     \eps \cdot  \frac{\Gamma(\sigma_x)}{\Gamma(\rho)\Gamma(\sigma_x-\rho)} \exp \left( (\rho-\sigma_x) \delta\right)  \\ \times
    (1-\exp(-\alpha))^{1-\sigma_x} \cdot (1-\exp(\delta-\alpha))^{\sigma_x-\rho-1}  \cdot (1-\exp(-\delta))^{\rho-1}.
\end{multline}
Dividing by $\eps$ we obtain the asymptotic density of $|\delta|$.
Making the change of variables, we find that the density of $\frac{\exp(-\delta)-\exp(-\alpha)}{1-\exp(-\alpha)}$ at a point $0<y<1$ is
\begin{equation}
       \frac{\Gamma(\sigma_x)}{\Gamma(\rho)\Gamma(\sigma_x-\rho)}
     \cdot    y^{\sigma_x-\rho-1} \left(1-y\right)^{\rho-1} . \qedhere
\end{equation}
\end{proof}
The next step is to study asymptotics of the conditional distribution \eqref{eq_x39}.

Suppose that we are given $N$ positive real numbers $\alpha_1,\dots,\alpha_N$ and an additional number $\delta$ such that $0<\delta<|\alpha|=\alpha_1+\dots+\alpha_N$. Then we define  $N$--dimensional vector $(\delta_1,\dots,\delta_N)$ through $\delta_1+\delta_2+\dots+\delta_N=|\delta|=\delta$ and
  \begin{equation}\label{eq_x47}
     \exp\left(- \sum_{j=i}^N \delta_{j}\right)= \exp\left(-\sum_{j=i}^N \alpha_{j}\right) + \left(1-\exp\left(-\sum_{j=i}^N \alpha_{j}\right)\right) \eta, \quad i=1,2,\dots,N,
  \end{equation}
  where $\eta$ is the constant found from the $i=1$ condition. In addition, we define real numbers $v_i$ by
  \begin{equation} \label{eq_x48}
    \frac{1}{v_i}= 1+ \frac{\partial^2}{\partial x^2} \biggl[ Li_2(\exp(- x))+Li_2(\exp(x-\alpha_i)) \biggr]_{x=\delta_i}, \quad i=1,\dots,N.
  \end{equation}
 All $v_i$, $i=1,\dots,N$, are positive, as follows from the inequality
 \begin{equation}
 \label{eq_x52}
  \frac{\partial^2}{\partial x^2} \biggl[ Li_2(\exp(- x))+Li_2(\exp(x-\alpha_i)) \biggr]>-1, \quad 0<x<\alpha_i.
 \end{equation}
\begin{proposition} \label{Proposition_conditional_LLN}
  Suppose that as $\eps\to 0$, the numbers $A_i$, $i=1,\dots,N$, and $|\D|$ scale so that $A_i=\alpha_i \eps^{-1}$, $|\D|=\delta\eps^{-1}$ for some $\alpha_i$ staying in a compact subset of $(0,+\infty)$ and $\delta$ staying in a compact subset of $(0,\sum_{i} \alpha_i)$. Then the law of $D_i$, $i=1,\dots,N$, given by \eqref{eq_x39} satisfies the law of large numbers and central limit theorem as $\eps\to 0$: For $\delta_i$, $v_i$ given by \eqref{eq_x47}, \eqref{eq_x48}, the vector
  \begin{equation}
  \label{eq_x46}
   \left(\xi_i=\frac{D_i - \eps^{-1} \delta_i}{\eps^{-1/2}} \right)_{i=1}^N
  \end{equation}
  weakly converges to the (degenerate) $N$--dimensional centered Gaussian vector, which is supported on the hyperplane $\xi_1+\dots+\xi_N=0$ and whose density is proportional to the restriction to this hyperplane of the density of  the vector with independent Gaussian components of mean $0$ and variances $v_i$, $i=1,\dots,N$.
\end{proposition}
\begin{remark}
 The computation below leading to \eqref{eq_x47} is quite involved. Another way to see the result (which is how we originally arrived at this formula) is to make the computation recurrently, reducing it to the $N=2$ case. At each step $k=1,2,\dots,N-1$, one looks at the distribution of the pair $(D_k, D_{k+1}+D_{k+2}+\dots+D_N)$ conditional on $D_k+D_{k+1}+\dots+D_N$,  in order to determine  $\delta_{k+1}+\dots+\delta_N$.  (For instance, at the first step, $k=1$, we study the asymptotic of $(D_1, D_2+\dots+D_N)$ given the fixed sum of all coordinates, $|\D|=D_1+\dots+D_N$.) Using Lemma \ref{Lemma_q_Binomial_split}, one notices that this conditional distribution has the form of \eqref{eq_x46} with $N=2$ and two coordinates given by $D_k$ and $D_{k+1}+D_{k+2}+\dots+D_N$ (the sum of these coordinates is fixed by to the conditioning).
\end{remark}
\begin{proof}[Proof of Proposition \ref{Proposition_conditional_LLN}]
Using Lemma \ref{Lemma_Poch_asymptotics}, as a function of all $D_i$, subject to the condition of the fixed $|\D|$, \eqref{eq_x39} is proportional to
\begin{multline}
\label{eq_x40}
q^{\sum_{i<j} D_i (A_j-D_j)} \prod_{i=1}^\infty (q^{D_i};q)_\infty (q^{A_i-D_i};q)_\infty= \exp\Biggl( -\eps^{-1}  \sum_{i<j} (\eps D_i) (\eps A_j-\eps D_j)  \\ - \eps^{-1} \sum_i \biggl[ Li_2(\exp(-\eps D_i))+Li_2(\exp(\eps D_i-\eps A_i)) \biggr]\\  + \frac{1}{2} \ln(1-\exp(-\eps D_i)) +\frac12  \ln(1-\exp(-\eps (A_i-D_i))) +O(\eps) \Biggr).
\end{multline}
We are interested in $\eps^{-1}$--part of this expression as a function of $\eps D_1,\dots,\eps D_N$ subject to the condition of the fixed sum of the coordinates. We are going to show that this function is strictly concave and has a critical point inside the domain $0<\eps D_i <\eps A_i$, $i=1,\dots,N$, $\sum_{i=1}^N (\eps D_i)=\eps |\D|$. Hence, this critical point is a (unique) maximum of the function inside the domain and, due to $\eps^{-1}$ prefactor, the measure is concentrated near this point, which leads to the law of large numbers. For the central limit theorem, we Taylor expand the density near the critical point up to second order and the quadratic approximation gives the desired Gaussian limit.

We remark that in order for the asymptotic expansion \eqref{eq_x40} to be valid, we need $q^{D_i}$ and $q^{A_i-D_i}$ to be bounded away from $1$, which means that $\eps D_i$ and $\eps(A_i-D_i)$ should be bounded away from $0$. For boundary values, i.e., for small $\eps D_i$ or $\eps(A_i-D_i)$, the $\eps^{-1}$ part in the exponential of \eqref{eq_x40}
$$
 \exp\Biggl( -\eps^{-1}  \sum_{i<j} (\eps D_i) (\eps A_j-\eps D_j)  \\ - \eps^{-1} \sum_i \biggl[ Li_2(\exp(-\eps D_i))+Li_2(\exp(\eps D_i-\eps A_i)) \biggr]  + O(1)  \Biggr)
$$
becomes an upper bound for the probability \eqref{eq_x39}, as follows from the unimodality of $q$-Binomial coefficients ${n\choose k}_q$ as a function of $k$. This upper bound implies that the boundary values of $D_i$ give negligible contribution in the asymptotics and we can (and will) ignore them.

We proceed with more details of the argument. For the law of large numbers, i.e., the convergence of $\eps^{-1/2} \xi_i$ from \eqref{eq_x46} to zero, we assume that $\A$ and $\D$ are large by setting $A_i=\eps^{-1} \alpha_i$ and $D_i=\eps^{-1} \delta_i$. We further seek the maximum of the $\eps^{-1}(\dots)$ part of the exponent in \eqref{eq_x40}, subject to the condition $\sum_i \delta_i =|\delta|=\delta$. Introducing a Lagrange multiplier $\lambda$, we need to maximize
$$
 -\sum_{i<j} \delta_i(\alpha_j-\delta_j) - \sum_i \left[Li_2(e^{-\delta_i})+Li_2(e^{\delta_i-\alpha_i})\right]+\lambda\left( -\delta+\sum_i \delta_i\right).
$$
Differentiating in $\delta_i$, we conclude that the following expression needs to be equal to $0$:
\begin{multline}
 \lambda-\sum_{j>i} (\alpha_j-\delta_j)+ \sum_{j<i} \delta_j  - \exp(-\delta_i) \frac{\ln(1- \exp(-\delta_i))}{\exp(-\delta_i)}+ \exp(\delta_i-\alpha_i)\frac{\ln(1-\exp(\delta_i-\alpha_i))}{\exp(\delta_i-\alpha_i)}
\\= \lambda+\delta-\left(\sum_{j>i} \alpha_j \right)-\delta_i  - \ln(1- \exp(-\delta_i))+ \ln(1-\exp(\delta_i-\alpha_i))
\\= \lambda+\delta-\left(\sum_{j>i} \alpha_j \right) + \ln\left(\frac{\exp(-\delta_i)-\exp(-\alpha_i)}{1- \exp(-\delta_i)}\right).
\end{multline}
Thus, the extremum is found by solving the equations
\begin{equation}
\label{eq_x45}
 \frac{1- \exp(-\delta_i)}{\exp(-\delta_i)-\exp(-\alpha_i)}=\exp\left(\lambda+\delta-\sum_{j>i} \alpha_j \right),\quad i=1,2,\dots.
\end{equation}
Dividing $i$th and $(i-1)$st equations by each other, we get rid of $\lambda$ and get a system of $N-1$ equations on $\exp(-\delta_1)$, \dots, $\exp(-\delta_N)$, which is supplemented by $\delta_1+\dots+\delta_N=\delta$. We claim that the solution (its uniqueness follows from the strict concavity that we prove below) is given by the following equivalent form of \eqref{eq_x47}:
\begin{equation}
\label{eq_x43}
 \exp(-\delta_{\ges i})= \exp(-\alpha_{\ges i}) + (1-\exp(-\alpha_{\ges i})) \eta, \quad i=1,2,\dots
\end{equation}
where $\delta_{\ges i}=\delta_i + \delta_{i+1}\dots$; $\alpha_{\ges i}=\alpha_i+\alpha_{i+1}+\dots$; and $\eta$ does not depend on $i$, i.e., it is found from the $i=1$  case of \eqref{eq_x43}. Indeed, \eqref{eq_x43} gives
\begin{equation}
\label{eq_x44}
 \exp(-\delta_{i})= \frac{\exp(-\alpha_{\ges i}) + (1-\exp(-\alpha_{\ges i})) \eta} {\exp(-\alpha_{\ges i+1}) + (1-\exp(-\alpha_{\ges i+1})) \eta}.
\end{equation}
This implies
\begin{multline}
  \frac{1- \exp(-\delta_i)}{\exp(-\delta_i)-\exp(-\alpha_i)}=
  \dfrac{1- \frac{\exp(-\alpha_{\ges i}) + (1-\exp(-\alpha_{\ges i})) \eta} {\exp(-\alpha_{\ges i+1}) + (1-\exp(-\alpha_{\ges i+1})) \eta}}{ \frac{\exp(-\alpha_{\ges i}) + (1-\exp(-\alpha_{\ges i})) \eta} {\exp(-\alpha_{\ges i+1}) + (1-\exp(-\alpha_{\ges i+1})) \eta}-\exp(-\alpha_i)}
  \\=
  \frac{(1-\eta)\exp(-\alpha_{\ges i+1})(1-\exp(-\alpha_i))  }
  {  \eta(1-\exp(-\alpha_i))}=\frac{1-\eta}{\eta}  \exp(-\alpha_{\ges i+1}),
\end{multline}
which matches \eqref{eq_x45} with $\lambda=-\delta+\ln\left(\frac{1-\eta}{\eta}\right)$.

\medskip

For the strict concavity, as well as for the central limit theorem, we need to compute the Hessian matrix of the $\eps^{-1}(\cdot)$ part of \eqref{eq_x40}. We rewrite this part in the coordinates $\delta_i=\eps D_i$ as
\begin{equation}
\label{eq_x50}
-  \sum_{i<j} \delta_i \alpha_j-\frac{1}{2}\left(\sum_{i=1}^N \delta_i\right)^2 + \sum_{i=1}^N \frac{\delta_i^2}{2} -  \sum_{i=1}^N \biggl[ Li_2(\exp(-\delta_i))+Li_2(\exp(\delta_i-\alpha_i)) \biggr].
\end{equation}
The first term in \eqref{eq_x50} is linear and does not contribute to the Hessian. The second term is constant due to the constraint. We conclude that the Hessian matrix is the diagonal matrix of second derivatives of the last two terms restricted to the hyperplane $\sum_{i=1}^N \delta_i=\delta$. This diagonal matrix is negative-definite, as follows from the inequality \eqref{eq_x52}. We conclude that its restriction on the hyperplane is also negative--definite. This implies the desired strict concavity.

For the central limit theorem, we need to Taylor expand \eqref{eq_x40} up to second order in a neighborhood of $\eps \D_i\approx \delta_i$. The zeroth order cancels with the normalization constant of the probability measure, the first order term vanishes, as we are expanding near the maximizer. The quadratic form appearing in the second order terms gives the desired Gaussian approximation. Our Hessian computation implies that in the $\xi_i$--variables the quadratic form is given by:
$$
 \exp\Biggl( \frac{1}{2}(\xi_1+\dots+\xi_N)^2  - \sum_i \frac{\xi_i^2}{2} \left( \frac{\partial^2}{\partial x^2} \biggl[ Li_2(\exp(- x))+Li_2(\exp(x-\alpha_i)) \biggr]_{x=\delta_i}+1\right)\Biggr).
$$
 Since $\xi_1+\dots+\xi_N=0$ deterministically, the last formula matches the description in the statement of the theorem.
\end{proof}

Let us summarize the continuous model which we obtained as $\eps\to 0$ limit of the vertex model. At this point we ignore the smaller order Gaussian component --- it can be readily reconstructed by utilizing Proposition \ref{Proposition_conditional_LLN}.

We first combine the results of Propositions \ref{Proposition_colorless_limit} and \ref{Proposition_conditional_LLN}:
\begin{corollary} \label{Corollary_continuous_vertex} The vertex weight $\mathcal{W}_{\mathfrak l,\mathfrak m,q}( \A, \B; \C, \D)
$ of \eqref{eq_higher_spin_weight_continued_3} treated as a stochastic sampling rule for $(\C,\D)$ given $(\A,\B)$, weakly converges in the limit regime
$$
 \eps\to 0,\quad  q=\exp(-\eps),\quad \mathfrak l=q^\rho ,\quad \mathfrak m=q^{\sigma},$$
  $$\eps \A\to \alpha,\quad \eps \B\to\beta, \quad \eps\C\to\gamma,\quad \eps\D\to \delta, \qquad |\alpha|>0,
$$
to the following sampling procedure for the outgoing masses of colors $\gamma_i$, $\delta_i$, $i=1,2,\dots$, given the incoming masses $\alpha_i$, $\beta_i$, $i=1,2,\dots$ entering into the vertex from the left and from the right. We take a Beta random variable $\eta\sim B(\sigma-\rho,\rho)$ and define $\delta_i$,  $0 < \delta_i < \alpha_i$ for all $1\le i \le \max(i^*: \alpha_{i^*}>0)$ through
  $$
     \exp\left(- \delta_{\ges i} \right)= \exp\left(-\alpha_{\ges i}\right) + \left(1-\exp\left(- \alpha_{\ges i}\right)\right) \eta, \quad i=1,2,\dots,
  $$
  where $\delta_{\ges i}=\sum_{j=i}^\infty \delta_{j}$, $\alpha_{\ges i}=\sum_{j=i}^\infty \alpha_{j}$. We also set
  $$
    \gamma_i=\alpha_i+\beta_i-\delta_i, \qquad i=1,2,\dots.
  $$
\end{corollary}

Hence, the resulting continuous model has the following self-contained description.

\begin{itemize}
 \item Each lattice edge of the positive quadrant has a random real-valued vector $(u_1,u_2,\dots)$ attached to it. Coordinate $u_i\ge 0$ is interpreted as the (real) \emph{mass} of color $i$.
 \item For edges entering the quadrant from below all masses are zero.
 \item For the edge entering the quadrant from the left at ordinate $y$ only the color $y$ is present. To find its mass, we sample (independently for each $y$) a Beta random variable $\eta\sim B(\sigma_0-\rho,\rho)$ and set the mass to be $-\ln(\eta)$.
  \item Given the incoming masses of colors $\alpha_i$, $\beta_i$, $i=1,2,\dots$ entering into the vertex $(x,y)$ from the bottom and from the left, we define the outgoing masses $\gamma_i$, $\delta_i$ by using  (independently over $x$ and $y$) a  Beta random variable $\eta\sim B(\sigma_x-\rho,\rho)$ and set $0 \le \delta_i \le \alpha_i$ through
  $$
     \exp\left(- \delta_{\ges i} \right)= \exp\left(-\alpha_{\ges i}\right) + \left(1-\exp\left(- \alpha_{\ges i}\right)\right) \eta, \quad i=1,2,\dots,
  $$
  where $\delta_{\ges i}=\sum_{j=i}^\infty \delta_{j}$, $\alpha_{\ges i}=\sum_{j=i}^\infty \alpha_{j}$. We also set
  $$
    \gamma_i=\alpha_i+\beta_i-\delta_i, \qquad i=1,2,\dots.
  $$
  \item The definition implies that only colors $\le y$ pass through a general vertex $(x,y)$, and no colors at all pass through $(x,y)$ with $x>y$.
\end{itemize}

We can also define the height functions $\H^{\ges i}(x,y)$, $i=1,2,\dots$ of the continuous model. As before, we assume that $x,y\in \frac{1}{2}+\mathbb Z_{\ge 0}$ in this definition. In words, $\H^{\ges i}(x,y)$ counts the total mass of paths of colors $\geqslant i$  below the point $(x,y)$. Equivalently, the height function is defined by setting $\H^{\ges i}(\frac12,\frac12)=0$ and
$$
\H^{\ges i}(x,y+1)-\H^{\ges i}(x,y)=\text{ total mass of  colors } \geqslant i\text{ at } (x,y+\tfrac12),
$$
$$
\H^{\ges i}(x+1,y)-\H^{\ges i}(x,y)=\text{ total mass of colors } \geqslant i\text{ at } (x+\tfrac12,y).
$$
Corollary \ref{Corollary_continuous_vertex} immediately implies the following statement.
\begin{corollary} Consider the analytically continued fused colored model in the quadrant, as in Theorem \ref{Theorem_6v_invariance_infinite}, in the limit regime $$
 \eps\to 0,\quad  q=\exp(-\eps),\quad \mathfrak l=q^\rho ,\quad \mathfrak m_x=q^{\sigma_x}, \quad x=0,1,2,\dots .$$
 Let $\H_\eps^{\ges i}(x,y)$ denote the height function of this discrete model. Then in finite-dimensional distributions
 $$
  \lim_{\eps \to 0} \eps \H_\eps^{\ges i}(x,y)= \H^{\ges i}(x,y),\quad x,y\in \tfrac{1}{2}+\mathbb Z_{\ge 0},
 $$
 where  $\H^{\ges i}(x,y)$ is the just defined height function of the continuous vertex model.
\end{corollary}

\section{Directed polymers}

\label{Section_polymers}

\subsection{Beta polymer}

\label{Section_Beta}

Following \cite{Bar-Cor} (who considered the colorless case) we identify the model of Section \ref{Section_cont_limit} with random directed polymers. Another possible interpretation (which we do not pursue in this text) is that of a random walk in random environment.

Let us  present an alternative definition of the height function $\H^{\ges i}(x,y)$ of the continuous vertex model. For that, we start with a sequence $\eta(x,y)$, $x\in\mathbb Z_{\ge 0}$, $y\in\mathbb Z_{>0}$, of independent Beta random variables with distributions $B(\sigma_x-\rho,\rho)$.\footnote{In this construction one can also allow vertical inhomogeneities $\rho_y$, so that the Beta random variables have distributions $B(\sigma_x-\rho_y,\rho_y)$. One would expect, as in Remark \ref{Remark_vertical_inhomogeneities}, that the shift invariance still holds in such an extended setting, but we do not have a proof at this time.}
\begin{proposition} \label{Proposition_continuous_recurrence}
 $\H^{\ges i}(x,y)$ vanishes for $x\ge y$, while for $x<y$ it can be found by solving the following recurrence:
   \begin{multline} \label{eq_polymer_recurrence} \exp\bigl(-\H^{\ges i}(x,y)\bigr)=\eta\left(x-\tfrac12,y-\tfrac12\right) \exp\bigl(-\H^{\ges i}(x,y-1)\bigr)\\+\left(1-\eta\left(x-\tfrac12,y-\tfrac12\right)\right) \exp\bigl(-\H^{\ges i}(x-1,y-1)\bigr)\end{multline}
 with boundary conditions
$$ \H^{\ges i}\left(x,x\right)=0, \qquad  \H^{\ges i}\left(\tfrac{1}{2},y\right)= \sum_{u=i}^{y-1/2} \ln(-\eta(0,u)), \qquad x,y=\tfrac12, \tfrac32,\tfrac52,\dots.$$
\end{proposition}
\begin{proof}
 The boundary conditions are checked in a straighforward way. For the recurrence we notice that, by definition of the system:
 \begin{equation}
 \label{eq_x49}
   \H^{\ges i}(x,y)=\H^{\ges i}(x,y-1)+ \delta_{\ges i}(\alpha_1,\alpha_2,\dots),
 \end{equation}
 where $\delta_{\ges i}(\alpha_1,\alpha_2,\dots)$  stands for the mass of the colors $\geqslant i$ leaving the vertex $(x-\frac12,y-\frac12)$ to the right (this is a (random) function of masses $(\alpha_1,\alpha_2,\dots)$ entering from below). Corollary \ref{Corollary_continuous_vertex} reads
 $$
  \exp(-\delta_{\ges i})=\exp(-\alpha_{\ges i})+\bigl(1-\exp(-\alpha_{\ges i})\bigr) \eta\left(x-\tfrac12,y-\tfrac12\right).
 $$
 Exponentiating \eqref{eq_x49} and using $\alpha_{\ges i}=\H^{\ges i}(x-1,y-1)-\H^{\ges i}(x,y-1)$, we get
 \begin{multline}
  \exp\left(-\H^{\ges i}(x,y)\right)= \exp\left(-\H^{\ges i}(x,y-1)\right)   \Biggl( \exp(\H^{\ges i}(x,y-1)-\H^{\ges i}(x-1,y-1))\\ +\biggl(1-\exp(\H^{\ges i}(x,y-1)-\H^{\ges i}(x-1,y-1))\biggr) \eta\left(x-\tfrac12,y-\tfrac12\right) \Biggr)
  \\= \exp\bigl(-\H^{\ges i}(x,y-1)\bigr) \eta\left(x-\tfrac12,y-\tfrac12\right)\\ +\exp\bigl(-\H^{\ges i}(x-1,y-1)\bigr) \left(1-\eta\left(x-\tfrac12,y-\tfrac12\right)\right).\qedhere
 \end{multline}
\end{proof}

The recurrence \eqref{eq_polymer_recurrence} can be solved in the language of directed polymers:
\begin{itemize}
\item  We consider the grid $\mathbb Z_{\ge 0}\times \mathbb Z_{\ge 0}$ with diagonal and vertical edges, i.e., the edges link $(x,y)$ with $(x+1,y+1)$ and with $(x,y+1)$.
\item Each edge has a weight $w(\text{edge})$ assigned to it. For vertical edges $(x,y-1)\to (x,y)$ the weight is an independent Beta random variable $B_{xy}$ whose parameters are allowed to depend on $x$ via $B(\sigma_{x-1}-\rho,\rho)$ for a sequence of real numbers $\sigma_0,\sigma_1,\dots,$ and $\rho<\min_{x\ge 1}\sigma_{x-1}$, as in the previous section. For diagonal edges $(x-1,y-1)\to (x,y)$ we set $w(\text{edge})=1-B_{xy}$, cf.\ Figure \ref{Fig_polymer_square}.
\item For two points $(x',y')$, $(x,y)$ in $\mathbb Z_{\ge 0}\times \mathbb Z_{\ge 0}$ with $x\ge x'$ and $y\ge y'+(x-x')$, we define the partition function $\mathfrak Z^B_{(x',y')\to(x,y)}$ of the \emph{delayed} Beta--polymer  as a sum over all lattice paths  joining $(x',y')$ with $(x,y)$
\begin{equation}
\label{eq_x51}
 \mathfrak Z^B_{(x',y')\to(x,y)}= \sum_{(x',y')=\pi_0\to \pi_1 \to \dots \to \pi_{y-y'}=(x,y)} \,\, \prod_{k=f(\pi)}^{y-y'} w( \pi_{k-1}\to \pi_k),
\end{equation}
where each edge $\pi_k-\pi_{k-1}$ is either $(1,1)$ or $(1,0)$, and $f(\pi)=\min \{i: \pi_{i}-\pi_{i-1}=(1,0)\}$.
\end{itemize}
The empty  product in \eqref{eq_x51} is taken to be $1$, so that $\mathfrak Z_{(x,y)\to(x+k,y+k)}=1$, $k=0,1,2,\dots$. The appearance of $f(\pi)$  in \eqref{eq_x51} means that in the computation of the weight we \emph{ignore} the initial segment of diagonal edges starting from $(x',y')$ and only multiply the weights starting from the first vertical edge on the path. One can also interpret such a partition function as a point-to-half-line, rather than point-to-point polymer, see \cite{Bar-Cor}.

\begin{figure}[t]
\begin{center}
{\scalebox{1.5}{\includegraphics{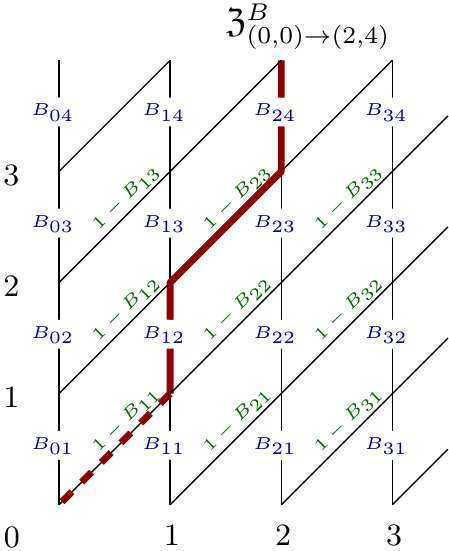}}}
 \caption{Continuous vertex model treated as a directed polymer in the quadrant. The first few diagonal edges (until the first up step) do not collect any weights.
 \label{Fig_polymer_square}}
\end{center}
\end{figure}

\begin{proposition}
\label{Proposition_vertex_model_as_polymer}
 The joint law of the Beta polymer delayed partition functions
 $$\bigl[\mathfrak Z^B_{(0,y')\to (x,y)}\bigr]_{x\ge 0,\, y\ge y' \ge 0}$$
 is the same as that of the exponentiated height functions of the continuous vertex model from Section \ref{Section_cont_limit}
 $$\left[\exp\left(-\H^{\ges (y'+1)}\left(x+\tfrac12,y+\tfrac12\right)\right)\right]_{x\ge 0,\, y\ge y'\ge 0}.$$
 \end{proposition}
 \begin{proof}
   Let us identify the random variables $\eta(x,y)$ of Proposition \ref{Proposition_continuous_recurrence} with $B_{xy}$ in the definition of the polymer. Then the partition functions  $\mathfrak Z_{(0,y')\to (x,y)}$ satisfy the same recurrence in $(x,y)$ and boundary conditions as the corresponding exponentials in Proposition \ref{Proposition_continuous_recurrence}.
 \end{proof}

We remark that \cite{Bar-Cor} was using slightly different notations for the Beta polymer. Their polymers live on transposed and shifted by $(0,1)$ grid, as in Figure \ref{Fig_polymer_BC}.

\begin{figure}[t]
\begin{center}
{\scalebox{1.0}{\includegraphics{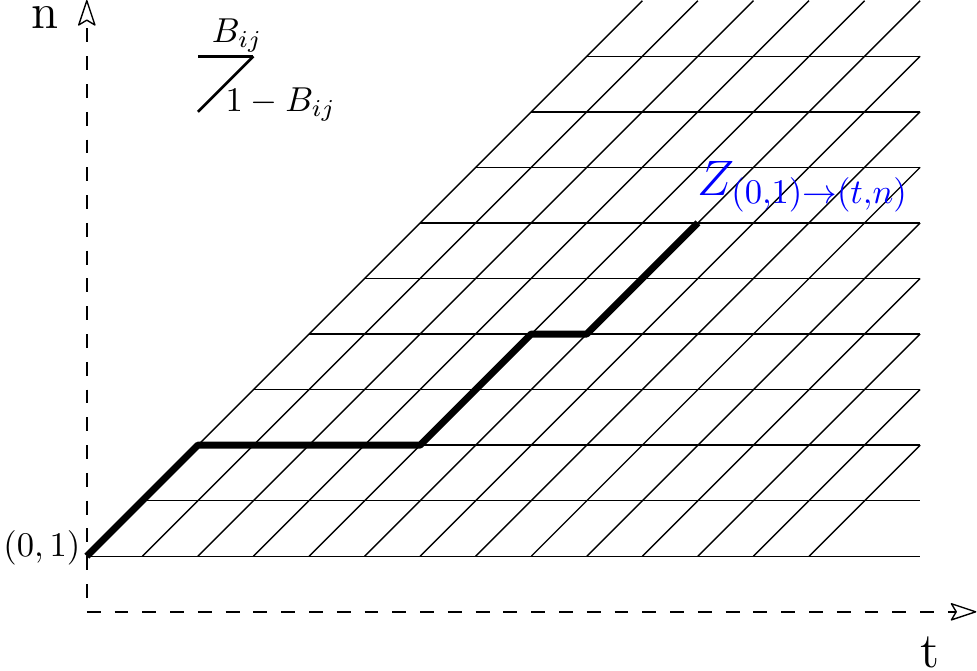}}}
 \caption{Another way to interpret the continuous vertex model as a directed polymer, following \cite{Bar-Cor}. For the computation of $Z_{(0,1)\to (t,n)}$ the weights along the $y=x+1$ line are all treated as $1$.
 \label{Fig_polymer_BC}}
\end{center}
\end{figure}

\medskip

We now restate the shift invariance in the language of polymers. Recall that for two points $\mathcal U=(u_1,u_2)$, $\mathcal V=(v_1,v_2)$, we write $\mathcal U\succeq \mathcal V$ if $u_1\le v_1$ and $u_2\ge v_2$.

\begin{theorem} \label{Theorem_Beta_polymer_invariance} For the Beta polymer of Proposition \ref{Proposition_vertex_model_as_polymer}, choose a set of positive integers $k_j$ and a set of  points in the quadrant $\mathcal U_j$, $j=1,\dots,n$. Fix an  index $1\le \iota \le n$ and an integer $\Delta>0$. Set
 $$
  k_j'=\begin{cases} k_j,& i\ne \iota,\\ k_\iota+\Delta, & j=\iota,\end{cases} \qquad \qquad \mathcal U'_j=\begin{cases}\mathcal U_j, & j\ne \iota \\ \mathcal U_\iota + (0,\Delta), & j=\iota. \end{cases}
 $$
Suppose that
 $$
  0\le k_1\le k_2\le \dots\le k_n, \qquad 0\le k'_1\le k'_2\le \dots\le k'_n,
 $$
 $$\mathcal U_1,\dots, \mathcal U_{\iota-1}\succeq \mathcal U_\iota \succeq \mathcal U_{\iota+1},\dots,\mathcal U_n, \qquad \mathcal U'_1,\dots, \mathcal U'_{\iota-1}\succeq \mathcal U'_\iota \succeq \mathcal U'_{\iota+1},\dots,\mathcal U'_n.
 $$
 Then the distribution of the vector of the polymer partition functions
 $$
  \bigl(\mathfrak Z^B_{(0,k_1)\to \mathcal U_1},\, \mathfrak Z^B_{(0,k_2)\to \mathcal U_2},\, \dots, \mathfrak Z^B_{(0,k_n)\to \mathcal U_n}\bigr)
 $$
 coincides with the distribution of the vector with shifted $\iota$-th coordinate
 $$
  \bigl(\mathfrak Z^B_{(0,k'_1)\to \mathcal U'_1},\, \mathfrak Z^B_{(0,k'_2)\to \mathcal U'_2},\, \dots, \mathfrak Z^B_{(0,k'_n)\to \mathcal U'_n}\bigr).
 $$
\end{theorem}
\begin{remark}
 As in Theorems \ref{Theorem_6v_invariance_fused}, \ref{Theorem_6v_invariance_infinite}, one can expect the possibility of the extension of the theorem to the case of different $\rho_y$ (which need to be interchanged together with shifts). Also, it should be possible to move end-points of the polymer in the horizontal direction as well, as long as the intersection of the polymer trajectories is not impacted.
\end{remark}
\begin{proof}[Proof of Theorem \ref{Theorem_Beta_polymer_invariance}]
 We start from Theorem \ref{Theorem_6v_invariance_infinite} and send $q\to 1$ in the regime \eqref{eq_limit_regime}. As a result, we get a similar theorem for the continuous vertex model of Section \ref{Section_cont_limit}. Proposition \ref{Proposition_vertex_model_as_polymer} recasts the vertex model as a directed polymer. Theorem \ref{Theorem_Beta_polymer_invariance} is then a direct restatement of the $q\to 1$ limit of Theorem \ref{Theorem_6v_invariance_infinite}.
\end{proof}

\subsection{Gamma polymer}
\label{Section_Gamma}

In this and five subsequent sections we degenerate the Beta polymer to several other  probabilistic systems. From now on we only consider \emph{homogeneous case} when all $\sigma_x$ are the same. Inhomogeneous versions are certainly possible to consider, but we leave them out of the present text.

\begin{figure}[t]
\begin{center}
{\scalebox{1.3}{\includegraphics{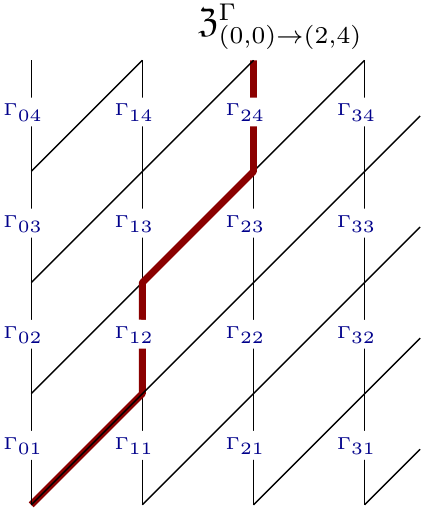}}}
 \caption{Directed polymer with independent Gamma weights $\Gamma_{ij}$ with parameter $\kappa$ on vertical edges.
 \label{Fig_Gamma_polymer}}
\end{center}
\end{figure}

Recall that a random variable $\xi$ has Gamma distribution with parameter $\kappa>0$ if its density with respect to the Lebesgue measure on the positive semiaxis is
$$
 \frac{1}{\Gamma(\kappa)} x^{\kappa-1} \exp(-x), \quad x>0.
$$
The Gamma distribution with parameter $\kappa$ can be obtained from the Beta distribution as the limit
$$
 \lim_{\eps \to 0} \eps^{-1}B(\kappa, \eps^{-1}).
$$
Hence, setting $\rho=\eps^{-1}$, $\sigma=\rho+\kappa$ and sending $\eps\to 0$, the polymer of the previous section converges to the following one:

\begin{itemize}
  \item We deal with square grid $\mathbb Z_{\ge 0}\times \mathbb Z_{\ge 0}$ with vertical and diagonal edges.
  \item Each vertical edge $(x,y-1)\to (x,y)$, $x\ge 0$, $y\ge 1$, is equipped with a weight $w(\text{edge})=\Gamma_{xy}$, which is an independent Gamma random variable with parameter $\kappa$.
  \item Each diagonal edge $(x-1,y-1)\to(x,y)$, $x,y\ge 1$, has the weight $w(\text{edge})=1$.
  \item For each $x\ge x'$, $y\ge y'+(x-x')$ we define a partition function of the polymer as the sum over lattice paths linking $(x',y')$ to $(x,y)$:
      \begin{equation}
      \label{eq_Gamma_polymer}
        \mathfrak Z^\Gamma_{(x',y')\to (x,y)}=\sum_{(x',y')=\pi_0\to\pi_1\to\dots\to \pi_{x-x'}=(x,y)} \, \prod_{k=1}^{y-y'} w( \pi_{k-1}\to \pi_k),
      \end{equation}
      where for each $k\ge 1$ the difference $\pi_k-\pi_{k-1}$ is either $(1,1)$ or $(0,1)$.
\end{itemize}
Note that, as opposed to the Beta polymer, we no longer need to deal with delayed partition functions; this is because diagonal edges already have weight $1$ and, therefore, changing weights of diagonal edges to $1$ does not change anything. See Figure \ref{Fig_Gamma_polymer} for an illustration. This polymer was studied under the name ``strict-weak polymer'' in \cite{CSS} and ``random polymer with gamma-distributed weights'' in \cite{OCO}.

By a straightforward limit transition, Theorem \ref{Theorem_Beta_polymer_invariance} leads to the  same shift invariance statement for the Gamma polymer \eqref{eq_Gamma_polymer}.

\subsection{O'Connell-Yor polymer}

\label{Section_Brownian_polymer}

Let us consider the Gamma polymer in a thin vertical rectangle. For that we take a large $L\to\infty$, set the parameter of the Gamma distributions to $\kappa=L$ and consider the following limit for the polymer of Section \ref{Section_Gamma}:
\begin{equation}
\label{eq_Gamma_limit}
 \exp\left(\tfrac{y-y'}{2}\right)\lim_{L\to\infty} \frac{\mathfrak Z^\Gamma_{(x', Ly')\to (x, Ly)}}{L^{L(y-y')}}=\exp\left(\tfrac{y-y'}{2}\right) \lim_{L\to\infty} \frac{\mathfrak Z^\Gamma_{(x', Ly')\to (x, Ly)}}{L^{L(y-y')-(x-x')} \cdot L^{x-x'}}.
\end{equation}
The factor $L^{L(y-y')-(x-x')}$ can be absorbed into the Gamma random factors which polymer collects (there are precisely $L(y-y')-(x-x')$ of those on each path from $(x,Ly)$ to $(x',Ly')$). Gamma random variables as their parameter $\kappa=L\to\infty$ have the following asymptotics:
$$
 \frac{\Gamma_{ij}}{L}= 1+ \frac{\eta_{ij}}{\sqrt L},
$$
where $\eta_{ij}$ are centered i.i.d.\ random variables, which asymptotically become standard Gaussians $N(0,1)$. (For integral $L$ this can be seen by identifying $\Gamma_{ij}$ with the sum of $L$ independent Gaussian distributions and using the Central Limit Theorem.) Hence, moving the noise into the exponent, we can write
$$
 \ln\left(\frac{\Gamma_{ij}}{L}\right)= \frac{\eta_{ij}}{\sqrt L} - \frac{\eta_{ij}^2}{2 L}+ o\left(\frac{1}{L}\right).
$$
At this point, we can use functional central limit theorem for the sums of logarithms of $\Gamma_{ij}$. The prefactor $L^{x'-x}$ leads to conversion of the (Riemann) sums into integrals, while the products themselves turn into the exponentials of increments of Brownian motions. For the term $-\frac{1}{2 L} \sum \eta_{ij}^2$ we can apply the Law of Large Numbers, resulting in the factor that precisely cancels the exponential prefactor in \eqref{eq_Gamma_limit}.

 The  $L\to\infty$ limit of \eqref{eq_Gamma_limit} is known as the \emph{O'Connell-Yor polymer} or semi-discrete Brownian directed polymer, first introduced in \cite{OCY}. We rename the variables $x$ and $y$ into $n$ and $t$, respectively, to match the standard notations.\footnote{Note, however, that usually the continuous coordinate $t$ is drawn in the horizontal direction. We direct it vertically, in order to keep all the shifts vertical, as was the case throughout the paper.} Here is the resulting object:

\begin{itemize}
  \item We deal with semi-discrete grid $\mathbb Z_{\ge 0}\times \mathbb R_{\ge 0}$ with coordinates $(n,t)$.
  \item Each vertical line has independent standard Brownian motion $B_n(t)$, $t\ge 0$, attached to it.
  \item For each $n\ge n' \ge 0$, $t\ge t'\ge 0$ we define a partition function of the polymer as the integral over monotone grid paths linking $(n',t')$ to $(n,t)$:
      $$
        \mathfrak Z^{OY}_{(n',t')\to (n,t)}=\int_{t'=t_0< t_1< \dots< t_{n-n'+1}=t} \exp\left[\sum_{i=0}^{n-n'} \bigl(B_{i+n'}(t_{i+1})-B_{i+n'}(t_i)\bigr) \right]   dt_1 \cdots dt_{n-n'}.
      $$
\end{itemize}

Recall that for $\mathcal U=(u_1,u_2)$, $\mathcal V=(v_1,v_2)$, we write $U\succeq V$ if $u_1\le v_1$ and $u_2\ge v_2$.

\begin{theorem} \label{Theorem_cont_polymer_invariance} In the above setting of the semi-discrete Brownian directed polymer in the quadrant, choose a set of positive reals $k_j$ and a set of  points in the quadrant $\mathcal U_j$, $j=1,\dots,n$. Fix an index $1\le \iota \le n$ and  $\Delta>0$. Set
 $$
  k_j'=\begin{cases} k_j,& i\ne \iota,\\ k_\iota+\Delta, & j=\iota,\end{cases} \qquad \qquad \mathcal U'_j=\begin{cases}\mathcal U_j, & j\ne \iota \\ \mathcal U_\iota + (0,\Delta), & j=\iota. \end{cases}
 $$
Suppose that
 $$
  0\le k_1\le k_2\le \dots\le k_n, \qquad 0\le k'_1\le k'_2\le \dots\le k'_n,
 $$
 $$\mathcal U_1,\dots \mathcal U_{\iota-1}\succeq\mathcal U_\iota\succeq\mathcal U_{\iota+1},\dots \mathcal U_n, \qquad \mathcal U'_1,\dots \mathcal U'_{\iota-1}\succeq\mathcal U'_\iota\succeq \mathcal U'_{\iota+1},\dots \mathcal U_n,
 $$
 Then the distribution of the vector of the polymer partition functions
 $$
  \bigl(\mathfrak Z^{OY}_{(0,k_1)\to \mathcal U_1},\, \mathfrak Z^{OY}_{(0,k_2)\to \mathcal U_2},\, \dots, \mathfrak Z^{OY}_{(0,k_n)\to \mathcal U_n}\bigr)
 $$
 coincides with the distribution of the vector with shifted $\iota$-th coordinate
 $$
  \bigl(\mathfrak Z^{OY}_{(0,k'_1)\to \mathcal U'_1},\, \mathfrak Z^{OY}_{(0,k'_2)\to \mathcal U'_2},\, \dots, \mathfrak Z^{OY}_{(0,k'_n)\to \mathcal U'_n}\bigr).
 $$
\end{theorem}
We refer to Figure \ref{Fig_Brownian_polymer} for an example. The proof of Theorem \ref{Theorem_cont_polymer_invariance} is a direct limit transition from the similar shift-invariance statement for Gamma polymer mentioned at the end of Section \ref{Section_Gamma}.

 Theorem \ref{Theorem_cont_polymer_invariance} deals with vertical shifts along the continuous coordinate of the polymer. Can we also shift in the discrete horizontal direction? The answer is positive, but we then need to deal with \emph{delayed} partition functions as we did for Beta polymer.

 For each $n\ge n' \ge 0$, $t\ge t'\ge 0$ we define the delayed partition function of the O'Connell--Yor polymer as the integral over monotone grid paths linking $(n',t')$ to $(n,t)$:
     $$
        \mathfrak Z^{OY;\text{delay}}_{(n',t')\to (n,t)}=\int_{t'=t_0< t_1< \dots< t_{n-n'+1}=t} \exp\left[\sum_{i=1}^{n-n'} \bigl(B_{i+n'}(t_{i+1})-B_{i+n'}(t_i)\bigr) \right]   dt_1 \cdots dt_{n-n'}.
      $$
      where we emphasize that the noise $B_{n'}(t_1)-B_{n'}(t_0)$ \emph{is not} being collected.

\begin{theorem} \label{Theorem_cont_polymer_invariance_delayed} In the  setting of the delayed semi-discrete Brownian directed polymer in the quadrant, choose a set of positive reals $k_j$ and a set of  points in the quadrant $\mathcal U_j$, $j=1,\dots,n$. Fix an index $1\le \iota \le n$ and  $\Delta=1,2,\dots$. Set
 $$
  k_j'=\begin{cases} k_j,& i\ne \iota,\\ k_\iota+\Delta, & j=\iota,\end{cases} \qquad \qquad \mathcal U'_j=\begin{cases}\mathcal U_j, & j\ne \iota \\ \mathcal U_\iota + (\Delta,0), & j=\iota. \end{cases}
 $$
Suppose that
 $$
  0\le k_1\le k_2\le \dots\le k_n, \qquad 0\le k'_1\le k'_2\le \dots\le k'_n,
 $$
 $$\mathcal U_1,\dots \mathcal U_{\iota-1}\preceq\mathcal U_\iota\preceq\mathcal U_{\iota+1},\dots \mathcal U_n, \qquad \mathcal U'_1,\dots \mathcal U'_{\iota-1}\preceq\mathcal U'_\iota\preceq \mathcal U'_{\iota+1},\dots \mathcal U_n,
 $$
 Then the distribution of the vector of the polymer partition functions
 $$
  \bigl(\mathfrak Z^{OY;\text{delay}}_{(k_1,0)\to \mathcal U_1},\, \mathfrak Z^{OY;\text{delay}}_{(k_2,0)\to \mathcal U_2},\, \dots, \mathfrak Z^{OY;\text{delay}}_{(k_n,0)\to \mathcal U_n}\bigr)
 $$
 coincides with the distribution of the vector with shifted $\iota$-th coordinate
 $$
  \bigl(\mathfrak Z^{OY;\text{delay}}_{(k'_1,0)\to \mathcal U'_1},\, \mathfrak Z^{OY;\text{delay}}_{(k'_2,0)\to \mathcal U'_2},\, \dots, \mathfrak Z^{OY;\text{delay}}_{(k'_n,0)\to \mathcal U'_n}\bigr).
 $$
\end{theorem}

Note the reverse inequalities for the points $\mathcal U_j$, as compared to Theorem \ref{Theorem_cont_polymer_invariance}.
\begin{proof}[Proof of Theorem \ref{Theorem_cont_polymer_invariance_delayed}]
 We take another limit from the Beta polymer to a version of the Gamma polymer by taking a limit of $1-B_{ij}$ (instead of $B_{ij}$) to the Gamma random variables. The resulting polymer is like in Section \ref{Section_Gamma}, but the Gamma random variables are placed on diagonal rather than vertical edges of the grid. Note that for this modified Gamma polymer, we prove the shift-invariance only for the delayed partition functions $\mathfrak Z^{\Gamma;\text{delay}}$. We then consider the limit of this polymer as its end-points move far away, while staying in a finite neighborhood of the diagonal $x=y$:
 \begin{equation}
\label{eq_Gamma_limit_2}
 \mathfrak Z^{OY;\text{delay}}_{(x',y')\to(x,y)}= \exp\left(\tfrac{y-y'}{2}\right)\lim_{L\to\infty} \frac{\mathfrak Z^{\Gamma;\text{delay}}_{(Lx', Lx'+y')\to (Lx, Lx+y)}}{L^{L(x-x')+(y-y')}}.
\end{equation}
Note that the coordinate system $(x,y)$ is transposed to the one we used for $\mathfrak Z^{OY}$ above, i.e., the first rather than the second coordinate is now continuous. Taking the limits of Theorem \ref{Theorem_Beta_polymer_invariance} first to the shift-invariance for $\mathfrak Z^{\Gamma;\text{delay}}$ and then for $\mathfrak Z^{OY;\text{delay}}$ we arrive at the desired statement.
\end{proof}

\begin{remark}
 Is shift-invariance in the discrete (horizontal) direction also true for $\mathfrak Z^{OY}$? In other words, is it necessary to pass to the delayed partition functions $ \mathfrak Z^{OY;\text{delay}}$? Our proofs do not give any answer.
\end{remark}

\subsection{Brownian Last Passage Percolation}

\label{Section_BLPP}

For the next limit transition we consider large $t-t'$ in the definition of the semi-discrete Brownian directed polymer; more precisely, we multiply $t$ and $t'$ by $L$ and send $L\to\infty$. If we use the scale-invariance of the Brownian motion, we can write

\begin{multline}
\int_{Lt'=t_0< t_1< \dots< t_{n-n'+1}=Lt} \exp\left[\sum_{i=0}^{n-n'} \bigl(B_{i+n'}(t_{i+1})-B_{i+n'}(t_i)\bigr) \right]   dt_1 \cdots dt_{n-n'}
\\=L^{n-n'}\int_{t'=t_0< t_1< \dots< t_{n-n'+1}=t} \exp\left[\sqrt{L} \cdot \sum_{i=0}^{n-n'} \bigl(B_{i+n'}(t_{i+1})-B_{i+n'}(t_{i})\bigr) \right]   dt_1 \cdots dt_{n-n'}.
\end{multline}
Almost surely, the integrand is a continuous function of $t_1,\dots,t_{n-n'}$. Hence, as $L\to\infty$, the integral is dominated by the points where the exponent $\sum_{i=0}^{n-n'} \bigl(B_{i+n'}(t_{i+1})-B_{i+n'}(t_{i})\bigr)$ is maximized (cf.\, standard proofs of the convergence of $L_p$ norms to $L_\infty$ norm as $p\to\infty$). Rescaled logarithm of the integral converges (in law, jointly over finitely many $(t',n')$ and $(t,n)$) to the \emph{Brownian Last Passage time}, defined as follows:
\begin{itemize}
  \item We deal with semi-discrete grid $ \mathbb Z_{\ge 0}\times \mathbb R_{\ge 0}$ with coordinates $(n,t)$.
  \item Each vertical line has independent standard Brownian motion $B_n(t)$, $t\ge 0$, attached to it.
  \item For each $n\ge n' \ge 0$, $t\ge t'\ge 0$ we assign the passage time as the maximum over grid paths linking $(n',t')$ to $(n,t)$:
      \begin{equation}
      \label{eq_Brownian_passage_time}
        \mathfrak Z^{BLPP}_{(n',t')\to (n,t)}=\max_{t'=t_0< t_1< \dots< t_{n-n'+1}=t} \left[\sum_{i=0}^{n-n'} \bigl(B_{i+n'}(t_{i+1})-B_{i+n'}(t_{i})\bigr) \right] .
      \end{equation}
\end{itemize}

Direct limit of Theorem \ref{Theorem_cont_polymer_invariance} yields the same shift-invariance statement for the  passage times \eqref{eq_Brownian_passage_time}. If we introduce a delayed version of \eqref{eq_Brownian_passage_time}, in which we ignore the noise $B_{n'}(t_1)-B_{n'}(t_0)$, then Theorem \ref{Theorem_cont_polymer_invariance_delayed} also yields the same shift-invariance statement for the delayed Brownian Last Passage times.

\subsection{Continuous directed polymer and KPZ equation}
\label{Section_KPZ}

Another limiting object can be obtained from the O'Connell--Yor polymer by considering large $n-n'$. In this regime the $n$--indexed sequence of white noises (derivatives of the Brownian motions $B_n(t)$) turns into the two-dimensional white-noise. The polymer collecting such noises is known as the Continuum Directed Random Polymer, see, e.g., \cite{AKQ1}, \cite{Q_CDM}, and its logarithm can be identified with a solution to the Kardar--Parisi--Zhang stochastic partial differential equation.

 In more detail, let us introduce the normalized version of the semi-discrete Brownian polymer through
\begin{equation}
 \tilde {\mathfrak Z}^{OY}_{(n',t')\to (n,t)}= \frac{(n-n')!}{ (t-t')^{n-n'} } \cdot \exp\left(-\frac{t-t'}{2}\right)\cdot \mathfrak Z^{OY}_{(n',t')\to (n,t)}.
\end{equation}
The first factor is introduced to compensate for the volume of the simplex $\{t'=t_0<t_1<\dots<t_{n-n'}=t\}$ in the definition of the polymer, while the second one compensates the expectation of the exponential of the Brownian motions (which can be computed using the identity $\E \exp( \sqrt{c} \cdot \xi)= \exp( c/2 )$ for Gaussian $\xi\sim N(0,1)$). Thus, we have $\E \tilde {\mathfrak Z}^{OY}_{(n',t')\to (n,t)}=1$.

We have the following convergence result for the Brownian directed polymer:
\begin{equation}
\label{eq_convergence_to_KPZ}
 \lim_{L\to\infty} \left[\tilde {\mathfrak Z}^{OY}_{(0, y)\to (tL, t\sqrt{L}+ x) } \right] \cdot \frac{1}{\sqrt{2\pi t}} \exp\left(- \tfrac{(x-y)^2}{2 t}\right)= \mathcal Z^{(y)}(t,x),
\end{equation}
where $\mathcal Z^{(y)}(x,t)$ is the solution to the stochastic heat equation with multiplicative noise started at $t=0$ time from $\delta$--function initial condition at $y$:
\begin{equation}
\label{eq_multi_SHE}
 \mathcal Z^{(y)}_{t}=\tfrac12 \mathcal Z^{(y)}_{xx}+ \eta \mathcal Z^{(y)}, \quad t\ge 0, x\in\mathbb R;\quad \mathcal Z^{(y)}(0,x)=\delta(x-y).
\end{equation}
Here $\eta$ is the space-time $2d$ white noise (the same for each $y$). We remark that the literature (see \cite{AKQ2}, \cite{Nica}) typically states the convergence result \eqref{eq_convergence_to_KPZ} only for fixed $y$ (usually $y=0$), yet the technique extends to the joint convergence in law for finitely many $y$'s.\footnote{Let us follow the approach of \cite{Nica}. In that paper, the author introduces a coupling between the 2d white noise $\eta$ in the definition of $\mathcal Z^{(y)}(x,t)$ and Brownian motions $B_n(t)$ in the definition of the Brownian directed polymer. Using the Wiener chaos series expansion, \cite[Section 2.5]{Nica} then shows the convergence \eqref{eq_convergence_to_KPZ} for fixed values of $x$, $y$, and $t$ in the $L_2$ space of random variables on the probability space where the white noise $\eta$ is defined. The $L_2$--convergence implies the convergence of joint distributions, if we use the following abstract statement. If for random variables $X_n$, $Y_n$, $X$, $Y$ on the same probability space, we have $\lim_{n\to\infty}\E|X_n-X|^2=\lim_{n\to\infty}\E|Y_n-Y|^2=0$, then the distribution of the vector $(X_N,Y_N)$ converges as $n\to\infty$ to that of the vector $(X,Y)$, and similarly for $k>2$ dimensional vectors.}

The Feynman-Kac representation for the solution to \eqref{eq_multi_SHE} leads to its representation as an integral of exponentiated noise over paths of Brownian bridges, thus clarifying the convergence \eqref{eq_convergence_to_KPZ} and the name Continuum Directed Random Polymer for $\mathcal Z^{(y)}(x,t)$.

Computing (formally) the logarithm of $\mathcal Z$, one finds that $\mathcal H:=-\ln(\mathcal Z^{(y)})$ satisfies the Kardar--Parisi--Zhang (KPZ) stochastic partial differential equation:
\begin{equation}
 \mathcal H_t= \tfrac12 \mathcal H_{xx}-\tfrac12 (\mathcal H_x)^2-\eta.
\end{equation}
Taking the limit of Theorem \ref{Theorem_cont_polymer_invariance} we arrive at the following statement for $\mathcal Z^{(y)}$ (or, equivalently, for $\mathcal H$), cf.\ Figure \ref{Fig_KPZ_shift}:

\begin{theorem}[modulo making convergence in \eqref{eq_convergence_to_KPZ} joint in $x$'s and $y$'s] \label{Theorem_KPZ_invariance}
 Fix $t>0$, $\Delta>0$, $x,y\in\mathbb R$. In addition, choose a collection of points $(x_i,y_i)$, $i=1,\dots,n$, such that for each $i$ either $x_i<x$, $y_i>y+\Delta$ or $x_i>x+\Delta$, $y_i<y$. Then we have an identity in distribution:
 $$
   \bigl(\mathcal Z^{(y)}(t,x); \, \mathcal Z^{(y_i)}(t,x_i), i=1,\dots,n \bigr)\stackrel{d}{=}
   \bigl(\mathcal Z^{(y+\Delta)}(t,x+\Delta); \, \mathcal Z^{(y_i)}(t,x_i), i=1,\dots,n\bigr).
 $$
\end{theorem}
\begin{remark}
 While Theorem \ref{Theorem_cont_polymer_invariance} allowed for different time coordinates $t_i$, there is only a single $t$ in Theorem \ref{Theorem_KPZ_invariance}. This is because of the nature of rescaling in the limit transition \eqref{eq_convergence_to_KPZ}, which collapses all points $\mathcal U_i$ of Theorem \ref{Theorem_KPZ_invariance} onto a single line.
\end{remark}

\subsection{Airy sheet}

\label{Section_Airy_sheet}

We proceed to the universal object in the KPZ-universality class known as the Airy sheet. It was recently proven in \cite[Theorem 1.3]{DOV} that the Brownian Last Passage times admit the following limiting behavior:
\begin{equation}
\label{eq_Airy_sheet}
 \frac{\mathfrak Z_{(0,2x n^{2/3})\to (n,n+2 y n^{-1/3})}-2n -2n^{2/3}(y-x)+(x-y)^2n^{1/3}}{  n^{1/3}} \to \mathcal A(x,y), \quad x,y\in\mathbb R.
\end{equation}
The formula \eqref{eq_Airy_sheet} can be taken as the definition of the Airy sheet $\mathcal A(x,y)$. Let us mention a notational subtlety on whether to add $(x-y)^2 n^{1/3}$ in the definition, as we did. In keeping with the traditional usage of the stationary forms of the objects in the Airy$_2$ process and the Airy line ensemble, and also following \cite{CQR}, we chose the definition such that $\mathcal A(x,y)$ has the Tracy-Widom distribution for each $x,y\in\mathbb R$.
The article \cite{DOV} does not have the $(x-y)^2 n^{1/3}$ terms and proves that the whole two-dimensional field $\mathfrak Z_{(t',n')\to (t,n)}$ converges to a four-dimensional extension of the Airy sheet that the authors called the Directed Landscape\footnote{If we follow the notations of \cite{CQR} and add the parabola, as in \eqref{eq_Airy_sheet}, then the limiting object has been known under the name of space-time Airy sheet.}. Note, however, that our result deals only with the function of two variables $\mathcal A(x,y)$.

It is believed that the Airy sheet appears universally in scaling limits of directed percolation models. \cite{CQR} also argues that it should be related to the large time scaling limit of the KPZ equation (proving this rigorously remains an open problem at this time).

Nevertheless, taking the limit of the shift invariance statement for the Brownian Last Passage Percolation, we arrive at a similar statement for the Airy sheet.

\begin{theorem} \label{Theorem_Airy_invariance}
 Fix $t>0$, $\Delta>0$, $x,y\in\mathbb R$. In addition, choose a collection of points $(x_i,y_i)$, $i=1,\dots,n$, such that for each $i$ either $x_i<x$, $y_i>y+\Delta$ or $x_i>x+\Delta$, $y_i<y$. Then we have an identity in distribution:
 $$
   \bigl(\mathcal A(x,y); \, \mathcal A(x_i,y_i), i=1,\dots,n \bigr)\stackrel{d}{=}
   \bigl(\mathcal A(x+\Delta,y+\Delta); \, \mathcal A(x_i,y_i), i=1,\dots,n\bigr).
 $$
\end{theorem}

\subsection{Back to additive SHE}
\label{Section_back_to_SHE}

We end the discussion of degenerations by a remark that the small time limit of the KPZ equation is Gaussian and is given, after proper recentering and rescaling by the stochastic heat equation with additive noise, namely
$$
 \mathfrak H^{(y)}_t=\tfrac{1}{2} \mathfrak H^{(y)}_{xx}+ \frac{1}{\sqrt{2\pi t}} \exp\left(-\tfrac{(x-y)^2}{2t}\right) \cdot \eta, \qquad t\ge 0, x\in\mathbb R,\qquad \mathfrak H^{(y)}(0,x)=0,
$$
where $\eta$ is the $2d$ white noise, see \cite[Section 6.2]{ACQ} for the discussion on how the small time limit can be readily obtained from the Wiener chaos expansion. The shift-invariance for the last equation is essentially equivalent to the shift invariance for the colored SHE of Section \ref{Section_SHE} and can be proven directly in the same way.

\section{Appendix: fusion}

\label{Section_fusion}

In this section we explain the fusion procedure, which produces higher spin/higher rank vertex model from the colored six-vertex model. Our aim is to obtain the weight \eqref{eq_higher_spin_weight} as a result of the summation of the products of weights of all vertices in $M\times L$ rectangle (over all possible choices of such vertices) for the colored six-vertex model.

%\textcolor{green}{[TODO: $(N,M)$ should be changed to $(L,M)$. $n$ should be changed to $N$. Arcs in the drawings of the vertices should be removed]}

%\textcolor{green}{[TODO: This section should be proving Theorem \ref{Theorem_fusion}. We need to define $q$--Gibbs property, state [BW, Proposition B.2.2] and add a couple of sentences explaining that summations of this proposition give the desired claim]}

\subsection{Row-vertices}
We follow the notations of Section \ref{ssec:fundamental}.

The key combinatorial object in the fusion procedure is the {\it row-vertex}. It is obtained by horizontally concatenating $M$ of the $R$-vertices \eqref{R-vert}, and specializing the spectral parameters to a geometric progression $z,qz,q^2z,\dots,q^{M-1}z$. More specifically, for two fixed integers $b, d \in \{0,1,\dots,N\}$ and two vectors of integers $(a_1,\dots,a_M), (c_1,\dots,c_M) \in \{0,1,\dots,N\}^M$, we define
\begin{multline}
\label{R-row}
\index{R@$R_{z}((a_1,\dots,a_M),b; (c_1,\dots,c_M),d)$; row-vertices}
R_{z}\Big((a_1,\dots,a_M),b; (c_1,\dots,c_M),d \Big)
\\
=
\sum_{i_1 = 0}^{N}
\cdots
\sum_{i_{M-1} = 0}^{N}
R_{q^{M-1} z}(a_1,b;c_1,i_1)
R_{q^{M-2} z}(a_2,i_1;c_2,i_2)
\dots
R_{z}(a_M,i_{M-1};c_M,d),
\end{multline}
or graphically,
\begin{align}
\label{R-row-graph}
R_{z}\Big((a_1,\dots,a_M),b; (c_1,\dots,c_M),d \Big)
=
\tikz{0.9}{
%lines
\draw[lgray,line width=1.5pt,->] (-1,0) -- (6,0);
\draw[lgray,line width=1.5pt,->] (0,-1) -- (0,1);
\draw[lgray,line width=1.5pt,->] (1,-1) -- (1,1);
\draw[lgray,line width=1.5pt,->] (2,-1) -- (2,1);
\draw[lgray,line width=1.5pt,->] (3,-1) -- (3,1);
\draw[lgray,line width=1.5pt,->] (4,-1) -- (4,1);
\draw[lgray,line width=1.5pt,->] (5,-1) -- (5,1);
%%arcs
%\draw[densely dotted] (0.5,0) arc (0:90:0.5);
%\node[above right] at (-0.1,-0.1) {\tiny $\alpha_1$};
%\draw[densely dotted] (1.5,0) arc (0:90:0.5);
%\node[above right] at (0.9,-0.1) {\tiny $\alpha_2$};
%\draw[densely dotted] (2.5,0) arc (0:90:0.5);
%\draw[densely dotted] (3.5,0) arc (0:90:0.5);
%\draw[densely dotted] (4.5,0) arc (0:90:0.5);
%\draw[densely dotted] (5.5,0) arc (0:90:0.5);
%\node[above right] at (4.85,-0.1) {\tiny $\alpha_M$};
%horizontal labels
\node[left] at (-1,0) {\tiny $b$};\node[right] at (6,0) {\tiny $d$};
%vertical labels
\node[below] at (0,-1) {\tiny $a_1$};\node[above] at (0,1) {\tiny $c_1$};
\node[below] at (1,-1) {\tiny $a_2$};\node[above] at (1,1) {\tiny $c_2$};
\node[below] at (3,-1) {\tiny $\cdots$};\node[above] at (3,1) {\tiny $\cdots$};
\node[below] at (5,-1) {\tiny $a_M$};\node[above] at (5,1) {\tiny $c_M$};
},
\end{align}
where the spectral parameters of the vertices are given by $q^{M-i} z$ (with $i$ increasing as one reads from left to right), and each internal horizontal edge is summed over all possible values in the set $\{0,1,\dots,N\}$.

\subsection{$M$-fused vertices}
\label{sec:M-fus}

\begin{defn}
Let $M \geq 1$ and consider a vector of nonnegative integers $(a_1,\dots,a_M) \in \{0,1,\dots,N\}^M$. From this we define another vector,
\begin{align*}
\mathcal{C}(a_1,\dots,a_M)
:=
(A_1,\dots,A_N),
\qquad
A_i = \#\{j : a_j = i \},
\quad
\forall\ 1 \leq i \leq N,
\end{align*}
which keeps track of the multiplicity of each color $1 \leq i \leq N$ within $(a_1,\dots,a_M)$.
\end{defn}

Let us fix two vectors $\A = (A_1,\dots,A_N)$ and $\C = (C_1,\dots,C_N)$ whose components are nonnegative integers, such that $|\A| \leq M$ and $|\C| \leq M$. We define the weight of an $M$-fused vertex as follows:
\begin{align}
\label{M-fus}
\mathcal{L}^{(M)}_{z}(\A,b; \C,d)
\index{L4@$\mathcal{L}^{(M)}_{z}(\A,b; \C,d)$; $M$-fused vertices}
=
\frac{1}{Z_q(M;\A)}
\sum_{\substack{
\mathcal{C}(a_1,\dots,a_M) = \A
\\
\mathcal{C}(c_1,\dots,c_M) = \C
}}
q^{{\rm inv}(a_1,\dots,a_M)}
R_{z}\Big((a_1,\dots,a_M),b; (c_1,\dots,c_M),d \Big),
\end{align}
where the summation is over all vectors of integers
$(a_1,\dots,a_M), (c_1,\dots,c_M) \in \{0,1,\dots,N\}^M$ such that $\mathcal{C}(a_1,\dots,a_M) = \A$ and $\mathcal{C}(c_1,\dots,c_M) = \C$. The exponent appearing in the sum is given by
\begin{align*}
{\rm inv}(a_1,\dots,a_M) = \#\{ i<j : a_i > a_j \},
\end{align*}
while the normalization takes the form of a $q$-multinomial coefficient:
\begin{align*}
\index{Z5@$Z_q(M;\I)$}
Z_q(M;\A)
=
\sum_{\mathcal{C}(a_1,\dots,a_M) = \A}
q^{{\rm inv}(a_1,\dots,a_M)}
=
\frac{(q;q)_M}{(q;q)_{A_0} (q;q)_{A_1} \dots (q;q)_{A_N}},
\quad
A_0 := M - \sum_{i=1}^{N} A_i.
\end{align*}
We shall represent the $M$-fused vertex \eqref{M-fus} pictorially as
\begin{align}
\label{L-graphical}
\mathcal{L}^{(M)}_{z}(\A,b; \C,d)
=
\tikz{0.9}{
\draw[lgray,line width=1.5pt,->] (-1,0) -- (1,0);
\draw[lgray,line width=4pt,->] (0,-1) -- (0,1);
%\draw[densely dotted] (0.5,0) arc (0:90:0.5);
\node[left] at (-1,0) {\tiny $b$};\node[right] at (1,0) {\tiny $d$};
\node[below] at (0,-1) {\tiny $\A$};\node[above] at (0,1) {\tiny $\C$};
},
\quad
b,d \in \{0,1,\dots,N\},
\quad
\A,\C \in \mathbb{N}^N,
\end{align}
where the spectral parameter of the vertex is equal to $z$.

Probabilistically, one should think that the distribution of incoming colors satisfies $q$--exhangeability\footnote{\cite[Proposition B.2.2]{BW} then yields $q$--exhangeability also for the the outgoing colors.}. Then \eqref{M-fus} computes the partition function of $M$ vertices under such distribution.

\subsection{Evaluation of $M$-fused vertices}

The following result is taken from \cite[Theorem B.4.1]{BW}.

\begin{theorem}
Fix two vectors $\A, \C \in \mathbb{N}^N$ such that $|\A| \leq M$, $|\C| \leq M$, and two integers $b,d \in \{0,1,\dots,N\}$. The $M$-fused vertex \eqref{M-fus} has the following explicit evaluation:
\begin{align}
\label{fus-wt}
\mathcal{L}^{(M)}_{z}(\A,b;\C,d)
=
\bm{1}_{(\A + \bm{e}_b = \C + \bm{e}_{d})}
\cdot
\frac{1}
{1-q^M z}
\cdot
\left\{
\begin{array}{ll}
(1-q^{A_d} z) q^{\As{d+1}{N}}, & \quad b = d,
\\ \\
(1-q^{A_d}) q^{\As{d+1}{N}}, & \quad b < d,
\\ \\
z (1-q^{A_d}) q^{\As{d+1}{N}}, & \quad b > d,
\end{array}
\right.
\end{align}
where $\As{d+1}{N}=A_{d+1}+A_{d+2}+\dots+A_N$. Or more compactly,
\begin{align}
\label{fus-compact}
\mathcal{L}^{(M)}_{z}(\A,b;\C,d)
=
\bm{1}_{(\A + \bm{e}_b = \C + \bm{e}_{d})}
\cdot
z^{\bm{1}_{b >d}}
\cdot
\frac{1-q^{A_{d}} z^{\bm{1}_{b=d}}}{1-q^M z}
\cdot
q^{\As{d+1}{N}},
\end{align}
where by agreement $A_0 = M - \sum_{i=1}^{N} A_i$.
\end{theorem}

\subsection{Column vertices}

We begin by defining {\it column vertices} by taking towers of height $L$ of the $M$-fused vertices \eqref{L-graphical}:
\begin{align}
\label{tower}
\mathcal{L}^{(M)}_{z}
\Big(
\A,(b_1,\dots,b_L);\C,(d_1,\dots,d_L)
\Big)
=
\tikz{0.9}{
\draw[lgray,line width=1.5pt,->] (-1,5) -- (1,5);
\draw[lgray,line width=1.5pt,->] (-1,4) -- (1,4);
\draw[lgray,line width=1.5pt,->] (-1,3) -- (1,3);
\draw[lgray,line width=1.5pt,->] (-1,2) -- (1,2);
\draw[lgray,line width=1.5pt,->] (-1,1) -- (1,1);
\draw[lgray,line width=1.5pt,->] (-1,0) -- (1,0);
\draw[lgray,line width=4pt,->] (0,-1) -- (0,6);
\node[left] at (-1,0) {\tiny $b_1$};\node[right] at (1,0) {\tiny $d_1$};
\node[left] at (-1,1) {\tiny $b_2$};\node[right] at (1,1) {\tiny $d_2$};
\node[left] at (-1,5) {\tiny $b_L$};\node[right] at (1,5) {\tiny $d_L$};
\node[below] at (0,-1) {\tiny $\A$};\node[above] at (0,6) {\tiny $\C$};
%%arcs
%\draw[densely dotted] (0.5,0) arc (0:90:0.5);
%\node[above right] at (-0.1,-0.1) {\tiny $\alpha_1$};
%\draw[densely dotted] (0.5,1) arc (0:90:0.5);
%\node[above right] at (-0.1,0.9) {\tiny $\alpha_2$};
%\draw[densely dotted] (0.5,2) arc (0:90:0.5);
%\draw[densely dotted] (0.5,3) arc (0:90:0.5);
%\draw[densely dotted] (0.5,4) arc (0:90:0.5);
%\draw[densely dotted] (0.5,5) arc (0:90:0.5);
%\node[above right] at (-0.1,4.9) {\tiny $\alpha_L$};
}
\end{align}
where the spectral parameters of the vertices are given by $q^{i-1} z$ (with $i$ increasing as one reads from bottom to top).

\subsection{$(L,M)$-fused vertices}

Now we fuse the horizontal lines, to produce $(L,M)$-fused vertices. Fix two more nonnegative integer vectors $\B = (B_1,\dots,B_N)$ and $\D = (D_1,\dots,D_N)$.

\begin{defn}
In a similar vein to Section \ref{sec:M-fus}, we define
\begin{align}
\label{MN-fus}
W_{L,M} (z;\A,\B;\C,\D)
=
\frac{1}{Z_q(L;\B)}
\sum_{\substack{
\mathcal{C}(b_1,\dots,b_L) = \B
\\
\mathcal{C}(d_1,\dots,d_L) = \D
}}
q^{{\rm inv}(b_L,\dots,b_1)}
\mathcal{L}^{(M)}_{1/z}
\Big(
\A,(b_1,\dots,b_L);\C,(d_1,\dots,d_L)
\Big),
\end{align}
where the normalization takes the form
\begin{align*}
Z_q(L;\B)
=
\sum_{\mathcal{C}(b_1,\dots,b_L) = \B}
q^{{\rm inv}(b_L,\dots,b_1)}
=
\frac{(q;q)_L}{(q;q)_{B_0} (q;q)_{B_1} \dots (q;q)_{B_N}},
\quad
B_0 := L - \sum_{i=1}^{N} B_i.
\end{align*}
Note that, in this definition, we reverse the order of the entries of the vector
$(b_1,\dots,b_L)$ before applying the {\rm inv} function.
\end{defn}

The main result of this section is Theorem \ref{thmq} below, which identifies \eqref{MN-fus} with \eqref{eq_higher_spin_weight}. For now we take \eqref{MN-fus} as the definition.

\subsection{Recursion relation}

We write a recursion for the $(L,M)$-fused vertices \eqref{MN-fus}, by peeling away the bottom-most vertex from the column \eqref{tower}:
\begin{multline}
\label{bot-rec}
W_{L,M} (z;\A,\B;\C,\D)
=\\
\frac{1}{1-q^L}
\sum_{i=0}^{N}
\sum_{j=0}^{N}
(1-q^{B_i})
q^{B_{i+1}}\ldots q^{B_N}
W_{1,M}(z;\A,\e_i;\A^{+-}_{ij},\e_j)
W_{L-1,M}(qz;\A^{+-}_{ij},\B^{-}_i;\C,\D^{-}_j)
\end{multline}
where we defined $\B_i^{-} = \B - \e_i$ and $\A_{ij}^{+-} = \A + \e_i - \e_j$ and $\e_i$ is the $i$th basis vector. Equivalently, in terms of the $M$-fused vertices \eqref{M-fus}, we have
\begin{multline}
\label{bot-rec2}
W_{L,M} (z;\A,\B;\C,\D)
=\\
\frac{1}{1-q^L}
\sum_{i=0}^{N}
\sum_{j=0}^{N}
(1-q^{B_i})
q^{B_{i+1}}\ldots q^{B_N}
\mathcal{L}^{(M)}_{1/z}(\A,i;\A^{+-}_{ij},j)
W_{L-1,M}(qz;\A^{+-}_{ij},\B^{-}_i;\C,\D^{-}_j).
\end{multline}

\subsection{Explicit formula for $(L,M)$-fused weights}

For any two vectors $\alpha = (\alpha_1,\dots,\alpha_N)$ and $\beta = (\beta_1,\dots,\beta_N)$, define the function
\begin{align*}
\phi(\alpha,\beta)
=
\sum_{1 \leq i<j \leq N}
\alpha_i \beta_j.
\end{align*}
For any two vectors $\lambda = (\lambda_1,\dots,\lambda_N)$ and $\mu = (\mu_1,\dots,\mu_N)$ such that $\lambda_i \leq \mu_i$ for all $1 \leq i \leq N$, define further
\begin{align}
\label{phi-def}
\Phi(\lambda,\mu;x,y)
=
q^{\phi(\mu-\lambda,\lambda)}
\left(y/x\right)^{|\lambda|}
\frac{(x;q)_{|\lambda|} (y/x;q)_{|\mu-\lambda|}}{(y;q)_{|\mu|}}
\prod_{i=1}^N \binom{\mu_i}{\lambda_i}_{q}.
\end{align}

\begin{prop}
Note the following recursive properties of the function $\Phi$:
\begin{align*}
\Phi(\lambda,\mu;x,qy)
=
q^{\lambda}
\frac{(1-y)(x-yq^{\mu-\lambda})}{(x-y)(1-yq^{\mu})}
\Phi(\lambda,\mu;x,y),
\\
\Phi(\lambda,\mu-\e_j;x,qy)
=
x
q^{\sum_{i \leq j} \lambda_i}
\frac{(1-y)(1-q^{\mu_j-\lambda_j})}{(x-y)(1-q^{\mu_j})}
\Phi(\lambda,\mu;x,y),
\end{align*}
where $j \geq 1$. Both of these are immediate from the definition \eqref{phi-def}.
\end{prop}

\begin{theorem}
\label{thmq}
Fix two integers $M,L \geq 1$, and four vectors $\A,\B,\C,\D$ such that $0 \leq |\A|, |\C| \leq M$ and $0 \leq |\B|, |\D| \leq L$. The $(L,M)$-fused weights are given by the following formula:
\begin{multline}
\label{wMN}
W_{L,M}(z;\A,\B;\C,\D)
=
\\
\bm{1}_{(\A + \B = \C + \D)}
\times
z^{\D-\B}
q^{\A L -\D M}
\times
\sum_{\Pp}
\Phi(\C-\Pp,\C+\D-\Pp;q^{L-M}z,q^{-M}z)
\Phi(\Pp,\B;q^{-L}/z,q^{-L})
\end{multline}
where the sum is taken over  vectors $\Pp = (P_1,\dots,P_N)$ such that $0 \leq P_i \leq \min(B_i,C_i)$ for all $1 \leq i \leq N$.
\end{theorem}

\subsection{Proof of Theorem \ref{thmq} for $L=1$}
 Because of the constraint $0 \leq |\B|, |\D| \leq L$, when $L=1$ the vectors $\B,\D$ must be of the form $\B = \e_i$ and $\D = \e_j$, for some $0 \leq i,j \leq N$.

In the case $\B = \e_0$ the summation in \eqref{wMN} reduces to a single term, namely $\Pp = \e_0$. After some simplifications, we find that
\begin{align*}
W_{1,M}(z;\A,\e_0;\C,\e_b)
&=
\bm{1}_{(\A = \C + \e_b)}
\Phi(\C,\C+\e_b;q^{1-M}z,q^{-M}z)
\Phi(\e_0,\e_0;q^{-1}/z,q^{-1})
z^{\theta_b} q^{\A-M\theta_b}
\\
&=
\bm{1}_{(\A = \C + \e_b)}
\frac{(q^{1-M}z;q)_{|\C|} (q^{-1};q)_{\theta_b}}{(q^{-M}z;q)_{|\C|+\theta_b}}
\binom{C_b + \theta_b}{C_b}_q
z^{\theta_b} q^{(\C_{[b+1,N]}-M+1)\theta_b},
\end{align*}
where we have introduced the notation
\begin{align*}
\theta_b = \left\{ \begin{array}{ll} 0, & b=0, \\ 1, & b \geq 1. \end{array} \right.
\end{align*}
Analysing separately the cases $b=0$ and $b \geq 1$ of the above equation, we see that
\begin{align*}
W_{1,M}(z;\A,\e_0;\C,\e_0)
&=
\bm{1}_{(\A = \C)}
\left( \frac{1-q^{\C-M}z}{1-q^{-M}z} \right)
=
\bm{1}_{(\A = \C)}
\left( \frac{q^{\A}-q^M/z}{1-q^M/z} \right)
\end{align*}
for $b=0$, and
\begin{align*}
W_{1,M}(z;\A,\e_0;\C,\e_b)
&=
\bm{1}_{(\A = \C+\e_b)}
\left(
\frac{q^{C_b+1}-1}{1-q^{-M}z}
\right)
z q^{\C_{[b+1,N]}-M}
=
\bm{1}_{(\A = \C+\e_b)}
\left(
\frac{1-q^{A_b}}{1-q^{M}/z}
\right)
q^{\A_{[b+1,N]}}
\end{align*}
for $b \geq 1$. The expressions obtained are in agreement with the weights $\mathcal{L}^{(M)}_{1/z}(\A,0;\C,0)$ and $\mathcal{L}^{(M)}_{1/z}(\A,0;\C,b)$ respectively, as is easily verified using \eqref{fus-wt}.

\smallskip

In the case $\B = \e_a$ for $1\leq a \leq N$, the summation over $\Pp$ is constrained to $\Pp \in \{\e_0,\e_a\}$, and we obtain
\begin{multline*}
W_{1,M}(z;\A,\e_a; \C,\e_b)
=
\bm{1}_{(\A+\e_a = \C+\e_b)}
\Big(
\Phi(\C,\C+\e_b; q^{1-M}z, q^{-M}z)
\Phi(\e_0,\e_a; q^{-1}/z, q^{-1})
+
\\
\Phi(\C-\e_a,\C-\e_a+\e_b; q^{1-M}z, q^{-M}z)
\Phi(\e_a,\e_a; q^{-1}/z, q^{-1})
\Big) z^{\theta_b-\theta_a} q^{\A-M\theta_b}.
\end{multline*}
This can then be divided into two distinct cases, namely (i) $a \not= b$ and (ii) $a=b$.

\smallskip

Using the definition \eqref{phi-def}, in case (i) one finds that
\begin{multline*}
W_{1,M}(z;\A,\e_a; \C,\e_b)
=
\bm{1}_{(\A+\e_a = \C+\e_b)}
\frac{(q^{-1};q)_{\theta_b}}{(1-q^{-1})}
\binom{C_b + \theta_b}{C_b}_q
\\
\left\{
\frac{(q^{1-M}z;q)_{|\C|}}{(q^{-M}z;q)_{|\C|+\theta_b}}
(1-z)
+
q^{-\theta_b \bm{1}_{b<a}}
\frac{(q^{1-M}z;q)_{|\C|-1}}{(q^{-M}z;q)_{|\C|-1+\theta_b}}
(qz-1)
\right\}
z^{\theta_b-1} q^{(\C_{[b+1,N]}-M+1)\theta_b-\theta_a}.
\end{multline*}
This can in turn be subdivided into three cases, (i1) $b=0$, (i2) $0\not=b < a$ and (i3) $a<b$. In case (i1), we calculate
\begin{multline*}
W_{1,M}(z;\A,\e_a; \C,\e_0)
=
\\
\bm{1}_{(\A+\e_a = \C)}
\frac{1}{(1-q^{-1})}
\left\{
\frac{1-q^{\C-M}z}{1-q^{-M}z}
(1-z)
+
\frac{1-q^{\C-1-M}z}{1-q^{-M}z}
(qz-1)
\right\}
z^{-1} q^{-1}
\\
=
\bm{1}_{(\A+\e_a = \C)}
\frac{1-q^{\C-M-1}}{1-q^{-M}z}
=
\bm{1}_{(\A+\e_a = \C)}
\left(\frac{1-q^{M-\A}}{1-q^M/z}\right)
z^{-1} q^{\A}.
\end{multline*}
In case (i2), we calculate
\begin{multline*}
W_{1,M}(z;\A,\e_a; \C,\e_b)
=
\bm{1}_{(\A+\e_a = \C+\e_b)}
\frac{1-q^{C_b+1}}{1-q}
\left\{
\frac{1-z}{1-q^{-M}z}
+
q^{-1}
\frac{qz-1}{1-q^{-M}z}
\right\}
q^{\C_{[b+1,N]-M}}
\\
=
\bm{1}_{(\A+\e_a = \C+\e_b)}
\left(
\frac{1-q^{C_b+1}}{1-q^{M}/z}
\right) z^{-1} q^{\C_{[b+1,N]}-1}
=
\bm{1}_{(\A+\e_a = \C+\e_b)}
\left(
\frac{1-q^{A_b}}{1-q^{M}/z}
\right) z^{-1} q^{\A_{[b+1,N]}}.
\end{multline*}
Both (i1) and (i2) agree with $\mathcal{L}^{(M)}_{1/z}(\A,a;\C,b)$ for $a>b$. In case (i3), we calculate
\begin{multline*}
W_{1,M}(z;\A,\e_a; \C,\e_b)
=
\bm{1}_{(\A+\e_a = \C+\e_b)}
\frac{1-q^{C_b+1}}{1-q}
\left\{
\frac{1-z}{1-q^{-M}z}
+
\frac{qz-1}{1-q^{-M}z}
\right\}
q^{\C_{[b+1,N]-M}}
\\
=
\bm{1}_{(\A+\e_a = \C+\e_b)}
\left(
\frac{1-q^{C_b+1}}{1-q^{M}/z}
\right) q^{\C_{[b+1,N]}}
=
\bm{1}_{(\A+\e_a = \C+\e_b)}
\left(
\frac{1-q^{A_b}}{1-q^{M}/z}
\right) q^{\A_{[b+1,N]}},
\end{multline*}
which agrees with $\mathcal{L}^{(M)}_{1/z}(\A,a;\C,b)$ for $a<b$.

\subsubsection{}

Similarly, in case (ii), we obtain
\begin{multline*}
W_{1,M}(z;\A,\e_a; \C,\e_a)
=
\bm{1}_{(\A = \C)}
\left\{
\binom{C_a + 1}{C_a}_q
\frac{(q^{1-M}z;q)_{|\C|}}{(q^{-M}z;q)_{|\C|+1}}
(1-z)
\right.
\\
+
\left.
\binom{C_a}{C_a-1}_q
\frac{(q^{1-M}z;q)_{|\C|-1}}{(q^{-M}z;q)_{|\C|}}
(qz-1)
\right\}
q^{\C_{[a+1,N]}-M}.
\end{multline*}
This we can readily simplify as
\begin{align*}
W_{1,M}(z;\A,\e_a; \C,\e_a)
&=
\frac{\bm{1}_{(\A = \C)}}{(1-q)(1-q^{-M}z)}
\Big\{
(1-q^{C_a+1})(1-z)
-
(1-q^{C_a})(1-qz)
\Big\}
q^{\C_{[a+1,N]}-M}
\\
&=
\bm{1}_{(\A = \C)}
\left(
\frac{q^{C_a}-z}{1-q^{-M}z}
\right)
q^{\C_{[a+1,N]}-M}
=
\bm{1}_{(\A = \C)}
\left(
\frac{1-q^{A_a}/z}{1-q^M/z}
\right)
q^{\A_{[a+1,N]}},
\end{align*}
in agreement with the form \eqref{fus-wt} of $\mathcal{L}^{(M)}_{1/z}(\A,a;\C,a)$.

\subsection{Proof of general case in Theorem \ref{thmq}}

So far we have proved that \eqref{wMN} holds for $L=1$. We now show that \eqref{wMN} obeys the recursion relation \eqref{bot-rec}, which is sufficient to show that \eqref{wMN} holds in general by induction on $L$. We can assume that
\begin{multline}
\label{W_MN-1}
W_{L-1,M}(qz;\A^{+-}_{ab},\B^{-}_a;\C,\D^{-}_b)
=
\delta_{\A + \B = \C + \D}
\times
z^{\theta_a-\theta_b}
q^{M \theta_b + L( \theta_a - \theta_b)-\C}
z^{\D - \B}
q^{\A L -\D M}
\\
\sum_{\Pp}
\Phi(\C-\Pp,\C+\D-\e_b-\Pp;q^{L-M}z,q^{-M+1} z)
\Phi(\Pp,\B-\e_a;q^{-L}/z,q^{-L+1})
\end{multline}
holds for some $L \geq 2$, where the summation is over all $\Pp$ such that $0 \leq P_i \leq \min(B_i,C_i)$ for all $i \not= a$, and $0 \leq P_a \leq \min(B_a-\theta_a,C_a)$.

Up to prefactors and slight modifications of the arguments of $\Phi$, this very closely resembles the form of $W_{L,M}(z;\A,\B;\C,\D)$. Substituting \eqref{W_MN-1} into \eqref{bot-rec2}, and requiring this to be equal to \eqref{wMN} leads to the identity stated in Lemma \ref{lem} below. The proof of that lemma completes the proof of Theorem \ref{thmq}.

\begin{lem}
\label{lem}
Let $\lambda = (\lambda_1,\dots,\lambda_N)$ and $\mu = (\mu_1,\dots,\mu_N)$ be two vectors of arbitrary variables, and write $|\lambda| = \sum_{i=1}^{N} \lambda_i$ and $|\mu| = \sum_{i=1}^{N} \mu_i$ for their sums. Let $x$ and $y$ be two further arbitrary parameters, and let $a$ be an integer satisfying $0 \leq a \leq N$. We define the following function:
\begin{align*}
\rho_a(\lambda,\mu; x,y)
=
\left\{
\begin{array}{ll}
\left( \dfrac{xq^{|\lambda|}-yq^{|\mu|}}{1-yq^{|\mu|}} \right)
\left( \dfrac{1-y}{x-y} \right),
&
a=0
\\ \\
x
q^{\sum_{i=1}^{a} \lambda_i}
\left( \dfrac{1-q^{\mu_a-\lambda_a}}{1-q^{\mu_a}} \right)
\left( \dfrac{1-y}{x-y} \right),
&
a \geq 1.
\end{array}
\right.
\end{align*}
Fix four vectors of variables $\A,\B,\C,\D$ such that $\A +\B = \C + \D$. Let $\Pp$ be a further vector of arbitrary variables. Fix two further parameters $M$ and $L$, and define $A_0 = M - \sum_{i=1}^{N} A_i$ and $B_0 = L - \sum_{i=1}^{N} B_i$. The following summation identity holds:
\begin{multline}
\label{miracle_sum}
\sum_{a=0}^{N}
\sum_{b=0}^{N}
\left(\frac{q^{-\C}}{1-q^L}\right)
z^{\theta_a-\theta_b}
q^{M \theta_b + L( \theta_a - \theta_b)}
(1-q^{B_a})
q^{B_{a+1}} \dots q^{B_N}
\mathcal{L}^{(M)}_{1/z}(\A,a;\A^{+-}_{ab},b)
\\
\times
\rho_a(\Pp,\B;q^{-L}/z,q^{-L})
\rho_b(\C-\Pp,\C+\D-\Pp;q^{L-M}z,q^{-M} z)
=
1.
\end{multline}
\end{lem}

\begin{proof}
The summation over $a$ and $b$ can be broken down into six cases. These are {\bf 1.} $a=b=0$, {\bf 2.} $ N \geq a > b = 0$, {\bf 3.} $N \geq a > b \geq 1$, {\bf 4.} $N \geq a = b \geq 1$, {\bf 5.} $1 \leq a < b \leq N$, {\bf 6.} $0 = a < b \leq N$. Let $S_{a,b}$ denote the summand of \eqref{miracle_sum}, \ie\ we write the left hand side as
\begin{align*}
\sum_{a=0}^{N}
\sum_{b=0}^{N}
S_{a,b}.
\end{align*}
One can write down the summand explicitly in each case. In what follows we use the notations $\B_a := \sum_{i=1}^{a}$ and $\B_{b,a} := \sum_{i=b}^{a} B_i$. The six cases read:

{\bf Case 1:}
\begin{align*}
S_{0,0} =
\frac{
(q^L - q^{\D})
(q^{\B}z - q^{\Pp})
(q^M - q^{\A} z)
}
{
(1-z)
(1-q^L)
(q^{M+\Pp} - q^{\C + \D}z)
}.
\end{align*}

{\bf Case 2:}
\begin{align*}
S_{a,0} =
z q^{\Pp_a - \B_a}
\frac{
(q^L - q^{\D})
(q^{M+\B} - q^{\A + \B})
(1- q^{B_a-P_a})
}
{
(1-z)
(1-q^L)
(q^{M+\Pp} - q^{\C + \D}z)
},
\quad
a \geq 1.
\end{align*}

{\bf Case 3:}
\begin{align*}
S_{a,b} =
q^{\Pp_{b,a}-\B_{b,a}+\D_{b,N}}
\frac{
(1-q^{D_b})
(q^{B_b - D_b} - q^{C_b})
(1-q^{B_a-P_a})
}
{
(1-z)
(1-q^L)
(q^{C_b+D_b} - q^{P_b})
},
\quad
a > b \geq 1.
\end{align*}

{\bf Case 4:}
\begin{align*}
S_{b,b} =
q^{P_b-B_b+\D_{b,N}}
\frac{
(1-q^{D_b})
(1-q^{B_b-P_b})
(q^{B_b-D_b} z -q^{C_b})
}
{
(1-z)
(1-q^L)
(q^{C_b+D_b} - q^{P_b})
},
\quad
b \geq 1.
\end{align*}

{\bf Case 5:}
\begin{align*}
S_{a,b} =
z q^{\B_{a,b} - \Pp_{a,b-1}+ \D_{b+1,N}}
\frac{
(1-q^{D_b})
(1-q^{C_b+D_b-B_b})
(q^{P_a-B_a}-1)
}
{
(1-z)
(1-q^L)
(q^{C_b+D_b} - q^{P_b})
},
\quad
1\leq a < b.
\end{align*}

{\bf Case 6:}
\begin{align*}
S_{0,b} =
q^{-\Pp_{b-1}-\B_{b,N} +\D_{b,N}}
\frac{
(1-q^{D_b})
(q^{B_b-D_b}-q^{C_b})
(q^{\B}z-q^{\Pp})
}
{
(1-z)
(1-q^L)
(q^{C_b+D_b} - q^{P_b})
},
\quad
b \geq 1.
\end{align*}
We can then perform some partial summations over the index $a$, keeping $b$ fixed. These sums telescope, and we easily find that (summing Case 1 and 2 terms)
\begin{align*}
S_{0,0}
+
\sum_{a=1}^{N} S_{a,0}
&=
\frac{
(q^L - q^{\D})
}
{
(1-z)
(1-q^L)
(q^{M+\Pp} - q^{\C + \D}z)
}
\Big(
(q^{\B}z - q^{\Pp})
(q^M - q^{\A} z)
+
z
(q^{M} - q^{\A})
(q^{\Pp}-q^{\B})
\Big)
\\
&=
\frac{
(q^L - q^{\D})(q^{\A+\B}z-q^{M+\Pp})
}
{
(1-q^L)
(q^{M+\Pp} - q^{\C + \D}z)
}
=
\frac{
(q^{\D}-q^L)
}
{
(1-q^L)
},
\end{align*}
and on the other hand (summing Case 3--6 terms)
\begin{multline*}
S_{0,b}
+
\sum_{a=1}^{b-1}
S_{a,b}
+
S_{b,b}
+
\sum_{a=b+1}^{N}
S_{a,b}
=
\\
\frac{
q^{\D_{b,N}}
(1-q^{D_b})
}
{
(1-z)
(1-q^L)
(q^{C_b+D_b} - q^{P_b})
}
\Big(
q^{-\Pp_{b-1}-\B_{b,N}}
(q^{B_b-D_b}-q^{C_b})
(q^{\B}z-q^{\Pp})
+
z
(q^{B_b-D_b}-q^{C_b})
(1-q^{\B_{b-1} - \Pp_{b-1}})
\\
+
(q^{P_b - B_b}-1)
(q^{B_b-D_b} z - q^{C_b})
+
q^{P_b-B_b}
(q^{B_b-D_b}-q^{C_b})
(q^{\Pp_{b+1,N}-\B_{b+1,N}}-1)
\Big)
=
\frac{q^{\D_{b+1,N}}(1-q^{D_b})}{(1-q^L)}.
\end{multline*}
Finally one can sum over all $0 \leq b \leq N$. The remaining terms then telescope to give $1$.
\end{proof}


\begin{thebibliography}{KMMO}

\bibitem[A1]{Ag} A.~Aggarwal, Convergence of the stochastic six-vertex model to the ASEP, Mathematical Physics, Analysis and Geometry 20, no.\ 3  (2017). arXiv:1607.08683

\bibitem[A2]{Ag_Duke} A.~Aggarwal, Current Fluctuations of the Stationary ASEP and Six-Vertex Model, Duke Mathematical Journal 167, no.\ 2 (2018), 269-384. arXiv:1608.04726

\bibitem[AB]{AB} A.~Aggarwal, A.~Borodin,
Phase Transitions in the ASEP and Stochastic Six-Vertex Model,  Annals of Probability
47, no.\ 2 (2019), 613-689, arXiv:1607.08684



\bibitem[ABB]{ABB} A.~Aggarwal, A.~Borodin, A.~Bufetov,
Stochasticization of Solutions to the Yang-Baxter Equation, Annales Henri Poincar\'{e}, 20, no.\ 8 (2019),  2495--2554, arXiv:1810.04299.

\bibitem[AKQ1]{AKQ1} T.~Alberts, K.~Khanin, J.~Quastel, The Continuum Directed Random Polymer, Journal of Statistical Physics, 154, no.\ 1-2 (2014),  305-326 arXiv:1202.4403


\bibitem[AKQ2]{AKQ2} T.~Alberts, K.~Khanin, J.~Quastel,  The intermediate disorder regime for directed polymers in dimension 1 + 1, Annals of Probability 42, no.\ 3 (2014), 1212-1256.  arXiv:1609.00298

\bibitem[ACQ]{ACQ} G.~Amir, I.~Corwin, J.~Quastel, Probability Distribution of the Free Energy of the
Continuum Directed Random Polymer in 1+1 dimensions. Communications on Pure and Applied Mathematics, 64, no.\ 4 (2011), 466-537. arXiv:1003.0443

\bibitem[AAR]{AAR} G.~E.~Andrews, R.~Askey, R.~Roy, Special functions. Cambridge University Press, 1999.

\bibitem[BC]{Bar-Cor} G.~Barraquand, I.~Corwin, Random-walk in Beta-distributed random environment,
Probability Theory and Related Fields, 167, no.\ 3 (2017),  1057-1116. arXiv:1503.04117

\bibitem[Bar]{Bary} Y.~Baryshnikov, GUEs and queues, Probability Theory and Related Fields,  119, no.\ 2 (2001),  256–274.

\bibitem[Baz]{Bazhanov} V.~V.~Bazhanov, Trigonometric solutions of triangle equations and classical lie algebras, Physics Letters B, 159, no.\ 4–6 (1985), 321--324.


\bibitem[BeGi]{BeGi} L.~Bertini, G.~Giacomin,
Stochastic Burgers and KPZ equations from particle system, Communications in Mathematical Physics,
183 (1997), 571--607.

\bibitem[BiY]{BianeYor} Ph.~Biane and M.~Yor. Sur la loi des temps locaux Browniens pris en un temps
exponentiel. In S\'{e}minaire de Probabilités XXII, pages 454-466. Springer, 1988.
Lecture Notes in Math. 1321.


\bibitem[B]{ANBorodin} A.~N.~Borodin. Brownian local time. Russian Math. Surveys, 44, no.\ 2 (1989), 1-51.

\bibitem[Bo]{Bor_rat} A.~Borodin, On a family of symmetric rational functions,  Advances in Mathematics, 306 (2017),
973–1018.


\bibitem[BB]{BB_sym} A.~Borodin, A.~Bufetov,
Color-position symmetry in interacting particle systems,
arXiv:1905.04692

%\bibitem[BC]{BigMac}  A.~Borodin, I.~Corwin, Macdonald processes, Probability Theory and Related Fields, 158 (2014) 225--400, arXiv:1111.4408

\bibitem[BCG]{BCG} A.~Borodin, I.~Corwin, V.~Gorin. Stochastic six-vertex model. Duke Mathematical Journal, 165, no.\ 3 (2016), 563-624. arXiv:1407.6729

\bibitem[BG]{BG_tele} A.~Borodin, V.~Gorin, A stochastic telegraph equation from the six-vertex model,  Annals of Probability 47, no.\ 6 (2019), 4137-4194. arXiv:1803.09137

\bibitem[BO]{BO} A.~Borodin, G.~Olshanski, The ASEP and determinantal point processes, Communications in Mathematical Physics 353 (2017), 853-903. arXiv:1608.01564

\bibitem[BP]{BP1} A.~Borodin, L.~Petrov, Higher spin six vertex model and symmetric rational
    functions.   Selecta Mathematica, 24, no.\ 2  (2018),  751–874.

\bibitem[BW]{BW} A.~Borodin, M.~Wheeler, Coloured stochastic vertex models and their spectral theory, arXiv:1808.01866

\bibitem[BM]{BM} G.~Bosnjak, V.~V.~Mangazeev, Construction of $R$-matrices for symmetric tensor representations related to $U_q(\hat {sl_n})$. J. Phys. A: Math. Theor. 49 (2016) 495204, arXiv:1607.07968

\bibitem[C]{Corwin} I.~Corwin. The Kardar-Parisi-Zhang equation and universality class,
Random Matrices: Theory and Applications, 1, no.\ 1 (2012). arXiv:1106.1596

\bibitem[CP]{CP} I.~Corwin, L.~Petrov, Stochastic Higher Spin Vertex Models on the Line, Communications
in Mathematical Physics, 343, no.\ 2 ( 2016), pp 651--700.

\bibitem[CQR]{CQR} I.~Corwin, J.~Quastel, D.~Remenik, Renormalization fixed point of the KPZ universality class, Journal of Statistical Physics, 160, no.\ 4 (2012),  815-834. arXiv:1103.3422

\bibitem[CS]{CS} I.~Corwin and H.~Shen. Open ASEP in the weakly asymmetric regime. arXiv:1610.04931, to appear in
Communications on Pure and Applied Mathematics, 2016.

\bibitem[CST]{CST} I.~Corwin, H.~Shen, and L.-C.~Tsai. ASEP(q, j) converges to the KPZ equation. Annales Institut Henri Poincar\'{e}:
Probability and Statistics 54, no.\ 2 (2018), 995-1012. arXiv:1602.01908

\bibitem[CSS]{CSS} I.~Corwin, T.~Seppalainen, H.~Shen, The strict-weak lattice polymer. Journal of Statistical Physics, 160, no.\ 4 (2015), 1027–1053. arXiv:1409.1794

\bibitem[CT]{CT} I.~Corwin, L.-C.~Tsai, KPZ equation limit of higher-spin exclusion processes, Annals of Probability 45, no.\ 3  (2017), 1771-1798. arXiv:1505.04158

\bibitem[CGST]{CGST}  I.~Corwin, P.~Ghosal, H.~Shen, L.-C.~Tsai, Stochastic PDE Limit of the Six Vertex Model, arXiv:1803.08120

\bibitem[DT]{DT} A.~Dembo and L.-C.~Tsai. Weakly asymmetric non-simple exclusion process and the Kardar–Parisi–Zhang
equation. Communications in Mathematical Physics 341, no.\ 1 (2016), 219–261. arXiv:1302.5760

\bibitem[DOV]{DOV} D.~Dauvergne, J.~Ortmann, B.~Virag, The directed landscape. arXiv:1812.00309

\bibitem[FRT]{FRT} L.~D.~Faddeev, N.~Y.~Reshetikhin, and L.~A.~Takhtajan. Quantization of Lie groups and Lie algebras. In
Algebraic analysis, pages 129-139. Elsevier, 1988.

\bibitem[GR]{GR} G.~Gasper, M.~Rahman, Basic hypergeometric series, Encyclopedia of Mathematics and Its Applications, vol.\ 96. Cambridge University Press, 2004

%\bibitem[J1]{J_Annals} Kurt Johansson, Discrete orthogonal polynomial ensembles and the Plancherel measure, Ann. of Math. (2) 153 (2001), no. 2, 259--296, arXiv:math/9906120


\bibitem[GS]{GwaSpohn} L.-H.~Gwa, H.~Spohn, Six-vertex model, roughened surfaces, and an asymmetric spin
Hamiltonian, Physical Review Letters, 68, no.\ 6 (1992), 725--728.


\bibitem[HQ]{HQ} M.~Hairer and J.~Quastel. A class of growth models rescaling to KPZ. Forum Math. Pi, 6:e3, 112, 2018. arXiv:1512.07845

\bibitem[HS]{HS} M.~Hairer and H.~Shen. A central limit theorem for the KPZ equation, Annals of Probability 45, no.\ 6B (2017), 4167–4221. arXiv:1507.01237

\bibitem[JR]{JR} K.~Johanson, M.~Rahman, Multi-time distribution in discrete polynuclear growth, arXiv:1906.01053.

\bibitem[Ji1]{Jimbo1} M.~Jimbo. A $q$-analogue of $U(\mathfrak{gl}(N + 1))$, Hecke algebra, and the Yang–Baxter equation. Letters in
Mathematical Physics, 11(3):247–252, 1986.

\bibitem[Ji2]{Jimbo2} M.~Jimbo. Quantum $R$ matrix for the generalized Toda system. Communications in Mathematical Physics,
102, no.\ 4 (1986), 537-547.

\bibitem[Ki]{K} A.~N.~Kirillov, Dilogarithm identities. Progress of Theoretical Physics Supplement No. 118, 1995.   arXiv:hep-th/9408113

\bibitem[KRS]{KRS}   P.~Kulish, N.~Reshetikhin, and E.~Sklyanin. Yang-Baxter equation and representation theory: I. Letters
in Mathematical Physics, 5, no.\ 5 (1981), 393–403.

\bibitem[KMMO]{KMMO} A.~Kuniba, V.~Mangazeev, S.~Maruyama, and M.~Okado. Stochastic $R$ matrix for $U_q(A^{(1)}_n )$. Nuclear
Physics B, 913, 248–277, 2016. arXiv:1604.08304

\bibitem[Kua]{Kuan}   J.~Kuan, An algebraic construction of duality functions for the stochastic $U_q(A_n^{(1)})$ vertex model and its degenerations. Communications in Mathematical Physics (2018) 359, no.\ 1 (2018), 121-187. arXiv:1701.04468

\bibitem[Kup]{Kuperberg} G.~Kuperberg, Random words, quantum statistics, central limits, random matrices, Methods Appl. Anal. 9 (2002), no. 1, 101-119, arXiv:math/9909104

\bibitem[N]{Nica} M.~Nica, Intermediate disorder limits for multi-layer
semi-discrete directed polymers, arXiv:1609.00298

\bibitem[OCO]{OCO} N.~O'Connell, J.~Ortmann, Tracy-Widom asymptotics for a random polymer model with gamma-distributed weights, Electronic Journal of Probability,
20 (2015), paper no.\ 25. arXiv:1408.5326

\bibitem[OCY]{OCY}  N.~O'Connell and M.~Yor,  Brownian analogues of Burke’s theorem. Stochastic
Processes and Applications 96 (2001), 285–304.

\bibitem[Pa]{Pauling} L.~Pauling, The Structure and Entropy of Ice and of Other Crystals with Some Randomness of Atomic Arrangement. Journal of the American Chemical Society. 57, no.\ 12  (1935),  2680--2684.

\bibitem[P]{Pitman} J.~Pitman, The distribution of local times of a Brownian bridge. Séminaire de probabilités (Strasbourg), tome 33 (1999), 388-394

\bibitem[Q1]{Quastel} J.~Quastel, Diffusion of color in the simple exclusion process. Communications on Pure and Applied Mathematics, 45, no.\ 6 (1992), 623-679

\bibitem[Q2]{Q_CDM} J.~Quastel, Introduction to KPZ, Current developments in mathematics 2011 (1).

\bibitem[QS]{QS} J.~Quastel and H.~Spohn. The one-dimensional KPZ equation and its universality class. Journal of Statistical Physics
160, no.\ 4 (2015), 965-984. arXiv:1503.06185

\bibitem[R]{Ray} D.~B.~Ray. Sojourn times of a diffusion process. Illinois Journal of Mathematics 7, 615-630. 1963.

\bibitem[W]{Williams}  D.~Williams. Path decomposition and continuity of local time for one dimensional
diffusions I. Procedings of London Mathematical Society s3-28, no.\ 4 (1974), 738-768.

\end{thebibliography}
\end{document}